\documentclass[11pt]{book}
\pdfoutput=1

\usepackage[english]{babel}
\usepackage[utf8]{inputenc}
\usepackage[T1]{fontenc}
\usepackage{lmodern}
\usepackage{geometry}
\usepackage[indentafter]{titlesec}
\usepackage{enumitem}
\usepackage{fancyhdr}
\usepackage{amsmath,amsfonts,amssymb}
\usepackage[amsmath,thmmarks]{ntheorem}
\usepackage{array}
\usepackage[all]{xy}

\renewcommand{\sc}{\scshape}
\newcommand{\md}{\mdseries}     
\renewcommand{\bf}{\bfseries}

\newcommand{\Dbb}{\mathbb{D}}
  
\newcommand{\Fbb}{\mathbb{F}}
\newcommand{\Gbb}{\mathbb{G}}

\newcommand{\Nbb}{\mathbb{N}}

\newcommand{\Qbb}{\mathbb{Q}}  

\newcommand{\Sbb}{\mathbb{S}}

\newcommand{\Wbb}{\mathbb{W}}

\newcommand{\Zbb}{\mathbb{Z}}

\newcommand{\Ccal}{\mathcal{C}}

\newcommand{\Fcal}{\mathcal{F}}
\newcommand{\Gcal}{\mathcal{G}}  

\newcommand{\Ical}{\mathcal{I}}

\newcommand{\Ocal}{\mathcal{O}}

\newcommand{\Scal}{\mathcal{S}}

\newcommand{\ra}{\rightarrow}       
\newcommand{\raa}{\longrightarrow}

\newcommand{\Lra}{\Leftrightarrow}

\newcommand{\Raa}{\Longrightarrow}

\newcommand{\mto}{\mapsto}          
\newcommand{\mtoo}{\longmapsto}

\newcommand{\q}{\quad}              
\newcommand{\qq}{\qquad}
      
\newcommand{\qqqq}{\qquad\qquad}

\renewcommand{\l}{\left}
\renewcommand{\r}{\right}
\newcommand{\lan}{\langle}          
\newcommand{\ran}{\rangle}

\renewcommand{\epsilon}{\varepsilon}    
\renewcommand{\phi}{\varphi}
\renewcommand{\subset}{\subseteq}       
\renewcommand{\supset}{\supseteq}
                
\newcommand{\x}{\times}
\newcommand{\ox}{\otimes}               
\newcommand{\op}{\oplus}
            
\newcommand{\iso}{\cong}
\renewcommand{\tilde}{\widetilde}       
\renewcommand{\bar}{\overline}       
\renewcommand{\mod}{\bmod}

\qedsymbol{\ensuremath{\Box}}

\theoremstyle{plain}
\theoremheaderfont{\normalfont\bfseries}
\theorembodyfont{\itshape}
\theoremseparator{.}
\theoremsymbol{}

\newtheorem{theorem}{\iflanguage{french}{Th\'eor\`eme}{Theorem}}[chapter]
\newtheorem{proposition}[theorem]{Proposition}
\newtheorem{corollary}[theorem]{\iflanguage{french}{Corollaire}{Corollary}}
\newtheorem{lemma}[theorem]{\iflanguage{french}{Lemme}{Lemma}}
\newtheorem*{nntheorem}{\iflanguage{french}{Th\'eor\`eme}{Theorem}}

\theorembodyfont{\normalfont}

\newtheorem*{definition}{\iflanguage{french}{D\'efinition}{Definition}}

\newtheorem{notation}[theorem]{Notation}

\newtheorem{remark}[theorem]{\iflanguage{french}{Remarque}{Remark}}

\newtheorem{example}[theorem]{\iflanguage{french}{Exemple}{Example}}

\theoremheaderfont{\itshape}
\theoremsymbol{\hfill\ensuremath{\Box}}

\newtheorem*{proof}{\iflanguage{french}{Preuve}{Proof}}

\xymatrixrowsep{8mm} \xymatrixcolsep{8mm}

\setitemize[1]{label=$\bullet$}
\setitemize[2]{label=$\circ$}
\setitemize[3]{label=$-$}
\setenumerate[1]{label=\arabic*)}
\setenumerate[2]{label=\alph*)}
\setenumerate[3]{label=\roman*)}
\setdescription{font=\bf}

\makeatletter 
\renewcommand{\@makechapterhead}[1]{
\vspace*{0mm}
\begin{flushleft}
	{\LARGE\bf \chaptername\ \thechapter :\\}
\end{flushleft}
\vspace{-2mm}
\begin{center}
	{\LARGE\bf #1}
\end{center} 
\nobreak\vspace{10mm}
} \makeatother

\makeatletter 
\renewcommand{\@makeschapterhead}[1]{
\markboth{#1}{#1}
\vspace*{0mm}
\begin{flushleft}
	{\LARGE\bf #1}
\end{flushleft} 
\nobreak\vspace{10mm}
} \makeatother

\titleformat{\section}{}%
{\normalfont\large\bfseries\thetitle.}{0.5em}%
{\normalfont\large\bfseries}
\titlespacing{\section}{0pt}{*5}{*2}
\titleformat{\subsection}{}%
{\normalfont\normalsize\bfseries\thetitle.}{0.5em}%
{\normalfont\normalsize\bfseries}
\titlespacing{\subsection}{0pt}{*4}{*2}

\begin{document} 


\begin{titlepage}
\begin{center}
    {\Large\sc
        Institut de Recherche Mathématique Avancée\\
        Université de Strasbourg\\
    }
    \vspace{40mm}
    {\bf\huge
        Doctoral Thesis\\
        \vspace{8mm}
        Finite subgroups of extended Morava stabilizer groups\\
        \vspace{10mm}
        \Large by\\ 
        \vspace{4mm}
        Cédric Bujard\\
    }
\end{center}

\vfill

\noindent Defended on June 4, 2012.\\  
Under the supervision of Prof. Hans-Werner Henn.

\vspace{4mm}

\end{titlepage}


\renewcommand{\contentsname}{\LARGE\bf \vspace*{0mm} 
    Table of Contents \vspace{-6mm}}
\tableofcontents 
\markboth{Table of Contents}{Table of Contents}


\chapter*{Introduction} 
\addcontentsline{toc}{chapter}{Introduction}
\markboth{Introduction}{Introduction}

\section*{The Morava stabilizer groups}

Let $n$ be a positive integer and $K$ a separably closed field of
characteristic $p > 0$. If $F$ is a formal group law of height $n$
defined over $K$, then the Dieudonné-Lubin theorem \ref{276} says that
the $K$-automorphism group $Aut_K(F)$ of $F$ can be identified with the
units in the maximal order $\Ocal_n$ of the central division algebra
$\Dbb_n = D(\Qbb_p, 1/n)$ of invariant $1/n$ over $\Qbb_p$. In the case
where $F=F_n$ is the Honda formal group law of height $n$, as given by
theorem \ref{274}, we have
\[
    Aut_K(F_n) \iso Aut_{\Fbb_{p^n}}(F_n).
\]
We define \label{314}
\[
    \Sbb_n := Aut_{\Fbb_{p^n}}(F_n) \iso \Ocal_n^\x
\]
to be the \emph{$n$-th (classical) Morava stabilizer group}.

More generally, we are interested in the category $\mathcal{FGL}_n$
whose objects are pairs $(F,k)$ for $k$ a perfect field of
characteristic $p$ and $F$ a formal group law of height $n$ defined over
$k$, and whose morphisms are given by pairs
\[
    (f, \phi): (F_1, k_1) \raa (F_2, k_2),
\]
where $\phi: k_1 \ra k_2$ is a field homomorphism and $f: \phi_\ast F_1
\ra F_2$ is an isomorphism from the formal group law given by applying
$\phi$ on the coefficients of $F_1$. If $(f, \phi)$ is an endomorphism
of $(F, k)$, then $\phi$ is an automorphism of $k$ and $\phi \in
Gal(k/\Fbb_p)$. We let
\[
    Aut_{\mathcal{FGL}_n}(F, k)
    = \{ (f, \phi): (F, k) \ra (F, k)\ |\ \phi \in Gal(k/\Fbb_p)
    \text{ and } f: \phi_\ast F \iso F \}
\]
denote the group of automorphisms of $(F, k)$ in $\mathcal{FGL}_n$. If
$F$ is already defined over $\Fbb_p$, the Frobenius automorphism $X^p
\in End_{K}(F)$ defines an element $\xi_F \in \Ocal_n$. Then proposition
\ref{280} says that $End_K(F) = End_{\Fbb_{p^n}}(F)$ if and only if the
minimal polynomial of $\xi_F$ over $\Zbb_p$ is $\xi_F^n-up$ with $u \in
\Zbb_p^\x$. For such an $F$, we define \label{316}
\[
    \Gbb_n(u) := Aut_{\mathcal{FGL}_n}(F, \Fbb_{p^n})
\]
to be the \emph{$n$-th extended Morava stabilizer group} associated to
$u$. We often note $\Gbb_n = \Gbb_n(1)$.

Here $\phi_\ast F = F$ for any $\phi \in Gal(\Fbb_{p^n}/\Fbb_p)$. The
group $\Gbb_n(u)$ contains $\Sbb_n$ as the subgroup of elements of the
form $(f, id_{\Fbb_{p^n}})$, and there is an extension
\[
    1 \raa \Sbb_n \raa \Gbb_n(u) \raa Gal(\Fbb_{p^n}/\Fbb_p) \raa 1
\]
where an element $f \in \Sbb_n$ is mapped to the pair $(f,
id_{\Fbb_{p^n}})$ and where the image of a pair $(f, \phi) \in
\Gbb_n(u)$ in the Galois group is the automorphism $\phi$ of
$\Fbb_{p^n}$. Moreover, the Frobenius automorphism $\sigma \in
Gal(\Fbb_{p^n}/\Fbb_p) \iso \Zbb/n$ splits as the pair $(id_F, \sigma)$
in $\Gbb_n(u)$, and we get
\[
    \Gbb_n(u) \iso \Sbb_n \rtimes_F Gal(\Fbb_{p^n}/\Fbb_p),
\]
where the action on $\Sbb_n$ is induced by conjugation by $\xi_F$. In
terms of division algebras (see appendix \ref{273}), this extension
translates into a split exact sequence
\[
    1 \raa \Ocal_n^\x \raa \Dbb_n^\x/\lan \xi_F^n 
    \ran \raa \Zbb/n \raa 1,
\]
so that
\[
    \Gbb_n(u) \iso \Dbb_n^\x/\lan pu \ran.
\]

In the text we address the problem of classifying the finite subgroups
of $\Gbb_n(u)$ up to conjugation. In particular, we give necessary and
sufficient conditions on $n$, $p$ and $u$ for the existence in
$\Gbb_n(u)$ of extensions of the form
\[
    1 \raa G \raa F \raa \Zbb/n \raa 1
\]
with $G$ maximal finite in $\Sbb_n$, and if such extensions exist, we
establish their classification as finite subgroups of $\Gbb_n(u)$ up to
conjugation.

\section*{Motivation} 

Given a prime $p$ and for $K(n)$ the $n$-th Morava $K$-theory at $p$, the
stable homotopy category of $p$-local spectra can be analysed from
the category of $K(n)$-local spectra in the sense of \cite{henn2}
section 1.1. In particular, letting $L_n = L_{K(0) \vee\ldots\vee K(n)}$
be the localization functor with respect to $K(0) \vee\ldots\vee K(n)$,
there is a tower of localization functors
\[
    \ldots \raa L_n \raa L_{n-1} \raa \ldots \raa L_0
\]
together with natural maps $X \ra L_nX$, such that for every $p$-local
finite spectrum $X$ the natural map $X \ra holim L_nX$ is a weak
equivalence. Furthermore, the maps $L_nX \ra L_{n-1}X$ fit into a
natural commutative homotopy pullback square
\[
    \xymatrix{
        L_nX \ar[r] \ar[d] & L_{K(n)}X \ar[d] \\
        L_{n-1}X \ar[r] & L_{n-1}L_{K(n)}X.
    }
\]
In this way, the Morava $K$-theory localizations $L_{K(n)}X$ form the
basic building blocks for the homotopy type of a $p$-local finite
spectrum $X$, and of course, the localization of the sphere
$L_{K(n)}S^0$ plays a central role in this approach.  

The spectrum $L_{K(n)}S^0$ can be identified with the homotopy fixed
point spectrum $E_n^{h\Gbb_n}$ of the $n$-th Lubin-Tate spectrum $E_n$,
and the Adams-Novikov spectral sequence for $L_{K(n)}S^0$ can be
identified with the spectral sequence
\[
    E_2^{s,t}
    = H^s(\Gbb_n, (E_n)_t)\
    \Raa\ \pi_{t-s}L_{K(n)}S^0.
\]
Here the ring $(E_n)_0$ is isomorphic to the universal deformation ring
$E(F,\Fbb_{p^n})$ (in the sense of Lubin and Tate) associated to a
formal group law $F$ of height $n$ over $\Fbb_{p^n}$, and $(E_n)_\ast$
is a graded version of $E(F,\Fbb_{p^n})$. The functor
\[
    E(\_, \_): \mathcal{FGL}_n \raa Rings_{cl}
\]
to the category of complete local rings defines the action of
$\Gbb_n(u)$ on the universal ring $E(F, \Fbb_{p^n})$, which in turn
induces an action on $(E_n)_\ast$.

There is good hope that $L_{K(n)}S^0$ can be written as the inverse
limit of a tower of fibrations whose successive fibers are of the form
$E^{hF}_n$ for $F$ a finite subgroup of $\Gbb_n(u)$. This is at least
true in the case $n=2$, $p=3$ and $u=1$, which is the object of
\cite{goerss-henn}. In \cite{henn2} the case $n=p-1$, $p>2$ and $u=1$ is
investigated. Moreover, the importance of the subgroups of $\Gbb_2(-1)$
for $p=3$ is exemplified in \cite{behrens}. As shown in the present
text, the choice of $u$ plays an important role in the determination of
the finite subgroups of $\Gbb_n(u)$. 

For example, when $n=2$ and $p=3$ theorem \ref{261} shows that the
maximal finite subgroups of $\Gbb_n(u)$ are represented up to
conjugation by $SD_{16}$, the semidihedral group of order $16$, and by a
semi-direct product of the cyclic group of order $3$ with either the
quaternion group $Q_8$ if $u \equiv 1 \mod 3$ or the dihedral group
$D_8$ of order $8$ if $u \equiv -1 \mod 3$.

Another example is given by theorem \ref{264} in the case $n=2$ and
$p=2$: the maximal finite conjugacy classes are given by two or four
classes depending on $u$. When $u \equiv 1 \mod 8$, there are two of
them given by a metacyclic group of order $12$ and by
\[
    \begin{cases}
        O_{48} & \text{if } u \equiv 1 \mod 8,\\
        T_{24} \rtimes C_2 & \text{if } u \equiv -1 \mod 8,
    \end{cases}
\]
for $O_{48}$ the binary octahedral group of order $48$, $C_2$ the cyclic
group of order $2$ and $T_{24}$ the binary tetrahedral group of order
$24$. On the other hand when $u \not\equiv 1 \mod 8$, there are four of
them given by $T_{24}$, by two distinct metacyclic groups of order $12$,
and by
\[
    \begin{cases}
        D_8 & \text{if } u \equiv 3 \mod 8,\\
        Q_8 & \text{if } u \equiv -3 \mod 8.
    \end{cases}
\]

The group $\Gbb_2(-1)$ is the Morava stabilizer group associated
to the formal group law of a supersingular elliptic curve, while in
general $\Gbb_n = \Gbb_n(1)$ is the one associated to the Honda formal
group law of height $n$.

\section*{Overview}

In the first chapter of the text, we establish a classification up to
conjugation of the maximal finite subgroups of $\Sbb_n$ for a prime $p$
and a positive integer $n$. When $n$ is not a multiple of $p-1$ the
situation remains simple as no non-trivial finite $p$-subgroup exist. In
this case, all finite subgroups are subgroups in the unique conjugacy
class isomorphic to
\[
    \begin{cases}
    C_{p^n-1}
    & \text{ if } p>2,\\
    C_{2(p^n-1)}
    & \text{ if } p=2,
    \end{cases}
\]
where $C_l$ denotes the cyclic group of order $l$. Otherwise, $n =
(p-1)p^{k-1}m$ with $m$ prime to $p$. For $\alpha \leq k$ and Euler's
totient function $\phi$, we let $n_\alpha = \frac{n}{\phi(p^\alpha)}$
and we obtain:

\begin{nntheorem}
If $p>2$ and $n = (p-1)p^{k-1}m$ with $m$ prime to $p$, the group
$\Sbb_n$ has exactly $k+1$ conjugacy classes of maximal finite subgroups
represented by
\[
    G_0 = C_{p^n-1} 
    \qq \text{and} \qq 
    G_\alpha = C_{p^\alpha} \rtimes C_{(p^{n_\alpha}-1)(p-1)} \qq 
    \text{for} \q 1 \leq \alpha \leq k.
\] 
\end{nntheorem}

\begin{nntheorem}
Let $p=2$ and $n = 2^{k-1}m$ with $m$ odd. The group $\Sbb_n$,
respectively $\Dbb_n^\x$, has exactly $k$ maximal conjugacy
classes of finite subgroups. If $k \neq 2$, they are represented by 
\[
    G_\alpha 
    = C_{2^\alpha(2^{n_\alpha}-1)} \qq 
    \text{for} \q 1 \leq \alpha \leq k.
\]
If $k=2$, they are represented by $G_\alpha$ for $\alpha \neq 2$ and by
the unique maximal nonabelian conjugacy class
\[
    Q_8 \rtimes C_{3(2^m-1)} \iso T_{24} \x C_{2^m-1},
\]
the latter containing $G_2$ as a subclass.
\end{nntheorem}

The classification of the isomorphism classes of the finite subgroups of
$\Sbb_n$ has already been found by Hewett in \cite{hewett}; it is based
on a previous classification made by Amitsur in \cite{amitsur}. Our
approach is different: it has the advantage of being more direct,
exploiting the structure of $\Dbb_n^\x$ in terms of Witt vectors, and
lays the foundations for our study of the extended groups $\Gbb_n(u)$. A
further attempt by Hewett to extend his classification from isomorphism
classes to conjugacy classes can be found in \cite{hewett2}, but the
results turn out to be false (see remarks \ref{171} and \ref{172}). In
example \ref{049}, we provide an explicit family of counter examples in
the case $p>2$.

\medskip

In chapter 2, we present a theoretical framework for the classification
of the finite subgroups of $\Gbb_n(u) \iso \Dbb_n^\x/\lan pu \ran$. Most
of the work is done in $\Dbb_n^\x$ via a bijection (see proposition
\ref{074}) between the set of (conjugacy classes of) finite subgroups of
$\Gbb_n(u)$ and the set of (conjugacy classes of) subgroups of
$\Dbb_n^\x$ containing $\lan pu \ran$ as a subgroup of finite index. For
a finite subgroup $F$ of $\Gbb_n(u)$ for which $F \cap \Sbb_n$ has an
abelian $p$-Sylow subgroup (the remaining case of a quaternionic
$p$-Sylow is quite specific and is treated in chapter 4), we consider
its correspondent $\tilde{F}$ in $\Dbb_n^\x$ via the above bijection.
This group fits into a chain of successive extensions
\[
    \tilde{F_0} \subset \tilde{F_1} \subset \tilde{F_2} 
    \subset \tilde{F_3} = \tilde{F},
\]
where $F_0 = \lan F \cap S_n, Z_{p'}(F\cap\Sbb_n) \ran$ is cyclic for
$S_n$ the $p$-Sylow subgroup of $\Sbb_n$ and $Z_{p'}(F\cap\Sbb_n)$ the
$p'$-part of the center of $F \cap \Sbb_n$, and where
\begin{align*}
    \tilde{F_0} 
    &= F_0 \x \lan pu \ran,
    & \tilde{F_2} 
    &= \tilde{F} \cap C_{\Dbb_n^\x}(F_0) 
    = C_{\tilde{F}}(F_0), \\
    \tilde{F_1} 
    &= \tilde{F} \cap \Qbb_p(F_0)^\x,
    & \tilde{F_3} 
    &= \tilde{F} \cap N_{\Dbb_n^\x}(F_0) 
    = N_{\tilde{F}}(F_0).
\end{align*}
Referring to the above classification of the finite subgroups of
$\Sbb_n$, we note that $F_0$ is a subgroup of a cyclic group of order
$p^\alpha(p^{n_\alpha}-1)$ for an $\alpha \leq k$, and that the whole
(nonabelian) groups of type $G_\alpha$ when $p>2$ can only be recovered
in the last stage of the chain of extensions.  We then provide
cohomological criteria (see theorem \ref{180}, \ref{185}, \ref{190} and
\ref{193}) for the existence and uniqueness up to conjugation of each of
these successive group extensions. We are mostly interested in the cases
where each successive $\tilde{F_i}$ is maximal, that is, $F_0$ is a
maximal abelian finite subgroup of $\Sbb_n$, and for $1 \leq i \leq 3$,
each $\tilde{F_i}$ is a maximal subgroup of the respective group
$\Qbb_p(F_0)^\x$, $C_{\Dbb_n^\x}(F_0)$, $N_{\Dbb_n^\x}(F_0)$ containing
$\tilde{F_0}$ as a subgroup of finite index.

\medskip

In chapter 3, we treat the abelian cases which are covered up to the
second extension type $\tilde{F_2}$. Given $F_0$, we let
$\tilde{\Fcal}_u(\Qbb_p(F_0), \tilde{F_0}, r_1)$ denote the set of all
$\tilde{F_1}$'s which give rise to a finite subgroup $F_1$ of
$\Gbb_n(u)$ extending $F_0$ by a cyclic group of order $r_1$. Then:

\begin{nntheorem}
If $F_0$ is a maximal abelian finite subgroup of $\Sbb_n$, then
$\tilde{\Fcal}_u(\Qbb_p(F_0), \tilde{F_0}, r_1)$ is non-empty if and
only if
\[
    r_1 
    \q \text{divides} \q
    \begin{cases}
        1 & \text{if } p>2 \text{ with } \zeta_p \not\in F_0,\\
        p-1 & \text{if } p>2 \text{ with } \zeta_p \in F_0,\\
        1 & \text{if } p=2 
        \text{ with } \zeta_3 \not\in F_0 
        \text{ and } u \not\equiv \pm 1 \mod 8,
        \text{ or with } \zeta_4 \not\in F_0,\\
        2 & \text{if } p=2 
        \text{ with } \zeta_4 \in F_0
        \text{ and either } u \equiv \pm 1 \mod 8
        \text{ or } \zeta_{3} \in F_0.
    \end{cases}
\]
\end{nntheorem}

\noindent Furthermore, given $F_0 \subset F_1$, we let
$\tilde{\Fcal}_u(C_{\Dbb_n^\x}(F_0), \tilde{F_1}, r_2)$ denote the set
of all $\tilde{F_2}$'s which give rise to a finite subgroup $F_2$ of
$\Gbb_n(u)$ extending $F_1$ by a group of order $r_2$. Then:

\begin{nntheorem}
If $r_1$ is maximal such that $\tilde{\Fcal}_u(\Qbb_p(F_0), \tilde{F_0},
r_1) \neq \emptyset$ and if $\tilde{F_1}$ belongs to this set, then
$\tilde{\Fcal}_u(C_{\Dbb_n^\x}(F_0), \tilde{F_1}, r_2)$ is non-empty if
and only if $r_2$ divides $\frac{n}{[\Qbb_p(F_0):\Qbb_p]}$.
\end{nntheorem}

\noindent In the particular case where $F_0$ is a maximal abelian 
finite subgroup of $\Sbb_n$, we have $\tilde{F_1}=\tilde{F_2}$. 

\medskip

In chapter 4, we treat (nonabelian) finite extensions of $\tilde{F_2}$
in the case where $\Qbb_p(\tilde{F_2})$ is a maximal subfield in
$\Dbb_n$; any such field is of degree $n$ over $\Qbb_p$. We provide
necessary and sufficient conditions on $n$, $p$ and $u$ for the
existence of $\tilde{F}$'s such that $|\tilde{F}/\tilde{F_0}| = n$ and
$\tilde{F_1} = \tilde{F_2}$:

\begin{nntheorem}
Let $p > 2$, $n = (p-1)p^{k-1}m$ with $m$ prime to $p$, $u \in
\Zbb_p^\x$, $F_0 = C_{p^\alpha} \x C_{p^{n_\alpha}-1}$ be a maximal
abelian finite subgroup in $\Sbb_n$, $G = Gal(\Qbb_p(F_0)/\Qbb_p)$,
$G_{p'}$ be the $p'$-part of $G$, and let $\tilde{F_1} = \lan x_1 \ran
\x F_0 \subset \Qbb_p(F_0)^\x$ be maximal as a subgroup of
$\Qbb_p(F_0)^\x$ having $\tilde{F_0}$ as subgroup of finite index.
\begin{enumerate}
    \item For any $0 \leq \alpha \leq k$, there is an extension of
    $\tilde{F_1}$ by $G_{p'}$; this extension is unique up to
    conjugation.

    \item If $\alpha \leq 1$, there is an extension of $\tilde{F_1}$ by
    $G$; this maximal extension is unique up to conjugation.

    \item If $\alpha \geq 2$, there is an extension of $\tilde{F_1}$ by
    $G$ if and only if    
    \[
        \alpha = k
        \qq \text{and} \qq
        u \not\in \mu(\Zbb_p^\x) 
        \x \{ x \in \Zbb_p^\x\ |\ x \equiv 1 \mod (p^2) \},
    \]
    in which case this maximal extension is unique up to conjugation.
\end{enumerate}
\end{nntheorem}

\begin{nntheorem}
Let $p=2$, $n = 2^{k-1}m$ with $m$ odd, $u \in \Zbb_2^\x$, $F_0 =
C_{2^\alpha} \x C_{2^{n_\alpha}-1}$ be a maximal abelian finite
subgroup of $\Sbb_n$, $G = Gal(\Qbb_2(F_0)/\Qbb_2)$, $G_{2'}$ be the odd
part of $G$, and let $\tilde{F_1} = \lan x_1 \ran \x F_0 \subset
\Qbb_2(F_0)^\x$ be maximal as a subgroup of $\Qbb_2(F_0)^\x$ having
$\tilde{F_0}$ as subgroup of finite index. 
\begin{enumerate}
    \item For any $1 \leq \alpha \leq k$, there is an extension of
    $\tilde{F_1}$ by $G_{2'}$; this extension is unique up to
    conjugation.
    
    \item If $\alpha =1$, there is an extension of $\tilde{F_1}$ by $G$;
    the number of such extensions up to conjugation is
    \[
        \begin{cases}
            1 & \text{if } n \text{ is odd},\\
            2 & \text{if } n \text{ is even}.
        \end{cases}
    \]

    \item If $\alpha = 2$, there is an extension of $\tilde{F_1}$ by $G$
    if and only if $k=2$; the number of such extensions up to
    conjugation is
    \[
        \begin{cases}
            1 & \text{if } u \equiv \pm 1 \mod 8,\\
            2 & \text{if } u \not\equiv \pm 1 \mod 8.
        \end{cases}
    \]

    \item If $\alpha \geq 3$, there is no extension of $\tilde{F_1}$ by
    $G$.
\end{enumerate}
\end{nntheorem}

\noindent We then treat the specific remaining case where $F \cap
\Sbb_n$ has a quaternionic $p$-Sylow subgroup; this only occurs when
$p=2$ and $n \equiv 2 \mod 4$.

\begin{nntheorem}
Let $p=2$, $n=2m$ with $m$ odd, and $u \in \Zbb_2^\x$. A subgroup $G$
isomorphic to $T_{24} \x C_{2^m-1}$ in $\Sbb_n$ extends to a maximal
finite subgroup $F$ of order $n|G|=48m(2^m-1)$ in $\Gbb_n(u)$ if and
only if $u \equiv \pm 1 \mod 8$; this extension is unique up to
conjugation.
\end{nntheorem}

\noindent We end the chapter by explicitly analysing the case $n=2$,
where we obtain:

\begin{nntheorem}
Let $n=2$, $p=3$ and $u \in \Zbb_p^\x$. The conjugacy classes of maximal
finite subgroups of $\Gbb_2(u)$ are represented by
\[
    SD_{16}
    \qq \text{and} \qq
    \begin{cases}
        C_3 \rtimes Q_8 & \text{if } u \equiv 1 \mod 3,\\
        C_3 \rtimes D_8 & \text{if } u \equiv -1 \mod 3.
    \end{cases}
\]
\end{nntheorem}

\begin{nntheorem}
Let $n=2$, $p=2$ and $u \in \Zbb_2^\x$. The conjugacy classes of maximal
finite subgroups of $\Gbb_2(u)$ are represented by
\[
    \begin{cases}
        C_{6} \rtimes C_2,\ O_{48} 
          & \text{if } u \equiv 1 \mod 8,\\
        C_{3} \rtimes C_4,\ T_{24} \rtimes C_2 
          & \text{if } u \equiv -1 \mod 8,\\
        C_{3} \rtimes C_4,\ C_{6} \rtimes C_2,\ 
        D_8  \text{ and } T_{24}
          & \text{if } u \equiv 3 \mod 8,\\
        C_{3} \rtimes C_4,\ C_{6} \rtimes C_2,\ 
        Q_8 \text{ and } T_{24}
          & \text{if } u \equiv -3 \mod 8.
    \end{cases}
\]
\end{nntheorem}


\chapter{Finite subgroups of $\Sbb_n$} 
\label{036}

From now on, we will always consider $p$ a prime, $n$ a strictly
positive integer, and 
\[
    \Dbb_n := D(\Qbb_p, 1/n)
\]
the central division algebra of invariant $1/n$ over $\Qbb_p$. The
reader may refer to appendix \ref{333} and \ref{033} for the essential
background on division algebras. We identify $\Sbb_n$ as the group of
units $\Ocal_n^\x$ of the maximal order $\Ocal_n$ of $\Dbb_n$.

\vspace{5ex}

\section{The structure of $\Dbb_n$ and its finite subgroups} 
\label{037}

The structure of $\Dbb_n$ can be explicitly given by the following
construction; see appendix \ref{033} or appendix 2 of \cite{ravenel} for
more details.  Let $\Wbb_n = \Wbb(\Fbb_{p^n})$ be the ring of Witt
vectors on the finite field $\Fbb_{p^n}$ with $p^n$ elements.  Here
$\Wbb_n$ can be identified with the ring $\Zbb_p[\zeta_{p^n-1}]$ of
integers of the unramified extension of degree $n$ over $\Qbb_p$. It is
a complete local ring with maximal ideal $(p)$ and residue field
$\Fbb_{p^n}$ whose elements are written uniquely as
\[
    w = 
    \sum_{i\geq 0} w_ip^i 
    \qq \text{with} \q 
    w_i^{p^n} = w_i.
\]
The Frobenius automorphism $x \mapsto x^p \in Gal(\Fbb_{p^n}/\Fbb_p)$
can be extended to an automorphism $\sigma: w \mapsto w^\sigma$ of
$\Wbb_n$ generating $Gal(\Wbb_n/\Zbb_p)$ by setting
\[
    w^\sigma 
    = \sum_{i\geq 0} w_i^pp^i 
    \qq \text{for each} \q 
    w = \sum_{i\geq 0} w_ip^i\ \in\ \Wbb_n.
\]
We then add to $\Wbb_n$ a non-commutative element $S$ \label{303}
satisfying $S^n=p$ and $Sw=w^\sigma S$ for all $w \in \Wbb_n$; the
non-commutative ring we obtain in this way can be identified with 
\[
    \Ocal_n \iso \Wbb_n \lan S \ran / (S^n=p, Sw=w^\sigma S),
\]
and
\[
    \Dbb_n \iso \Ocal_n \ox_{\Zbb_p} \Qbb_p.
\] 

\medskip

The valuation map $v_{\Qbb_p}: \Qbb_p^\x \ra \Zbb$ satisfying $v(p)=1$
extends uniquely to a valuation $v = v_{\Dbb_n}$ on $\Dbb_n$, with value
group 
\[
    v(\Dbb_n^\x) = \frac{1}{n}\Zbb,
\]
in such a way that
\[
    v(S) = \frac{1}{n} 
    \qq \text{and} \qq 
    \Ocal_n = \{x \in \Dbb_n\ |\ v(x) \geq 0\}.
\]
Because $v(x^{-1}) = -v(x)$, we have
\[
    \Ocal_n^\x = \{x \in \Dbb_n\ |\ v(x) = 0\}.
\]

\begin{proposition} \label{016}
A finite subgroup of $\Dbb_n^\x$ is a subgroup of $\Ocal_n^\x$.
\end{proposition}

\begin{proof}
An element $\zeta \in \Dbb_n^\x$ of finite order $i \geq 1$ satisfies
\[
    0 = v(1) = iv(\zeta),
\]
and it follows that $v(\zeta) = 0$.
\end{proof}

As we will now see, the structure of $\Dbb_n$ given above greatly
reduces the possibilities of what form a finite subgroup of $\Ocal_n^\x$
can have. 

The element $S \in \Dbb_n^\x$ generates a two-sided maximal ideal
$\mathfrak{m}$ of $\Ocal_n$ with residue field $\Ocal_n/\mathfrak{m}
\iso \Fbb_{p^n}$. This maximal ideal satisfies
\[
    \mathfrak{m} = \{x \in \Dbb_n\ |\ v(x) > 0\}.
\]
The kernel of the group epimorphism $\Ocal_n^\x \ra \Fbb_{p^n}^\x$ which
results from this quotient is denoted $S_n$ \label{315}. We thus have a
group extension
\[
    1 \raa S_n \raa \Ocal_n^\x \raa \Fbb_{p^n}^\x \raa 1.
\]
The groups $\Ocal_n^\x$ and $S_n$ have natural profinite structures
induced by the filtration of subgroups 
\[
    S_n = U_1 \supset U_2 \supset U_3 \supset \ldots
\]
given by
\begin{align*}
    U_i := U_i(\Wbb_n^\x) 
    &= \{ x \in S_n\ |\ v(x-1) \geq \frac{i}{n} \}\\[1ex] 
    &= \{ x \in S_n\ |\ x \equiv 1 \mod S^i\},
    \qq \text{for } i \geq 1.
\end{align*}
The intersection of these groups is trivial and $S_n = \lim_i S_n/U_i$.
We also have canonical isomorphisms
\[
    U_i/U_{i+1} \iso \Fbb_{p^n}
    \qq \text{given by} \qq 
    1+aS^i \mapsto \bar{a} 
\]
for $a \in \Ocal_n$ and $\bar{a}$ the residue class of $a$ in
$\Ocal_n/\mathfrak{m} = \Fbb_{p^n}$. In particular, all quotients
$S_n/U_i$ are finite $p$-groups and $S_n$ is a profinite $p$-subgroup of
the profinite group $\Ocal_n^\x$. By uniqueness of the maximal ideal
$\mathfrak{m}$, we know that $S_n$ is the unique $p$-Sylow subgroup of
$\Ocal_n^\x$. Consequently:

\begin{proposition} \label{017}
All $p$-subgroups of $\Ocal_n^\x$, and only those, are subgroups of
$S_n$. \qed
\end{proposition}

Throughout the text we let $\phi$ \label{337} denote Euler's totient
function, which for each positive integer $i$ associates the number
$\phi(i)$ of integers $1 \leq j \leq i$ for which $(i; j) = 1$. 

\begin{proposition} \label{021}
The group $S_n$, respectively $\Ocal_n^\x$, has elements of order $p^k$
for $k \geq 1$ if and only if $\phi(p^k)=(p-1)p^{k-1}$ divides $n$.
\end{proposition}

\begin{proof}
This is a straightforward consequence of the embedding theorem
\ref{013}, together with proposition \ref{342} which states that
\[
    [\Qbb_p(\zeta_{p^k}):\Qbb_p] 
    = \phi(p^k) 
    = (p-1)p^{k-1},
\]
for $\zeta_{p^k}$ a primitive $p^k$-th root of unity.
\end{proof}

\begin{proposition} \label{018}
Every abelian finite subgroup of $\Dbb_n^\x$ is cyclic.
\end{proposition}

\begin{proof}
If $G$ is a finite multiplicative abelian subgroup of a division algebra
of type $\Dbb_n$, then it lies within the local commutative field $F =
\Qbb_p(G)$ in $\Dbb_n$ and is a subgroup of $F^\x$. Because $G$ is
finite, proposition \ref{341} implies that $G$ is a subgroup of the
cyclic group $\mu(F)$.
\end{proof}

In the text, we are lead to use group cohomology $H^\ast(G, M)$
extensively for some group $G$ and $G$-module $M$. Most often, we will
exploit the tools of low dimensional cohomology to study group
extensions. A good introduction to the subject is provided in
\cite{brown} chapter IV.  In particular, we will invoke the following
classic results; see \cite{brown} section IV.4 for proposition \ref{019}
(with exercise 4), and see \cite{brown} chapter IV corollary 3.13 and
the following remark for proposition \ref{020}.

\begin{proposition} \label{019}
If $G$ is a finite $p$-group whose abelian finite subgroups are cyclic,
then $G$ is either cyclic or a generalized quaternion group
\[
    G 
    \iso Q_{2^k} 
    = \lan x,y\ |\ x^{2^{k-1}}=1,\ 
    yxy^{-1}=x^{-1},\ x^{2^{k-2}}=y^2 \ran,
\]
this last possibility being valid only when $p=2$.
\end{proposition}

\begin{proposition}[Schur-Zassenhaus] \label{020}
If $G$ is a finite group of order $mn$ with $m$ prime to $n$ containing
a normal subgroup $N$ of order $m$, then $G$ has a subgroup of order $n$
and any two such subgroups are conjugate by an element in $G$.
\end{proposition}

It follows that every finite subgroup $G$ of $\Dbb_n^\x$ is contained in
$\Sbb_n$ and determines a split extension
\[
    1 \raa P \raa G \raa C \raa 1,
\]
where $P := G \cap S_n$ is a finite normal $p$-subgroup which is the
$p$-Sylow subgroup of $G$, and $C := G/P$ is a cyclic group of order
prime to $p$ which embeds into $\Fbb_{p^n}^\x$ via the reduction
homomorphism. Moreover, $P$ is either cyclic or a generalized quaternion
group if $p=2$. If $P$ is cyclic of order $p^\alpha$ with $\alpha \geq
1$, we know that $n$ is a multiple of $\phi(p^\alpha) =
(p-1)p^{\alpha-1}$.

\begin{proposition} \label{022}
If $n$ is odd or is not divisible by $(p-1)$, then 
\[
    \begin{cases}
    C_{p^n-1} \iso \Fbb_{p^n}^\x 
    & \text{ if } p>2,\\
    C_{2(p^n-1)} \iso \Fbb_{p^n}^\x\!\x\!\{\pm 1\} 
    & \text{ if } p=2,
    \end{cases}
\]
represents the only isomorphic class of maximal finite subgroups of
$\Ocal_n^\x$.
\end{proposition}

\begin{proof}
Under the given assumptions, proposition \ref{021} implies that the
$p$-Sylow subgroup $P$ of a maximal finite subgroup $G$ of $\Dbb_n^\x$
is trivial if $p$ is odd, and is $\{\pm 1\}$ if $p=2$. The result then
follows from the Skolem-Noether theorem \ref{010}.
\end{proof}

By proposition \ref{022}, only those cases where $n$ is even and
divisible by $p-1$ remain to be studied. From now on, we will adopt the
following notations.

\begin{notation} \label{131}
Fix a prime $p$ and $n$ a multiple of $p-1$. Then we define integers $k$
and $m$ to satisfy
\[
    n = (p-1)p^{k-1}m
    \qq \text{with} \q (m;p)=1,
\]
and for $0 \leq \alpha \leq k$ we set
\[
    n_\alpha 
    := \frac{n}{\phi(p^\alpha)} =
    \begin{cases}
        n & \text{if } \alpha = 0,\\
        p^{k-\alpha}m & \text{if } \alpha > 0.
    \end{cases}
\]
\end{notation}

\begin{notation} \label{170}
For a finite subgroup $G \subset \Dbb_n^\x$ and a commutative ring $R$
extending $\Zbb_p$ in $\Dbb_n$, respectively a commutative field
extending $\Qbb_p$ in $\Dbb_n$, we denote by
\[
    R(G) = \{\sum_{g \in G} x_g g\ |\ x_g \in R\}
\]
the $R$-subalgebra of $\Dbb_n$ generated by $G$. 

For example if $R=\Qbb_p$, $G$ is a finite cyclic group and $\zeta$ is a
generator of $G$, then $\Qbb_p(G) = \Qbb_p(\zeta)$ is the cyclotomic
field generated by $\zeta$ over $\Qbb_p$. 

We note that $R(G)$ is not in general isomorphic to the group ring
$R[G]$, although there is a unique surjective homomorphism of
$R$-algebras from $R[G]$ to $R(G)$ extending the embedding of $G$ (seen
as abstract group) into $\Dbb_n^\x$.
\end{notation}

\section{Finite subgroups of $\Dbb_n^\x$ with cyclic $p$-Sylow} 
\label{155}

Let $n=(p-1)p^{k-1}m$ with $m$ prime to $p$ as in notation \ref{131}.
If $G$ is a finite subgroup of $\Dbb_n^\x$, it is then a subgroup of
$\Ocal_n^\x$ which determines an extension
\[
    1 \raa P \raa G \raa C \raa 1,
\]
and whose $p$-Sylow subgroup $P = G \cap S_n$ is either cyclic of order
$p^\alpha$ for $0 \leq \alpha \leq k$, or a generalized quaternion
group. The latter case only occurs when $p=2$; it is studied in section
\ref{156}.

For now, we fix an integer $1 \leq \alpha \leq k$ and assume that $P$ is
cyclic of order $p^\alpha$. We know from proposition \ref{342} that
$\Qbb_p(P)$ is a totally ramified extension of degree $\phi(p^\alpha)$
over $\Qbb_p$. As $P$ is abelian and normal in $G$, there are inclusions
of subgroups 
\[
    P \subset C_G(P) \subset N_G(P) = G,
\]
and the group $C = G/P$ injects into $\Fbb_{p^n}^\x$. The following
result establishes a stronger condition on $C$.

\begin{proposition} \label{157}
The group $C_{G}(P)/P$ injects into $\Fbb_{p^{n_\alpha}}^\x$ via the
reduction homomorphism, and $N_{G}(P)/C_{G}(P)$ identifies canonically
with a subgroup of the $p'$-part of $Aut(P)$.
\end{proposition}

\begin{proof}
First note that $P$ generates a cyclotomic extension $K = \Qbb_p(P)$,
and $C_G(P)$ is contained in $C_{\Dbb_n^\x}(K)$. By the centralizer
theorem \ref{007}, $C_{\Dbb_n}(K)$ is itself a central division algebra
over $K$. Since
\[
    n = \phi(p^\alpha)n_\alpha = [\Qbb_p(P): \Qbb_p]n_\alpha,
\]
it is of dimension $n_\alpha^2$ over its center $K$ and has residue
field $\Fbb_{p^{n_\alpha}}$. The reduction homomorphism in this division
algebra induces a map $C_G(P) \ra \Fbb_{p^{n_\alpha}}^\x$ whose kernel
is $P$; this shows the first assertion. 

The second assertion follows from the facts that 
\[
    P \subset C_G(P) 
    \qq \text{and} \qq
    G = N_G(P),
\]
and hence that $N_G(P)/C_G(P) \subset C$ must be prime to $p$.
\end{proof}

\begin{corollary} \label{158}
The group $C$ is contained in the cyclic subgroup of order
$(p^{n_\alpha}-1)(p-1)$ in $\Fbb_{p^n}^\x$.
\end{corollary}

\begin{proof}
This follows from proposition \ref{157} and the fact that the $p'$-part
of
\[
    Aut(P) \iso
    \begin{cases}
        C_{p-1} \x C_{p^{\alpha-1}} & \text{if } p>2,\\
        C_{2} \x C_{2^{\alpha-2}} & \text{if } p=2,
    \end{cases}
\]
is of order $p-1$.
\end{proof}

We now proceed to the existence of such finite groups. Recall from
proposition \ref{021} that $\Dbb_n^\x$ has cyclic subgroups of order
$p^\alpha$ for any $1 \leq \alpha \leq k$.

\begin{proposition} \label{159}
If $P_\alpha$ is a cyclic subgroup of order $p^\alpha > 1$ in
$\Dbb_n^\x$ and $v = v_{\Dbb_n}$, then
\[
    v(C_{\Dbb_n^\x}(P_\alpha)) = \frac{1}{n}\Zbb
    \qq \text{and} \qq
    N_{\Dbb_n^\x}(P_\alpha)/C_{\Dbb_n^\x}(P_\alpha)
    \iso N_{\Ocal_n^\x}(P_\alpha)/C_{\Ocal_n^\x}(P_\alpha).
\]
\end{proposition}

\begin{proof}
From the Skolem-Noether theorem \ref{010}, we know that
\[
    N_{\Dbb_n^\x}(P_\alpha)/C_{\Dbb_n^\x}(P_\alpha)
    \iso Aut(P_\alpha).
\]
This means that for any $f$ in $Aut(P_\alpha)$, there is an element $a$
in $\Dbb_n^\x$ such that
\[
    f(x) = axa^{-1} 
    \q \text{for all } x \in K_\alpha = \Qbb_p(P_\alpha).
\]
As explained in appendix \ref{033}, the fact that $K_\alpha$ is a
totally ramified extension of $\Qbb_p$ implies that the value group of
$C_{\Dbb_n^\x}(K_\alpha)$ is that of $\Dbb_n^\x$; in other words  
\[
    v(C_{\Dbb_n^\x}(P_\alpha)) = \frac{1}{n}\Zbb.
\]
Hence there is an element $b$ in $C_{\Dbb_n^\x}(K_\alpha)$ such that
\[
    v(ab) = 0
    \qq \text{and} \qq
    (ab)x(ab)^{-1} = axa^{-1} = f(x)
\]
for all $x \in K_\alpha$. In particular $ab \in \Ocal_n^\x$ and
\[
    N_{\Ocal_n^\x}(P_\alpha)/C_{\Ocal_n^\x}(P_\alpha)
    \iso Aut(P_\alpha),
\]
as was to be shown.
\end{proof}

\begin{lemma} \label{160}
If $P_\alpha$ is a cyclic subgroup of order $p^\alpha > 1$ in
$\Dbb_n^\x$, the image of $N_{\Ocal_n^\x}(P_\alpha)$ in $\Fbb_{p^n}^\x$
via the reduction homomorphism is cyclic of order
$(p^{n_\alpha}-1)(p-1)$.
\end{lemma}

\begin{proof}
Since the residue field of the division algebra
\[
    C_{\Dbb_n}(\Qbb_p(P_\alpha)) = C_{\Dbb_n}(P_\alpha)
\]
is $\Fbb_{p^{n_\alpha}}$, the image of $C_{\Ocal_n^\x}(P_\alpha) =
C_{\Dbb_n^\x}(P_\alpha) \cap \Ocal_n^\x$ via the reduction homomorphism
is cyclic of order $p^{n_\alpha}-1$ in $\Fbb_{p^n}^\x$.  Furthermore,
there is a canonical surjection
\[
    N_{\Ocal_n^\x}(P_\alpha) \raa Aut(P_\alpha) \raa C_{p-1}.
\]
Clearly, $C_{\Ocal_n^\x}(P_\alpha)$ is in the kernel of this projection,
and since $p-1$ is prime to $p$, the $p$-Sylow subgroup of
$N_{\Ocal_n^\x}(P_\alpha)$ must be contained in the kernel as well. It
follows that $N_{\Ocal_n^\x}(P_\alpha)$ contains a group which is sent
surjectively onto $C_{p-1}$ and whose image in $\Fbb_{p^n}^\x$ is the
cyclic subgroup of order $(p^{n_\alpha}-1)(p-1)$.
\end{proof}

\begin{theorem} \label{161}
For each $1 \leq \alpha \leq k$ and each cyclic subgroup $P_\alpha$ of
order $p^\alpha$ in $\Dbb_n^\x$, there exists a subgroup $G_\alpha$ of
$\Ocal_n^\x$ such that
\[
    G_\alpha \cap S_n = P_\alpha
    \qq \text{and} \qq
    G_\alpha/P_\alpha 
    \iso C_{(p^{n_\alpha}-1)(p-1)} \subset \Fbb_{p^n}^\x.
\]
\end{theorem}

\begin{proof}
We want to show that the cyclic subgroup of order
$(p^{n_\alpha}-1)(p-1)$ in $\Fbb_{p^n}^\x$ obtained from lemma \ref{160}
can be lifted to an element of finite order in
$N_{\Ocal_n^\x}(P_\alpha)$. 

Let $\bar{x}$ be an element of order $(p^{n_\alpha}-1)(p-1)$ in
$\Fbb_{p^n}^\x$. By lemma \ref{160}, $\bar{x}$ has a preimage $x$ in
$N_{\Ocal_n^\x}(P_\alpha)$ generating by conjugation an element of order
$p-1$ in $Aut(P_\alpha)$. The closure $\lan x \ran$ in $\Ocal_n^\x$ of
the group generated by $x$ fits into the exact sequence
\[
    1 \raa H \raa \lan x \ran \raa C \raa 1,
\]
where
\[
    H = \lan x \ran \cap S_n
    \qq \text{and} \qq
    C = \lan x \ran / H.
\]
The group $H$ being a cyclic profinite $p$-group, it must be isomorphic
to $\Zbb_p$ or to a finite cyclic $p$-group. As $l := |C|$ is prime to
$p$, any element in $H$ is $l$-divisible, and because $x^l \in H$, there
is a $y \in H$ such that $x^l = y^l$. Since $x,y \in \lan x \ran$
commute with each other, $(xy^{-1})^l = 1$ and $xy^{-1}$ is the desired
element of finite order in $N_{\Dbb_n^\x}(P_\alpha)$.
\end{proof}

\begin{remark} \label{162}
One can show that the isomorphism class of such a $G_\alpha$ is uniquely
determined by $\alpha$. This however is a consequence of the uniqueness
of $G_\alpha$ up to conjugation, a fact established in theorem \ref{024}
and \ref{032}.
\end{remark}

\section{Finite subgroups of $\Dbb_n^\x$ with quaternionic $2$-Sylow}  
\label{156}

Continuing our investigation of the finite subgroups $G$ of $\Dbb_n^\x$,
we now consider the case where the $p$-Sylow subgroup $P$ of $G$ is
non-cyclic. We know from proposition \ref{019} that in this case $p=2$
and $P$ is a generalized quaternion group $Q_{2^\alpha}$ with $\alpha
\geq 3$. Throughout this section we assume $p=2$. 

We first look at the case $n=2$. Consider the filtration of $\Zbb_2^\x
\iso \Zbb/2 \x \Zbb_2$ given by
\[
    U_i 
    = U_i(\Zbb_2^\x) 
    = 1+2^i\Zbb_2 
    = \{ x \in \Zbb_2^\x\ |\ x \equiv 1 \mod 2^i \},
    \qq \text{for } i \geq 1.
\]
As $-7 \equiv 1$ mod $2^3$, we have
\[
    -7 \in U_3 = (\Zbb_2^\x)^2.
\]
So let $\rho$ be an element of $\Zbb_2^\x$ such that $\rho^2 = -7$. 

\begin{remark} \label{210}
By remark \ref{334}, we know that
\begin{align*}
    \Dbb_2\ 
    &\iso\ \Qbb_2(\omega)\lan S \ran 
    / (S^2 = 2, S x = x^\sigma S)\\
    &\iso\ \Qbb_2(\omega)\lan T \ran 
    / (T^2 = -2, T x = x^\sigma T),
\end{align*}
for $\omega$ a primitive third root of unity which satisfies 
\[
    1+\omega+\omega^2=0,
\]
and for $S, T$ two elements generating the Frobenius $\sigma$. Letting
$T = x+yS \in \Dbb_2^\x$ for $x, y \in \Qbb_2(\omega)$, we have
\begin{align*}
    -2 = T^2 = (x^2+2yy^\sigma) + (xy+yx^\sigma)S
    \qq \Lra \qq 
    \begin{cases}
    x^2+2yy^\sigma = -2\\
    xy+yx^\sigma = 0.
    \end{cases}
\end{align*}
Taking for solution
\[
    x=0 \q \text{and} \q y = \frac{3+2\omega}{\rho},\qq 
    \text{so that} \q T = \frac{3+2\omega}{\rho} S,
\]
we obtain an isomorphism between these two representations of $\Dbb_2$.

Via these representations, we may further exhibit an explicit embedding
of $Q_8 = \lan i, j \ran$ into $\Dbb_2^\x$ in the following way. We
first look for an element $i = a+bT$ with $a, b \in \Wbb(\Fbb_4)$
satisfying
\[
    -1 = i^2 = (a^2 - 2bb^\sigma) + (a+a^\sigma)bT.
\]
Hence either $b=0$ or $a+a^\sigma = 0$. The first case being impossible
as $a^2=-1$ has no solution in $\Wbb(\Fbb_4)$, we must have $a+a^\sigma
= 0$. A possible solution is
\[
    a = \frac{-1}{1+2\omega} 
    = \frac{1}{3}(1+2\omega)
    \qq \text{and} \qq 
    b = \frac{1}{1+2\omega} 
    = -\frac{1}{3}(1+2\omega),
\]
meaning that
\begin{align*}
    i\ &=\ \frac{1}{3}(1+2\omega) - \frac{1}{3}(1+2\omega)T\\[1ex]
    &=\ \frac{1}{3}(1+2\omega) + \frac{1}{3\rho}(1-4\omega)S.
\end{align*}
We then look for an element $j=a'+b'T$ with $a', b' \in \Wbb(\Fbb_4)$
satisfying $j^2=-1$ and $ij=-ji$, in other words
\begin{align*}
    (a'^2 - 2b'b'^\sigma) + (a'+a'^\sigma)b'T\ &=\ -1 \\[1ex]
    \text{and} \qq 
    (aa'-2bb'^\sigma)+(ab'+ba'^\sigma)T\ 
    &=\ -(a'a-2b'b^\sigma)-(a'b+b'a^\sigma)T.
\end{align*}
As $a+a^\sigma=0$ and $a=-b=b^\sigma$, these relations are equivalent to
\[
    \begin{cases}
    a'+a'^\sigma=0 \\
    2aa' = 2(bb'^\sigma + b'b^\sigma) = 2b(b'^\sigma - b')
    \end{cases}
    \qq \Lra \qq
    \begin{cases}
    a'+a'^\sigma=0 \\
    a'=b'-b'^\sigma.
    \end{cases}
\]
A possible solution is
\[
    a' = a = \frac{1}{3}(1+2\omega) 
    \qq \text{and} \qq  
    b' = (1+\omega)a' = \frac{1}{3}(-1+\omega),
\] 
meaning that
\begin{align*}
    j\ &=\ \frac{1}{3}(1+2\omega) + \frac{1}{3}(-1+\omega) T\\[1ex]
    &=\ \frac{1}{3}(1+2\omega) - \frac{1}{3\rho}(5+\omega)S.
\end{align*}
\end{remark}

\begin{proposition} \label{027}
The quaternion group $Q_8$ embeds in $\Dbb_n^\x$ if and only if $n
\equiv 2 \mod 4$.
\end{proposition}

\begin{proof}
The $\Qbb_2$-algebra $\Qbb_2(i,j)$ generated by $\lan i, j \ran \iso
Q_8$ is non-commutative and is at least of dimension $4$ over $\Qbb_2$.
By remark \ref{210} we know that $\Qbb_2(i,j) \subset \Dbb_2$, and it
follows that $\Qbb_2(i,j) = \Dbb_2$. Thus in particular, $Q_8$ embeds in
$\Dbb_n^\x$ if and only if $\Dbb_2^\x$ does, and by corollary \ref{040}
this happens if and only if $n \equiv 2 \mod 4$.
\end{proof}

\begin{remark} \label{028}
Using the elements $i$ and $j$ obtained in remark \ref{210}, and
defining
\[
    k 
    := ij 
    = -\frac{1}{3}(1+2\omega) - \frac{1}{3\rho}(4+5\omega)S,
\]
we note that
\[
    \omega^2 i \omega^{-2} 
    = \omega j \omega^{-1} 
    = -k.
\]
This implies that the group \label{319}
\[
    T_{24} 
    := Q_8 \rtimes C_3 
    \iso \lan i,j,\omega \ran
\]
embeds as a maximal finite subgroup of $\Dbb_2^\x$. This group of order
$24$ is the binary tetrahedral group; it is explicitly given by
\[
    T_{24} 
    = \{\pm 1, \pm i, \pm j, \pm k, 
    \frac{1}{2}(\pm 1 \pm i \pm j \pm k)\}.
\]
From proposition \ref{027}, we have obtained
\[
    \Dbb_2^\x 
    \iso \Qbb_2(Q_8) 
    \iso \Qbb_2(T_{24}).
\]
\end{remark}

\begin{proposition} \label{029}
A generalized quaternion subgroup of $\Dbb_n^\x$ is isomorphic to $Q_8$.
\end{proposition}

\begin{proof}
Assume that $Q_{2^{\alpha+1}}$ embeds as a subgroup of $\Dbb_n^\x$ for
$\alpha \geq 2$. Then $Q_8$ embeds and
\[
    n \equiv 2 \mod 4
\]
by proposition \ref{027}. On the other hand, the cyclic group
$C_{2^\alpha}$ embeds as well and generates a cyclotomic extension of
degree $\phi(2^\alpha) = 2^{\alpha-1}$ over $\Qbb_p$. Hence
\[
    n \equiv 0 \mod 2^{\alpha-1}
\]
by the embedding theorem. Therefore $\alpha = 2$.
\end{proof}

\begin{proposition} \label{163}
If $Q_8$ is a quaternion subgroup of $\Dbb_n^\x$ and $v = v_{\Dbb_n}$,
then
\[
    v(C_{\Dbb_n^\x}(Q_8)) = \frac{2}{n}\Zbb, \qq
    v(N_{\Dbb_n^\x}(Q_8)) = \frac{1}{n}\Zbb,
\]
and $N_{\Ocal_n^\x}(Q_8)/C_{\Ocal_n^\x}(Q_8)$ injects into
$N_{\Dbb_n^\x}(Q_8)/C_{\Dbb_n^\x}(Q_8)$ as a subgroup of index 2.
\end{proposition}

\begin{proof}
Using the centralizer theorem \ref{007}, together with remark \ref{028},
we know that
\[
    \Dbb_n \iso \Qbb_2(Q_8) \ox_{\Qbb_2} C_{\Dbb_n}(Q_8),
\]
where $C_{\Dbb_n}(Q_8)$ is a central division algebra of dimension
$n^2/4$ over $\Qbb_2$ whose ramification index is
$e(C_{\Dbb_n}(Q_8)/\Qbb_2) = n/2$ by proposition \ref{011}. In
particular,
\begin{align}
    v(C_{\Dbb_n^\x}(Q_8)) = \frac{2}{n}\Zbb.\tag{$\ast$}
\end{align}
Now the existence of $Q_8$ in $\Dbb_n^\x$ implies by proposition
\ref{027} that $n \equiv 2 \mod 4$, so that $n = 2(2r+1)$ for an integer
$r \geq 0$.  As $2$ and $2r+1$ are prime to each other, there are
integers $a, b \geq 1$ satisfying
\[
    (2r+1)a + 2b = 1 
    \qq \Lra \qq 
    \frac{a}{2}+\frac{b}{2r+1} = \frac{1}{n}.
\]
By ($\ast$) we can choose an element $x \in C_{\Dbb_n^\x}(Q_8)$ having
valuation $2/n = 1/(2r+1)$. On the other hand since
\[
    (1+i)j(1+i)^{-1} = k \in Q_8 \qq \text{and} \qq (1+i)^2 = 2i,
\]
we know that $1+i$ is an element of $N_{\Dbb_n^\x}(Q_8)$ having
valuation $1/2$.  We thus have found an element $(1+i)^ax^b$ in
$N_{\Dbb_n^\x}(Q_8)$ of valuation
\begin{align*}
    v( (1+i)^ax^b ) 
    = av(1+i)+bv(x) 
    = \frac{a}{2}+\frac{b}{2r+1} 
    = \frac{1}{n},
\end{align*}
so that 
\[
    v(N_{\Dbb_n^\x}(Q_8)) = \frac{1}{n}\Zbb.
\]
This result, together with ($\ast$), implies the last assertion of the
proposition.
\end{proof}

\begin{proposition} \label{176}
$|Aut(Q_8)| = |Aut(T_{24})| = 24$.
\end{proposition}

\begin{proof}
Let $Q_8 = \lan i, j \ran$ and $T_{24} = \lan Q_8, \omega \ran$ with $i,
j, k, \omega$ as defined in remark \ref{210} and \ref{028}. Counting on
which of the $6$ elements $\{\pm i, \pm j, \pm k\}$ of order $4$ the
generators $i$ and $j$ may be sent via an automorphism, we know that
$|Aut(Q_8)|$ divides $24$. The inner automorphism group of $Q_8$ has
order $|Q_8/\{\pm 1\}| = 4$; it is generated by conjugation by $i$ and
$j$. Let 
\[
    c_{Q_8}: T_{24} \raa Aut(Q_8)
\]
be the conjugation action of $Q_8$ by elements of $T_{24}$. As noted in
remark \ref{028}, the conjugation by $\omega$ has order $3$, and hence
the cardinality of the image of $c_{Q_8}$ is $12$. Since the element
$(1+i) \in \Dbb_n^\x$ acts by conjugation on $Q_8$ by $i \mto i$ and $j
\mto k$, it follows that the automorphism of $Q_8$ induced by $(1+i)$ is
not in the image of $c_{Q_8}$. Because $|Aut(Q_8)| \leq 24$, we obtain
$|Aut(Q_8)|=24$.

Now using that $Q_8$ is the (normal) $2$-Sylow subgroup of $T_{24}$,
consider the canonical map $\phi: Aut(T_{24}) \ra Aut(Q_8)$; it is
surjective since $(1+i)$ also induces an automorphism of $T_{24}$. Let
$\sigma \in Aut(T_{24})$ be such that $\sigma |_{Q_8} = id_{Q_8}$. Then
for any $t \in T_{24}$ and $q \in Q_8$ we have
\[
    c_{Q_8}(\sigma(t))(q)
    = \sigma(t) q \sigma(t)^{-1}
    = \sigma(tqt^{-1})
    = tqt^{-1}
    = c_{Q_8}(t)(q).
\]
Hence $\sigma(t)t^{-1} \in Ker(c_{Q_8}) = \{\pm 1\}$ and $\sigma(t) =
\pm t$ for any $t \in T_{24}$. In fact, $t=s q$ with $q \in Q_8$ and $s$
an element of order $3$ in $T_{24}$, and we have
\[
    \sigma(t)t^{-1}
    = \sigma(s)\sigma(q)q^{-1}s^{-1}
    = \sigma(s)s^{-1}.
\]
Because $s$ is of order $3$ and $-s$ is of order $6$, the case
$\sigma(t)=-t$ is impossible and we must have $\sigma(t) = t$ for all $t
\in T_{24}$. Therefore the map $\phi$ is bijective, and as
$|Aut(Q_8)|=24$, it follows that $|Aut(T_{24})|=24$.
\end{proof}

Now assume $n=2m$ with $m$ odd and consider a finite subgroup $G$ of
$\Dbb_n^\x$ whose $2$-Sylow subgroup $P$ is isomorphic to $Q_8$. Such a
group determines a subgroup $C = G/P$ of $\Fbb_{2^n}^\x$.

\begin{proposition} \label{164}
If $G$ is a finite subgroup of $\Dbb_n^\x$ with a quaternionic $2$-Sylow
subgroup $P \iso Q_8$, then $G/P$ embeds into the cyclic subgroup of
order $3(2^m-1)$ in $\Fbb_{2^n}^\x$.
\end{proposition}

\begin{proof}
Recall that $\Qbb_2(P) \iso \Dbb_2^\x$ and note that $C_G(P)$ is
contained in 
\[
    C_{\Dbb_n^\x}(P) 
    = C_{\Dbb_n^\x}(\Qbb_2(P)) 
    \iso C_{\Dbb_n^\x}(\Dbb_2)
\]
which consists of the non-zero elements of a central division algebra of
dimension $m^2$ over $\Qbb_2$. Its residue field is $\Fbb_{2^m}$, and
$C_G(P)/P \cap C_G(P) \iso P \cdot C_G(P)/ P$ injects via the reduction
homomorphism into $\Fbb_{2^m}^\x$.

Furthermore, we have an injection
\[
    N_G(P)/C_G(P) \raa N_{\Ocal_n^\x}(P)/C_{\Ocal_n^\x}(P)
    \subset N_{\Dbb_n^\x}(P)/C_{\Dbb_n^\x}(P) \iso Aut(Q_8),
\]
where the last isomorphism is due to the Skolem-Noether theorem. Since
$|Aut(Q_8)| = 24$, proposition \ref{163} implies that $|N_G(P)/C_G(P)|$
divides $12$.  As $P \cap C_G(P) = \{\pm 1\}$ is of index $4$ in $P$, we
know that $C_G(P)$ is of index $4$ in $P \cdot C_G(P)$, and consequently
that $P \cdot C_G(P)$ is of index a divisor of $3$ in $N_G(P)$.

We have thus obtained a chain of subgroups
\[
    P \subset P \cdot C_G(P) \subset N_G(P) = G,
\]
where the first group is of index a divisor of $2^m-1$ in the second
group, and the latter is of index a divisor of $3$ in the third group.
\end{proof}

\begin{theorem} \label{165}
If $p=2$ and $n=2m$ with $m$ odd, the group
\[
    T_{24} \x C_{2^m-1} = Q_8 \rtimes C_{3(2^m-1)}
\]
embeds as a maximal finite subgroup of $\Dbb_n^\x$.
\end{theorem}

\begin{proof}
By the centralizer theorem
\[
    \Dbb_n 
    \iso \Dbb_2 \ox_{\Qbb_2} C_{\Dbb_n}(\Dbb_2)
    \iso \Qbb_2(Q_8) \ox_{\Qbb_2} C_{\Dbb_n}(Q_8).
\]
By remark \ref{028}, $T_{24} = Q_8 \rtimes C_3$ embeds as a subgroup of
$\Dbb_2^\x$; more precisely $\Qbb_2(T_{24}) = \Dbb_2$. Moreover, since
$C_{\Dbb_n}(\Dbb_2)$ is a central division algebra of dimension $m^2$
over $\Qbb_2$, its maximal unramified extension of degree $m$ over
$\Qbb_2$ contains a cyclic subgroup $C_{2^m-1}$ of order $2^m-1$ which
centralizes $T_{24}$.  Since $m$ is odd, $2^m-1$ is not a multiple of
$3$ and $\Dbb_n^\x$ contains a subgroup isomorphic to
\[
    T_{24} \x C_{2^m-1} \iso Q_8 \rtimes C_{3(2^m-1)};
\]
its maximality as a finite subgroup then follows from proposition
\ref{164}.
\end{proof}

\begin{corollary} \label{174}
The center of $T_{24} \x C_{2^m-1}$ is
\[
    Z(T_{24} \x C_{2^m-1})
    = \{\pm 1\} \x C_{2^m-1}
    \iso C_{2(2^m-1)}.
\]
\end{corollary}

\begin{proof}
This follows from the proof of theorem \ref{165} and the obvious fact
that the center of $Q_8$ is $\{\pm 1\}$.
\end{proof}

\section{Conjugacy classes in $\Sbb_n$} 
 
In this section, we establish a classification of the finite subgroups
of $\Sbb_n$ up to conjugation. We say that two subgroups $G_1, G_2
\subset \Dbb_n^\x$ are \emph{conjugate} in $\Dbb_n^\x$, respectively in
$\Ocal_n^\x$, if there is an element $a$ in $\Dbb_n^\x$, respectively in
$\Ocal_n^\x$, satisfying
\[
    a G_1 a^{-1} = G_2.
\]
We will see that two finite subgroups $G_1$ and $G_2$ whose respective
$p$-Sylow subgroups $P_1$ and $P_2$ are isomorphic, and for which the
quotient groups $G_1/P_1$ and $G_2/P_2$ are also isomorphic, are not
only isomorphic but even conjugate in $\Ocal_n^\x$. This will imply that
the maximal subgroups of $\Ocal_n^\x$ are classified up to conjugation
by the type of their $p$-Sylow subgroups. To do this, we will exploit
the tools of nonabelian cohomology of profinite groups as introduced in
\cite{serre3} chapter I paragraph 5. 

For any subgroup $G$ of a group $H$, we set
\[
    S_H(G) := \{ G' \leq H\ |\ G' \iso G \}
    \qq \text{and} \qq
    \Ccal_H(G) := S_H(G)/\sim_H
\]
where $\sim_H$ designates the relation of conjugation by an element in
$H$.

\begin{lemma} \label{166}
If $P$ is a finite $p$-subgroup of $\Ocal_n^\x$, then
$|\Ccal_{\Ocal_n^\x}(P)| = 1$.
\end{lemma}

\begin{proof}
Let $Q$ be a finite $p$-subgroup of $\Ocal_n^\x$ isomorphic to $P$. We
have seen that these two groups are either cyclic or quaternionic. In
either case, the Skolem-Noether theorem implies the existence of an
element $a$ in $\Dbb_n^\x$ such that
\[
    \Qbb_p(Q) = a \Qbb_p(P) a^{-1}.
\]
In the cyclic case, this clearly implies $Q = aPa^{-1}$. In the
quaternionic case, this yields two quaternion groups $Q$ and $aPa^{-1}$
within $\Qbb_2(Q) \iso \Dbb_2^\x$ in which we can use Skolem-Noether
once more to obtain an element $a' \in \Qbb_2(Q)$ such that $Q =
a'aP(a'a)^{-1}$. Now by proposition \ref{159} and \ref{163}, we know
that
\[
    v(N_{\Dbb_n^\x}(P)) = \frac{1}{n}\Zbb = v(\Dbb_n^\x).
\]
Thus there is an element $b$ in $\Dbb_n^\x$ such that
\[
    v(ab) = 0
    \qq \text{and} \qq
    P = b P b^{-1},
\]
and $ab$ is an element of $\Ocal_n^\x$ conjugating $P$ into $Q$.
\end{proof}

\begin{lemma} \label{167}
Let $P$ be a profinite $p$-group of the form $P = \lim_n P_n$ where each
$P_n$ is a finite $p$-group and the homomorphisms in the inverse system
are surjective, and let $R$ be a finite group of order prime to $p$
which acts by group homomorphisms on all $P_n$ in such a way that the
homomorphisms in the inverse system are $R$-equivariant. Then the
(nonabelian) cohomology group $H^1(R, P)$ is trivial.
\end{lemma}

\begin{proof}
Denote by $j_n: P_n \ra P_{n-1}$ the homomorphisms of the inverse
system, and consider the map $\delta: \prod_n P_n \ra \prod_n P_n$
defined by
\[
    \delta(f_n) = (-1)^n f_n + (-1)^{n+1} j_{n+1}(f_{n+1}),
\]
for $f = (f_n) \in \prod_n P_n$. Then note that $\delta$ is surjective
and that $Ker(\delta)$ is the set of all $f = (f_n) \in \prod_n P_n$
such that $j_{n}(f_n) = f_{n-1}$ for all $n$.  Hence there is a short
exact sequence
\[
    1 \raa P \raa \prod_n P_n \overset{\delta}{\raa} \prod_n P_n \raa 1
\]
which induces a long exact sequence
\[
    1 \ra P^R \ra \prod_n P_n^R \ra \prod_n P_n^R 
    \ra H^1(R, P) \ra H^1(R, \prod_n P_n) \ra H^1(R, \prod_n P_n), 
\]
where $P^R$, respectively $P_n^R$, denotes the $R$-invariants. Using the
canonical isomorphism
\[
    H^1(R, \prod_n P_n) \iso \prod_n H^1(R, P_n),
\]
and noting that each group $H^1(R, P_n)$ is trivial by the
Schur-Zassenhaus theorem \ref{020}, it is enough to show that the
homomorphism
\[
    \prod_n P_n^R \raa \prod_n P_n^R
\]
in the above exact sequence is surjective, and hence that each
homomorphism $j_{n+1}^R: P_{n+1}^R \ra P_n^R$ is surjective by the
definition of $\delta$.

For each $n$, let $K_{n+1}$ be the kernel of the map $j_{n+1}: P_{n+1}
\ra P_n$. For each short exact sequence of finite $p$-groups with action
of $R$
\[
    1 \raa K_{n+1} \raa P_{n+1} \raa P_n \raa 1,
\]
there is an associated exact cohomology sequence
\[
    1 \raa K_{n+1}^R \raa P_{n+1}^R \overset{j_{n+1}^R}{\raa} 
    P_n^R \raa H^1(R, K_{n+1}).
\]
Applying the Schur-Zassenhaus theorem once more, we obtain that $H^1(R,
K_{n+1})$ is trivial and that the homomorphism $j_{n+1}^R$ is
surjective.
\end{proof}

We recall the following fact from \cite{serre3} chapter I \S 5.1:

\begin{lemma} \label{269}
If $P$ is an $R$-group with trivial (nonabelian) $H^1(R, P)$, and if
\[
    1 \raa P \raa N \raa R \raa 1
\]
is a split extension, then two splittings of $R$ in $N$ are conjugate by
an element in $P$.
\end{lemma}

\begin{theorem} \label{168}
Two finite subgroups $G_1$ and $G_2$ of $\Ocal_n^\x$ with respective
isomorphic $p$-Sylow subgroups $P_1 \iso P_2$ and isomorphic quotient
groups $G_1/P_1 \iso G_2/P_2$ are conjugate in $\Ocal_n^\x$.
\end{theorem}

\begin{proof}
The groups $G_1$ and $G_2$ fit into exact sequences
\begin{align*}
    & 1 \raa P_1 \raa G_1 \raa C \raa 1\\[1ex]
    & 1 \raa P_2 \raa G_2 \raa C \raa 1,
\end{align*}
where $C$ is the subgroup of $\Fbb_{p^n}^\x$ isomorphic to $G_1/P_1 \iso
G_2/P_2$. We know from lemma \ref{166} that $P_1$ and $P_2$ are
conjugate in $\Ocal_n^\x$. By conjugating $G_2$, we can therefore assume
that
\[
    P_1 = P_2 =: P
    \qq \text{and} \qq
    G_1, G_2 \subset N_{\Ocal_n^\x}(P).
\]
Moreover, the latter groups fit into a split exact sequence
\[
    1 \raa N_{\Ocal_n^\x}(P) \cap S_n \raa N_{\Ocal_n^\x}(P)
    \raa R \raa 1,
\]
where $R \subset \Fbb_{p^n}^\x$ is a finite cyclic group of order prime
to $p$ containing $C$. It follows from lemma \ref{167} that $H^1(R,
N_{\Ocal_n^\x}(P) \cap S_n)$ is trivial, and hence by lemma \ref{269}
that $G_1$ and $G_2$ are conjugate in $N_{\Ocal_n^\x}(P) \subset
\Ocal_n^\x$.
\end{proof}

\begin{remark} \label{169}
Alternatively, we may directly apply \cite{ribes} theorem 2.3.15, which
shows that if $K$ is the $p$-Sylow subgroup of a profinite group $G$,
then there is up to conjugation in $G$ a unique closed subgroup $H$ of
$G$ such that $G=KH$ and $K \cap H = 1$. Indeed, since in our case both
extensions
\begin{align*}
    & 1 \raa P \raa G_1 \raa C \raa 1\\[1ex]
    & 1 \raa P \raa G_2 \raa C \raa 1
\end{align*}
are split by the Schur-Zassenhaus theorem, we obtain that both of the
corresponding sections are conjugate in $N_{\Ocal_n^\x}(P)$, and hence
that $G_1$ and $G_2$ are conjugate in $N_{\Ocal_n^\x}(P)$.
\end{remark}

\begin{corollary} \label{173}
Two finite subgroups of $\Ocal_n^\x$ are conjugate if and only if they
are isomorphic. \qed
\end{corollary}

\begin{theorem} \label{024}
If $p$ is an odd prime and $n = (p-1)p^{k-1}m$ with $m$ prime to $p$,
the group $\Sbb_n$, respectively $\Dbb_n^\x$, has exactly $k+1$
conjugacy classes of maximal finite subgroups; they are represented by
\[
    G_0 = C_{p^n-1} 
    \qq \text{and} \qq 
    G_\alpha = C_{p^\alpha} \rtimes C_{(p^{n_\alpha}-1)(p-1)} \qq 
    \text{for} \q 1 \leq \alpha \leq k.
\] 
Moreover, when $p-1$ does not divide $n$, the only class of maximal
finite subgroups is that of $G_0$.
\end{theorem}

\begin{proof}
First note that proposition \ref{022} and theorem \ref{168} imply that
there is a unique maximal conjugacy class $G_0$ of finite subgroups of
order prime to $p$ in $\Ocal_n^\x \iso \Sbb_n$, respectively in
$\Dbb_n^\x$ by proposition \ref{016}, and that this class is the only
one among finite subgroups if $n$ is not a multiple of $p-1$.

Now assume that $1 \leq \alpha \leq k$. By theorem \ref{161}, there is a
finite subgroup $G_\alpha$ in $\Dbb_n^\x$ realized as an extension
\[
    1 \raa C_{p^\alpha} \raa G_\alpha 
    \raa C_{(p^{n_\alpha}-1)(p-1)} \raa 1,
\]
where
\[
    C_{p^\alpha} = G_\alpha \cap S_n
    \qq \text{and} \qq
    G_\alpha/C_{p^\alpha} 
    \iso C_{(p^{n_\alpha}-1)(p-1)} 
    \subset \Fbb_{p^n}^\x.
\]
The Schur-Zassenhaus theorem implies that this extension splits, in
other words that
\[
    G_\alpha = C_{p^\alpha} \rtimes C_{(p^{n_\alpha}-1)(p-1)}.
\]
Corollary \ref{158} and theorem \ref{168} ensure that $G_\alpha$
represents the unique maximal conjugacy class of finite subgroups of
$\Ocal_n^\x \iso \Sbb_n$ which have a $p$-Sylow subgroup of order
$p^\alpha$.
\end{proof}

\begin{corollary} \label{175}
If $p>2$ and $1 \leq \alpha \leq k$, then
\[
    Z(G_\alpha) \iso C_{(p^{n_\alpha}-1)}.
\]
\end{corollary}

\begin{proof}
This follows from theorem \ref{024} and proposition \ref{157}, where the
latter shows that $C_{(p^{n_\alpha}-1)}$ embeds into $Z(G_\alpha)$ and
that $C_{(p^{n_\alpha}-1)(p-1)}/C_{(p^{n_\alpha}-1)} \iso C_{p-1}$ acts
faithfully on $C_{p^\alpha}$.
\end{proof}

The following case can be explicitly analyzed. As noted in remark
\ref{171}, this provides a counter example to the main results of
\cite{hewett2}.

\begin{example} \label{049}
Assume that $p$ is odd and $n = (p-1)p^{k-1}m$ with $(p;m)=1$. Let
$\omega \in \Dbb_n^\x$ be a primitive $(p^n\!-\!1)$-th root of unity in
$\Ocal_n^\x$. Define
\[
    X := \omega^{\frac{p-1}{2}}S\ \in \Ocal_n^\x 
    \qq \text{and} \qq 
    Z := X^l
    \q \text{with }
    l = \frac{n}{p-1}.
\]
A simple calculation shows
\[
    Z^{p-1} = X^n = -p.
\]
We can show (see \cite{henn2} lemma 19) that $\Qbb_p(Z)$ contains a
primitive $p$-th root of unity $\zeta_p$. Because the fields $\Qbb_p(Z)$
and $\Qbb_p(\zeta_p)$ are of the same degree $p-1$ over $\Qbb_p$, they
must be identical. We set
\[
    K := \Qbb_p(Z) = \Qbb_p(\zeta_p). 
\]
We note that $p^n-1$ is divisible by $(p^l-1)(p-1)$ and let
\[
    \tau 
    := \omega^{\frac{p^n-1}{(p^l-1)(p-1)}}\ \in \Fbb_q^\x.
\]
We have 
\[
    \tau Z \tau^{-1} 
    = \omega^{\frac{p^n-1}{p-1}}Z 
    = \zeta_{p-1} Z,
\] 
for $\zeta_{p-1}$ a primitive $(p\!-\!1)$-th root of unity in
$\Ocal_n^\x$. Hence $\tau$ induces an automorphism of $K$ of order $p-1$
which sends $\zeta_p$ to another root of unity of the same order, and
$\tau$ normalizes the group generated by $\zeta_p$.  The group $G$
generated by $\zeta_p$ and $\tau$ is clearly of order $p(p^l-1)(p-1)$;
it is therefore maximal. Since $X$ commutes with all elements of $K$, it
necessarily commutes with $\zeta_p$. Moreover the fact that
\[
    X \tau X^{-1} = \tau^p
\]
shows that $X$ belongs to the normalizer $N_{\Dbb_n^\x}(G)$. The
valuation of $X$ is $\frac{1}{n}$ by definition, and we have
\[
    v(N_{\Dbb_n^\x}(G)) = \frac{1}{n}\Zbb.
\]
As in lemma \ref{166}, we can then apply the Skolem-Noether theorem to
obtain that there is only one conjugacy class of subgroups of
$\Ocal_n^\x$ that are isomorphic to $G$. 

In particular, if $p=3$ and $n=4$, then $k=1$, $m=2$, the order of
$\omega$ is $80$, and a maximal finite $3$-Sylow subgroup in
$\Ocal_n^\x$ is isomorphic to $C_3$. Here 
\[
    X=\omega S,\qq Z=\omega^4 S^2 \qq \text{and} \qq Z^2 = -3. 
\]
In order to find an element $\zeta_3$ in $\Qbb_3(X^2)$, we may solve the
equation
\[
    (x+yZ)^3 = 1 \qq \text{with} \q x, y \in \Qbb_3.
\]
We find $x = \pm y$ with $x = -\frac{1}{2}$, from which we obtain the
primitive third roots of unity
\[
    \zeta_3 = -\frac{1}{2}(1+\omega^4 S^2) 
    \qq \text{and} \qq 
    \zeta_3^2 = -\frac{1}{2}(1-\omega^4 S^2)
\]
in the field $\Qbb_3(Z)$. Here $\tau = \omega^5$ is of order $16$ and
we easily verify the relations
\[
    \tau \zeta_3 \tau^{-1} = \zeta_3^2, \qq 
    X\zeta_3X^{-1} = \zeta_3, \qq 
    X\tau X^{-1} = \tau^3,
\]
showing as expected that 
\[
    v(N_{\Dbb_n^\x}(C_{3} \rtimes C_{2(3^2-1)})) 
    = \frac{1}{4}\Zbb
    \qq \text{and} \qq
    |\Ccal_{\Ocal_n^\x}(C_{3} \rtimes C_{2(3^2-1)})| = 1.
\]
\end{example}

\begin{remark} \label{171}
Theorem \ref{024} and example \ref{049} (in particular the case where
$n=4$ and $p=3$) bring a contradiction to the main results of
\cite{hewett2}.  In the latter, a central result concerning the
nonabelian finite groups when $p > 2$ is proposition 3.9: it states that
for $\alpha \geq 1$ the normalizer of $G_\alpha$ in $\Dbb_n^\x$ has
valuation group
\[
    v(N_{\Dbb_n^\x}(G_\alpha))
    = \frac{f(\Qbb_p(\zeta_{p^\alpha(p^{n_\alpha}-1)}/\Qbb_p))}{n} \Zbb
    = \frac{n_\alpha}{n} \Zbb,
\]
where $f$ denotes the residue degree of the given cyclotomic extension.
As a consequence of this incorrect result propositions 3.10 to 3.12 in
\cite{hewett2} are incorrect as well.
\end{remark}

\begin{theorem} \label{032}
Let $p=2$ and $n = 2^{k-1}m$ with $m$ odd. The group $\Sbb_n$,
respectively $\Dbb_n^\x$, has exactly $k$ maximal conjugacy
classes of finite subgroups. If $k \neq 2$, they are represented by 
\[
    G_\alpha 
    = C_{2^\alpha(2^{n_\alpha}-1)} \qq 
    \text{for} \q 1 \leq \alpha \leq k.
\]
If $k=2$, they are represented by $G_\alpha$ for $\alpha \neq 2$ and by
the unique maximal nonabelian conjugacy class
\[
    Q_8 \rtimes C_{3(2^m-1)} \iso T_{24} \x C_{2^m-1},
\]
the latter containing $G_2$ as a subclass.
\end{theorem}

\begin{proof}
The argument for the cyclic classes $G_\alpha$ is identical to that of
theorem \ref{024} except that in this case $G_0 = C_{2^n-1}$ is
contained in $G_1$.

Furthermore, proposition \ref{029} ensures that a nonabelian finite
subgroup may only exist in $\Ocal_n^\x \iso \Sbb_n$, respectively in
$\Dbb_n^\x$, when its $2$-Sylow subgroup is isomorphic to $Q_8$, and
proposition \ref{027} shows that such a group occurs if and only if
$k=2$.  In fact, assuming $k=2$, the group $Q_8 \rtimes C_{3(2^m-1)}$
embeds in $\Ocal_n^\x$ as a maximal finite subgroup by theorem
\ref{165}, and its conjugacy class is unique among maximal nonabelian
finite subgroups by theorem \ref{168}.
\end{proof}

\begin{remark} \label{172}
Theorem \ref{032} contradicts theorem 5.3 in \cite{hewett2}. According
to the latter, we should have two distinct conjugacy classes in
$\Ocal_n^\x$ for the finite groups containing $T_{24} = Q_8 \rtimes
C_3$. Letting $Inn(T_{24})$ and $Out(T_{24})$ denote the inner and outer
automorphisms of $T_{24}$, the error occurs before theorem 5.1 where it
is said that $Out(T_{24})$ is trivial. This is absurd given that
\[
    |Aut(T_{24})| = 24 
    \qq \text{and} \qq 
    Inn(T_{24}) \iso T_{24}/\{\pm 1\}.
\]
All results given in section 5 of \cite{hewett2} are then wrong in this
case.
\end{remark}

\begin{corollary} \label{212}
The abelian finite subgroups of $\Dbb_n^\x$ are classified up to
conjugation in $\Ocal_n^\x$, respectively in $\Dbb_n^\x$, by the pairs
of integers $(\alpha, d)$ satisfying
\[
    0 \leq \alpha \leq k
    \qq \text{and} \qq
    1 \leq d\ |\ p^{n_\alpha}-1;
\]
each such pair represents the cyclic class $C_{p^\alpha d}$.
\end{corollary}

\begin{proof}
By corollary \ref{173}, the finite cyclic subgroups are classified up to
conjugation by their isomorphism classes. The result then follows from
the maximal finite classes provided by theorem \ref{024} and \ref{032}.
\end{proof}

\begin{remark} \label{026}
We restricted ourselves in considering the finite subgroups of
$\Dbb_n^\x$ as split extensions of subgroups of $\Fbb_{p^n}^\x$ by
finite $p$-subgroups in $S_n$.  It is also possible to express these
finite groups as subextensions of short exact sequences of the form
\[
    1 \raa C_{\Dbb_n^\x}(C_{p^\alpha}) 
    \raa N_{\Dbb_n^\x}(C_{p^\alpha})
    \raa Aut(C_{p^\alpha}) \raa 1,
\]
as induced by the Skolem-Noether theorem. A finite group of type
$G_\alpha \subset \Dbb_n^\x$ can be seen as a metacyclic extension
\[
    1 \raa \lan A \ran \raa G_\alpha 
    \raa \lan B \ran \raa 1,
\]
with
\[
    \lan A \ran = G_\alpha' \x Z(G_\alpha) 
    \qq \text{and} \qq 
    \lan B \ran \iso C_{p-1},
\]
where $G_\alpha'$ denotes the commutator subgroup of $G_\alpha$. The
classification given in \cite{hewett} follows this approach, but has the
disadvantages of being less direct and relying on a classification
previously established in \cite{amitsur}.
\end{remark}


\chapter{A classification scheme for finite subgroups} 
\label{091}

We fix a prime $p$, a positive integer $n$ which is a multiple of
$(p-1)$, and a unit $u \in \Zbb_p^\x$. Given these, we adopt notation
\ref{131}.  In this chapter, we provide necessary and sufficient
conditions for the existence of finite subgroups of 
\[
    \Gbb_n(u) = \Dbb_n^\x/\lan pu \ran
\]
whose intersection with $\Sbb_n$ have a cyclic $p$-Sylow subgroup. The
remaining case of a quaternionic $2$-Sylow will be treated in chapter
\ref{268}. 

\vspace{5ex}

\section{A canonical bijection} 
\label{073}

Let 
\[
    \pi: \Dbb_n^\x \raa \Gbb_n(u) 
\]
denote the canonical homomorphism. In order to study a finite subgroup
$F$ of $\Gbb_n(u)$, it is often more convenient to analyse its preimage 
\[
     \tilde{F} := \pi^{-1}(F)\ \in \Dbb_n^\x.
\]
For any group $G$ we define $\Fcal(G)$ \label{306} to be the set of all
finite subgroups of $G$; and if $G$ is a subgroup of $\Dbb_n^\x$ we
define $\tilde{\Fcal}_u(G)$ to be the set, eventually empty, of all
subgroups of $G$ which contain $\lan pu \ran$ as a subgroup of finite
index.

\begin{proposition} \label{074}
The map $\pi$ induces a canonical bijection
\[
    \tilde{\Fcal}_u(\Dbb_n^\x) \raa \Fcal(\Gbb_n(u)).    
\]
This bijection passes to conjugacy classes.
\end{proposition}

\begin{proof}
For any $F \in \Fcal(\Gbb_n(u))$, it is clear that $\lan pu \ran$ is a
subgroup of finite index in $\pi^{-1}(F)$. Moreover, the fact that $\pi$
is surjective implies that $\pi \pi^{-1}(F) = F$. On the other hand, for
$G \in \tilde{\Fcal}_u(\Dbb_n^\x)$, as $Ker(\pi) = \lan pu \ran$ is
always a subgroup of $G$, we have $\pi^{-1}\pi(G) = G$.

In order to show the second assertion, let $F_1, F_2$ be two subgroups
of $\Gbb_n(u)$ with $\tilde{F_i} = \pi^{-1}(F_i)$ for $i \in \{1, 2\}$.
If there is an element $a \in \Dbb_n^\x$ such that $\tilde{F_2} = a
\tilde{F_1} a^{-1}$, then since $\pi$ is a group homomorphism we have
\begin{align*}
    F_2 &= \pi(a \tilde{F_1} a^{-1})\\ 
    &= \pi(a) \pi(\tilde{F_1}) \pi(a)^{-1}\\
    &= \pi(a) F_1 \pi(a)^{-1}.
\end{align*}
Conversely, if $F_2 = b F_1 b^{-1}$ for some $b \in \Gbb_n(u)$, and if
$\tilde{b} \in \Dbb_n^\x$ satisfies $\pi(\tilde{b})=b$, then from the
above identity we have
\[
    \pi(\tilde{b}\tilde{F_1}\tilde{b}^{-1}) = F_2,
\]
as was to be shown.
\end{proof}

\begin{remark} \label{075}
In a similar way, the map $\pi$ induces a bijection between the set of
all subgroups of $\Gbb_n(u)$ and the set of all subgroups of $\Dbb_n^\x$
containing $\lan pu \ran$.
\end{remark}

\begin{notation}
For a subgroup $G$ of $\Gbb_n(u)$, we denote by
\[
    \tilde{G} = \pi^{-1}(G)
\]
its preimage under the canonical map $\pi: \Dbb_n^\x \ra \Gbb_n(u)$.
From now on, when introducing a tilded group, its non-tilded
correspondent will be implicitly defined.
\end{notation}

\begin{remark} \label{133}
The valuation $v = v_{\Dbb_n}: p \mto 1$ on $\Dbb_n^\x$ induces a
commutative diagram with exact rows and columns
\[
    \xymatrix{
    && \lan pu \ran \ar[d] \ar[r]^{v} & \Zbb \ar[d] \\
    1 \ar[r] & \Sbb_n \ar[d] \ar[r] 
      & \Dbb_n^\x \ar[d]^{\pi} \ar[r]^{v} 
      & \frac{1}{n}\Zbb \ar[d] \ar[r] & 1\\
    1 \ar[r] & \Sbb_n \ar[r] & \Gbb_n(u) \ar[r]^{v} 
      & \frac{1}{n}\Zbb/\Zbb \ar[r] & 1.
    }
\]
Subgroups of $\Sbb_n$ can therefore be considered as subgroups of both
$\Gbb_n(u)$ and $\Dbb_n^\x$. 
\end{remark}

\begin{proposition} \label{339}
If $F \subset \Sbb_n$, then $\tilde{F} = F \x \lan pu \ran$.
\end{proposition}

\begin{proof}
This follows from the exact commutative diagram of remark \ref{133} and
the fact that $\lan pu \ran$ is central in $\Dbb_n^\x$.
\end{proof}

\section{Chains of extensions} 
\label{076}

For $F$ a finite subgroup of $\Gbb_n(u)$ such that $F \cap S_n$ is
cyclic, we set \label{317}
\[
    G := F \cap \Sbb_n
    \qq \text{and} \qq
    F_0 :=\ \lan F \cap S_n, Z_{p'}(G) \ran,
\]
for $S_n$ the $p$-Sylow subgroup of $\Sbb_n$ and $Z_{p'}(G)$ the
$p'$-part of the center $Z(G)$ of $G$. As previously seen, $F_0$ is the
maximal abelian subgroup of $G$ equal to $P \x Z_{p'}(G)$ for $P$ the
cyclic $p$-Sylow subgroup of $G$. 

\begin{remark} \label{177}
From proposition \ref{339}, we know that
\[
    \tilde{F_0} = F_0 \x \lan pu \ran.
\]
\end{remark}

\begin{remark} \label{178}
By definition, $G$ consists of the elements of $F$ which are of
valuation zero in $\Dbb_n^\x$. Hence $G$ is normal in $F$ and there is a
short exact sequence
\[
    1 \raa G \raa F \raa F/G \raa 1,
\]
where the quotient embeds via the valuation into $\frac{1}{n}\Zbb/\Zbb$.
\end{remark}

\begin{proposition} \label{132}
We have
\[
    \tilde{C_F(F_0)} 
    = C_{\tilde{F}}(\tilde{F_0}) 
    \supset \tilde{F_0}
    \qq \text{and} \qq
    C_F(F_0)/F_0 
    \iso C_{\tilde{F}}(\tilde{F_0})/\tilde{F_0}.
\]
\end{proposition}

\begin{proof}
It is obvious that
\[
    \tilde{C_F(F_0)}\ 
    \supset\ C_{\tilde{F}}(\tilde{F_0})\ 
    \supset\ \tilde{F_0},
\]
and the second assertion is a direct consequence of the first one. 

It remains to check that $\tilde{C_F(F_0)} \subset
C_{\tilde{F}}(\tilde{F_0})$.  If $\tilde{x} \in \tilde{C_F(F_0)}$ and
$\tilde{f} \in \tilde{F_0}$, then there is a unique element
$z=z(\tilde{x}, \tilde{f}) \in \lan pu \ran$ such that
\[
    z\tilde{f} = \tilde{x}\tilde{f}\tilde{x}^{-1}.
\]
Because $\lan pu \ran$ is central, we have
\[
    z(\tilde{x}, \tilde{f}\tilde{g})  
    = z(\tilde{x}, \tilde{f})z(\tilde{x}, \tilde{g}),
\]
for every $\tilde{f}, \tilde{g} \in \tilde{F_0}$. This yields an exact
sequence
\[
    1 \raa C_{\tilde{F}}(\tilde{F_0}) \raa \tilde{C_F(F_0)}
    \raa Hom(\tilde{F_0}, \lan pu \ran),
\]
where the image of $\tilde{x}$ is the homomorphism $\tilde{f} \mapsto
z(\tilde{x}, \tilde{f})$. As stated in remark \ref{177} we know that
$\tilde{F_0} = F_0 \x \lan pu \ran$. Because $\lan pu \ran$ is central,
the image of $\tilde{C_F(F_0)}$ in $Hom(\tilde{F_0}, \lan pu \ran)$ is
contained in the subgroup of those homomorphisms which are trivial on
$\lan pu \ran \subset \tilde{F_0}$ and hence factors through $F_0$.
Because $F_0$ is finite and $\lan pu \ran$ is torsion free, it follows
that this image is trivial and $C_{\tilde{F}}(\tilde{F_0}) =
\tilde{C_F(F_0)}$.
\end{proof}

\begin{proposition} \label{077}
We have
\[
    F = N_F(F_0)
    \qq \text{and} \qq
    \tilde{F} = N_{\tilde{F}}(\tilde{F_0}).
\]
\end{proposition}

\begin{proof}
Because $P$ is the unique $p$-Sylow subgroup of $G = F \cap \Sbb_n$, it
is a characteristic subgroup of $G$. Moreover as $F_0 = P \x Z_{p'}(G)$
and $Z_{p'}(G)$ is also a characteristic subgroup of $G$, it follows
that $F_0$ is a characteristic subgroup of $G$; in other words
\[
    N_{\Gbb_n(u)}(G) \subset N_{\Gbb_n(u)}(F_0)
    \qq \text{and} \qq
    N_{\Dbb_n^\x}(G) \subset N_{\Dbb_n^\x}(F_0).
\]
Since $G$ is by definition normal in $F$, its subgroup $F_0$ is normal
in $F$. Proposition \ref{074} finally implies that $\tilde{F_0}$ is
normal in $\tilde{F}$.
\end{proof}

\begin{corollary} \label{179}
There are short exact sequences
\begin{align*}
    & 1 \raa F_0 \raa F \raa F/F_0 \raa 1,\\[1ex]
    & 1 \raa \tilde{F_0} \raa \tilde{F} \raa F/F_0 \raa 1.
\end{align*}
\end{corollary}

\begin{proof}
This follows from that facts that $F_0$ is normal in $F$, $\tilde{F_0}$
is normal in $\tilde{F}$, and that $\tilde{F}/\tilde{F_0} \iso F/F_0$.
\end{proof}

Note that in $\Dbb_n^\x$ we have $\Qbb_p(F_0) = \Qbb_p(\tilde{F_0})$,
and there are inclusions
\[
    \tilde{F_0} 
    \subset \Qbb_p(F_0)^\x 
    \subset C_{\Dbb_n^\x}(F_0)
    \subset N_{\Dbb_n^\x}(F_0).
\]
Given $F$ and $F_0$, the second extension of corollary \ref{179} can
then be broken into three pieces via the chain of subgroups
\[
    \tilde{F_0} \subset \tilde{F_1} \subset \tilde{F_2} 
    \subset \tilde{F_3} = \tilde{F}
\]
defined by: \label{318}
\begin{itemize}
    \item $\tilde{F_1} := \tilde{F} \cap \Qbb_p(F_0)^\x$;

    \item $\tilde{F_2} := \tilde{F} \cap C_{\Dbb_n^\x}(F_0)
    = C_{\tilde{F}}(F_0)$;

    \item $\tilde{F_3} := \tilde{F} \cap N_{\Dbb_n^\x}(F_0) 
    = N_{\tilde{F}}(F_0) = \tilde{F}$.
\end{itemize}
Clearly the groups $\tilde{F_0}$, $\tilde{F_1}$ are abelian, and since
$\tilde{F_2}/\tilde{F_0} \iso F_2/F_0 \subset \frac{1}{n}\Zbb/\Zbb$ is
cyclic by proposition \ref{132} and remark \ref{178}, the group
$\tilde{F_2}$ is also abelian. Moreover we note that for $0 \leq i \leq
2$, each $\tilde{F_{i}}$ is normal in $\tilde{F_{i+1}}$. In particular,
any finite subgroup $F \subset \Gbb_n(u)$ determines successive group
extensions with abelian kernel
\[
    1 \raa \tilde{F_{i}} \raa \tilde{F_{i+1}} 
    \raa \tilde{F_{i+1}}/\tilde{F_{i}} \raa 1
    \qq \text{for} \q 0 \leq i \leq 2.
\]
In the following sections, we analyse these extensions recursively. 

\begin{remark} \label{207}
For $F_0$ the situation is completely understood from chapter \ref{036}
(see corollary \ref{212}), where we have shown that the conjugacy
classes of
\[
    F_0 \iso C_{p^\alpha} \x C_d
\]
are classified by the pairs of integers $(\alpha, d)$ satisfying
\[
    0 \leq \alpha \leq k
    \qq \text{and} \qq
    1 \leq d\ |\ p^{n_\alpha}-1.
\]
\end{remark}

\section{Existence and uniqueness in cohomological terms} 

The following general approach will be applied to the $F_i$'s that can
be understood through extensions with abelian kernel.

Let $\rho: G \ra Q$ be a group homomorphism whose kernel $Ker(\rho)$ is
not necessarily supposed to be abelian. Let $A$ be an abelian normal
subgroup of $G$ which is contained in the center of $ker(\rho)$, and let
$B$ be a subgroup of $Im(\rho)$.
\[
    \Gcal_\rho(G, A, B)
    :=
    \{ H \leq G\ |\ H \cap Ker(\rho) = A 
    \text{ and } H/A \iso B \text{ via } \rho \}.
\]
When $Ker(\rho)$ is abelian, we let $e_{\rho} \in H^2(Im(\rho),
Ker(\rho))$ denote the cohomology class of the extension
\[
    1 \raa Ker(\rho) \raa G \raa Im(\rho) \raa 1,
\]
and we define $e_{\rho}(B) \in H^2(B, Ker(\rho))$ to be the image of
$e_{\rho}$ under the map 
\[
    j^\ast = H^2(j, Ker(\rho))
\]
induced by the inclusion $j$ of $B$ into $Im(\rho)$.

\begin{theorem} \label{204}
If $Ker(\rho)$ is abelian, then the set $\Gcal_\rho(G, A, B)$ is
non-empty if and only if $e_\rho(B)$ becomes trivial in $H^2(B,
Ker(\rho)/A)$.
\end{theorem}

\begin{proof}
Let $H$ be an element of $\Gcal_\rho(G, A, B)$, and let $e_H \in H^2(B,
A)$ be the extension class of
\[
    1 \raa A \raa H \raa B \raa 1.
\]
Define $H'$ to be the pushout of the diagram
\[
    \xymatrix{
        Ker(\rho)
        & A \ar[l]_{\qq i} \ar[r]
        & H
    }
\]
given by the canonical inclusions of $A$ into $Ker(\rho)$ and $H$
respectively.  Then $H'$ fits into the commutative diagram
\[
    \xymatrix
    {
        1 \ar[r] & A \ar[r] \ar[d]_{i}
        & H \ar[d] \ar[r] & B \ar[r] \ar@{=}[d] & 1 \\
        1 \ar[r] & Ker(\rho) \ar[r] 
        & H' \ar[r] & B \ar[r] & 1, 
    }
\]
where the horizontal sequences are exact; the top extension class being
$e_H$, and the bottom extension class being $i_B^\ast(e_H)$, the image
of $e_H$ via the map $i_B^\ast = H^2(B, i)$. Furthermore, define $G'$ to
be the pullback of the diagram
\[
    \xymatrix{
        G \ar[r]
        & Im(\rho)
        & B \ar[l]_{\qq j}
    }
\]
given by the canonical inclusions of $G$ and $B$ into $Im(\rho)$. Then
$G'$ fits into the commutative diagram
\[
    \xymatrix
    {
        1 \ar[r] & Ker(\rho) \ar[r] \ar@{=}[d]
        & G' \ar[d] \ar[r] & B \ar[r] \ar[d]_{j} & 1 \\
        1 \ar[r] & Ker(\rho) \ar[r] 
        & G \ar[r] & Im(\rho) \ar[r] & 1, 
    }
\]
where the horizontal sequences are exact; the top extension class being
$e_\rho(B)$, and the bottom extension class being $e_\rho$. From the
universal properties of the pushout and the pullback, the above maps
\[
    H \raa B
    \qq \text{and} \qq
    Ker(\rho) \raa G
\]
determine a homomorphism from $H'$ to $G'$ merging the above diagrams
into
\[
    \xymatrix{
        1 \ar[r] & A \ar[r] \ar[d]_{i}
        & H \ar[d] \ar[r] & B \ar[r] \ar@{=}[d] & 1 \\
        1 \ar[r] & Ker(\rho) \ar[r] \ar@{=}[d]
        & H' \ar[r] \ar[d]_{\iso} & B \ar[r] \ar@{=}[d] & 1 \\
        1 \ar[r] & Ker(\rho) \ar[r] \ar@{=}[d]
        & G' \ar[d] \ar[r] & B \ar[r] \ar[d]_{j} & 1 \\
        1 \ar[r] & Ker(\rho) \ar[r] 
        & G \ar[r] & Im(\rho) \ar[r] & 1, 
    }
\]
so that $i_B^\ast(e_H) = e_\rho(B)$. Now we have a short exact sequence
\[
    1 \raa A \overset{i}{\raa} Ker(\rho) 
    \raa Ker(\rho)/A \raa 1,
\]
which induces an exact sequence in cohomology
\[
    H^2(B, A) \overset{i_B^\ast}{\raa} H^2(B, Ker(\rho)) 
    \raa H^2(B, Ker(\rho)/A).
\]
Since
\[
    e_\rho(B) \in H^2(B, Ker(\rho))    
\]
is in the image of $i_B^\ast$, it must become trivial in $H^2(B,
Ker(\rho)/A)$.

Conversely, if $e_\rho(B)$ becomes trivial in $H^2(B, Ker(\rho)/A)$,
then there is an element $e_H$ in $H^2(B, A)$ satisfying $i_B^\ast(e_H)
= e_\rho(B)$.  This means that there is an extension
\[
    1 \raa A \raa H \raa B \raa 1,
\]
and a connecting map from the pushout $H'$ to the pullback $G'$ which
induces the commutative diagram
\[
    \xymatrix{
        1 \ar[r] & A \ar[r] \ar[d]_{i}
        & H \ar[d] \ar[r] & B \ar[r] \ar[d]^{j} & 1 \\
        1 \ar[r] & Ker(\rho) \ar[r] 
        & G \ar[r] & Im(\rho) \ar[r] & 1,
    }
\]
with $H \in \Gcal_\rho(G, A, B)$.
\end{proof}

\begin{remark} \label{205}
We may interpret theorem \ref{204} by saying that $\Gcal_\rho(G,A,B)$ is
non-empty if and only if the associated extension
\[
    1 \raa Ker(\rho)/A \raa G/A \raa Im(\rho) \raa 1
\]
splits when pulled back to $B \subset Im(\rho)$.
\end{remark}

Denote by $\Gcal_\rho(G, A, B)/\sim_{Ker(\rho)}$ the set of orbits with
respect to the conjugation action of $Ker(\rho)$ on $\Gcal_\rho(G, A,
B)$. Given a distinguished element $H_0$ in $\Gcal_\rho(G,A,B)$, we have
an action of $B$ on $Ker(\rho)/A$ induced by the conjugation action of
$G$ on $Ker(\rho)$.  Indeed, since $A$ is normal in $G$, this
conjugation action determines a homomorphism $G \ra Aut(Ker(\rho)/A)$,
which in turn descends to a homomorphism $G/A \ra Aut(Ker(\rho)/A)$ as
$A$ is in the center of $Ker(\rho)$.  We thus obtain a canonical
homomorphism
\[
    B \iso H_0/A \subset G/A \raa Aut(Ker(\rho)/A),
\]
which allows us to consider $H^1(B, Ker(\rho)/A)$.  The latter can be
identified with the set of $Ker(\rho)/A$-conjugacy classes of sections
of the split extension of remark \ref{205}, as explained in \cite{brown}
chapter IV proposition 2.3 for the abelian case and \cite{serre3}
chapter I section 5.1 (see exercise 1) for the nonabelian case.

\begin{theorem} \label{206}
If $\Gcal_\rho(G, A, B)$ is non-empty and $H_0$ is an element of
$\Gcal_\rho(G, A, B)$, then there exists a bijection
\[
    \psi_{H_0}: H^1(B, Ker(\rho)/A) 
    \raa \Gcal_\rho(G, A, B)/\sim_{Ker(\rho)},
\]
which depends on the choice of $H_0$.
\end{theorem}

\begin{proof}
Let $\Scal(B, \pi)$ denote the set of all sections $s:B \ra G/A$ of the
canonical projection $\pi: G/A \ra G/Ker(\rho)$, that is
\[
    \Scal(B, \pi)
    :=
    \{ \text{group homomorphism } s:B \ra G/A\ 
    |\ (\pi \circ s)(b) = b \text{ for all } b \in B \}.
\]
For any $s \in \Scal(B,\pi)$, we denote by $\tilde{s}:B \ra G$ the
choice of a set theoretical lift of $s$. This defines maps
\begin{align*}
    \Gcal_\rho(G, A, B) & \raa \Scal(B, \pi)\\
    H & \mtoo s: B \iso H/A \ra G/A, \\[3ex]
    \Scal(B,\pi) & \raa \Gcal_\rho(G, A, B)\\
    s & \mtoo \lan A, \tilde{s}(B) \ran,
\end{align*}
which can easily be checked to be mutually inverse to each other and
compatible with the obvious actions of $Ker(\rho)$ by conjugation. The
desired result then follows from the usual interpretation of $H^1(B,
Ker(\rho)/A)$ as conjugacy classes of sections.
\end{proof}

\section{The first extension type} 
\label{194}

In this section we consider the first extension in the chain of section
\ref{076}. Recall that for a given subgroup $\tilde{F} \in
\tilde{\Fcal}_u(\Dbb_n^\x)$, we let $\tilde{F_0}$ be such that $F_0 =
\lan \tilde{F} \cap S_n, Z_{p'}(\tilde{F} \cap \Sbb_n) \ran$.

\begin{lemma} \label{326}
Let $H_0$ be an abelian finite subgroup of $\Sbb_n$, $\tilde{H_0} = H_0
\x \lan pu \ran$, and let $\tilde{\Fcal}_u(\Qbb_p(H_0)^\x, \tilde{H_0})$
be the set of all $\tilde{F} \in \tilde{\Fcal}_u(\Qbb_p(H_0)^\x)$ such
that
\begin{itemize}
    \item $\tilde{H_0} \subset \tilde{F}$, and

    \item the valuation $v: \tilde{F} \ra \frac{1}{n}\Zbb$ induces a
    monomorphism $\tilde{F}/\tilde{H_0} \ra \frac{1}{n}\Zbb/\Zbb$.
\end{itemize}
If $\tilde{F} \in \tilde{\Fcal}_u(\Qbb_p(H_0)^\x, \tilde{H_0})$, then
$\tilde{F_0} = \tilde{H_0}$.
\end{lemma}

\begin{proof}
By assumption on $\tilde{F}$, we have in $\Gbb_n(u)$
\[
    H_0 = Ker(v: F \ra \Zbb/n) = F \cap \Sbb_n = F_0, 
\]
and therefore $\tilde{H_0} = \tilde{F_0}$.
\end{proof}

We now fix $F_0$ and analyse the set $\tilde{\Fcal}_u(\Qbb_p(F_0)^\x,
\tilde{F_0})$ of all $\tilde{F_1} \in \tilde{\Fcal}_u(\Qbb_p(F_0)^\x)$
such that
\begin{itemize} \label{307}
    \item $\tilde{F_1}$ contains $\tilde{F_0}$ as a subgroups of finite
    index, and

    \item $\tilde{F_1}/\tilde{F_0}$ injects via $v$ into
    $\frac{1}{n}\Zbb/\Zbb$.
\end{itemize}
Lemma \ref{326} ensures that the elements of
$\tilde{\Fcal}_u(\Qbb_p(F_0)^\x, \tilde{F_0})$ are extensions of the
form
\[
    1 \raa \tilde{F_0} \raa \tilde{F_1} 
    \raa \tilde{F_1}/\tilde{F_0} \raa 1
\]
with $\lan \tilde{F_1} \cap S_n, Z_{p'}(\tilde{F_1} \cap \Sbb_n) \ran =
F_0$. By definition, such an extension fits into a commutative diagram
\[
    \xymatrix
    {
        1 \ar[r] & \tilde{F_0} \ar[r] \ar[d]_{i}
        & \tilde{F_1} \ar[d] \ar[r] 
        & \tilde{F_1}/\tilde{F_0} \ar[r] \ar[d]_{j} & 1 \\
        1 \ar[r] & \Zbb_p(F_0)^\x \x \lan pu \ran \ar[r] 
        & \Qbb_p(F_0)^\x \ar[r] 
        & \frac{1}{e(\Qbb_p(F_0))} \Zbb/\Zbb \ar[r] & 1,
    }
\]
where $e(\Qbb_p(F_0))$ denotes the ramification index of $\Qbb_p(F_0)$
over $\Qbb_p$, where the horizontal maps form exact sequences and where
the vertical maps are the canonical inclusions. As 
\[
    \frac{1}{e(\Qbb_p(F_0))} \Zbb/\Zbb \iso \Zbb/e(\Qbb_p(F_0)),
\]
the quotient group $\tilde{F_1}/\tilde{F_0}$ must be cyclic of order a
divisor of $e(\Qbb_p(F_0))$. Let
\[
    e_u(F_0) \in H^2(\Zbb/e(\Qbb_p(F_0)),\ 
    \Zbb_p(F_0)^\x \x \lan pu \ran)
\]
denote the class of this last extension. Furthermore, for $r_1$ a
divisor of $e(\Qbb_p(F_0))$, we let \label{308}
\[
    \tilde{\Fcal}_u(\Qbb_p(F_0)^\x, \tilde{F_0}, r_1)
    :=
    \{ \tilde{F_1} \in \tilde{\Fcal}_u(\Qbb_p(F_0)^\x, \tilde{F_0})\
    |\ |\tilde{F_1}/\tilde{F_0}| = r_1 \},
\]
and we define the cohomology class
\[
    e_u(F_0, r_1) \in 
    H^2(\Zbb/r_1,\ \Zbb_p(F_0)^\x \x \lan pu \ran)
\]
to be the image of $e_u(F_0)$ under the induced homomorphism
\[
    j^\ast = H^2(j,\ \Zbb_p(F_0)^\x \x \lan pu \ran),
\]
for $j$ the canonical inclusion of $\tilde{F_1}/\tilde{F_0}$ into
$\frac{1}{e(\Qbb_p(F_0))}\Zbb/\Zbb$.

\begin{theorem} \label{180}
Let $r_1$ be a divisor of $e(\Qbb_p(F_0))$. 
\begin{enumerate}
    \item The set $\tilde{\Fcal}_u(\Qbb_p(F_0)^\x, \tilde{F_0}, r_1)$ is
    non-empty if and only if $e_u(F_0, r_1)$ becomes trivial in
    $H^2(\Zbb/r_1, \Zbb_p(F_0)^\x/F_0)$.

    \item If $\tilde{\Fcal}_u(\Qbb_p(F_0)^\x, \tilde{F_0}, r_1)$ is
    non-empty and if $\tilde{F_1} = \lan \tilde{F_0}, x_1 \ran$ belongs
    to this set with $v(x_1) = \frac{1}{r_1}$, then there is a bijection
    \begin{align*}
        \psi_1: H^1(\Zbb/r_1, \Zbb_p(F_0)^\x/F_0) 
        & \raa \tilde{\Fcal}_u(\Qbb_p(F_0)^\x, \tilde{F_0}, r_1) \\
        y & \mtoo \lan \tilde{F_0}, y x_1 \ran.
    \end{align*}
\end{enumerate}
\end{theorem}

\begin{proof}
Statements 1) and 2) are the respective specializations of theorem
\ref{204} and \ref{206} in the case where
\[
    \rho: G = \Qbb_p(F_0)^\x \raa \frac{1}{n}\Zbb/\Zbb = Q
\]
is induced by the valuation, $G$ (and hence $Ker(\rho)$) acts trivially
on $\tilde{\Fcal}_u(\Qbb_p(F_0)^\x, \tilde{F_0}, r_1)$, and
\[
    A = \tilde{F_0} = F_0 \x \lan pu \ran, \qq
    B = \frac{1}{r_1}\Zbb/\Zbb;
\]
in particular,
\[
    Ker(\rho) = \Zbb_p(F_0)^\x \x \lan pu \ran, \qq
    Im(\rho) = \frac{1}{e(\Qbb_p(F_0))}\Zbb/\Zbb,
\]
and
\[
    e_\rho = e_u(F_0), \qq
    e_\rho(B) = e_u(F_0,r_1).
\]
\end{proof}

\begin{remark} \label{181}
Note that $\tilde{F_1} \in \tilde{\Fcal}_u(\Qbb_p(F_0)^\x)$ belongs to
$\tilde{\Fcal}_u(\Qbb_p(F_0)^\x, \tilde{F_0}, r_1)$ if and only if there
exists an element $x_1 \in \Qbb_p(F_0)^\x$ with $\tilde{F_1} = \lan
\tilde{F_0}, x_1 \ran$ satisfying
\[
    v(x_1)=\frac{1}{r_1}
    \qq \text{and} \qq
    x_1^{r_1} \in \tilde{F_0}.
\]
Clearly, $\tilde{F_1}$ uniquely determines $x_1$ modulo $F_0$.
\end{remark}

\begin{corollary} \label{183}
If $F_0 = \mu(\Qbb_p(F_0))$ is the group of roots of unity in
$\Qbb_p(F_0)$, then 
\[
    |\tilde{\Fcal}_u(\Qbb_p(F_0)^\x, \tilde{F_0}, r_1)| \leq 1.
\]
\end{corollary}

\begin{proof}
By proposition \ref{341} we have $\Zbb_p(F_0)^\x \iso \mu(\Qbb_p(F_0))
\x \Zbb_p^{[\Qbb_p(F_0):\Qbb_p]}$. Since the action of $\Zbb/r_1$ is
trivial on $\Zbb_p(F_0)^\x$, we obtain
\begin{align*}
    H^1(\Zbb/r_1,\ \Zbb_p(F_0)^\x / F_0)\
    &\iso\ 
    H^1(\Zbb/r_1,\ \Zbb_p^{[\Qbb_p(F_0): \Qbb_p]} 
    \x \mu(\Qbb_p(F_0))/F_0)\\
    &\iso\ H^1(\Zbb/r_1,\ \mu(\Qbb_p(F_0))/F_0)\\
    &\iso\ \{1\}.
\end{align*}
The result then follows from theorem \ref{180}.2.
\end{proof}

\begin{remark} \label{184}
The condition $F_0 = \mu(\Qbb_p(F_0))$ is equivalent to the maximality
of $F_0$ as a finite subgroup of $\Qbb_p(F_0)^\x$. In section \ref{228}
we will see that if $p$ is odd and $F_0 = \mu(\Qbb_p(F_0))$, then the
set $\tilde{\Fcal}_u(\Qbb_p(F_0), \tilde{F_0}, r_1)$ is non-empty if and
only if $p$ does not divide $r_1$. As for the case $p=2$, we will see in
section \ref{229} that this depends on $u$ and $F_0$.
\end{remark}

\section{The second extension type} 
\label{195}

In this section we consider the second extension in the chain of section
\ref{076}. Recall that for a given subgroup $\tilde{F} \in
\tilde{\Fcal}_u(\Dbb_n^\x)$, we let $\tilde{F_1} = \tilde{F} \cap
\Qbb_p(F_0)^\x$.

\begin{lemma} \label{327}
Let $H_0$ be an abelian finite subgroup of $\Sbb_n$, $\tilde{H_0} =
H_0 \x \lan pu \ran$, $\tilde{H_1} \in \tilde{\Fcal}_u(\Qbb_p(H_0)^\x,
\tilde{H_0})$, and let $\tilde{\Fcal}_u(C_{\Dbb_n^\x}(H_0),
\tilde{H_1})$ be the set of all $\tilde{F} \in
\tilde{\Fcal}_u(C_{\Dbb_n^\x}(H_0))$ such that
\begin{itemize}
    \item $\tilde{H_1} = \tilde{F} \cap \Qbb_p(H_0)^\x$, and

    \item the valuation $v: \tilde{F} \ra \frac{1}{n}\Zbb$ induces a
    monomorphism $\tilde{F}/\tilde{H_1} \ra
    \frac{1}{n}\Zbb/v(\tilde{H_1})$.
\end{itemize}
If $\tilde{F} \in \tilde{\Fcal}_u(C_{\Dbb_n^\x}(H_0), \tilde{H_1})$,
then $\tilde{F_1} = \tilde{H_1}$.
\end{lemma}

\begin{proof}
Clearly $\tilde{F}/\tilde{H_1}$ injects via $v$ into
$\frac{1}{n}\Zbb/v(\tilde{H_1})$ if and only if we have in $\Gbb_n(u)$ a
monomorphism $F/H_0 \ra \Zbb/n$, and this is true if and only if $H_0 =
F \cap \Sbb_n$. Therefore $F_0 = H_0$ and consequently
\[
    \tilde{F_1}
    = \tilde{F} \cap \Qbb_p(F_0)^\x
    = \tilde{F} \cap \Qbb_p(H_0)^\x
    = \tilde{H_1}.
\]
\end{proof}

We now fix $F_0$ and $r_1$ such that $\tilde{\Fcal}_u(\Qbb_p(F_0),
\tilde{F_0}, r_1)$ is non-empty, and fix a group $\tilde{F_1} \in
\tilde{\Fcal}_u(\Qbb_p(F_0), \tilde{F_0}, r_1)$. We consider the set
\label{309} $\tilde{\Fcal}_u(C_{\Dbb_n^\x}(F_0), \tilde{F_1})$ of all
$\tilde{F_2} \in \tilde{\Fcal}_u(C_{\Dbb_n^\x}(F_0))$ such that
\begin{itemize}
    \item $\tilde{F_2} \cap \Qbb_p(F_0)^\x = \tilde{F_1}$, and

    \item $\tilde{F_2}/\tilde{F_1}$ injects via $v$ into
    $\frac{1}{n}\Zbb/v(\tilde{F_1})$.
\end{itemize}
Lemma \ref{327} ensures that the elements of
$\tilde{\Fcal}_u(C_{\Dbb_n^\x}(F_0), \tilde{F_1})$ are extensions of the
form
\[
    1 \raa \tilde{F_1} \raa \tilde{F_2} 
    \raa \tilde{F_2}/\tilde{F_1} \raa 1
\]
with $\tilde{F_2} \cap \Qbb_p(F_0)^\x = \tilde{F_1}$ and $\lan
\tilde{F_2} \cap S_n, Z_{p'}(\tilde{F_2} \cap \Sbb_n) \ran = F_0$. For
$r_2$ a divisor of $n/r_1$, we define \label{310}
\[
    \tilde{\Fcal}_u(C_{\Dbb_n^\x}(F_0), \tilde{F_1}, r_2)
    :=
    \{ \tilde{F_2} \in \tilde{\Fcal}_u(C_{\Dbb_n^\x}(F_0), 
    \tilde{F_1})\ |\ |\tilde{F_2}/\tilde{F_1}| = r_2 \}.
\]
Note that any $\tilde{F_2} \in \tilde{\Fcal}_u(C_{\Dbb_n^\x}(F_0),
\tilde{F_1}, r_2)$ determines a commutative field extension $L =
\Qbb_p(\tilde{F_2})$ of degree $r_2$ over $\Qbb_p(F_0)$ which is
obtained by adjoining to $\Qbb_p(F_0)$ an element $x_2 \in
C_{\Dbb_n^\x}(F_0)$ which satisfies
\[
    v(x_2) = \frac{1}{r_1r_2}
    \qq \text{and} \qq
    x_2^{r_2} \in \tilde{F_1}.
\]
We can thus partition our sets
\[
    \tilde{\Fcal}_u(C_{\Dbb_n^\x}(F_0), \tilde{F_1}, r_2)
    = \coprod_{\stackrel{L \supset \Qbb_p(F_0)}{[L:\Qbb_p(F_0)]=r_2}}
    \tilde{\Fcal}_u(C_{\Dbb_n^\x}(F_0), \tilde{F_1}, L),
\]
according to all $L \supset \Qbb_p(F_0)$ obtained from $\Qbb_p(F_0)$ via
irreducible equations of the form $X^{r_2}-x_1$ for $x_1$ an element of
valuation $\frac{1}{r_1}$ in $\tilde{F_1}$, where \label{311}
\[
    \tilde{\Fcal}_u(C_{\Dbb_n^\x}(F_0), \tilde{F_1}, L)
    :=
    \{ \tilde{F_2} \in \tilde{\Fcal}_u(C_{\Dbb_n^\x}(F_0), 
    \tilde{F_1})\ |\ \Qbb_p(\tilde{F_2}) = L \}.
\]
Clearly, $L$ determines $r_2$ and we have
\begin{align*}
    \tilde{\Fcal}_u(C_{\Dbb_n^\x}(F_0), \tilde{F_1})\
    &=\ \coprod_{r_2 | \frac{n}{r_1}}
    \tilde{\Fcal}_u(C_{\Dbb_n^\x}(F_0), \tilde{F_1}, r_2)\\
    &=\ \coprod_{[L:\Qbb_p(F_0)] | \frac{n}{r_1}}
    \tilde{\Fcal}_u(C_{\Dbb_n^\x}(F_0), \tilde{F_1}, L).
\end{align*}

\begin{theorem} \label{185}
Let $x_1$ be an element of $\tilde{F_1} \in \tilde{\Fcal}_u(\Qbb_p(F_0),
\tilde{F_0}, r_1)$ with $v(x_1) = \frac{1}{r_1}$, let $L$ be an
extension of $\Qbb_p(F_0)$ of degree $r_2$, and let $L_{r_1}^\x$ denote
the group of all $x \in L^\x$ such that $v(x) \in \frac{1}{r_1}\Zbb$.  
\begin{enumerate}
    \item The set $\tilde{\Fcal}_u(C_{\Dbb_n^\x}(F_0), \tilde{F_1}, L)$
    is non-empty if and only if $r_2[\Qbb_p(F_0):\Qbb_p]$ divides $n$,
    there exists a $\delta \in F_0$ such that the equation
    $X^{r_2}-\delta x_1$ is irreducible over $\Qbb_p(F_0)$ and $L =
    \Qbb_p(x_2)$ for $x_2$ a root of this equation.

    \item If $\tilde{\Fcal}_u(C_{\Dbb_n^\x}(F_0), \tilde{F_1}, L)$ is
    non-empty and if $\tilde{F_2} = \lan \tilde{F_1}, x_2 \ran$ belongs
    to this set with $v(x_2) = \frac{1}{r_1r_2}$, then there is a
    bijection
    \begin{align*}
        \psi_2: H^1(\Zbb/r_2, L_{r_1}^\x/\tilde{F_1}) 
        & \raa \tilde{\Fcal}_u(C_{\Dbb_n^\x}(F_0), \tilde{F_1}, L) \\
        y & \mtoo \lan \tilde{F_1}, y x_2 \ran.
    \end{align*}
\end{enumerate}
\end{theorem}

\begin{proof}
1) This is a direct consequence of the embedding theorem.

2) This is a specialization of theorem \ref{206} in the case where
\[
    \rho: G = L^\x \raa \frac{1}{n}\Zbb/\frac{1}{r_1}\Zbb = Q
\]
is induced by the valuation, $G$ (and hence $Ker(\rho)$) acts trivially
on $\tilde{\Fcal}_u(C_{\Dbb_n^\x}(F_0), \tilde{F_1}, L)$, and
\[
    A = \tilde{F_1}, \qq
    B = \frac{1}{r_1r_2}\Zbb/\frac{1}{r_1}\Zbb;
\]
in particular, $Ker(\rho) = L_{r_1}^\x$.
\end{proof}

\begin{remark} \label{186}
Note that 
\[
    \tilde{F_2} \in \tilde{\Fcal}_u(C_{\Dbb_n^\x}(F_0), \tilde{F_1})
\]
satisfies $\Qbb_p(\tilde{F_2}) = L$ if and only if there exists an
element $x_2 \in L$ with $\tilde{F_2} = \lan \tilde{F_1}, x_2 \ran$
satisfying
\[
    v(x_2)=\frac{1}{r_1r_2}
    \qq \text{and} \qq
    x_2^{r_2} \in \tilde{F_1}.
\]
Moreover, $\tilde{F_2}$ uniquely determines such an $x_2$ modulo $F_0$.
\end{remark}

\begin{corollary} \label{188}
If $F_0 = \mu(L)$ is the group of roots of unity in $L$, then 
\[
    |\tilde{\Fcal}_u(C_{\Dbb_n^\x}(F_0), \tilde{F_1}, L)| \leq 1.
\]
\end{corollary}

\begin{proof}
We know from proposition \ref{341} that $L_{r_1}^\x \iso \Zbb\lan x_1
\ran \x \mu(L) \x \Zbb_p^{[L:\Qbb_p]}$. Since the action of $\Zbb/r_2$
is trivial on $L_{r_1}^\x$, we obtain
\begin{align*}
    H^1(\Zbb/r_2,\ L_{r_1}^\x / \tilde{F_1})\
    &\iso\ 
    H^1(\Zbb/r_2,\ \Zbb_p^{[L: \Qbb_p]} \x \mu(L)/F_0)\\
    &\iso\ H^1(\Zbb/r_2,\ \mu(L)/F_0)\\
    &\iso\ \{1\}.
\end{align*}
The result then follows from theorem \ref{185}.2.
\end{proof}

\section{The third extension type} 
\label{150}

In this section we consider the third extension in the chain of section
\ref{076}. Recall that for a given subgroup $\tilde{F} \in
\tilde{\Fcal}_u(\Dbb_n^\x)$, we let $\tilde{F_2} = \tilde{F} \cap
C_{\Dbb_n^\x}(F_0)$.

\begin{lemma} \label{328}
Let $H_0$ be an abelian finite subgroup of $\Sbb_n$, $\tilde{H_0} = H_0
\x \lan pu \ran$, $\tilde{H_1} \in \tilde{\Fcal}_u(\Qbb_p(H_0)^\x,
\tilde{H_0})$, $\tilde{H_2} \in \tilde{\Fcal}_u(C_{\Dbb_n^\x}(H_0),
\tilde{H_1})$, and let $\tilde{\Fcal}_u(N_{\Dbb_n^\x}(H_0),
\tilde{H_2})$ be the set of all $\tilde{F} \in
\tilde{\Fcal}_u(N_{\Dbb_n^\x}(H_0))$ such that
\begin{itemize}
    \item $\tilde{H_2} = \tilde{F} \cap C_{\Dbb_n^\x}(H_0)$, and

    \item $\tilde{H_2}$ is normal in $\tilde{F}$.
\end{itemize}
If $\tilde{F} \in \tilde{\Fcal}_u(N_{\Dbb_n^\x}(H_0), \tilde{H_2})$,
then
\begin{description}
    \item[\q \md a)] $\tilde{F_i} = \tilde{H_i}$ for $0 \leq i \leq 2$,
    or

    \item[\q \md b)] $p=2$, $n \equiv 2 \mod 4$, $H_0 \cap S_n \iso
    C_4$, $\tilde{F} \cap S_n \iso Q_8$ and $Z_{p'}(\tilde{F} \cap
    \Sbb_n) \iso Z_{p'}(\tilde{H_0} \cap \Sbb_n)$.
\end{description}
\end{lemma}

\begin{proof}
First note that the condition $\tilde{H_2} = \tilde{F} \cap
C_{\Dbb_n^\x}(H_0)$ implies that the canonical homomorphism $\tilde{F}
\ra Aut(H_0)$ induces an injective homomorphism $\tilde{F}/\tilde{H_2}
\ra Aut(H_0)$. In particular, we have a monomorphism $F \cap \Sbb_n /
H_0 \ra Aut(H_0)$ where
\[
    (H_0 \cap S_n) \x Z_{p'}(H_0)
    = H_0
    \subset F \cap \Sbb_n,
\]
for $Z_{p'}$ the $p'$-part of (the center of) $H_0$. By theorem
\ref{024} and \ref{032} we know that $F \cap \Sbb_n$ acts trivially on
$Z_{p'}(H_0)$, so that we have a monomorphism 
\begin{align}
    F \cap \Sbb_n / H_0 \raa Aut(H_0 \cap S_n). \tag{$\ast$}
\end{align}

Assume for the moment that $F \cap S_n$ is abelian; this must be the
case if $p>2$, or if $p=2$ with either $n \not\equiv 2 \mod 4$ or $H_0
\cap S_n \not\iso C_4$. Then $F \cap \Sbb_n = (F \cap S_n) \x Z_{p'}(F
\cap \Sbb_n)$ is cyclic. Since $(F \cap S_n) / (H_0 \cap S_n)$ injects
into the kernel of the injective map $(\ast)$, we have $F \cap S_n = H_0
\cap S_n$. Furthermore, the $p'$-part of $Aut(H_0 \cap S_n)$ is a cyclic
group of order $p-1$, and the quotient group $Z_{p'}(F \cap \Sbb_n) /
Z_{p'}(H_0)$ injects into $C_{p-1} \subset Aut(H_0 \cap S_n)$. Hence
$Z_{p'}(F \cap \Sbb_n) = Z_{p'}(H_0)$ by theorem \ref{024} and
\ref{032}, so that $F_0 = H_0$ and $\tilde{F_0} = \tilde{H_0}$.
Therefore
\[
    \tilde{F_2}
    = \tilde{F} \cap C_{\Dbb_n^\x}(F_0)
    = \tilde{F} \cap C_{\Dbb_n^\x}(H_0)
    = \tilde{H_2},
\]
and
\[
    \tilde{F_1}
    = \tilde{F} \cap \Qbb_p(F_0)^\x
    = \tilde{F} \cap \Qbb_p(H_0)^\x
    = \tilde{F} \cap C_{\Dbb_n^\x}(H_0) \cap \Qbb_p(H_0)^\x
    = \tilde{H_2} \cap \Qbb_p(H_0)^\x
    = \tilde{H_1}.
\]

Finally, if $F \cap S_n$ is not abelian, then $p=2$, $n \equiv 2 \mod
4$, $H_0 \cap S_n \iso C_4$ and $F \cap S_n \iso Q_8$ by theorem
\ref{032}.  As seen above, the quotient group of the $2'$-part of $F
\cap \Sbb_n$ by the $2'$-part of $H_0$ injects into the trivial group.
\end{proof}

We now fix a chain $F_0 \subset F_1 \subset F_2$ such that condition b)
of lemma \ref{328} is not satisfied, and we let $L$ be a subfield of
$\Dbb_n$ such that $\tilde{F_2}$ belongs to
$\tilde{\Fcal}_u(C_{\Dbb_n^\x}(F_0), \tilde{F_1}, L)$; in particular $L
= \Qbb_p(\tilde{F_2})$. We consider the set \label{312}
$\tilde{\Fcal}_u(N_{\Dbb_n^\x}(F_0), \tilde{F_2})$ of all $\tilde{F_3}
\in \tilde{\Fcal}_u(N_{\Dbb_n^\x}(F_0))$ such that
\begin{itemize}
    \item $\tilde{F_3} \cap C_{\Dbb_n^\x}(F_0) = \tilde{F_2}$, and

    \item $\tilde{F_2}$ is normal in $\tilde{F_3}$.
\end{itemize}

\begin{proposition} \label{329}
If $\tilde{F_3} \in \tilde{\Fcal}_u(N_{\Dbb_n^\x}(F_0), \tilde{F_2})$,
there is a commutative diagram of obvious group homomorphisms
\[
    \xymatrix{
        \tilde{F_3}/\tilde{F_2} \ar@{=}[d] \ar[r]
        & Aut(\Qbb_p(\tilde{F_2}), \Qbb_p(F_0)) \ar[r]
        & Aut(\Qbb_p(F_0)) \ar[d] \\
        \tilde{F_3}/\tilde{F_2} \ar[r]
        & Aut(\tilde{F_2}, F_0) \ar[r]
        & Aut(F_0)
    }
\]
in which all compositions starting at $\tilde{F_3}/\tilde{F_2}$ are
injective.
\end{proposition}

\begin{proof}
Clearly, the condition $\tilde{F_2} = \tilde{F_3} \cap
C_{\Dbb_n^\x}(F_0)$ is equivalent to the fact that the canonical
homomorphism $\tilde{F_3} \ra Aut(F_0)$ induces an injective
homomorphism $\tilde{F_3}/\tilde{F_2} \ra Aut(F_0)$.  Furthermore, an
automorphism of the field $\Qbb_p(F_0)$ induces an automorphism of the
group $\mu(\Qbb_p(F_0))$ of roots of unity in $\Qbb_p(F_0)$, and since
this group is cyclic and contains $F_0$, it also induces an automorphism
of $F_0$. This determines an injective homomorphism $Aut(\Qbb_p(F_0))
\ra Aut(F_0)$. The homomorphism $\tilde{F_3}/\tilde{F_2} \ra Aut(F_0)$
clearly takes its values into the subgroup $Aut(\Qbb_p(F_0))$.

The condition that $\tilde{F_2}$ is normal in $\tilde{F_3}$ yields
canonical homomorphisms $\tilde{F_3} \ra Aut(\tilde{F_2})$ and
$\tilde{F_3} \ra Aut(\Qbb_p(\tilde{F_2}))$. Since $\tilde{F_2}$ is
abelian, these induce canonical homomorphisms $\tilde{F_3}/\tilde{F_2}
\ra Aut(\tilde{F_2})$ and $\tilde{F_3}/\tilde{F_2} \ra
Aut(\Qbb_p(\tilde{F_2}))$. Moreover, as $\tilde{F_3}/\tilde{F_2} \ra
Aut(\tilde{F_2})$ takes its values into the subgroup $Aut(\tilde{F_2},
F_0)$ of those automorphisms of $\tilde{F_2}$ which leave $F_0$
invariant, and as $\tilde{F_3}/\tilde{F_2} \ra Aut(\Qbb_p(\tilde{F_2}))$
takes its values into the subgroup $Aut(\Qbb_p(\tilde{F_2}),
\Qbb_p(F_0))$ of those automorphisms which leave $\Qbb_p(F_0)$
invariant, we end up with the given commutative diagram.

From the injectivity of the map $\tilde{F_3}/\tilde{F_2} \ra Aut(F_0)$,
we then obtain that all compositions of homomorphism in the diagram
starting at $\tilde{F_3}/\tilde{F_2}$ are injective.
\end{proof}

Let $Aut(L, \tilde{F_2}, F_0)$ denote the subgroup of all elements of
$Aut(L)$ which leave both $\tilde{F_2}$ and $F_0$ invariant. By
proposition \ref{329}, we may partition the set
\[
    \tilde{\Fcal}_u(N_{\Dbb_n^\x}(F_0), \tilde{F_2})
    = \coprod_{W}
    \tilde{\Fcal}_u(N_{\Dbb_n^\x}(F_0), \tilde{F_2}, W)
\]
according to all subgroups $W$ of $Aut(F_0)$ which lift to $Aut(L,
\tilde{F_2}, F_0)$, where \label{313}
\[
    \tilde{\Fcal}_u(N_{\Dbb_n^\x}(F_0), \tilde{F_2}, W)
    :=
    \{ \tilde{F_3} \in \tilde{\Fcal}_u(N_{\Dbb_n^\x}(F_0), 
    \tilde{F_2})\ |\ \tilde{F_3} / \tilde{F_2} = W \}.
\]
Let us fix such a $W$. Under our assumptions, lemma \ref{328} and
proposition \ref{329} ensure that the elements of
$\tilde{\Fcal}_u(N_{\Dbb_n^\x}(F_0), \tilde{F_2}, W)$ are extensions of
the form
\[
    1 \raa \tilde{F_2} \raa \tilde{F_3} \raa W \raa 1
\]
with $\tilde{F_3} \cap C_{\Dbb_n^\x}(F_0) = \tilde{F_2}$, $\tilde{F_3}
\cap \Qbb_p(F_0)^\x = \tilde{F_1}$ and $\lan \tilde{F_3} \cap S_n,
Z_{p'}(\tilde{F_3} \cap \Sbb_n) \ran = F_0$. Define 
\[
    K := L^W \subset L = \Qbb_p(\tilde{F_2})
\]
to be the subfield of all elements of $L$ that are fixed by the action
of $W$. Clearly, $K$ is an extension of $\Qbb_p$ and the respective
dimensions of $K$ and $L$ over $\Qbb_p$ divide $n$. Recall from section
\ref{338} that an element $e \in H^2(W,L^\x)$ defines a central simple
crossed $K$-algebra $(L/K, e)$ up to isomorphism.

\begin{lemma} \label{330}
There is a generator of $H^2(W,L^\x)$ whose associated crossed algebra
embeds into $\Dbb_n$ if and only if $|W|$ is prime to
$n[L:\Qbb_p]^{-1}$. 
\end{lemma}

\begin{proof}
Consider the tower of extensions $\Qbb_p \subset K := L^W \subset L$.
Let $k := [K:\Qbb_p]$, $l := [L:\Qbb_p]$, $w := |W|$, and let $e$ be a
generator of $H^2(W,L) \iso \Zbb/w \subset \Qbb/\Zbb$. By proposition
\ref{331}, we know that the crossed algebra $(L/K, e) = \sum_{\sigma \in
W} Lu_\sigma$ is a central division algebra over $K$ of invariant $r/w
\in Br(K)$ for some integer $r$ prime to $w$. If $q = n/l$, then the
invariant of $D := C_{\Dbb_n}(K)$ is $1/qw \in Br(K)$ by proposition
\ref{014}.

Suppose $(L/K, e)$ can be embedded into $\Dbb_n$. Then it embeds into
$D$, and by the centralizer theorem
\[
    D
    \iso (L/K, e) \ox_K C_{D}(L/K, e),
\]
where $C_{D}(L/K, e)$ is central of dimension $q^2$ over $K$. On the
level of Hasse invariants we get a relation of the form
\begin{align}
    \frac{1}{qw} \equiv \frac{r}{w} + \frac{s}{q} \q \mod \Zbb
    \tag{$\ast$}
\end{align}
for a suitable integer $s$ which is prime to $q$. Hence
\[
    1 \equiv rq+sw \q \mod q\Zbb,
\]
and it follows that $q$ is prime to $w$.  Conversely if $q$ is prime to
$w$, then there is an $r$ prime to $w$ and an $s$ prime to $q$ such that
$(\ast)$ holds. Therefore the algebra $(L/K, e)$ embeds into $D$, and
consequently into $\Dbb_n$.
\end{proof}

\begin{theorem} \label{190}
Let $W$ be a subgroup of $Aut(F_0)$ which lifts to $Aut(L, \tilde{F_2},
F_0)$ and let 
\[
    i_W^\ast: H^2(W, \tilde{F_2}) \raa H^2(W, L^\x)
\]
be the map induced by the inclusion of $\tilde{F_2}$ into $L^\x$.  Then
$\tilde{\Fcal}_u(N_{\Dbb_n^\x}(F_0), \tilde{F_2}, W)$ is non-empty if
and only if $|W|$ is prime to $n[L:\Qbb_p]^{-1}$ and $i_W^\ast$ is
surjective.
\end{theorem}

\begin{proof}
Suppose that $|W|$ is prime to $n[L:\Qbb_p]^{-1}$ and $i_W^\ast$ is
surjective. By lemma \ref{330}, there is a generator $e \in H^2(W,L^\x)$
whose associated algebra $(L/L^W, e)$ embeds into $\Dbb_n$. The group of
units $(L/L^W, e)^\x$ contains
\[
    L^\x W = \coprod_{\sigma \in W} L^\x u_\sigma
\]
as a subgroup, and we get an embedding of $L^\x W$ into $\Dbb_n^\x$. By
the Skolem-Noether theorem we can assume that this embedding restricts
to the given embedding of $L$ into $\Dbb_n$. Since $L =
\Qbb_p(\tilde{F_2})$, we have $L^\x W \subset
N_{\Dbb_n^\x}(\tilde{F_2})$ and there is a commutative diagram
\[
    \xymatrix{
        1 \ar[r] 
        & L^\x \ar[r] \ar[d] 
        & L^\x W \ar[r] \ar[d] 
        & W \ar[r] \ar[d] & 1 \\
        1 \ar[r] 
        & C_{\Dbb_n^\x}(\tilde{F_2}) \ar[r] 
        & N_{\Dbb_n^\x}(\tilde{F_2}) \ar[r] 
        & Aut(L) \ar[r] & 1,
    }
\]
whose vertical maps are inclusions and whose horizontal sequences are
exact. Now, the surjectivity of $i_W^\ast$ implies the existence of an
element $e' \in H^2(W,\tilde{F_2})$ such that $i_W^\ast(e') = e$, in
which case the above diagram extends to a commutative diagram
\[
    \xymatrix{
        1 \ar[r] 
        & \tilde{F_2} \ar[r] \ar[d]
        & \tilde{F} \ar[r] \ar[d]
        & W \ar[r] \ar@{=}[d] & 1 \\
        1 \ar[r] 
        & L^\x \ar[r] \ar[d] 
        & L^\x W \ar[r] \ar[d] 
        & W \ar[r] \ar[d] & 1 \\
        1 \ar[r] 
        & C_{\Dbb_n^\x}(\tilde{F_2}) \ar[r] 
        & N_{\Dbb_n^\x}(\tilde{F_2}) \ar[r] 
        & Aut(L) \ar[r] & 1,
    }
\]
where the top exact sequence has extension class $e'$. Because of our
assumption that $W$ injects into $Aut(F_0)$, we have $\tilde{F} \cap
C_{\Dbb_n^\x}(F_0) = \tilde{F_2}$, and therefore $\tilde{F} \in
\tilde{\Fcal}_u(N_{\Dbb_n^\x}(F_0), \tilde{F_2}, W)$.
 
Conversely, if $\tilde{F} \in \tilde{\Fcal}_u(N_{\Dbb_n^\x}(F_0),
\tilde{F_2}, W)$, then $\tilde{F}$ extends $\tilde{F_2}$ by $W$ and
there are commutative diagrams 
\[
    \xymatrix{
        1 \ar[r] 
        & \tilde{F_2} \ar[r] \ar[d]
        & \tilde{F} \ar[r] \ar[d]
        & W \ar[r] \ar[d] & 1 \\
        1 \ar[r] 
        & C_{\Dbb_n^\x}(\tilde{F_2}) \ar[r] 
        & N_{\Dbb_n^\x}(\tilde{F_2}) \ar[r] 
        & Aut(L) \ar[r] & 1,
    }
\]
and
\[
    \xymatrix{
        1 \ar[r] 
        & \tilde{F_2} \ar[r] \ar[d]
        & \tilde{F} \ar[r] \ar[d]
        & W \ar[r] \ar@{=}[d] & 1 \\
        1 \ar[r] 
        & L^\x \ar[r]
        & L^\x W \ar[r]
        & W \ar[r] & 1,
    }
\]
whose vertical maps are inclusions and horizontal sequences are exact.
Using the universal property of the lower left pushout square, we may
extend the latter diagram in an obvious way to obtain an embedding of
extensions
\[
    \xymatrix{
        1 \ar[r] 
        & \tilde{F_2} \ar[r] \ar[d]
        & \tilde{F} \ar[r] \ar[d]
        & W \ar[r] \ar@{=}[d] & 1 \\
        1 \ar[r] 
        & L^\x \ar[r] \ar[d] 
        & L^\x W \ar[r] \ar[d]^{i} 
        & W \ar[r] \ar[d] & 1 \\
        1 \ar[r] 
        & C_{\Dbb_n^\x}(\tilde{F_2}) \ar[r] 
        & N_{\Dbb_n^\x}(\tilde{F_2}) \ar[r] 
        & Aut(L) \ar[r] & 1,
    }
\]
where $L^\x W \subset (L/L^W, e) = \sum_{\sigma \in W} L u_{\sigma}$ for
$e$ the image of the extension class of $\tilde{F}$ in $H^2(W,L^\x)$. By
definition of $L^\x W$, the map $i$ extends uniquely to an algebra
homomorphism 
\[
    \tilde{i}: (L/L^W, e) \raa \Dbb_n 
    : \sum_\sigma x_\sigma u_\sigma 
    \mtoo \sum_\sigma i(x_\sigma u_\sigma),
    \qq x_\sigma \in L^\x.
\]
Moreover since $(L/L^W, e)$ is simple and $i$ is non-trivial, the kernel
of $\tilde{i}$ is trivial. Hence $\tilde{i}$ is injective and $(L/L^W,
e)$, which embeds into $\Dbb_n$, is a division algebra by proposition
\ref{006}. It follows that $e$ is a generator of $H^2(W,L^\x)$ and
$i_W^\ast$ is surjective. Applying lemma \ref{330} we finally obtain
that $|W|$ is prime to $n[L:\Qbb_p]^{-1}$.
\end{proof}

\begin{theorem} \label{193}
Let $W$ be a subgroup of $Aut(F_0)$ which lifts to $Aut(L, \tilde{F_2},
F_0)$. If the set $\tilde{\Fcal}_u(N_{\Dbb_n^\x}(F_0), \tilde{F_2}, W)$
is non-empty and contains $\tilde{F_3}$, then there is a bijection
\begin{align*}
    \psi_3: H^1(W, C_{\Dbb_n^\x}(\tilde{F_2})/\tilde{F_2})
    & \raa \tilde{\Fcal}_u(N_{\Dbb_n^\x}(F_0), \tilde{F_2}, W)
    /\sim_{C_{\Dbb_n^\x}(\tilde{F_2})} \\
    c & \mtoo \lan \tilde{F_2}, cs_{\tilde{F_3}} \ran,
\end{align*}
for $c$ a cocycle and $s_{\tilde{F_3}}: W \ra \tilde{F_3}$ a set
theoretic section of the epimorphism $\tilde{F_3} \ra W$.
\end{theorem}

\begin{proof}
By proposition \ref{329}, we know that $W$ lifts to an automorphism of
$\tilde{F_2}$. The result is then a specialization of theorem \ref{206}
in the case where
\[
    \rho: G = N_{\Dbb_n^\x}(\tilde{F_2}) 
    \raa Aut(\tilde{F_2}) = Q
\]
is given by the canonical homomorphism induced by conjugation and
\[
    A = \tilde{F_2}, \qq B = W;
\]
in particular, $Ker(\rho)=C_{\Dbb_n^\x}(\tilde{F_2})$.
\end{proof}

\begin{corollary} \label{192}
If $\Qbb_p(F_0)$ is a maximal subfield of $\Dbb_n$ such that
$\mu(\Qbb_p(F_0))=F_0$, and if $i_W^\ast: H^2(W, \tilde{F_2}) \ra H^2(W,
L^\x)$ is an epimorphism for $W$ a subgroup of $Aut(F_0)$ which lifts to
$Aut(L, \tilde{F_2}, F_0)$, then there is a bijection between the
conjugacy classes of elements of $\tilde{\Fcal}_u(N_{\Dbb_n^\x}(F_0),
\tilde{F_2}, W)$ and the kernel of $i_W^\ast$.
\end{corollary}

\begin{proof}
Under the stated assumptions, we have $\tilde{F_2} = \tilde{F_1}$, $L =
\Qbb_p(F_0)$, $[L:\Qbb_p]=n$ and $C_{\Dbb_n^\x}(\tilde{F_2}) = L^\x$.
Hence there is a short exact sequence
\[
    1 \raa \tilde{F_2} \raa L^\x \raa L^\x/\tilde{F_2} \raa 1,
\]
which for $W \subset Aut(L,\tilde{F_2}) \subset Gal(L/\Qbb_p)$ induces
the long exact sequence
\[
    \xymatrix{
    \ldots \ar[r]
    & H^1(W, L^\x) \ar[r]
    & H^1(W, L^\x/\tilde{F_2}) \ar[d] 
    & Br(L/L^W) \ar@{=}[d] \\
    & 0 \ar@{=}[u]
    & H^2(W, \tilde{F_2}) \ar[r]^{i_W^\ast}
    & H^2(W, L^\x) \ar[r] 
    & \ldots,
    }
\]
where the left hand term is trivial by Hilbert's theorem 90. The group
$H^1(W, L^\x/\tilde{F_2})$ is therefore the kernel of $i_W^\ast$ and the
result follows from theorem \ref{193}.
\end{proof}

\begin{theorem} \label{332}
Let $\tilde{H} \in \tilde{\Fcal}_u(\Dbb_n^\x)$ be such that $H \cap S_n$ is
abelian and $\mu(\Qbb_p(H_0)) = \mu(\Qbb_p(\tilde{H_2}))$. Then there is
a subgroup $\tilde{F} \in \tilde{\Fcal}_u(\Dbb_n^\x)$ such that 
\[
    F_0 = \mu(\Qbb_p(H_0)), \qq 
    \Qbb_p(\tilde{F_2}) = \Qbb_p(\tilde{H_2})
    \qq \text{and} \qq
    \tilde{H_i} \subset \tilde{F_i}\ \text{ for } 0 \leq i \leq 3.
\]
\end{theorem}

\begin{proof}
We know that $\tilde{H_1} = \lan \tilde{H_0}, x_1 \ran$ where $x_1$
commutes with $\tilde{H_0}$, $v(x_1) = \frac{1}{r_1}$ and $x_1^{r_1} \in
\tilde{H_0}$, and furthermore that $\tilde{H_2} = \lan \tilde{H_1}, x_2
\ran$ where $x_2$ commutes with $\tilde{H_1}$, $v(x_2) =
\frac{1}{r_1r_2}$ and $x^{r_2} \in \tilde{H_1}$. Defining $F_0 =
\mu(\Qbb_p(H_0))$, $\tilde{F_0} = \lan F_0, pu \ran$, $\tilde{F_1} =
\lan \tilde{F_0}, x_1 \ran$ and $\tilde{F_2} = \lan \tilde{F_1}, x_2
\ran$, we have
\[
    F_0 = \tilde{F_2} \cap \Sbb_n, \qq
    \tilde{F_1} = \tilde{F_2} \cap \Qbb_p(F_0)
    \qq \text{and} \qq
    \tilde{F_2} \subset C_{\Dbb_n^\x}(F_0).
\]
It remains to show that the extension
\begin{align}
    1 \raa \tilde{H_2} \raa \tilde{H_3} = \tilde{H} \raa W \raa 1
    \tag{$\ast$}
\end{align}
can be extended to an extension in $\Dbb_n^\x$
\begin{align}
    1 \raa \tilde{F_2} \raa \tilde{F_3} = \tilde{F} \raa W \raa 1.
    \tag{$\ast\ast$}
\end{align}
Let $L := \Qbb_p(\tilde{H_2}) = \Qbb_p(\tilde{F_2})$. The existence of
$(\ast)$ implies that $W \subset Aut(H_0) \subset Aut(F_0)$ lifts to
$Aut(L, \tilde{H_2}, H_0)$. An automorphism $\sigma$ of $L$ which leaves
$H_0$ invariant, also leaves $\Qbb_p(H_0)$ invariant, and therefore the
subgroups $F_0 = \mu(\Qbb_p(H_0))$ and $\tilde{F_0} = \lan F_0, pu \ran$
are also left invariant. Hence
\[
    \sigma(x_1)^{r_1} \in \tilde{F_0} = \lan F_0, pu \ran
    \qq \text{and} \qq
    \left( \frac{\sigma(x_1)}{x_1} \right)^{r_1} \in F_0.
\]
Since $F_0$ is the (unique) maximal finite subgroup of $\Qbb_p(F_0)^\x$,
we have
\[
    \frac{\sigma(x_1)}{x_1} \in F_0
    \qq \text{and} \qq
    \sigma(x_1) \in \lan F_0, x_1 \ran = \tilde{F_1},
\]
and therefore $\sigma$ leaves $\tilde{F_1}$ invariant. This implies
\[
    \sigma(x_2)^{r_2} \in \tilde{F_1}
    \qq \text{and} \qq
    \left( \frac{\sigma(x_2)}{x_2} \right)^{r_2} 
    \in \tilde{F_1} \cap \Sbb_n = F_0.
\]
Using that $\mu(\Qbb_p(F_0)) = \mu(\Qbb_p(\tilde{F_2}))$, we obtain as
before that $\sigma(x_2) \in \lan F_0, x_2 \ran = \tilde{F_2}$ and
consequently that $\sigma$ leaves $\tilde{F_2}$ invariant. It follows
that $W$ lifts to $Aut(L, \tilde{F_2}, F_0)$. The chain of inclusions
$\tilde{H_2} \subset \tilde{F_2} \subset L^\x$ induces a commutative
diagram
\[
    \xymatrix{
        H^2(W, \tilde{H_2}) \ar[d] \ar@{->>}[dr] \\
        H^2(W, \tilde{F_2}) \ar[r]
        & H^2(W, L^\x)
    }
\]
whose oblique arrow is an epimorphism by theorem \ref{190}. The
horizontal homomorphism is therefore surjective and theorem \ref{190}
implies the existence of $(\ast\ast)$ in $\Dbb_n^\x$.
\end{proof}

\section{Classification of embeddings up to conjugation} 

In this section, we use the results obtained in this chapter to classify
the chains of subgroups
\[
    \tilde{F_0} \subset \tilde{F_1} \subset \tilde{F_2} 
    \subset \tilde{F_3}
\]
that occur in $\Dbb_n^\x$. In order to do this we proceed in four steps.

For a group $G$ and a $G$-set $\Scal$, we denote by $\Scal/\sim_G$ the
set of orbits with respect to the $G$-action on $\Scal$.

\subsection*{Classifying $\tilde{F_0}$'s}

As explained in remark \ref{207}, the map
\[
    \tilde{\Fcal}_u(\Dbb_n^\x) \raa \tilde{\Fcal}_u(\Sbb_n)\
    :\ \tilde{F} \mtoo \tilde{F_0} 
    := \lan \tilde{F} \cap S_n, Z_{p'}(G) \ran
\]
induces a well defined map
\[
    \phi_0:\ \tilde{\Fcal}_u(\Dbb_n^\x)/\sim_{\Dbb_n^\x} 
    \raa \tilde{\Fcal}_u(\Sbb_n)/\sim_{\Dbb_n^\x},
\]
whose image can be identified with the set
\[
    \{(\alpha, d) \in \Nbb \x \Nbb^\ast\ 
    |\ 0 \leq \alpha \leq k,\ d\ |\ p^{n_\alpha}-1\}.
\]

\subsection*{Classifying $\tilde{F_1}$'s}

Pick $\tilde{F_0} \in \tilde{\Fcal}_u(\Sbb_n)$ and define the sets
\begin{itemize}
    \item $\tilde{\Fcal}(\Dbb_n^\x, \tilde{F_0})$ of all subgroups
    $\tilde{F}$ of $\Dbb_n^\x$ such that $\tilde{F_0} = \tilde{F} \cap
    \Sbb_n$ is of finite index in $\tilde{F}$;

    \item $\tilde{\Fcal}(\Qbb_p(F_0)^\x, \tilde{F_0})$ of all subgroups
    $\tilde{F}$ of $\Qbb_p(F_0)^\x$ such that $\tilde{F_0} = \tilde{F}
    \cap \Sbb_n$ is of finite index in $\tilde{F}$.
\end{itemize}
Clearly the map
\[
    \tilde{\Fcal}(\Dbb_n^\x, \tilde{F_0}) 
    \raa \tilde{\Fcal}(\Qbb_p(F_0)^\x, \tilde{F_0})\
    :\ \tilde{F} \mtoo \tilde{F_1} 
    := \tilde{F} \cap \Qbb_p(F_0)^\x
\]
induces a well defined map
\[
    \phi_1:\ \tilde{\Fcal}(\Dbb_n^\x, \tilde{F_0})/
    \sim_{C_{\Dbb_n^\x}(\tilde{F_0})} 
    \raa \tilde{\Fcal}(\Qbb_p(F_0)^\x, \tilde{F_0}).
\]
As seen in section \ref{194}, every $\tilde{F_1} \in
\tilde{\Fcal}(\Qbb_p(F_0)^\x, \tilde{F_0})$ determines an integer $r_1 =
|\tilde{F_1}/\tilde{F_0}|$ which is a divisor of $n$. Furthermore,
according to theorem \ref{180}, if such a divisor $r_1$ is realized by a
subgroup $\tilde{F_1} \in \tilde{\Fcal}(\Qbb_p(F_0)^\x, \tilde{F_0})$,
then the set 
\[
    \{ \tilde{F_1} \in \tilde{\Fcal}(\Qbb_p(F_0)^\x, \tilde{F_0})\ |\
    |\tilde{F_1}/\tilde{F_0}| = r_1 \}
\]
is in bijection with the set $H^1(\Zbb/r_1, \Zbb_p[F_0]^\x/F_0)$.

\subsection*{Classifying $\tilde{F_2}$'s}

Pick $\tilde{F_1} \in \tilde{\Fcal}(\Qbb_p(F_0)^\x, \tilde{F_0})$ and
define the sets
\begin{itemize}
    \item $\tilde{\Fcal}(\Dbb_n^\x, \tilde{F_1})$ of all subgroups
    $\tilde{F}$ of $\Dbb_n^\x$ such that $\tilde{F_1} = \tilde{F} \cap
    \Qbb_p(F_0)^\x$ is of finite index in $\tilde{F}$;

    \item $\tilde{\Fcal}(C_{\Dbb_n^\x}(\tilde{F_1}), \tilde{F_1})$ of
    all subgroups $\tilde{F}$ of $C_{\Dbb_n^\x}(\tilde{F_1})$ such that
    $\tilde{F_1} = \tilde{F} \cap \Qbb_p(F_0)^\x$ is of finite index in
    $\tilde{F}$.
\end{itemize}
Then the map
\[
    \tilde{\Fcal}(\Dbb_n^\x, \tilde{F_1})
    \raa \tilde{\Fcal}(C_{\Dbb_n^\x}(\tilde{F_1}), \tilde{F_1})\
    :\ \tilde{F} \mtoo \tilde{F_2} 
    := \tilde{F} \cap C_{\Dbb_n^\x}(\tilde{F_1})
\]
induces a well defined map
\[
    \phi_2:\ \tilde{\Fcal}(\Dbb_n^\x, \tilde{F_1})
    /\sim_{C_{\Dbb_n^\x}(\tilde{F_1})} 
    \raa \tilde{\Fcal}(C_{\Dbb_n^\x}(\tilde{F_1}), \tilde{F_1})
    /\sim_{C_{\Dbb_n^\x}(\tilde{F_1})}.
\]
In order to describe the image of $\phi_2$, we recall that every
$\tilde{F_2} \in \tilde{\Fcal}(C_{\Dbb_n^\x}(\tilde{F_1}), \tilde{F_1})$
determines an extension $L := \Qbb_p(\tilde{F_2})$ of
$\Qbb_p(\tilde{F_1})$. Clearly, the isomorphism class of $L$ is constant
on each conjugacy class of $\tilde{F_2}$'s by elements in
$C_{\Dbb_n^\x}(\tilde{F_1})$, and hence determine the integer
$r_2 = [L:\Qbb_p(\tilde{F_1})]$ dividing $\frac{n}{r_1}$. By the
Skolem-Noether theorem, the set of isomorphism classes of extensions
$\Qbb_p(\tilde{F_1}) \subset L$ is in bijection with the set of
$C_{\Dbb_n^\x}(\tilde{F_1})$-conjugacy classes of $L$'s. Thus denoting
\[
    \tilde{\Fcal}(L^\x, \tilde{F_1})
    := \{ \tilde{F_2} 
    \in \tilde{\Fcal}(C_{\Dbb_n^\x}(\tilde{F_1}), \tilde{F_1})\
    |\ \Qbb_p(\tilde{F_2}) = L \},
\]
we have a bijection
\[
    \tilde{\Fcal}(C_{\Dbb_n^\x}(\tilde{F_1}), \tilde{F_1})
    /\sim_{C_{\Dbb_n^\x}(\tilde{F_1})}\
    \iso\ \coprod_{[L]} \tilde{\Fcal}(L^\x, \tilde{F_1}),
\]
where the union is taken over all isomorphism classes of extensions
$\Qbb_p(\tilde{F_1}) \subset L$.  Finally, if for a given $L$ the set
$\tilde{\Fcal}(L^\x, \tilde{F_1})$ is non-empty, then by theorem
\ref{185} it is in bijection with the set $H^1(\Zbb/r_2,
L_{r_1}^\x/\tilde{F_1})$, and we have
\[
    \tilde{\Fcal}(C_{\Dbb_n^\x}(\tilde{F_1}), \tilde{F_1})
    /\sim_{C_{\Dbb_n^\x}(\tilde{F_1})}\
    \iso\ \coprod_{[L]} H^1(\Zbb/r_2, L_{r_1}^\x/\tilde{F_1}).
\]

\subsection*{Classifying $\tilde{F_3}$'s}

Pick $\tilde{F_2} \in \tilde{\Fcal}(L^\x, \tilde{F_1})$ and define the
sets
\begin{itemize}
    \item $\tilde{\Fcal}(\Dbb_n^\x, \tilde{F_2})$ of all subgroups
    $\tilde{F}$ of $\Dbb_n^\x$ such that $\tilde{F} \cap S_n$ is abelian
    and $\tilde{F_2} = \tilde{F} \cap C_{\Dbb_n^\x}(F_0)$ is of finite
    index in $\tilde{F}$;

    \item $\tilde{\Fcal}(N_{\Dbb_n^\x}(\tilde{F_2}), \tilde{F_2})$ of
    all subgroups $\tilde{F}$ of $N_{\Dbb_n^\x}(\tilde{F_2})$ such that
    $\tilde{F} \cap S_n$ is abelian and $\tilde{F_2} = \tilde{F} \cap
    C_{\Dbb_n^\x}(F_0)$ is of finite index in $\tilde{F}$.
\end{itemize}
By proposition \ref{077}, each $\tilde{F}$ in $\tilde{\Fcal}(\Dbb_n^\x,
\tilde{F_2})$ satisfies $\tilde{F} \subset N_{\Dbb_n^\x}(F_0)$, in which
case $\tilde{F_2}$ is normal in $\tilde{F}$. Thus the map
\[
    \tilde{\Fcal}(\Dbb_n^\x, \tilde{F_2})
    \overset{\iso}{\raa} 
    \tilde{\Fcal}(N_{\Dbb_n^\x}(\tilde{F_2}), \tilde{F_2})\
    :\ \tilde{F} \mtoo \tilde{F_3} 
    := \tilde{F} \cap N_{\Dbb_n^\x}(\tilde{F_2})
\]
is a bijection and induces a well defined bijection
\[
    \phi_3:\ \tilde{\Fcal}(\Dbb_n^\x, \tilde{F_2})
    /\sim_{C_{\Dbb_n^\x}(\tilde{F_2})} 
    \overset{\iso}{\raa} 
    \tilde{\Fcal}(N_{\Dbb_n^\x}(\tilde{F_2}), \tilde{F_2})
    /\sim_{C_{\Dbb_n^\x}(\tilde{F_2})}.
\]
In order to describe the image of $\phi_3$, we recall that every
$\tilde{F_3} \in \tilde{\Fcal}(N_{\Dbb_n^\x}(\tilde{F_2}), \tilde{F_2})$
determines an extension
\[
    1 \raa \tilde{F_2} \raa \tilde{F_3} \raa W \raa 1,
\]
where $W$ canonically injects into $Aut(\tilde{F_2}, F_0)$. Via this
injection, $W$ is independent of the given representative in the
$C_{\Dbb_n^\x}(\tilde{F_2})$-conjugacy class of $\tilde{F_3}$. Thus
denoting
\[
    \tilde{\Fcal}(N_{\Dbb_n^\x}(\tilde{F_2}), \tilde{F_2}, W)
    :=
    \{ \tilde{F_3} 
    \in \tilde{\Fcal}(N_{\Dbb_n^\x}(\tilde{F_2}), \tilde{F_2})\
    |\ \tilde{F_3}/\tilde{F_2} = W \},
\]
we have a bijection
\[
    \tilde{\Fcal}(N_{\Dbb_n^\x}(\tilde{F_2}), \tilde{F_2})
    /\sim_{C_{\Dbb_n^\x}(\tilde{F_2})}\
    \iso\ \coprod_{W} 
    \tilde{\Fcal}(N_{\Dbb_n^\x}(\tilde{F_2}), \tilde{F_2}, W)
    /\sim_{C_{\Dbb_n^\x}(\tilde{F_2})}.
\]
Finally, if for a given $W$ the set
$\tilde{\Fcal}(N_{\Dbb_n^\x}(\tilde{F_2}), \tilde{F_2}, W)
/\sim_{C_{\Dbb_n^\x}(\tilde{F_2})}$ is non-empty, then by theorem
\ref{193} it is in bijection with the set $H^1(W,
C_{\Dbb_n^\x}(\tilde{F_2})/\tilde{F_2})$, and we have
\[
    \tilde{\Fcal}(N_{\Dbb_n^\x}(\tilde{F_2}), \tilde{F_2})
    /\sim_{C_{\Dbb_n^\x}(\tilde{F_2})}\
    \iso\ \coprod_{W} 
    H^1(W, C_{\Dbb_n^\x}(\tilde{F_2})/\tilde{F_2}).
\]


\chapter{On abelian finite subgroups of $\Gbb_n(u)$} 

Throughout this chapter we assume that $n=(p-1)p^{k-1}m$ with $m$ prime
to $p$. Given an abelian finite subgroup $F_0$ of $\Sbb_n$ whose
$p$-Sylow subgroup is cyclic of order $p^\alpha$ for $1 \leq \alpha \leq
k$, we want to determine what sequences of groups
\[
    F_0 \subset F_1 \subset F_2
\]
are realized in $\Gbb_n(u) = \Dbb_n^\x/\lan pu \ran$; here $F_2$ is an
abelian finite subgroup of $\Gbb_n(u)$ containing $F_0$ and $F_1$ is
such that $\tilde{F_1} = \tilde{F_2} \cap \Qbb_p(F_0)$. We know from
chapter \ref{091} that the tilded correspondents of these groups in
$\Dbb_n^\x$ are given by
\[
    \tilde{F_1} = \lan F_0, x_1 \ran
    \qq \text{and} \qq
    \tilde{F_2} = \lan F_0, x_2 \ran
\]
with $x_1, x_2 \in \Dbb_n^\x$ such that
\[
    v(x_1) = \frac{1}{r_1}, \qq
    x_1^{r_1} \in \tilde{F_0}, \qq
    v(x_2) = \frac{1}{r_1r_2}
    \qq \text{and} \qq
    x_2^{r_2} \in \tilde{F_1}.
\]
We want to determine for what pairs of positive integers $(r_1, r_2)$
the sets 
\[
    \tilde{\Fcal}_u(\Qbb_p(F_0), \tilde{F_0}, r_1)
    \qq \text{and} \qq
    \tilde{\Fcal}_u(C_{\Dbb_n^\x}(F_0), \tilde{F_1}, r_2)
\]
are non-empty.

\vspace{5ex}

\section{Elementary conditions on $r_1$} 
\label{288}

The question of determining for what $r_1$ the set
$\tilde{\Fcal}_u(\Qbb_p(F_0), \tilde{F_0}, r_1)$ is non-empty naturally
leads to studying the $r_1$-th roots of $pu$ in $\Qbb_p(F_0)$. Clearly
$r_1$ must be a divisor of $\phi(p^\alpha)$, the ramification index of
$\Qbb_p(F_0)$ over $\Qbb_p$.

\begin{proposition} \label{092}
Let $\zeta_{p^\alpha}$ be a primitive $p^\alpha$-th root of unity in
$\Qbb_p(F_0)^\x$. The principal ideal generated by $\zeta_{p^\alpha}-1$
is maximal in $\Zbb_p(F_0)$ and satisfies
\[
    (p) = (\zeta_{p^\alpha}-1)^{\phi(p^\alpha)}.
\]
\end{proposition}

\begin{proof}
If $a$ and $b$ are integers prime to $p$, one can solve the equation $a
\equiv bs \mod p^\alpha$, so that
\[
    \frac{\zeta_{p^\alpha}^a-1}{\zeta_{p^\alpha}^b-1} 
    = \frac{1-\zeta_{p^\alpha}^{bs}}{1-\zeta_{p^\alpha}^b} 
    = 1 + \zeta_{p^\alpha}^b + \ldots + \zeta_{p^\alpha}^{(s-1)b}\ 
    \in \Zbb_p[\zeta_{p^\alpha}].
\]
The same is true for
$\frac{\zeta_{p^\alpha}^b-1}{\zeta_{p^\alpha}^a-1}$, and
\[
    \frac{\zeta_{p^\alpha}^a-1}{\zeta_{p^\alpha}^b-1}\ 
    \in\ \Zbb_p[\zeta_{p^\alpha}]^\x
    \qq \text{whenever} \q (a;p) = (b;p) = 1.
\]
Moreover since
\[
    \sum_{i=0}^{p-1} x^{ip^{\alpha-1}}
    = \frac{1-x^{p^\alpha}}{ 1-x^{p^{\alpha-1}} }
    = \prod_{\stackrel{(a;p)=1}{1 \leq a<p^\alpha}} 
    (\zeta_{p^\alpha}^a-x),
\]
for $x=1$ we get
\[
    p 
    = \prod_{\stackrel{(a;p)=1}{1 \leq a<p^\alpha}} 
    (\zeta_{p^\alpha}^a-1)
    = (\zeta_{p^\alpha}-1)^{\phi(p^\alpha)} 
    \prod_{\stackrel{(a;p)=1}{1 \leq a<p^\alpha}} 
    \frac{\zeta_{p^\alpha}^a-1}{\zeta_{p^\alpha}-1},
\]
showing that $(p) = (\zeta_{p^\alpha}-1)^{\phi(p^\alpha)}$. The ideal
generated by $\zeta_{p^\alpha}-1$ is hence maximal in $\Zbb_p(F_0)$.
\end{proof}

\begin{corollary} \label{093}
We have
\[
    p 
    = \prod_{\stackrel{(a;p)=1}{1 \leq a<p^\alpha}} 
    (\zeta_{p^\alpha}^a-1) 
    \qq \text{and} \qq 
    v(\zeta_{p^\alpha}-1) = \frac{1}{\phi(p^\alpha)}.
\]
\qed
\end{corollary}

Let $\mu(\Qbb_p(F_0))$ denote the roots of unity in $\Qbb_p(F_0)$ and
fix $\zeta_{p^\alpha}$ a primitive $p^\alpha$-th root of unity in
$\mu(\Qbb_p(F_0))$. Define the unit
\[
    \epsilon_{\alpha} \label{304}
    \in \Zbb_p(\zeta_{p^\alpha})^\x 
    \subset \Qbb_p(F_0)^\x
    \qq \text{by} \qq
    (\zeta_{p^\alpha}-1)^{\phi(p^\alpha)} 
    = p \epsilon_{\alpha}.
\]
Obviously as $u \in \Zbb_p^\x$, we know that
$\frac{\epsilon_{\alpha}}{u}$ belongs to $\Zbb_p(\zeta_{p^\alpha})^\x$.
Let $\pi(e_u(F_0))$ denote the class of
\[
    e_u(F_0) 
    \in H^2(\Zbb/\phi(p^\alpha),\ \Zbb_p(F_0)^\x \x \lan pu \ran)
\]
in $H^2(\Zbb/\phi(p^\alpha),\ \Zbb_p(F_0)^\x)$ as defined in section
\ref{194}.

\begin{proposition} \label{094}
We have
\[
    \pi(e_u(F_0)) = \frac{\epsilon_{\alpha}}{u}
    \q \text{in} \q 
    H^2(\Zbb/\phi(p^\alpha),\ \Zbb_p(F_0)^\x)
    \iso \Zbb_p(F_0)^\x/(\Zbb_p(F_0)^\x)^{\phi(p^\alpha)}.
\]
\end{proposition}

\begin{proof}
This is a straightforward consequence of the fact that
\[
    p \epsilon_{\alpha}
    = pu \frac{\epsilon_{\alpha}}{u}
\]
belongs to the class of $e_u(F_0) \in H^2(\Zbb/\phi(p^\alpha),\
\Zbb_p(F_0)^\x \x \lan pu \ran)$.
\end{proof}

Recall from proposition \ref{341} that
\[
    \Zbb_p(F_0)^\x
    \iso \mu(\Qbb_p(F_0)) \x \Zbb_p^{[\Qbb_p(F_0):\Qbb_p]},
\]
so that
\begin{align} \label{135}
    H^2(\Zbb/r_1, \Zbb_p(F_0)^\x)\ 
    &\iso\ \Zbb_p(F_0)^\x / (\Zbb_p(F_0)^\x)^{r_1}\notag\\
    &\iso\
    \mu(\Qbb_p(F_0))/\mu(\Qbb_p(F_0))^{r_1} 
    \x (\Zbb_p / r_1 \Zbb_p)^{[\Qbb_p(F_0):\Qbb_p]}.
\end{align}

\begin{theorem} \label{098}
The set $\tilde{\Fcal}_u(\Qbb_p(F_0), \tilde{F_0}, r_1)$ is non-empty if
and only if
\[
    \frac{\epsilon_{\alpha}}{u} \equiv 1
    \qq \text{in} \q
    \Zbb_p(F_0)^\x/\lan F_0, (\Zbb_p(F_0)^\x)^{r_1} \ran.
\]
\end{theorem}

\begin{proof}
The unit 
\[
    \frac{\epsilon_{\alpha}}{u}   
    \in \Zbb_p(\zeta_{p^\alpha})^\x 
    \subset \Zbb_p(F_0)^\x,
\]
is equivalent to the trivial element if and only if $q_\ast(e_u(F_0,
r_1))$ is trivial in 
\[
    H^2(\Zbb/r_1, \Zbb_p(F_0)^\x/F_0),
\]
for $q_\ast = H^2(\Zbb/r_1, q)$ the map induced by the canonical
homomorphism 
\[
    q:\ \Zbb_p(F_0)^\x \x \lan pu \ran\
    \raa\ \Zbb_p(F_0)^\x \x \lan pu \ran / \tilde{F_0} 
    = \Zbb_p(F_0)^\x/F_0.
\]
By theorem \ref{180}, this is true if and only if
$\tilde{\Fcal}_u(\Qbb_p(F_0), \tilde{F_0}, r_1)$ is non-empty.
\end{proof}

\begin{corollary} \label{136}
If $\lan F_0, x_1 \ran \in \tilde{\Fcal}_u(\Qbb_p(F_0), \tilde{F_0},
r_1)$ with $v(x_1) = \frac{1}{r_1}$, then $x_1^{r_1}=pu\delta$ for a
$\delta$ in $F_0$ such that $\delta \equiv \frac{\epsilon_{\alpha}}{u}$
modulo $(\Zbb_p(F_0)^\x)^{r_1}$.
\end{corollary}

\begin{proof}
This follows from remark \ref{181} and theorem \ref{098}.
\end{proof}

\begin{corollary} \label{095}
Let $F_0 = \mu(\Qbb_p(F_0))$ and $r_1$ be prime to $p$.
\begin{enumerate}
    \item The set $\tilde{\Fcal}_u(\Qbb_p(F_0), \tilde{F_0}, r_1)$ is
    non-empty if and only if $r_1$ divides $p-1$.

    \item If $\tilde{\Fcal}_u(\Qbb_p(F_0), \tilde{F_0}, r_1)$ is
    non-empty with $r_1 > 1$, then $p$ is odd, and there are elements
    $\zeta_p \in F_0$ and $t \in \Zbb_p(\zeta_p)^\x$ such that
    \[
        x_1 = (\zeta_p-1)t
        \qq \text{and} \qq
        x_1^{p-1} \equiv pu\ \mod \mu_{p-1}.
    \]
\end{enumerate}
\end{corollary}

\begin{proof}
1) As $r_1$ divides the ramification index of $\Qbb_p(F_0)$, it must be
a divisor of $p-1$. The result then follows from corollary \ref{136},
the isomorphism (\ref{135}) and the fact that $\Zbb_p =
(p\!-\!1)\Zbb_p$.

2) The condition $r_1 > 1$ ensures that $p>2$ and $\zeta_p \in F_0$.
By 1) and theorem \ref{098}, we know that
\[
    \frac{u}{\epsilon_1}
    \in \lan \mu(\Qbb_p(\zeta_p)), (\Zbb_p(\zeta_p)^\x)^{p-1} \ran
    = \lan \mu_{p-1}, (\Zbb_p(\zeta_p)^\x)^{p-1} \ran.
\]
Hence there exists a $t \in \Zbb_p(\zeta_p)^\x$ such that
$u\epsilon_1^{-1} = t^{p-1} \delta$ for some $(p\!-\!1)$-th root of unity
$\delta \in \mu_{p-1}$. For $x_1 = (\zeta_p-1)t$, we then have
\[
    x_1^{p-1}
    = (\zeta_p-1)^{p-1} t^{p-1}
    = p\epsilon_1 \cdot \frac{u}{\epsilon_1} \delta^{-1}
    \equiv pu \q \mod \mu_{p-1}.
\]
\end{proof}

\begin{remark} \label{296}
When $F_0 = \mu(\Qbb_p(F_0))$, we know by corollary \ref{183} that
$\tilde{F_1}$ is unique in the set $\tilde{\Fcal}_u(\Qbb_p(F_0),
\tilde{F_0}, r_1)$. If $r_1$ divides $p-1$, we may therefore always
assume $\tilde{F_1} = \lan F_0, x_1 \ran$ with $x_1$ as given in
corollary \ref{095}.
\end{remark}

\begin{example} \label{096}
If $p$ is odd, then
\[
    \epsilon_{\alpha} \equiv -1 \q \mod\ (\Zbb_p(F_0)^\x)^{p-1}.
\]
Indeed, by example \ref{049} we know that
\[
    \Qbb_p(\zeta_p) \iso \Qbb_p(X^{\frac{n}{p-1}}),
\]
where $X = \omega^{\frac{p-1}{2}}S$ satisfies $X^n = -p$ for $\omega$ a
primitive $(p^n\!-\!1)$-th root of unity in $\Dbb_n^\x$. Furthermore, both
elements $X$ and $(\zeta_{p^\alpha}-1)$ belong to the field
$\Qbb_p(\zeta_{p^\alpha})$, and there is a $z \in
\Zbb_p(\zeta_{p^\alpha})^\x$ with
\[
    (\zeta_{p^\alpha}-1)^{p^{\alpha-1}} = X^{\frac{n}{p-1}} z.
\]
Since $\Qbb_p(\zeta_{p^\alpha})^\x \subset \Qbb_p(F_0)^\x$, we obtain
\[
    \epsilon_{\alpha} p 
    = (\zeta_{p^\alpha}-1)^{\phi(p^\alpha)}
    = X^n z^{p-1} 
    \equiv -p  \qq \mod\ (\Zbb_p(F_0)^\x)^{p-1}.
\]
Thus if $u$ is a root of unity, the set $\tilde{\Fcal}_u(\Qbb_p(F_0),
\tilde{F_0}, p\!-\!1)$ is non-empty.
\end{example}

\begin{example} \label{097}
If $p=2$, it is obvious that $\epsilon_1 = -1$. The case $\alpha = 1$
however is not interesting since then $r_1$ must divide the trivial
ramification index of $\Qbb_2(F_0)$ over $\Qbb_2$.

If $p=2$ and $\alpha \geq 2$, we have
\[
    \epsilon_{\alpha} 
    \equiv -\zeta_4 \q \mod\ (\Zbb_2(F_0)^\x)^{2}
\]
for a primitive $4$-th root of unity $\zeta_4 \in \Zbb_2(F_0)^\x$.
Indeed, the element $(\zeta_4-1)$ has valuation $\frac{1}{2}$ and
\[
    (\zeta_4-1)^2 = \zeta_4^2-2\zeta_4+1 = -2\zeta_4.
\]
Hence for $z \in \Zbb_2(F_0)^\x$ satisfying
\[
    (\zeta_{2^\alpha}-1)^{2^{\alpha-2}} = (\zeta_4-1)z,
\]
we obtain
\[
    2\epsilon_{\alpha}
    = (\zeta_{2^\alpha}-1)^{2^{\alpha-1}}
    = (\zeta_4-1)^2 z^2 
    \equiv -2\zeta_4 \qq \mod\ (\Zbb_2(F_0)^\x)^{2}.
\]
This shows that if $u = \pm 1$, the set $\tilde{\Fcal}_u(\Qbb_2(F_0),
\tilde{F_0}, 2)$ is non-empty.
\end{example}

\section{Change of rings} 
\label{287}

Assume $p$ to be any prime. For each $1 \leq \alpha \leq k$, we fix a
root of unity $\zeta_{p^\alpha}$ in $\Qbb_p(F_0)$, and we define
\[
    \label{305}
    \pi_\alpha := \zeta_{p^\alpha}-1
    \qq \text{and} \qq
    R_\alpha := \Zbb_p[\zeta_{p^\alpha}].
\]
Recall from proposition \ref{092} that $\pi_\alpha$ is a uniformizing
element of $R_\alpha$ where $(\pi_\alpha^{\phi(p^\alpha)}) = (p)$. Let 
\[
    i_\alpha: R_\alpha \raa R_{\alpha+1} 
\]
be the ring homomorphism defined by $i_\alpha(\zeta_{p^\alpha}) =
\zeta_{p^{\alpha+1}}^p$. By definition, $\epsilon_\alpha \in R_\alpha$
for each $\alpha$. In this section, we compare $\epsilon_{\alpha+1}$
with the image of $\epsilon_\alpha$ in $R_{\alpha+1}$.

\begin{lemma} \label{282}
For any prime $p$ and any $\alpha \geq 1$, we have
\[
    i_\alpha(\pi_\alpha)
    = \sum_{j=1}^p \binom{p}{j} \pi_{\alpha+1}^j.
\]
\end{lemma}

\begin{proof}
Clearly $i_\alpha(\pi_\alpha) = \zeta_{p^{\alpha+1}}^p-1$. This,
together with the identity
\[
    X^p-1
    = (X-1+1)^p-1
    = \sum_{j=1}^p \binom{p}{j}(X-1)^j
\]
applied to the case $X = \zeta_{p^{\alpha+1}}$, yields the result.
\end{proof}

\begin{corollary} \label{283}
For any prime $p$ and any $\alpha \geq 1$, we have
\[
    i_\alpha(\pi_\alpha)
    \equiv \pi_{\alpha+1}^p \q \mod (p\pi_{\alpha+1}).
\]
\end{corollary}

\begin{proof}
This follows from lemma \ref{282} and the $p$-divisibility of the
binomial coefficients for $1 \leq j < p$.
\end{proof}

For any prime $p$ and any $\alpha \geq 2$, define the positive integer
\[
    k_\alpha
    :=
    \begin{cases}
        p^\alpha-2p+1 & \text{if } p>2,\\
        2^\alpha-2 & \text{if } p=2.
    \end{cases}
\]

\begin{lemma} \label{284}
If $p>2$ and $\alpha \geq 2$, or if $p=2$ and $\alpha \geq 3$, then
\[
    i_\alpha(\pi_\alpha^j)
    \equiv \pi_{\alpha+1}^{jp} 
    \q \mod (\pi_{\alpha+1}^{p^{\alpha+1}+1})
    \qq \text{for any } j \geq k_\alpha.
\]
\end{lemma}

\begin{proof}
Let $j \geq k_\alpha$. Combining corollary \ref{283} with the binomial
formula yields
\[
    i_\alpha(\pi_\alpha^j)
    = (\pi_{\alpha+1}^p + p\pi_{\alpha+1}z)^j
    = \pi_{\alpha+1}^{jp} + w
\]
with
\[
    z \in R_{\alpha+1}
    \qq \text{and} \qq
    w
    = \sum_{k=0}^{j-1} \binom{j}{k} \pi_{\alpha+1}^{kp}
    (p\pi_{\alpha+1}z)^{j-k}.
\]
Note that the valuation of the $k$-th term is at least
$kp+(j-k)(\phi(p^{\alpha+1})+1)$. Hence for $0 \leq k \leq j-1$ its
valuation is at least
\[
    (j-1)p+\phi(p^{\alpha+1})+1
    \geq (k_\alpha-1)p + \phi(p^{\alpha+1})+1.
\]
If $p>2$ and $\alpha \geq 2$, we have
\begin{align*}
    (j-1)p+\phi(p^{\alpha+1})+1\
    &\geq\ (p^\alpha-2p)p+\phi(p^{\alpha+1})+1\\
    &=\ p^{\alpha+1}-2p^2+p^{\alpha+1}-p^\alpha+1\\
    &=\ p^{\alpha+1}+p^2(p^{\alpha-1}-p^{\alpha-2}-2)+1\\
    &\geq\ p^{\alpha+1}+1.
\end{align*}
Otherwise if $p=2$ and $\alpha \geq 3$, we obtain
\begin{align*}
    (j-1)p+\phi(p^{\alpha+1})+1\
    &\geq\ (2^\alpha-3)2+2^\alpha+1\\
    &=\ 2^{\alpha+1}-6+2^\alpha+1\\
    &\geq\ 2^{\alpha+1}+1.
\end{align*}
\end{proof}

\begin{lemma} \label{285}
If $p>2$ and $\alpha \geq 2$, or if $p=2$ and $\alpha \geq 3$, then
\[
    i_\alpha(\pi_\alpha^{\phi(p^\alpha)})
    \equiv \pi_{\alpha+1}^{\phi(p^{\alpha+1})}
    \q \mod (\pi_{\alpha+1}^{\phi(p^{\alpha+1})+p^{\alpha+1}+1}).
\]
\end{lemma}

\begin{proof}
Combining corollary \ref{283} with the binomial formula yields
\begin{align*}
    i_\alpha(\pi_\alpha^p)\
    &=\ (\pi_{\alpha+1}^p
    +\pi_{\alpha+1}^{\phi(p^{\alpha+1})+1}z_0)^p\\[1ex]
    &=\ \pi_{\alpha+1}^{p^2}
    + \sum_{j=1}^{p-1} \binom{p}{j}
    \pi_{\alpha+1}^{jp+(\phi(p^{\alpha+1})+1)(p-j)}z_0^{p-j}
    + \pi_{\alpha+1}^{(\phi(p^{\alpha+1})+1)p}z_0^p\\[1ex]
    &=\ \pi_{\alpha+1}^{p^2}+\pi_{\alpha+1}^{2\phi(p^{\alpha+1})+1}z_1
\end{align*}
for some suitable $z_0, z_1 \in R_{\alpha+1}$, where we have used that
the valuation of each term in the middle sum is greater or equal to
$(p-1)p+2\phi(p^{\alpha+1})+1$, while that of the last term is
$(\phi(p^{\alpha+1})+1)p > 2\phi(p^{\alpha+1})+1$. By iterating this
procedure we obtain some $z_k$ with
\[
    i_\alpha(\pi_\alpha^{p^k})
    = \pi_{\alpha+1}^{p^{k+1}}
    +\pi_{\alpha+1}^{(k+1)\phi(p^{\alpha+1})+1}z_k
    \qq \text{for every } k \geq 0.
\]
The required formula for $p=2$ and $\alpha \geq 3$ directly follows from
the case $k=\alpha-1$. Again, by taking $k=\alpha-1$ if $p>2$, we get
\begin{align*}
    i_\alpha(\pi_{\alpha}^{\phi(p^\alpha)})\
    &=\ (\pi_{\alpha+1}^{p^\alpha}
    +\pi_{\alpha+1}^{\alpha\phi(p^{\alpha+1})+1}z_{\alpha-1})^{p-1}\\
    &=\ \pi_{\alpha+1}^{\phi(p^{\alpha+1})}
    + \pi_{\alpha+1}^{\alpha\phi(p^{\alpha+1})+1+(p-2)p^\alpha}z
\end{align*}
for some $z \in R_{\alpha+1}$. The desired result for $p\geq 3$ and
$\alpha\geq 2$ then follows from the fact that
\begin{align*}
    \alpha\phi(p^{\alpha+1})+1+(p-2)p^\alpha\
    &=\ \phi(p^{\alpha+1})+(\alpha-1)\phi(p^{\alpha+1})+(p-2)p^\alpha+1\\
    &=\ \phi(p^{\alpha+1})+p^\alpha[(\alpha-1)(p-1)+p-2]+1\\
    &\geq\ \phi(p^{\alpha+1})+p^\alpha(2p-3)+1\\
    &\geq\ \phi(p^{\alpha+1})+p^{\alpha+1}+1.
\end{align*}
\end{proof}

\begin{corollary} \label{286}
If $p>2$ and $\alpha \geq 2$, or if $p=2$ and $\alpha \geq 3$, then
\[
    i_\alpha(\epsilon_\alpha)
    \equiv \epsilon_{\alpha+1}
    \q \mod (\pi_{\alpha+1}^{p^{\alpha+1}+1}).
\]
\end{corollary}

\begin{proof}
This follows from lemma \ref{285} together with the fact that
$p\epsilon_\alpha = \pi_\alpha^{\phi(p^\alpha)}$.
\end{proof}

\section{The $p$-part of $r_1$ for $p$ odd} 
\label{228}

Using notations introduced in sections \ref{288} and \ref{287}, we
assume $p$ to be an odd prime. Recall that $\alpha \geq 0$ is defined to
satisfy $|F_0 \cap S_n| = p^\alpha$. The goal of the section is to
establish that for $\alpha \geq 1$ and $F_0 = \mu(\Qbb_p(F_0))$, the set
$\tilde{\Fcal}_u(\Qbb_p(F_0), \tilde{F_0}, r_1)$ is non-empty if and
only if $r_1$ divides $p-1$. This is done by showing that
$\frac{\epsilon_{\alpha}}{u}$ is non-trivial in the group
$\Zbb_p(F_0)^\x/\lan \mu(\Qbb_p(F_0)), (\Zbb_p(F_0)^\x)^p \ran$ when
$\alpha \geq 2$, and hence that $p$ does not divide $r_1$.  

We know from proposition \ref{092} that $(\pi_{\alpha})$ is the maximal
ideal of $\Zbb_p(F_0)$. The situation is clear when $\alpha = 1$,
because the ramification index of $\Qbb_p(F_0)$ over $\Qbb_p$ is prime
to $p$ and hence the $p$-part of $r_1$ is trivial.

We need to establish a formula for the $\pi_{\alpha}$-adic expansion of
$\epsilon_{\alpha} = p^{-1}\pi_{\alpha}^{\phi(p^\alpha)}$.  For this we
begin by analysing the cyclotomic polynomials
\[
    Q_{\alpha}(X) 
    := \frac{(X+1)^{p^\alpha}-1}{(X+1)^{p^{\alpha-1}}-1}
    \q \in \Zbb[X].
\]
Note that $Q_\alpha(X)$ is the minimal polynomial of $\pi_\alpha$ over
$\Qbb_p$. We have
\[
    Q_{\alpha}(X) 
    = \sum_{k=0}^{p-1} (X+1)^{p^{\alpha-1}k}
    = \sum_{i=0}^{\phi(p^\alpha)} 
    \l( \sum_{k=0}^{p-1} \binom{p^{\alpha-1}k}{i} \r) X^i.
\]
Define $a_i^{(\alpha)}$ to be the coefficient of $X^i$ in
$Q_{\alpha}(X)$, and let 
\[
    b_i^{(\alpha)} 
    :=
    \begin{cases}
        \binom{p^{\alpha-1}}{i} 
        & \text{if } 0 \leq i \leq p^{\alpha-1},\\
        \hspace{1.5em} 0 
        & \text{if } i > p^{\alpha-1},
    \end{cases}
\]
be the coefficient of $X^i$ in $(X+1)^{p^{\alpha-1}}$.

\begin{lemma} \label{099}
For $\alpha, i \geq 1$, we have a strict identity
\[
    a_i^{(\alpha)} 
    = b_i^{(\alpha)} + \sum_{k=2}^{p-1} \sum_{i_1 + \ldots + i_k = i} 
    b_{i_1}^{(\alpha)} \ldots b_{i_k}^{(\alpha)}.
\]
\end{lemma}

\begin{proof}
This follows form the fact that
\[
    Q_{\alpha}(X) = \sum_{k=0}^{p-1} (X+1)^{p^{\alpha-1}k}.
\]
\end{proof}

\begin{lemma} \label{100}
For $\alpha \geq 3$ and $i \geq 1$, we have
\[
    b_i^{(\alpha)} \equiv
    \begin{cases}
    b_j^{(\alpha-1)} \mod p^2 & \text{if } i = pj,\\
    0 \hspace{2em} \mod p^2 & \text{if } i \not\equiv 0 \mod p.
    \end{cases}
\]
\end{lemma}

\begin{proof}
This is a consequence of the identity
\begin{align*}
    (1 + X)^{p^{\alpha-1}}\
    &=\ (1+X^p+pX(1+\ldots+X^{p-2}))^{p^{\alpha-2}}\\
    &\equiv\ (1+X^p)^{p^{\alpha-2}}
    + p^{\alpha-1}X(1+\ldots+X^{p-2})(1+X^p)^{p^{\alpha-2}-1}
    \q \mod (p^\alpha).
\end{align*}
\end{proof}

\begin{lemma} \label{101}
For $\alpha \geq 3$ and $i \geq 1$, we have
\[
    a_i^{(\alpha)}
    \equiv
    \begin{cases}
        a_{j}^{(\alpha-1)} \mod p^2 & \text{if } i=pj,\\
        0 \hspace{2.1em} \mod p^2 & \text{if } i \not\equiv 0 \mod p.
    \end{cases}
\]
\end{lemma}

\begin{proof}
If $i \not\equiv 0 \mod p$, the result is a direct consequence of lemma
\ref{100}. It remains to consider the case where $i=pj$ for some integer
$j\geq 1$. By lemma \ref{099} and \ref{100}, it suffices to show that
\[
    \sum_{k=2}^{p-1} \sum_{i_1+\ldots+i_k=pj} 
    b_{i_1}^{(\alpha)} \ldots b_{i_k}^{(\alpha)}\
    \equiv\ \sum_{k=2}^{p-1} \sum_{j_1+\ldots+j_k=j} 
    b_{j_1}^{(\alpha-1)} \ldots b_{j_k}^{(\alpha-1)}
    \q \mod p^2.
\]
Using lemma \ref{100} once again it suffices to show that
\[
    \sum_{k=2}^{p-1} \sum_{i_1+\ldots+i_k=pj}^{\ast} 
    b_{i_1}^{(\alpha)} \ldots b_{i_k}^{(\alpha)}\
    \equiv\ 0 \q \mod p^2,
\]
where the symbol $\sum_{i_1+\ldots+i_k=pj}^{\ast}$ denotes the sum over
all $k$-tuples $(i_1, \ldots, i_k)$ for which at least one (and hence at
least two) of the $i_k$ are not divisible by $p$. Then lemma \ref{100}
implies that this sum is congruent to $0$ modulo $p^2$.
\end{proof}

The case $\alpha=2$ is of particular interest.

\begin{remark} \label{289}
We note that
\[
    b_i^{(2)}\
    \begin{cases}
        \equiv 0\ \mod p & \text{if } 0<i<p,\\
        = 1 & \text{if } i \in \{0, p\},\\
        = 0 & \text{if } i>p.
    \end{cases}
\]
\end{remark}

\begin{lemma} \label{196}
We have
\[
    a_{(p-2)p+1}^{(2)} \equiv -p\ \mod p^2.
\]
\end{lemma}

\begin{proof}
By lemma \ref{099}
\[
    a_{(p-2)p+1}^{(2)} 
    = b_{(p-2)p+1}^{(2)} 
    + \sum_{k=2}^{p-1} 
    \sum_{i_1 + \ldots + i_k = {(p-2)p+1}} 
    b_{i_1}^{(2)} \ldots b_{i_k}^{(2)}.
\]
According to remark \ref{289}, the only nontrivial contributions in this
sum modulo $p^2$ happen when $k = p-1$ and come from tuples where all
but one $i_k$ are equal to $p$ (and hence the remaining one equal to
$1$). As there are $p-1$ of such contributions we obtain
\[
    a_{(p-2)p+1}^{(2)} 
    \equiv (p-1)b_1^{(2)} 
    \equiv (p-1)p 
    \equiv -p \q \mod p^2.
\]
\end{proof}

\begin{lemma} \label{291}
If $0 < j < p-1$, then
\[
    \sum_{k=1}^{p-1} \binom{k}{j} \equiv 0\ \mod p.
\]
\end{lemma}

\begin{proof}
For a fixed $0<j<p-1$, we have
\[
    \sum_{k=1}^{p-1} \binom{k}{j}
    = \frac{1}{j!} \sum_{k=1}^{p-1} k(k-1)\ldots(k-j+1),
\]
where the expression $k(k-1)\ldots(k-j+1)$ is a polynomial of degree $j$
in $\Zbb[k]$ with zero constant term. It is consequently enough to check
that
\[
    \sum_{k=1}^{p-1} k^r
    \equiv 0\ \mod p
    \q \text{for every } 0 < r < p-1.
\]
Given $a \in \Fbb_p^\x$ such that $a^r \neq 1$, we have
\[
    \sum_{x \in \Fbb_p} x^r
    = \sum_{x \in \Fbb_p} (ax)^r
    = a^r \sum_{x \in \Fbb_p} x^r,
\]
so that
\[
    \sum_{x \in \Fbb_p} x^r = 0
    \qq \text{and} \qq
    \sum_{k=1}^{p-1} k^r
    \equiv 0\ \mod p.
\]
\end{proof}

\begin{lemma} \label{213}
If $0 \leq r < p-2$ and $0 < j < p$, then
\[
    a_{pr+j}^{(2)} \equiv 0\ \mod p^2.
\]
\end{lemma}

\begin{proof}
By lemma \ref{099}, we have
\begin{align*}
    a_{pr+j}^{(2)}\ 
    &=\ b_{pr+j}^{(2)} + \sum_{k=2}^{p-1}\ 
    \sum_{i_1 + \ldots + i_k = pr+j} 
    b_{i_1}^{(2)} \ldots b_{i_k}^{(2)} \\
    &\equiv\ b_{pr+j}^{(2)} + \sum_{k=2}^{p-1}\ 
    \sum_{i_1 + \ldots + i_k = pr+j}^\ast
    b_{i_1}^{(2)} \ldots b_{i_k}^{(2)}
    \qq \mod p^2,
\end{align*}
where the last sum is taken over all $k$-tuples $(i_1, \ldots, i_k)$ in
which there is exactly one element $b_i^{(2)}$ with $i \not\in \{0,
p\}$. Furthermore, this $b_i^{(2)}$ is in fact $b_j^{(2)}$ and
$b_{i_1}^{(2)} \ldots b_{i_k}^{(2)} = b_j^{(2)}$. We hence get
\[
    a_{pr+j}^{(2)}
    \equiv b_{pr+j}^{(2)} + b_j^{(2)} \sum_{k=2}^{p-1} k \binom{k-1}{r}
    \q \mod p^2.
\]
If $r=0$, then
\[
    a_{j}^{(2)}
    \equiv b_{j}^{(2)} + b_j^{(2)} \sum_{k=2}^{p-1} k
    \equiv b_j^{(2)} \frac{p(p-1)}{2}
    \equiv 0 \q \mod p^2.
\]
If $r>0$, then $b_{pr+j}^{(2)} = 0$ and we have
\[
    a_{pr+j}^{(2)}
    \equiv b_j^{(2)} \sum_{k=2}^{p-1} k \binom{k-1}{r}
    \q \mod p^2.
\]
Since
\[
    \sum_{k=2}^{p-1} k \binom{k-1}{r}
    = \sum_{k=1}^{p-1} k \binom{k-1}{r}
    = (r+1) \sum_{k=1}^{p-1} \binom{k}{r+1}
    \equiv 0 \q \mod p
\]
by lemma \ref{291}, we get $a_{pr+j}^{(2)} \equiv 0 \mod p^2$.
\end{proof}

Let $F_0'$ denote the $p'$-part of $F_0$. Since $Q_{\alpha}(X)$ is the
minimal polynomial of $\pi_{\alpha}$ in $\Zbb_p(F_0')$, we have an
isomorphism of algebras
\[
    \phi_{F_0}: (\Zbb_p(F_0')[X]/(Q_{\alpha}(X)))
    \overset{\iso}{\raa} \Zbb_p(F_0)
    \q \text{given by} \q
    X \mtoo \pi_{\alpha},
\]
which restricts to an isomorphism on the groups of units 
\[
    \phi_{F_0}: (\Zbb_p(F_0')[X]/(Q_{\alpha}(X)))^\x 
    \overset{\iso}{\raa} \Zbb_p(F_0)^\x.
\]
Furthermore, there is a polynomial $\tilde{Q}_{\alpha}(X) \in \Zbb[X]$
of degree $\phi(p^\alpha)-1$ such that
\[
    Q_{\alpha}(X) = X^{\phi(p^\alpha)} + p \tilde{Q}_{\alpha}(X)
    \qq \text{and} \qq
    \tilde{Q}_{\alpha}(0) = 1,
\] 
and therefore we have
\[
    \phi_{F_0}(\tilde{Q}_{\alpha}(X)) 
    = -p^{-1} \pi_{\alpha}^{\phi(p^\alpha)}
    = -\epsilon_{\alpha}.
\]

Recall from proposition \ref{092} that $\pi_\alpha$ is a uniformizing
element in $\Zbb_p(F_0)$, so that $(\pi_\alpha)$ is the maximal ideal of
this ring, and that $(\pi_{\alpha}^{\phi(p^\alpha)}) = (p)$.  More
precisely, the $\pi_\alpha$-adic expansion of $p$ in $R_\alpha =
\Zbb_p[\pi_\alpha] \subset \Zbb_p(F_0)$ is given below.

\begin{proposition} \label{290}
If $p>2$ and $\alpha \geq 2$, then
\[
    p
    \equiv -\pi_\alpha^{\phi(p^\alpha)}
    + \frac{p-1}{2} \pi_\alpha^{p^\alpha}
    \q \mod (\pi_\alpha^{p^\alpha+1}).
\]
\end{proposition}

\begin{proof}
Recall that
\[
    Q_\alpha(X)
    = p+\sum_{i=1}^{\phi(p^\alpha)-1} a_i^{(\alpha)} X^i
    +X^{\phi(p^\alpha)}.
\]
By lemma \ref{101},
\[
    Q_\alpha(X)
    \equiv Q_{\alpha-1}(X^p)
    \equiv \ldots
    \equiv Q_2(X^{p^{\alpha-2}})
    \q \mod (p^2X).
\]
By lemma \ref{213} we know that $a_i^{(2)} \equiv 0 \mod p^2$ if
$0<i<p$. Furthermore, by lemma \ref{099} and remark \ref{289} we have
\[
    a_i^{(2)} 
    = \sum_{k=1}^{p-1} \sum_{i_1 + \ldots + i_k = i} 
    b_{i_1}^{(2)} \ldots b_{i_k}^{(2)}
    \qq \text{with} \q
    b_i^{(2)}\
    \begin{cases}
        \equiv 0\ \mod p & \text{if } 0<i<p,\\
        = 1 & \text{if } i \in \{0, p\},\\
        = 0 & \text{if } i>p.
    \end{cases}
\]
Then obviously 
\[
    a_p^{(2)}
    \equiv \sum_{k=1}^{p-1} k
    = \frac{p(p-1)}{2}\ \mod p^2, \qq
    a_i^{(2)} \equiv 0\ \mod p
    \q \text{if } i \not\equiv 0 \mod p,
\]
and if $i=pj$ we have
\[
    \sum_{i_1 + \ldots + i_k = i} 
    b_{i_1}^{(2)} \ldots b_{i_k}^{(2)}
    \equiv \binom{k}{j} \ \mod p^2.
\]
For a fixed $0<j<p-1$ we have by lemma \ref{291}
\[
    \sum_{k=1}^{p-1} \binom{k}{j} \equiv 0\ \mod p,
\]
so that
\[
    a_i^{(2)}
    \equiv
    \begin{cases}
        0\ \mod p & \text{if } i \not\equiv 0 \mod p,\\
        0\ \mod p & \text{if } i=jp \text{ for } 0<j<p-1,\\
        1\ \mod p & \text{if } i=p(p-1).
    \end{cases}
\]
Therefore
\[
    Q_2(X)
    \equiv p+p\frac{p-1}{2} X^p
    + X^{\phi(p^2)}
    \q \mod (p^2X, pX^{p+1}).
\]
Finally
\[
    Q_\alpha(\pi_\alpha)
    \equiv Q_2(\pi_\alpha^{p^{\alpha-2}})
    \equiv p + p \frac{p-1}{2} \pi_\alpha^{p^{\alpha-1}}
    + \pi^{\phi(p^\alpha)} 
    \q \mod (\pi_\alpha^{p^\alpha+1}),
\]
and consequently
\begin{align*}
    p\ 
    &\equiv\ -\pi_\alpha^{\phi(p^\alpha)}
    -p \frac{p-1}{2} \pi_\alpha^{p^{\alpha-1}}\\
    &\equiv\ -\pi_\alpha^{\phi(p^\alpha)}
    -\left(
      -\pi_\alpha^{\phi(p^\alpha)}
      -p \frac{p-1}{2} \pi_\alpha^{p^{\alpha-1}}
    \right)
    \frac{p-1}{2} \pi^{p^{\alpha-1}}\\
    &\equiv\ -\pi_\alpha^{\phi(p^\alpha)}
    + \frac{p-1}{2} \pi_\alpha^{\phi(p^\alpha)}
    \pi_\alpha^{p^{\alpha-1}}\\
    &\equiv\ -\pi_\alpha^{\phi(p^\alpha)}
    + \frac{p-1}{2} \pi_\alpha^{p^\alpha}
    & \mod (\pi_\alpha^{p^\alpha+1}).
\end{align*}
\end{proof}

Our interest in approximating modulo the ideal generated by
$\pi_\alpha^{p^\alpha+1}$ is explained in the following remark. Consider
the decreasing filtration
\[
    \Zbb_p(F_0)^\x = U_0 \supset U_1 \supset U_2 \supset \ldots
\]
given by $U_0 = \Zbb_p(F_0)^\x$ and
\[
    U_i
    = U_i(\Zbb_p(F_0)^\x)
    = \{ x \in U_0\ |\ x \equiv 1 \mod (\pi_{\alpha}^i)\}
    \qq \text{for } i \geq 1,
\]
where $U_0/U_1 = \mu_{p'}(\Qbb_p(F_0))$, and where $U_i/U_{i+1}$ is
isomorphic to the residue field of $\Qbb_p(F_0)$ for each $i \geq 1$.
Because $\Zbb_p(F_0)$ is complete with respect to the filtration, any $x
\in \Zbb_p(F_0)^\x$ admits a $\pi_\alpha$-adic expansion
\[
    x
    = \sum_{i \geq 0} \lambda_i \pi_\alpha^i,
\]
where the $\lambda_i$'s run in a given set of representative of the
residue field chosen in such a way that the representative in $\Fbb_p$
are integers between $0$ and $p-1$.

\begin{remark} \label{197}
For any $a$ in $U_{1}$, we have
\begin{align*}
    (1+a\pi_{\alpha}^i)^p\ 
    &\equiv\ 1+a^p\pi_{\alpha}^{ip}+ap\pi_{\alpha}^i\\
    &\equiv\ 1+a^p\pi_{\alpha}^{ip}-a\pi_{\alpha}^{\phi(p^\alpha)+i}
    \qq \mod (\pi_{\alpha}^{\phi(p^\alpha)+i+1}).
\end{align*}
As
\[
    ip < \phi(p^\alpha)+i
    \qq \Lra \qq
    i < p^{\alpha-1},
\]
we obtain
\[
    (1+a\pi_{\alpha}^i)^p \equiv
    \begin{cases}
        1+a^p\pi_{\alpha}^{ip} \hspace{3.1em} 
        \mod (\pi_\alpha^{ip+1}) 
        & \text{if } i < p^{\alpha-1},\\
        1+(a^p-a)\pi_{\alpha}^{p^\alpha}\ 
        \mod (\pi_{\alpha}^{p^\alpha+1})
        & \text{if } i = p^{\alpha-1},\\
        1+a\pi_{\alpha}^{\phi(p^\alpha)+i} \hspace{1.3em}
        \mod (\pi_{\alpha}^{\phi(p^\alpha)+i+1})
        & \text{if } i > p^{\alpha-1}.
    \end{cases}
\]
The last congruence and the completeness of the filtration imply that
every element of $U_i$ for $i > \phi(p^\alpha) + p^{\alpha-1} =
p^\alpha$ is a $p$-th power, and it follows that $U_{p^\alpha+1} \subset
U_1^p$.
\end{remark}

Setting $\mu := \mu(\Qbb_p(F_0))$, we have
\[
    U_0/\lan \mu, U_0^p \ran 
    \iso U_1/\lan \mu \cap U_1, U_1^p \ran,
\]
where $\mu \cap U_1$ is the subgroup generated by $\zeta_{p^\alpha} =
\pi_{\alpha}+1$. By remark \ref{197}, there is a commutative diagram 
\[
    \xymatrix{
        (\Zbb_p(F_0')[X]/(Q_{\alpha}(X)))^\x 
        \ar[rr]_{\qqqq\iso}^{\qqqq\phi_{F_0}} \ar@{->>}[d]
        && \Zbb_p(F_0)^\x \ar@{->>}[d] \\
        (\Zbb_p(F_0')[X]/(Q_{\alpha}(X), X^{p^\alpha+1}))^\x 
        \ar[rr]_{\qqqq\iso} \ar@{->>}[d]
        && \Zbb_p(F_0)^\x/(\pi_{\alpha}^{p^\alpha+1})
        \ar@{->>}[d] \\
        (\Zbb_p(F_0')[X]/(Q_{\alpha}(X), X^{p^\alpha+1}))^\x
        /\lan \mu, \text{p-th powers} \ran 
        \ar[rr]_{\qqqq\iso}
        && U_1/\lan \mu \cap U_1, U_1^p \ran,
    }
\]
in which the vertical maps are the canonical projections and the
horizontal maps are isomorphisms. Since $-1$ becomes trivial in $U_1$,
the images of $\tilde{Q}_{\alpha}(X)$ and $\epsilon_{\alpha}$ in the
quotient group $U_1/\lan \mu \cap U_1, U_1^p \ran$ are equal. We want to
prove that this image is non-trivial. In order to do so, we consider the
$\pi_\alpha$-adic expansion of $-\epsilon_\alpha$ in
$\Zbb_p(\pi_\alpha)^\x \subset \Zbb_p(F_0)^\x$:
\[
    -\epsilon_{\alpha}
    = \phi_{F_0}(\tilde{Q}_{\alpha}(X))
    = \sum_{i \geq 0} c_i^{(\alpha)} \pi_{\alpha}^i,
\]
where $c_0^{(\alpha)} = 1$ and $0 \leq c_i^{(\alpha)} < p$ for each $i
\geq 0$. For $0 \leq i \leq p^2$, we let $c_i := c_i^{(2)}$.

\begin{remark} \label{266}
By definition
\[
    c_i^{(\alpha)} \equiv \frac{a_i^{(\alpha)}}{p} \q \mod p
    \qq \text{for } \alpha \geq 2 
    \text{ and } 0 \leq i < \phi(p^\alpha).
\]
\end{remark}

\begin{lemma} \label{199}
Let $k_2=(p-2)p+1$. Then $c_{k_2} = -1$ and
\[
    -\epsilon_{2} 
    \equiv
    \sum_{i=0}^{p-2} c_{ip} \pi_{2}^{ip}
    - \pi_{2}^{k_2} 
    + \sum_{i=k_2+1}^{p^2} c_i \pi_{2}^{i}
    \q \mod (\pi_{2}^{p^2+1}).
\]
\end{lemma}

\begin{proof}
This follows from lemma \ref{196} and \ref{213}.
\end{proof}

\begin{lemma} \label{106}
If
\[
    x = 1+\sum_{i=1}^\infty a_i \pi_{\alpha}^i
    \qq \text{and} \qq
    y = 1+\sum_{i=1}^\infty b_i \pi_{\alpha}^i
\]
are two elements of $U_1 \subset \Zbb_p(F_0)^\x$ such that $0 \leq a_i,
b_i < p$ and $x \equiv y$ modulo
$(\pi_{\alpha}^{k})$ for an integer $k \geq 1$, then
\[
    \frac{x}{y} 
    \equiv 1+(a_{k}-b_{k})\pi_{\alpha}^{k}
    \q \mod (\pi_{\alpha}^{k+1}).
\]
\end{lemma}

\begin{proof}
Let $z = x-y$. Then
\[
    z \equiv 0\ \mod (\pi_{\alpha}^{k})
    \qq \text{and} \qq
    \frac{z}{x} 
    \equiv 0\ \mod (\pi_{\alpha}^{k}).
\]
Therefore
\begin{align*}
    \frac{x}{y}\ 
    &=\ \frac{1}{1-\frac{z}{x}}\
    =\ 1+\frac{z}{x}
    +\frac{z^2}{x^2}+\ldots\\
    &\equiv\ 1+\frac{z}{x}\\
    &\equiv\ 1+\frac{(a_{k}-b_{k})\pi_{\alpha}^{k}}{x}\\
    &\equiv\ 1+(a_{k}-b_{k})\pi_{\alpha}^{k}
    \hspace{6em} \mod (\pi_{\alpha}^{k+1}).
\end{align*}
\end{proof}

\begin{lemma} \label{198}
If $x \in U_1 \subset \Zbb_p(F_0)^\x$ is such that
\[
    x \equiv 1+a_k\pi_{\alpha}^k 
    \q \mod (\pi_{\alpha}^{k+1}) 
\]
with $2 \leq k < p^\alpha$ prime to $p$ and $a_k \not\equiv 0$ modulo
$(\pi_{\alpha})$, then $x \not\in \lan \mu \cap U_1, U_1^p
\ran$.
\end{lemma}

\begin{proof}
Recall that $\mu \cap U_1$ is generated by $1+\pi_{\alpha}$. One must
check that $x$ cannot be written in the form
\[
    x = (1+\pi_{\alpha})^i y^{p}
\]
with $0 \leq i \leq p-1$ and $y \in U_1$. Since $k \geq 2$ by
assumption, it follows that $i=0$ and it remains to verify that $x$ is
not a $p$-th power. As a direct consequence of remark \ref{197}, we have
\[
    (1+a\pi_{\alpha}^i)^p 
    \equiv 1+a^p\pi_{\alpha}^{ip} \q \mod (\pi_\alpha^{ip+1})   
\]
for any $a \in U_1$ and any $i \leq \frac{k}{p}$. Hence $x \not\in
U_1^p$.
\end{proof}

We are now in position to prove that $-\epsilon_{\alpha}$ (and hence
$\epsilon_{\alpha}$) is non-trivial in the group $U_1/\lan \mu \cap
U_1, U_1^p \ran$.

\begin{theorem} \label{200}
If $p$ is odd and $\alpha \geq 2$, then
\[
    \epsilon_{\alpha} 
    \not\in \lan \mu(\Qbb_p(F_0)), (\Zbb_p(F_0)^\x)^p \ran.
\]
\end{theorem}

\begin{proof}
In order to obtain the result, it suffices to show that
$-\epsilon_{\alpha} \in U_1$ is non-trivial in the quotient $U_1/\lan
\mu \cap U_1, U_1^p \ran$. This is done by induction on $\alpha \geq 2$.

First consider the case $\alpha = 2$, and let $k_2 := (p-2)p+1$.
According to lemma \ref{199}
\[
    -\epsilon_{2}
    \equiv \sum_{i=0}^{p-2} c_{pi} \pi_{2}^{pi}
    \equiv \prod_{i=0}^{p-2} (1+\tilde{c}_{pi} \pi_{2}^i)^p
    \qq \mod (\pi_{2}^{k_2}),
\]
where each $0 \leq \tilde{c}_{j} < p$ is such that $(\tilde{c}_j)^p
\equiv c_j \mod p$, and where the second equivalence is due to the facts
that $k_2-1 < \phi(p^2)$ and $(p)=(\pi_{2}^{\phi(p^2)})$. Letting
\[
    z_2 
    = \prod_{i=1}^{p-2} (1+\tilde{c}_{pi} \pi_2^i)
    \q \in \Zbb_p(\pi_2)^\x,
\]
we get by lemma \ref{106} and \ref{199}
\[
    \frac{-\epsilon_{2}}{z_2^p}
    \equiv 1 + \tilde{c}_{k_2} \pi_{2}^{k_2}
    \equiv 1 - \pi_{2}^{k_2}
    \qq \mod (\pi_{2}^{k_2+1}).
\]
As $k_2 < p^2$, it follows from lemma \ref{198} that $-\epsilon_{2}$
does not belong to $\lan \mu \cap U_1, U_1^p \ran$.

Now let $\alpha \geq 2$ and $k_{\alpha} = p^\alpha-2p+1$. Suppose there
is an element $z_{\alpha}$ in $\Zbb_p(\pi_\alpha)^\x$ such that
\[
    \frac{ -\epsilon_{\alpha} }{z_{\alpha}^p}
    \equiv 1 
    + \sum_{i=k_{\alpha}}^{p^{\alpha}} d_i \pi_{\alpha}^i
    \qq \mod (\pi_{\alpha}^{p^{\alpha}+1}),
\]
with $d_{k_\alpha} \not\equiv 0 \mod p$. By corollary \ref{286} and
lemma \ref{284} we have
\[
    \epsilon_{\alpha+1}
    \equiv i_\alpha(\epsilon_\alpha)
    \equiv -(1 + \sum_{i=k_{\alpha}}^{p^{\alpha}} 
      d_i \pi_{\alpha+1}^{pi})
      i_\alpha(z_{\alpha})^p
    \qq \mod (\pi_{\alpha+1}^{p^{\alpha+1}+1}),
\]
so that
\[
    \frac{ -\epsilon_{\alpha+1} }{(z_{\alpha+1}')^p}
    \equiv 1 
    + \sum_{i=k_{\alpha}}^{p^{\alpha}} d_i \pi_{\alpha+1}^{pi}
    \qq \mod (\pi_{\alpha+1}^{p^{\alpha+1}+1})
\]
for some $z_{\alpha+1}'$ in $\Zbb_p(\pi_{\alpha+1})^\x$. Let
\[
    k_{\alpha+1} := k_{\alpha}+\phi(p^{\alpha+1})
    \qq \text{and} \qq
    \tilde{z}_{\alpha+1}
    := \prod_{i=k_{\alpha}}^{k_{\alpha+1}-1} 
    (1+\tilde{d}_i \pi_{\alpha+1}^i)
    \q \in \Zbb_p(\pi_{\alpha+1})^\x
\]
with each $0 \leq \tilde{d}_i < p$ such that $\tilde{d}_i^p \equiv d_i
\mod p$. Then by proposition \ref{290}, there is an element
\[
    z_{\alpha+1}
    = z_{\alpha+1}' \tilde{z}_{\alpha+1}
    \q \in \Zbb_p(\pi_{\alpha+1})
\]
such that
\[
    \frac{ -\epsilon_{\alpha+1} }{ z_{\alpha+1}^p}
    \equiv 1 
    + \sum_{i=k_{\alpha+1}}^{p^{\alpha+1}} \tilde{d}_i \pi_{\alpha+1}^{i}
    \qq \mod (\pi_{\alpha+1}^{p^{\alpha+1}+1})
\]
with $\tilde{d}_{k_{\alpha+1}} \not\equiv 0 \mod p$. Since
\[
    k_{\alpha+1}
    = k_{\alpha} + \phi(p^{\alpha+1})
    = p^{\alpha+1} -2p +1
    < p^{\alpha+1},
\]
we can apply lemma \ref{198} to obtain that $-\epsilon_{\alpha+1}$ is
non-trivial in $U_1/\lan \mu, U_1^p \ran$.
\end{proof}

\begin{example} \label{292}
Let us have a look at the case $p=3$. A straightforward calculation
yields
\[
    -\epsilon_{2} 
    \equiv 
    1+\pi_{2}^{3}-\pi_{2}^{4}-\pi_{2}^{5}
    -\pi_{2}^{6}-\pi_{2}^{7}+\pi_{2}^{8}+\pi_{2}^{9}
    \qq \mod (\pi_2^{10}).
\]
Here $(p-2)p+1 = 4$, so letting $z_2 = (1+\pi_2)$ we get
\[
    \frac{-\epsilon_{2}}{z_2^3}
    \equiv 1-\pi_{2}^4 \q \mod (\pi_2^5).
\]
As $4<3^2$, we may apply lemma \ref{198} to obtain the result. Then if
$\alpha = 3$ we have
\[
    -\epsilon_{3} \equiv 
    1+\pi_{3}^{9}-\pi_{3}^{12}-\pi_{3}^{15}
    -\pi_{3}^{18}-\pi_{3}^{21}+\pi_{3}^{24}+\pi_{3}^{27}
    \qq \mod (\pi_3^{28}).
\]
Letting
\[
    z_3
    = (1+\pi_3^3)(1-\pi_3^4)(1-\pi_3^5)(1-\pi_3^6),
\]
so that
\begin{align*}
    z_3^3\ 
    &\equiv\
    (1+\pi_{3}^3)^3 (1-\pi_{3}^4)^3 (1-\pi_{3}^5)^3 (1-\pi_{3}^6)^3\\
    &\equiv\ 1+\pi_{3}^9-\pi_{3}^{12}-\pi_{3}^{15}-\pi_{3}^{18}
    -\pi_{3}^{21}+\pi_{3}^{22} \q \mod (\pi_{3}^{23}),
\end{align*}
we find
\[
    \frac{\epsilon_{3}}{z_3^3}
    \equiv 1-\pi_{3}^{22} \q \mod (\pi_{3}^{23}).
\]
As $22<3^3$, we may again apply lemma \ref{198} to obtain the result.
\end{example}

\begin{corollary} \label{201}
If $p$ is odd, $\alpha \geq 2$ and $u \in \Zbb_p^\x$, then
\[
    \frac{\epsilon_{\alpha}}{u} 
    \q \text{is non-trivial in} \q 
    \Zbb_p(F_0)^\x / \lan \mu(\Qbb_p(F_0)), (\Zbb_p(F_0)^\x)^p \ran.
\]
\end{corollary}

\begin{proof}
Let $u_1$ denote the projection of $u$ onto $U_1(\Zbb_p^\x) \subset
\Zbb_p^\x$. Then
\[
    u_1^{-1} 
    = 1 + \sum_{i\geq 1} v_i p^i\
    \equiv 1+v_1p
    \q \mod (\pi_{\alpha}^{p^\alpha+1}),
\]
for some $0 \leq v_i < p$. Clearly, the projection of $\epsilon_{\alpha}
u^{-1}$ in the group $U_1$ via the canonical decomposition $\Zbb_p^\x =
\mu_{p'} \x U_1$ is equal to $-\epsilon_{\alpha}u_1^{-1}$, and it is
enough to check that $-\epsilon_{\alpha}u_1^{-1}$ does not belong to
$\lan \mu \cap U_1, U_1^p \ran$.

If $\alpha=2$, then
\begin{align*}
    -\epsilon_{2}u_{1}^{-1}\
    &\equiv\ -\epsilon_{2} \hspace{6em}
    \mod\ (p) = (\pi_{\alpha}^{\phi(p^2)})\\
    &\equiv\ 1-\pi_{\alpha}^{(p-2)p+1} \qq 
    \mod\ \lan \mu \cap U_1, U_1^p, (\pi_{\alpha}^{\phi(p^2)}) \ran,
\end{align*}
and the result follows from theorem \ref{200}.

If $\alpha \geq 3$, we know from the proof of theorem \ref{200} that for
a suitable $z_\alpha \in \Zbb_p(\pi_\alpha)^\x$ we have
\[
    \frac{-\epsilon_\alpha}{z_\alpha^p}
    \equiv 1+(-1)^\alpha \pi_\alpha^{k_\alpha}
    \q \mod (\pi_\alpha^{k_\alpha+1}),
\]
for $k_\alpha = p^\alpha-2p+1$, so that by proposition \ref{290}
\begin{align*}
    \frac{-\epsilon_\alpha}{u_1 z_\alpha^p}\
    &\equiv\ (1+(-1)^\alpha \pi_\alpha^{k_\alpha})(1+v_1 p)\\
    &\equiv\ 1-v_1 \pi_\alpha^{\phi(p^\alpha)}
    +(-1)^\alpha \pi_\alpha^{k_\alpha}
    \qq \mod (\pi_\alpha^{k_\alpha+1}).
\end{align*}
As
\[
    k_\alpha
    = (p-2)p+1+\sum_{i=3}^\alpha \phi(p^i)
    < \sum_{i=2}^\alpha \phi(p^i),
\]
we obtain
\[
    \frac{-\epsilon_\alpha}{u_1 z_\alpha^p y_\alpha^p}
    \equiv 1+(-1)^\alpha \pi_\alpha^{k_\alpha}
    \q \mod (\pi_\alpha^{k_\alpha+1}),
\]
where after successive multiplications eliminating all terms in
$\pi_\alpha^k$ for $1 < k < k_\alpha$, we have have used the element
\[
    y_\alpha
    = (1-\tilde{v}_1 \pi_\alpha^{\phi(p^{\alpha-1})})
    \prod_{k=2}^{\alpha-1} 
    (1+(-1)^{k-1}\pi_\alpha^{[\sum_{i=1}^k \phi(p^{\alpha-i})]})
\]
with $\tilde{v}_1 \in \Zbb$ such that $(\tilde{v}_1)^p \equiv v_1 \mod
p$. It follows that $-\epsilon_\alpha u_1^{-1} \not\in \lan \mu \cap
U_1, U_1^p \ran$.
\end{proof}

\begin{corollary} \label{202}
If $p$ is odd and $F_0 = \mu(\Qbb_p(F_0))$, then
$\tilde{\Fcal}_u(\Qbb_p(F_0), \tilde{F_0}, r_1)$ is non-empty if and
only if
\[
    r_1 \q \text{divides} \q
    \begin{cases}
        1 & \text{if } \alpha = 0,\\
        p-1 & \text{if } \alpha \geq 1.
    \end{cases}
\]
\end{corollary}

\begin{proof}
The result for $\alpha=0$ is obvious; so let $\alpha \geq 1$. Since $F_0
= \mu(\Qbb_p(F_0))$, corollary \ref{095} applies if $r_1$ divides $p-1$.
Because $r_1$ must divide the ramification index $\phi(p^\alpha) =
(p-1)p^{\alpha-1}$ of $\Qbb_p(F_0)$ over $\Qbb_p$, it remains to show
that $\tilde{\Fcal}_u(\Qbb_p(F_0), \tilde{F_0}, r_1)$ is empty whenever
$p$ divides $r_1$ with $\alpha \geq 2$. This however is a direct
consequence of theorem \ref{098} and corollary \ref{201}.
\end{proof}

\section{The $p$-part of $r_1$ for $p = 2$} 
\label{229}

We now investigate the case where $p=2$. Recall that $\alpha \geq 1$ is
defined to satisfy $|F_0 \cap S_n| = 2^\alpha$. We know from example
\ref{097} that $\epsilon_{\alpha} = 2^{-1}\pi_\alpha^{\phi(2^\alpha)}$
always belongs to $(\Zbb_2(F_0)^\x)^2$ modulo the subgroup generated by
$-\zeta_4$ if $\alpha \geq 2$. We will hence look for $4$-th powers when
$\alpha \geq 3$.

Let $\mu := \mu(\Qbb_2(F_0))$ denote the group of roots of unity in
$\Qbb_2(F_0)$ and fix $\zeta_{2^\alpha}$ a primitive $2^\alpha$-th root
of unity in $\mu$. As in the previous section we consider the decreasing
filtration
\[
    \Zbb_2(F_0)^\x
    = U_{0} 
    \supset U_{1}
    \supset U_{2}
    \supset \ldots
\]
given by $U_0 = \Zbb_2(F_0)^\x$ and
\[
    U_{i}
    = \{ x \in U_0\ |\ x \equiv 1\ \mod (\pi_{\alpha}^i) \}
    \qq \text{for } i \geq 1, 
\]
where $U_0/U_1 = \mu_{2'}(\Qbb_2(F_0))$, and where $U_i/U_{i+1}$ is
isomorphic to the residue field of $\Qbb_2(F_0)$ for each $i \geq 1$.
Recall that $(\pi_{\alpha}^{2^{\alpha-1}}) = (2)$. Define 
\[
    Q_{\alpha}(X)
    := \frac{(X+1)^{2^\alpha}-1}{(X+1)^{2^{\alpha-1}}-1}
    = (X+1)^{2^{\alpha-1}}+1
    \q \in \Zbb[X]
\]
to be the minimal polynomial of $\pi_{\alpha}$ over $\Qbb_2$.

\begin{lemma} \label{267}
If $p=2$ and $\alpha \geq 3$, then
\[
    Q_{\alpha}(\pi_{\alpha})\
    \begin{cases}
        =\ 2+4\pi_{\alpha}+6\pi_{\alpha}^2+4\pi_{\alpha}^3+\pi_{\alpha}^{4}
        & \text{if } \alpha = 3,\\[1ex]
        \equiv\ 2+4\pi_{\alpha}^{2^{\alpha-3}}+6\pi_{\alpha}^{2^{\alpha-2}}
        +4\pi_{\alpha}^{3 \cdot 2^{\alpha-3}}+\pi_{\alpha}^{2^{\alpha-1}}
        \q \mod (8)
        & \text{if } \alpha \geq 4.
    \end{cases}
\]
\end{lemma}

\begin{proof}
The result for $\alpha = 3$ is clear. When $\alpha \geq 4$, the result
follows by induction on $4 \leq h \leq \alpha$ using the identity
\begin{align*}
    (1 + X)^{2^{h-1}}
    &=\ (1+X^2+2X)^{2^{h-2}}\\
    &\equiv\ (1+X^2)^{2^{h-2}} + 2^{h-1}X(1+X^2)^{2^{h-2}-1}
    \qq \mod (2^h).
\end{align*}
\end{proof}

\begin{proposition} \label{293}
If $p=2$ and $\alpha \geq 2$, then
\[
    2
    \equiv \pi_\alpha^{\phi(2^\alpha)} 
    + \pi_\alpha^{\phi(2^\alpha)+2^{\alpha-2}}
    \q \mod (\pi_\alpha^{2^\alpha}).
\]
\end{proposition}

\begin{proof}
By lemma \ref{267} we have
\[
    Q_\alpha(\pi_\alpha)
    \equiv 2 + 2\pi_\alpha^{2^{\alpha-2}} 
    + \pi_\alpha^{2^{\alpha-1}}
    \q \mod (4).
\]
Hence
\begin{align*}
    2\
    &\equiv\ -\pi_\alpha^{2^{\alpha-1}}-2\pi_\alpha^{2^{\alpha-2}}\\
    &\equiv\ \pi_\alpha^{2^{\alpha-1}}
    -(-\pi_\alpha^{2^{\alpha-1}}-2\pi_\alpha^{2^{\alpha-2}}) 
    \pi_\alpha^{2^{\alpha-2}}
    -(-\pi_\alpha^{2^{\alpha-1}}-2\pi_\alpha^{2^{\alpha-2}}) 
    \pi_\alpha^{2^{\alpha-1}}\\ 
    &\equiv\ \pi_\alpha^{2^{\alpha-1}}
    +\pi_\alpha^{2^{\alpha-1}+2^{\alpha-2}}
    +2\pi_\alpha^{2^{\alpha-1}}
    +\pi_\alpha^{2^\alpha}
    +2\pi_\alpha^{2^{\alpha-1}+2^{\alpha-2}}\\
    &\equiv\ \pi_\alpha^{2^{\alpha-1}}
    +\pi_\alpha^{2^{\alpha-1}+2^{\alpha-2}}
    -\pi_\alpha^{2^{\alpha}}
    +\pi_\alpha^{2^\alpha}\\
    &\equiv\ \pi_\alpha^{2^{\alpha-1}}
    +\pi_\alpha^{2^{\alpha-1}+2^{\alpha-2}}
    & \mod (\pi_\alpha^{2^\alpha}).
\end{align*}
\end{proof}

In the cases where $\Qbb_2(F_0)$ is completely ramified over $\Qbb_2$,
for example if $\mu_{2'}(\Qbb_2(F_0))$ is trivial, we have $U_0 = U_1$.
From the fact that $\epsilon_{\alpha} \in \Zbb_2(\pi_{\alpha}) \subset
\Zbb_2(F_0)^\x$ in general, it follows that $\epsilon_{\alpha}$ always
belongs to $U_1$. 

\begin{proposition} \label{111}
If $p=2$ and $\alpha \geq 3$, then
\[
    \epsilon_{\alpha}
    \equiv 1+\pi_{\alpha}^{2 \cdot 2^{\alpha-3}}
    +\pi_{\alpha}^{4 \cdot 2^{\alpha-3}}
    +\pi_{\alpha}^{5 \cdot 2^{\alpha-3}}
    +\pi_{\alpha}^{6 \cdot 2^{\alpha-3}}
    \q \mod (\pi_{\alpha}^{2^\alpha}).
\]
\end{proposition}

\begin{proof}
For simplicity we let $Z = \pi_{\alpha}^{2^{\alpha-3}}$. Since
\[
    -2 \equiv 4Z+6Z^2+4Z^3+Z^4 \q \mod (8)
\]
by lemma \ref{267}, we obtain
\begin{align*}
    \frac{ \pi_{\alpha}^{2^{\alpha-1}} }{2}\
    &\equiv\ -1-2Z-3Z^2-2Z^3\\[1ex]
    &\equiv\ 1+(Z^4+6Z^2+4Z+4Z^3)+Z(Z^4+6Z^2+4Z+4Z^3)\\
    &\q +Z^2+Z^3(Z^4+6Z^2+4Z+4Z^3)\\[1ex]
    &\equiv\ 1+Z^2+Z^4+Z^5+Z^7+2Z^2+2Z^3+2Z^5\\[1ex]
    &\equiv\ 1+Z^2+Z^4+Z^5+Z^7-Z^2(Z^4+6Z^2+4Z+4Z^3)\\
    &\q -Z^3(Z^4+6Z^2+4Z+4Z^3)-Z^5(Z^4+6Z^2+4Z+4Z^3)\\[1ex]
    &\equiv\ 1+Z^2+Z^4+Z^5+Z^6
    &\mod (4).
\end{align*}
\end{proof}

We will now prove that $\epsilon_{\alpha}$ is non-trivial in the
quotient group $U_1/\lan \mu \cap U_1, U_1^4 \ran$.

\begin{lemma} \label{203}
Let 
\[
    x 
    = 1+ \sum_{i \geq k} \lambda_i \pi_\alpha^i\
    \in \lan \mu \cap U_1, U_1^4 \ran
\]
with $\lambda_k \not\equiv 0 \mod \pi_\alpha$ and each $\lambda_i$
satisfying $\lambda_i^q = \lambda_i$ for $q$ the cardinality of the
residue field of $\Qbb_2(F_0)$. If $3 \leq k < 2^\alpha$, then $k \equiv
0 \mod 4$.  If $\alpha = 3$ and $k = 2$, then $\lambda_6 = 0$.
\end{lemma}

\begin{proof}
Recall that $\mu \cap U_1$ is generated by $1+\pi_{\alpha}$, so that $x$
is of the form
\[
    x = (1+\pi_{\alpha})^h y^{4}
    \qq \text{with } 0 \leq h \leq 3
    \text{ and } y \in U_1.
\]

If $3 \leq k < 2^\alpha$, it easily follows that $h=0$ and $x$ is a
$4$-th power. Furthermore, for any $a \in \Zbb_2(F_0)$ we have
\begin{align*}
    (1+a\pi_{\alpha}^i)^4\ 
    &=\ 1+4a\pi_{\alpha}^i+6a^2 \pi_{\alpha}^{2i}
    +4a^3\pi_{\alpha}^{3i}+a^4\pi_{\alpha}^{4i}\\
    &\equiv\ 1+6a^2 \pi_{\alpha}^{2i}+a^4\pi_{\alpha}^{4i}
    & \mod (\pi_{\alpha}^{i+2\phi(2^\alpha)}).
\end{align*}
As
\[
    4i \leq \phi(2^\alpha)+2i
    \qq \Lra \qq
    i \leq 2^{\alpha-2},
\]
we obtain
\[
    (1+a\pi_{\alpha}^i)^4 \equiv
    \begin{cases}
        1+a^4\pi_{\alpha}^{4i} \hspace{3.5em} 
        \mod (\pi_{\alpha}^{4i+1}) 
        & \text{if } i < 2^{\alpha-2},\\
        1+(a^4+a^2)\pi_{\alpha}^{2^\alpha}\ 
        \mod (\pi_{\alpha}^{2^\alpha+1})
        & \text{if } i = 2^{\alpha-2},    
    \end{cases}
\]
and the result for $3 \leq k < 2^\alpha$ follows.

Now suppose that $\alpha=3$ and $k=2$. In this case $h=2$ and $y =
1+b\pi_3$ for some $b \in \Zbb_2(F_0)$. Using proposition \ref{293} we
get
\begin{align*}
    x\
    &\equiv\ (1+\pi_3)^2 (1+b\pi_3)^4 \\
    &\equiv\ (1+\pi_3^2+\pi_3^5+\pi_3^7)(1+b^4\pi_3^4+b^2\pi_3^6)\\
    &\equiv\ 1+\pi_3^2+b^4\pi_3^4+\pi_3^5+(b^2+b^4)\pi_3^6+\pi_3^7
    \hspace{4em} \mod (\pi_3^8).
\end{align*}
If $\lambda_4=0$, then $b \equiv 0 \mod (\pi_3)$ and consequently
$\lambda_6 = 0$. On the other hand if $\lambda_4 = 1$, then $b \equiv 1
\mod (\pi_3)$ and once again $\lambda_6=0$.
\end{proof}

The idea of theorem \ref{113} is the same as in the case $p>2$: we
divide the $\pi_{\alpha}$-adic expression of $\epsilon_{\alpha}$ by a
$4$-th power in $U_1$ in such a way that the resulting expression is in a
form that allows lemma \ref{203} to be used.

\begin{theorem} \label{113}
If $p=2$ and $\alpha \geq 3$, then
\[
    \epsilon_{\alpha} \not\in
    \lan \mu(\Qbb_2(F_0)), (\Zbb_2(F_0)^\x)^4 \ran.
\]
\end{theorem}

\begin{proof}
Since $\epsilon_{\alpha} \in U_1$ it is enough to show that
$\epsilon_{\alpha} \not\in \lan \mu \cap U_1, U_1^4 \ran$.

In case $\alpha = 3$, we know from proposition \ref{111} that
\[
    \epsilon_{3}
    \equiv 1+\pi_3^{2}+\pi_3^{4}+\pi_3^{5}+\pi_3^{6}
    \q \mod (\pi_3^8),
\]
and a direct application of lemma \ref{203} yields the result.

Now assume $\alpha = 4$, and let $k_4 := 2^4-2 = 14$. By proposition
\ref{111}
\[
    \epsilon_{4}
    \equiv 1+\pi_{4}^{4}
    +\pi_{4}^{8}
    +\pi_{4}^{10}
    +\pi_{4}^{12}
    \q \mod (\pi_{4}^{16}).
\]
As
\begin{align*}
    (1+\pi_{4})^4 (1+\pi_{4}^2)^4\
    &\equiv\ (1+\pi_{4}^4+\pi_{4}^{10}+\pi_{4}^{14})
      (1+\pi_{4}^8+\pi_{4}^{12})\\
    &\equiv\ 1+\pi_{4}^4+\pi_{4}^8+\pi_{4}^{10}
      +\pi_{4}^{12}+\pi_{4}^{14}
    & \mod (\pi_{4}^{16}),
\end{align*}
letting
\[
    z_4 = (1+\pi_4) (1+\pi_4^2) (1+\pi_4^3)
    \q \text{in} \q \Zbb_2(\pi_4)^\x,
\]
we get
\[
    \frac{\epsilon_{4}}{z_4^4}
    \equiv 1+\pi_{4}^{k_4} \q \mod (\pi_{4}^{16}).
\]
Hence by lemma \ref{203}, it follows that $\epsilon_{4}$ does not
belong to $\lan \mu \cap U_1, U_1^4 \ran$.

Furthermore assume $\alpha \geq 4$, and let $k_{\alpha} := 2^\alpha-2$.
Suppose there is an element $z_{\alpha}$ in $\Zbb_2(\pi_\alpha)^\x$ such
that
\[
    \frac{ \epsilon_{\alpha} }{z_{\alpha}^4}
    \equiv 1 
    + d_{k_{\alpha}} \pi_{\alpha}^{k_{\alpha}}
    \qq \mod (\pi_{\alpha}^{2^\alpha}),
\]
with $d_{k_\alpha} \not\equiv 0 \mod 2$ in the residue field. By
corollary \ref{286} and lemma \ref{284} we have
\[
    \epsilon_{\alpha+1}
    \equiv i_\alpha(\epsilon_\alpha)
    \equiv (1 + d_{k_\alpha} \pi_{\alpha+1}^{2k_\alpha})
      i_\alpha(z_{\alpha})^4
    \qq \mod (\pi_{\alpha+1}^{2^{\alpha+1}}),
\]
so that
\[
    \frac{ \epsilon_{\alpha+1} }{(z_{\alpha+1}')^4}
    \equiv 1 + d_{k_\alpha} \pi_{\alpha+1}^{2k_\alpha}
    \qq \mod (\pi_{\alpha+1}^{2^{\alpha+1}})
\]
for a suitable $z_{\alpha+1}'$ in $\Zbb_2(\pi_{\alpha+1})^\x$. Let
$k_{\alpha+1} := k_{\alpha}+\phi(2^{\alpha+1}) = 2^{\alpha+1}-2$, and
let
\[
    \tilde{z}_{\alpha+1}
    :=  1+\tilde{d}_{k_\alpha} \pi_{\alpha+1}^{\frac{k_\alpha}{2}}
    \q \in \Zbb_2(\pi_{\alpha+1})^\x
\]
with $\tilde{d}_{k_\alpha}$ in the residue field such that
$\tilde{d}_{k_\alpha}^4 \equiv d_{k_\alpha} \mod 2$. Note that 
\[
    (1+\tilde{d}_{k_\alpha} \pi_{\alpha+1}^{\frac{k_\alpha}{2}})^4
    \equiv 1 + d_{k_\alpha} \pi_{\alpha+1}^{2k_\alpha}
    + \tilde{d}_{k_\alpha}^2 \pi_{\alpha+1}^{2^\alpha+k_\alpha}
    \qq \mod (\pi_{\alpha+1}^{2^{\alpha+1}}),
\]
where $2^\alpha+k_\alpha = 2^{\alpha+1}-2$. Then by proposition
\ref{293}, letting
\[
    z_{\alpha+1}
    = z_{\alpha+1}' \tilde{z}_{\alpha+1}
    \q \in \Zbb_2(\pi_{\alpha+1})^\x,
\]
we have
\[
    \frac{ \epsilon_{\alpha+1} }{ z_{\alpha+1}^4}
    \equiv 1 
    + d'_{k_{\alpha+1}} \pi_{\alpha+1}^{k_{\alpha+1}}
    \qq \mod (\pi_{\alpha+1}^{2^{\alpha+1}}),
\]
with $d'_{k_{\alpha+1}} \not\equiv 0 \mod 2$ in the residue field. Since
$k_{\alpha+1} < 2^{\alpha+1}$, we can apply lemma \ref{203} to obtain
that $\epsilon_{\alpha+1}$ is non-trivial in $U_1/\lan \mu, U_1^p \ran$.
The result then follows by induction on $\alpha \geq 4$.
\end{proof}

\begin{corollary} \label{114}
If $p=2$, $\alpha \geq 3$ and $u \in \Zbb_2^\x$, then
\[
    \frac{\epsilon_{\alpha}}{u} 
    \q \text{is non-trivial in} \q
    \Zbb_2(F_0)^\x / \lan \mu(\Qbb_2(F_0)), (\Zbb_2(F_0)^\x)^4 \ran.
\]
\end{corollary}

\begin{proof}
Since $u \in \Zbb_2^\x$, its inverse is of the form
\begin{align*}
    u^{-1}\ 
    &=\ 1 + \sum_{i\geq 1} v_i 2^i\\
    &\equiv\ 1+v_1 2\\
    &\equiv\ 1+v_1 (\pi_\alpha^{2^{\alpha-1}}
    +\pi_\alpha^{2^{\alpha-1}+2^{\alpha-2}})
    \qq \mod (4) = (\pi_{\alpha}^{2^\alpha}),
\end{align*}
where $v_i \in \{0, 1\}$ and where the last equivalence follows from
proposition \ref{293}. As in theorem \ref{113} it is enough to show that
$\epsilon_{\alpha}u^{-1}$ does not belong to $\lan \mu \cap U_1, U_1^4
\ran$.

If $\alpha=3$, we know from proposition \ref{111} that
\[
    \epsilon_3
    \equiv 1 + \pi_3^2 + \pi_3^4 + \pi_3^5 + \pi_3^6
    \q \mod (\pi_3^8).
\]
Hence
\begin{align*}
    \epsilon_3 u^{-1}\
    &\equiv\ (1+\pi_3^2+\pi_3^4+\pi_3^5+\pi_3^6)
    (1+v_1\pi_3^4+v_1\pi_3^6)\\
    &\equiv\ 1+\pi_3^2+(1+v_1)\pi_3^4+\pi_3^5+(1+2v_1)\pi_3^6\\
    &\equiv\ 1+\pi_3^2+(1+v_1)\pi_3^4+\pi_3^5+\pi_3^6
    \hspace{6em} \mod (\pi_3^8),
\end{align*}
and lemma \ref{203} implies $\epsilon_3 u^{-1} \not\in \lan \mu \cap
U_1, U_1^4 \ran$.

Now if $\alpha \geq 4$, we know from theorem \ref{113} that for a
suitable $z_\alpha \in \Zbb_2(\pi_\alpha)^\x$ we have
\[
    \frac{\epsilon_\alpha}{z^4}
    \equiv 1+\pi_\alpha^{k_\alpha}
    \q \mod (\pi_\alpha^{k_\alpha+1})
\]
for $k_\alpha = 2^\alpha-2$, so that by proposition \ref{293}
\begin{align*}
    \frac{\epsilon_\alpha}{u z_\alpha^4}
    &\equiv 1+v_1 \pi_\alpha^{\phi(2^\alpha)}
    + v_1 \pi_\alpha^{\phi(2^\alpha)+2^{\alpha-2}}
    + \pi_\alpha^{k_\alpha}
    \qq \mod (\pi_\alpha^{k_\alpha+1}).
\end{align*}
Letting
\[
    y_\alpha
    = (1+v_1 \pi_\alpha^{\phi(2^{\alpha-2})}),
\]
we have
\begin{align*}
    y_\alpha^4\
    &\equiv\ 1+v_1 \pi_\alpha^{\phi(2^\alpha)}
    + 2v_1 \pi_\alpha^{2^{\alpha-2}}\\
    &\equiv\ 1 + v_1 \pi_\alpha^{\phi(2^\alpha)}
    + v_1 \pi_\alpha^{\phi(2^\alpha)+2^{\alpha-2}}
    \qq \mod (\pi_\alpha^{2^\alpha}).
\end{align*}
Hence by lemma \ref{203}
\[
    \frac{\epsilon_\alpha}{u z^4 y_\alpha^4}
    \equiv 1+\pi_\alpha^{k_\alpha}
    \q \mod (\pi_\alpha^{k_\alpha+1}),
\]
and $\epsilon_\alpha u^{-1} \not\in \lan \mu \cap U_1, U_1^4 \ran$.
\end{proof}

\begin{corollary} \label{115}
If $p=2$, then $\tilde{\Fcal}_u(\Qbb_2(F_0), \tilde{F_0}, r_1)$ is
non-empty if and only if
\[
    r_1 
    \q \text{divides} \q
    \begin{cases}
        2 & 
        \text{if } \alpha \geq 2 
        \text{ with either } u \equiv \pm 1 \mod 8 
        \text{ or } \zeta_3 \in F_0,\\
        1 & 
        \text{if } \alpha \leq 1,
        \text{ or } u \equiv \pm 3 \mod 8 
        \text{ and } \zeta_{3} \not\in F_0.
    \end{cases}
\]
\end{corollary}

\begin{proof}
The case $\alpha \leq 1$ is clear; so let $\alpha \geq 2$. A necessary
condition for $\tilde{\Fcal}_u(\Qbb_2(F_0), \tilde{F_0}, r_1)$ to be
non-empty is that $r_1$ divides $2^{\alpha-1}$, the ramification index
of $\Qbb_2(F_0)$ over $\Qbb_2$. If $\alpha=2$, then $r_1$ divides $2$.
Otherwise if $\alpha \geq 3$, corollary \ref{114} and theorem \ref{098}
imply that $r_1$ must also be a divisor of $2$. 

By theorem \ref{098}, the integer $r_1$ may be any divisor of $2$ if and
only if $\frac{\epsilon_{\alpha}}{u}$ is a square of $\Zbb_2(F_0)^\x$
modulo $F_0$. In fact this is true if and only if $u$ is a square of
$\Zbb_2(F_0)^\x$ modulo $F_0$, since by example \ref{097} we have
$\epsilon_\alpha \in \lan (\Zbb_2(F_0)^\x)^2, F_0 \ran$. The result is
then obvious if $u \equiv \pm 1 \mod 8$. Otherwise if $u \equiv \pm 3
\mod 8$ we have $u = \pm 3 z^2$ for some $z \in \Zbb_2^\x$, and it
remains to verify that $-3$ belongs to $\lan (\Zbb_2(F_0)^\x)^2, F_0
\ran$ if and only if $F_0$ contains a 3rd root of unity $\zeta_3$. If
such a $\zeta_3$ exists, we let $\rho = 2\zeta_3+1$, so that
\[
    \rho^2 = -3
    \qq \text{and} \qq
    \Qbb_2(\rho) = \Qbb_2(\zeta_3).
\]
Conversely if there is a $\rho$ such that $\rho^2=-3$, we take $\zeta_3
= \frac{1}{2}(\rho-1)$.
\end{proof}

\begin{corollary} \label{294}
If $p=2$, $F_0=\mu(\Qbb_2(F_0))$, $[\Qbb_2(F_0):\Qbb_2]=n$, then
$\tilde{\Fcal}_u(\Qbb_2(F_0), \tilde{F_0}, r_1)$ is non-empty if and
only if
\[
    r_1 
    \q \text{divides} \q
    \begin{cases}
        2 &
        \text{if } \alpha \geq 2 
        \text{ with either } u \equiv \pm 1 \mod 8 
        \text{ or } n_\alpha \text{ even},\\
        1 & 
        \text{if } \alpha \leq 1,
        \text{ or } u \equiv \pm 3 \mod 8 
        \text{ and } n_\alpha \text{ odd}.
    \end{cases}
\]
\end{corollary}

\begin{proof}
Under these assumptions, $F_0 \iso C_{2^\alpha(2^{n_\alpha}-1)}$ by
proposition \ref{342}. The result follows from corollary \ref{115} and
the fact $\zeta_3 \in F_0$ if and only if $n_\alpha$ is even.
\end{proof}

\begin{remark} \label{297}
When $F_0 = \mu(\Qbb_2(F_0))$, we know by corollary \ref{183} that
$\tilde{F_1}$ is the unique element of $\tilde{\Fcal}_u(\Qbb_2(F_0),
\tilde{F_0}, r_1)$. We may therefore assume $\tilde{F_1}$ to be of the
form
\[
    \tilde{F_1} = \lan x_1 \ran \x F_0
    \qq \text{with} \qq
    x_1 =
    \begin{cases}
        2u & \text{if } r_1 = 1,\\
        (1+i)t & \text{if } r_1 = 2,
    \end{cases}
\]
for $i$ a primitive $4$-th root of unity in $\Qbb_2(F_0)^\x$ and
\[
    t \in
    \begin{cases}
        \Zbb_2^\x & \text{if } u \equiv \pm 1 \mod 8,\\
        \Zbb_2(\zeta_3)^\x & \text{if } u \equiv \pm 3 \mod 8,
    \end{cases}
    \qq \text{with} \qq
    t^2 =
    \begin{cases}
        u & \text{if } u \equiv 1 \text{ or } -3 \mod 8,\\
        -u & \text{if } u \equiv -1 \text{ or } 3 \mod 8.
    \end{cases}
\]
\end{remark}

\section{The determination of $r_2$} 

We fix $F_0$ and $r_1$ such that $\tilde{\Fcal}_u(\Qbb_p(F_0),
\tilde{F_0}, r_1)$ is non-empty, and fix an element $\tilde{F_1}$ in
$\tilde{\Fcal}_u(\Qbb_p(F_0), \tilde{F_0}, r_1)$. Corollary \ref{202}
and \ref{115} provide conditions on $r_1$ for this to happen, in which
case, according to remark \ref{181}, there is an element $x_1 \in
\Qbb_p(F_0)$ satisfying
\[
    v(x_1) = \frac{1}{r_1}
    \qq \text{and} \qq
    \tilde{F_1} = \lan F_0, x_1 \ran.
\]
We want to determine for which integer $r_2$ the set
$\tilde{\Fcal}_u(C_{\Dbb_n^\x}(F_0), \tilde{F_1}, r_2)$ is non-empty,
that is, for which $r_2$ dividing $\frac{n}{r_1}$ there exists an
element $x_2 \in \Dbb_n^\x$ such that $x_2^{r_2} = a$ with $a \in
\tilde{F_1}$ and $v(a)=v(x_1)$, and such that $\Qbb_p(F_0, x_2)$ is a
commutative field extension of $\Qbb_p(F_0) = \Qbb_p(\tilde{F_1})$.

As seen in theorem \ref{185}, the existence of such an $x_2$ is
equivalent to the irreducibility of the polynomial $X^{r_2}-a$ over
$\Qbb_p(F_0)$ with $r_2$ dividing $n[\Qbb_p(F_0):\Qbb_p]^{-1}$. 

\begin{theorem} \label{116}
Let $K$ be a field, $a \in K^\x$ and $r \geq 2$ an integer. Then $X^r-a$
is irreducible over $K$ if and only if for all primes $q$ dividing $r$
the class $a \in H^2(\Zbb/r, K^\x)$ is non-trivial in $H^2(\Zbb/q,
K^\x)$, and if $\frac{-a}{4}$ is non-trivial in $H^2(\Zbb/4, K^\x)$ when
$4$ divides $r$, where all cohomology groups are with trivial modules.
\end{theorem}

\begin{proof}
This is just a cohomological interpretation of \cite{lang} chapter
VI theorem 9.1.
\end{proof}

In general, there is a well defined map
\[
    \Xi:\ a \mtoo K[X]/(X^r-a)
\]
from $H^2(\Zbb/r, K^\x)$ to the set of isomorphism classes of algebra
extensions of $K$ by equations of the form $X^r-a=0$. This map is
injective: if there is an extension in which $X^r-a = 0$ and $X^r-b = 0$
both have solutions, then $\frac{a}{b}$ is a $r$-th power and becomes
trivial in $H^2(\Zbb/r, K^\x)$. Denote by
\[
    H_F^2(\Zbb/r, K^\x) \subset H^2(\Zbb/r, K^\x)
\]
the subset of all elements of $H^2(\Zbb/r, K^\x)$ that are sent to a
commutative field extension of $K$ via $\Xi$. Furthermore assuming that
$K = \Qbb_p(F_0)$ has ramification index $e(\Qbb_p(F_0))$ over $\Qbb_p$,
we consider the homomorphism
\[
    i^\ast: H^2(\Zbb/r, \tilde{F_1}) 
    \raa H^2(\Zbb/r, \Qbb_p(F_0)^\x)
\]
induced by the inclusion $\tilde{F_1} \subset \Qbb_p(F_0)^\x$. We are
interested in understanding the set
\[
    H_F^2(\Zbb/r, \Qbb_p(F_0)^\x) \cap i^\ast(H^2(\Zbb/r, \tilde{F_1})).
\]
Note that we have a non-canonically split exact sequence
\[
    1 
    \raa \mu(\Qbb_p(F_0)) \x \Zbb_p^{[\Qbb_p(F_0):\Qbb_p]}
    \raa \Qbb_p(F_0)^\x
    \raa \lan \pi_{F_0} \ran
    \raa 1,
\]
for $\pi_{F_0}$ a uniformizing element in $\Qbb_p(F_0)$. Moreover there
is a commutative diagram
\begin{align} \label{227}
    \xymatrix{
        H^2(\Zbb/r, \tilde{F_1}) \ar[r]^{i^\ast} \ar@{->>}[d]
        & H^2(\Zbb/r, \Qbb_p(F_0)^\x) \ar@{->>}[d] \\
        H^2(\Zbb/r, \lan x_1 \ran) \ar@{=}[d] \ar[r]
        & H^2(\Zbb/r, \lan \pi_{F_0} \ran) \ar@{=}[d] \\
        \Zbb/r \ar[r]
        & \Zbb/r,
    }    
\end{align}
where the top vertical arrows are non-canonically split surjective
homomorphisms respectively induced by the canonical surjections
\[
    \tilde{F_1} \iso F_0 \x \lan x_1 \ran \raa \lan x_1 \ran
    \qq \text{and} \qq
    \Qbb_p(F_0)^\x \raa \lan \pi_{F_0} \ran,
\]
and where the bottom horizontal map is the identity if $\Qbb_p(F_0)$ is
unramified over $\Qbb_p$, or otherwise is by multiplication with
\[
    \frac{e(\Qbb_p(F_0))}{r_1} = 
    \begin{cases}
        p^{\alpha-1} \frac{p-1}{r_1} 
        & \text{if } p > 2,\\
        \frac{2^{\alpha-1}}{r_1}
        & \text{if } p = 2.    
        \end{cases}
\]
Define $r_{F_0, r_1}$ to be the greatest divisor of
$\frac{n}{[\Qbb_p(F_0):\Qbb_p]}$ which is prime to
\[
    \begin{cases}
        1 
        & \text{if } e(\Qbb_p(F_0)) = 1,\\
        \frac{p-1}{r_1}
        & \text{if } p > 2 \text{ and } \alpha \geq 1,\\
        \frac{2}{r_1}
        & \text{if } p=2,\ \alpha \geq 2 
        \text{ and either } u \equiv \pm 1 \mod 8
        \text{ or } \zeta_3 \in F_0,\\
        1
        & \text{if } p=2,\ u \not\equiv \pm 1 \mod 8
        \text{ and } \zeta_3 \not\in F_0.
    \end{cases}
\]

\begin{theorem} \label{223}
Suppose $p > 2$ and $\tilde{F_1} \in \tilde{\Fcal}_u(\Qbb_p(F_0),
\tilde{F_0}, r_1) \neq \emptyset$. If $r_2$ divides $r_{F_0, r_1}$, then
$\tilde{\Fcal}_u(C_{\Dbb_n^\x}(F_0), \tilde{F_1}, r_2)$ is non-empty.
\end{theorem}

\begin{proof}
First note that if $\alpha = 0$, we must have $r_1 = 1$ and $x_1$ is a
uniformizing element of the unramified extension $\Qbb_p(F_0)/\Qbb_p$.
In this case $r_{F_0, r_1} = \frac{n}{[\Qbb_p(F_0):\Qbb_p]}$ and the
result follows from the embedding theorem.

Now assume that $\alpha \geq 1$, and let $r' = r_{F_0, r_1}'$ denote the
$p'$-part of $r = r_{F_0, r_1}$. Then for any prime $q$ dividing $r'$,
diagram (\ref{227}) can be extended to the commutative diagram
\[
    \xymatrix{
        H^2(\Zbb/r, \tilde{F_1}) \ar[r]^{i^\ast} \ar[d]
        & H^2(\Zbb/r, \Qbb_p(F_0)^\x) \ar[d] \ar[r]^{j^\ast}
        & H^2(\Zbb/q, \Qbb_p(F_0)^\x) \ar[d] \\
        \Zbb/r \ar[r]
        & \Zbb/r \ar@{->>}[r]
        & \Zbb/q,
    }
\]
where $j: \Zbb/q \ra \Zbb/r$ is the inclusion and the bottom right
horizontal map is the canonical projection. Since $r$ is prime to
$\frac{p-1}{r_1}$, it follows that $r'$, and hence $q$, are prime to
$\frac{e(\Qbb_p(F_0))}{r_1}$. Thus for any $\delta \in F_0$, the image
of $x_1 \delta \in \tilde{F_1}$ is non-trivial in $\Zbb/q = H^2(\Zbb/q,
\lan \pi_{F_0} \ran)$, and consequently non-trivial in $H^2(\Zbb/q,
\Qbb_p(F_0)^\x)$. In a similar way, it is equally non-trivial in
$H^2(\Zbb/4, \Qbb_p(F_0)^\x)$ if $4$ divides $r$. The result then
follows from theorem \ref{116} and corollary \ref{201}, where we have
shown that $x_1^{p-1}$, and hence $x_1$, is non-trivial in $H^2(\Zbb/p,
\Qbb_p(F_0)^\x)$. 
\end{proof}

\begin{theorem} \label{225}
Suppose $p = 2$ and $\tilde{F_1} \in \tilde{\Fcal}_u(\Qbb_2(F_0),
\tilde{F_0}, r_1) \neq \emptyset$. If $r_2$ divides $r_{F_0, r_1}$, then
$\tilde{\Fcal}_u(C_{\Dbb_n^\x}(F_0), \tilde{F_1}, r_2)$ is non-empty.
\end{theorem}

\begin{proof}
First note that if $\alpha \leq 1$, we must have $r_1 = 1$ and $x_1$ is
a uniformizing element of the unramified extension $\Qbb_2(F_0)/\Qbb_2$.
In this case $r_{F_0, r_1} = \frac{n}{[\Qbb_2(F_0):\Qbb_2]}$ and the
result follows from the embedding theorem.

Now assume that $\alpha \geq 2$. If $r_2$ is divisible by $2$, then so
is $r_{F_0, r_1}$ and we know from corollary \ref{115} that $x_1$ is
non-trivial in $H^2(\Zbb/2, \Qbb_2(F_0)^\x)$. Moreover if $r_2$ is
divisible by $4$, the fact that
\[
    (1+\zeta_4)^4 = -4
\]
imply that $\frac{-x_1}{4}$ is non-trivial in $H^2(\Zbb/4,
\Qbb_2(F_0)^\x)$. Furthermore for any odd prime $q$ dividing $r =
r_{F_0, r_1}$, diagram (\ref{227}) can be extended to the commutative
diagram
\[
    \xymatrix{
        H^2(\Zbb/r, \tilde{F_1}) \ar[r]^{i^\ast} \ar[d]
        & H^2(\Zbb/r, \Qbb_2(F_0)^\x) \ar[d] \ar[r]^{j^\ast}
        & H^2(\Zbb/q, \Qbb_2(F_0)^\x) \ar[d] \\
        \Zbb/r \ar[r]
        & \Zbb/r \ar@{->>}[r]
        & \Zbb/q,
    }
\]
where $j: \Zbb/q \ra \Zbb/r$ is the inclusion and the bottom right
horizontal map is the canonical projection. As $q$ is prime to
$\frac{2}{r_1}$, the image of $x_1$ is non-trivial in $\Zbb/q =
H^2(\Zbb/q, \lan \pi_{F_0} \ran)$, and consequently non-trivial in
$H^2(\Zbb/q, \Qbb_2(F_0)^\x)$. We may thus apply theorem \ref{116} to
obtain the desired result.
\end{proof}

We say that $r_1$ is \emph{maximal} if $\tilde{\Fcal}_u(\Qbb_p(F_0),
\tilde{F_0}, r_1)$ is non-empty and $\tilde{\Fcal}_u(\Qbb_p(F_0),
\tilde{F_0}, r)$ is empty whenever $r>r_1$.

\begin{corollary} \label{224}
Let $p$ be any prime. If $r_1$ is maximal, then
\[
    \tilde{\Fcal}_u(C_{\Dbb_n^\x}(F_0), \tilde{F_1}, r_2) \neq \emptyset
    \qq \text{if and only if} \qq
    r_2\ |\ \frac{n}{[\Qbb_p(F_0):\Qbb_p]},
\]
and any element $\tilde{F_2}$ in $\tilde{\Fcal}_u(C_{\Dbb_n^\x}(F_0),
\tilde{F_1}, n[\Qbb_p(F_0):\Qbb_p]^{-1})$ generates a maximal
commutative field $\Qbb_p(\tilde{F_2})$ in $\Dbb_n$. Moreover if $F_0 =
\mu(\Qbb_p(F_0))$, the number of such field extensions is equal to
\[
    |H^2(\Zbb/r_2, F_0)| = |F_0 \ox \Zbb/r_2|.
\]
\end{corollary}

\begin{proof}
By the maximality of $r_1$ we have
\[
    r_{F_0, r_1}
    = \frac{n}{[\Qbb_p(F_0):\Qbb_p]}
\]
and $\tilde{\Fcal}_u(C_{\Dbb_n^\x}(F_0), \tilde{F_1}, r_2) \neq
\emptyset$ implies $r_2\ |\ \frac{n}{[\Qbb_p(F_0):\Qbb_p]}$. The first
assertion then follows from theorem \ref{223} and \ref{225}.

As for the last assertion, if $F_0 = \mu(\Qbb_p(F_0))$, diagram
(\ref{227}) can be extended, via the short exact sequences 
\begin{align*}
    &1 \raa F_0 \raa \tilde{F_1} \raa \lan x_1 \ran \raa 1,\\[1ex]
    &1 \raa F_0 \x \Zbb_p^{[\Qbb_p(F_0):\Qbb_p]}
      \raa \Qbb_p(F_0)^\x \raa \lan \pi_{F_0} \ran \raa 1,\\[1ex]
    &1 \raa F_0 \raa F_0 \x \Zbb_p^{[\Qbb_p(F_0):\Qbb_p]}
      \raa \Zbb_p^{[\Qbb_p(F_0):\Qbb_p]} \raa 1,
\end{align*}
to the exact diagram
\begin{align*}
    \xymatrix{
        & H^1(\Zbb/r_2, \lan x_1 \ran) \ar[d]
        & H^1(\Zbb/r_2, \lan \pi_{F_0} \ran) \ar[d]\\
        H^1(\Zbb/r_2, \Zbb_p^{[\Qbb_p(F_0):\Qbb_p]}) \ar[r]
        & H^2(\Zbb/r_2, F_0) \ar[r] \ar[d]
        & H^2(\Zbb/r_2, F_0 \x \Zbb_p^{[\Qbb_p(F_0):\Qbb_p]}) \ar[d]\\
        & H^2(\Zbb/r_2, \tilde{F_1}) \ar[r]^{i^\ast} \ar@{->>}[d]
        & H^2(\Zbb/r_2, \Qbb_p(F_0)^\x) \ar@{->>}[d] \\
        & H^2(\Zbb/r_2, \lan x_1 \ran) \ar[r]
        & H^2(\Zbb/r_2, \lan \pi_{F_0} \ran), \\
    }    
\end{align*}
where all three first cohomology groups are trivial, and where we know
from diagram (\ref{227}) that the bottom vertical maps are surjective.
In particular $i^\ast$ is injective on the kernels of the bottom
vertical maps, and therefore the number of maximal fields of the form
$\Qbb_p(\tilde{F_2})$ in $\Dbb_n$ is given by the cardinality of
$H^2(\Zbb/r_2, F_0)$.
\end{proof}


\chapter{On maximal finite subgroups of $\Gbb_n(u)$} 
\label{268}

We consider a prime $p$, a positive integer $n=(p-1)p^{k-1}m$ with $m$
prime to $p$, and a unit $u \in \Zbb_p^\x$. In this chapter, we work in
the context of section \ref{150} in order to study the classes of
maximal nonabelian finite subgroups of $\Gbb_n(u)$.  

More particularly, we consider (nonabelian) finite extensions of
$\tilde{F_2}$ when $F_0$ is maximal as an abelian finite subgroup of
$\Sbb_n$ and the field $L = \Qbb_p(\tilde{F_2})$ is maximal in $\Dbb_n$.
In this case $C_{\Dbb_n^\x}(\tilde{F_2}) = L^\x$ and we have a short
exact sequence
\[
    1 \raa \tilde{F_2} \raa L^\x \raa L^\x/\tilde{F_2} \raa 1,
\]
which for $W \subset Aut(L, \tilde{F_2}, F_0) \subset Gal(L/\Qbb_p)$
induces the long exact sequence
\[
    \xymatrix{
    \ldots \ar[r]
    & H^1(W, L^\x) \ar[r]
    & H^1(W, L^\x/\tilde{F_2}) \ar[d] 
    & Br(L/L^W) \ar@{=}[d] \\
    & 0 \ar@{=}[u]
    & H^2(W, \tilde{F_2}) \ar[r]^{i_W^\ast}
    & H^2(W, L^\x) \ar[r] 
    & \ldots,
    }
\]
where the left hand term is trivial by Hilbert's theorem 90 (see for
example \cite{neukirch} chapter IV theorem 3.5). Then theorem \ref{190}
and \ref{193} become explicit if we can determine the homomorphism
$i_W^\ast$. We will use the following fact extensively.

\bigskip

\begin{proposition} \label{214}
If $i_W^\ast$ is an epimorphism, then $i_{W'}^\ast$ is an epimorphism
for every subgroup $W'$ of $W$.
\end{proposition}

\begin{proof}
The bifunctoriality of the cohomology induces a commutative square
\[
    \xymatrix{
        H^2(W, \tilde{F_2}) \ar[r]^{i_W^\ast} \ar[d]
        & H^2(W, L^\x) \ar[d] \\
        H^2(W', \tilde{F_2}) \ar[r]^{i_{W'}^\ast}
        & H^2(W', L^\x).
    }
\]
By corollary \ref{255}, the right hand map of this square is surjective.
Hence if $i_W^\ast$ is surjective, the bottom horizontal homomorphism is
surjective as well.
\end{proof}

The cases $p>2$ and $p=2$ are treated separately.

\section{Extensions of maximal abelian finite subgroups 
of $\Sbb_n$ for $p>2$} 

In this section, we assume $p$ to be odd, $F_0$ to be maximal abelian,
and $\tilde{F_1}$ to be maximal as a subgroup of $\Qbb_p(F_0)^\x$ having
$\tilde{F_0}$ as a subgroup of finite index; in other words
\[
    F_0 \iso C_{p^\alpha} \x C_{p^{n_\alpha}-1}
    \qq \text{with} \q 
    0 \leq \alpha \leq k, \q
    n_\alpha = \frac{n}{\phi(p^\alpha)}, 
\]
and
\[
    \tilde{F_1} =
    \begin{cases}
        \tilde{F_0} = F_0 \x \lan pu \ran & \text{if } \alpha = 0,\\
        F_0 \x \lan x_1 \ran & \text{if } \alpha \geq 1,
    \end{cases}
\]
where in the last case $x_1 \in \Qbb_p(\zeta_p) \subset \Qbb_p(F_0)$
satisfies
\[
    v(x_1) = \frac{1}{p-1}
    \qq \text{and} \qq
    x_1^{p-1} \in \tilde{F_0} = F_0 \x \lan pu \ran.
\]
In fact we may assume $x_1$ to satisfy $x_1^{p-1} = pu\delta$ for
$\delta \in \mu_{p-1}(\Qbb_p(\zeta_p))$ as given in corollary \ref{095}
and remark \ref{296}. By definition $\Qbb_p(F_0) = \Qbb_p(\tilde{F_1})$,
and because the latter is a maximal subfield of $\Dbb_n$ we have
$\tilde{F_1} = \tilde{F_2}$.  We let 
\[
    G 
    := Gal(\Qbb_p(F_0)/\Qbb_p) 
    \iso
    \begin{cases}
        C_n 
        & \text{if } \alpha = 0,\\
        C_{p-1} \x C_{p^{\alpha-1}} \x C_{n_\alpha}
        & \text{if } \alpha \geq 1,
    \end{cases}
\]
as given by proposition \ref{342}. From our choice of $x_1$, we know
that $\tilde{F_1}$ is stable under the action of a subgroup $W \subset
G$; this is because if $\sigma \in W$, then $\frac{\sigma(x_1)}{x_1}$ is
a $(p\!-\!1)$-th root of unity in $\Qbb_p^\x$, and hence $\sigma(x_1) \in
x_1 \lan \zeta_{p-1} \ran \subset \tilde{F_1}$ for $\zeta_{p-1} \in
\Qbb_p^\x$. The goal of the section is to determine necessary and
sufficient conditions on $n$, $p$, $u$ and $\alpha$ for the homomorphism
\[
    i_G^\ast: H^2(G, \tilde{F_1}) \raa H^2(G, \Qbb_p(F_0)^\x)
\]
to be surjective, and whenever this happens, we want to determine its
kernel. This is done via the analysis of
\[
    i_W^\ast: H^2(W, \tilde{F_1}) \raa H^2(W, \Qbb_p(F_0)^\x)
\]
for suitable subgroups $W \subset G$.

\subsection*{The case $\alpha = 0$}

The situation is much simpler when the $p$-Sylow subgroup of $F_0$ is
trivial.

\begin{lemma} \label{215}
If $\alpha = 0$ and $W = C_n$, then
\begin{align*}
    H^\ast(W, \tilde{F_1}) &\iso
    \begin{cases}
        \lan pu \ran \x C_{p-1} 
        & \text{if } \ast = 0,\\
        0 
        & \text{if } 0 < \ast \text{ is odd},\\
        \lan pu \ran / \lan (pu)^n \ran 
        & \text{if } 0 < \ast \text{ is even};
    \end{cases} \\[2ex]
    H^\ast(W, \Qbb_p(F_0)^\x) &\iso
    \begin{cases}
        \Qbb_p^\x 
        & \text{if } \ast = 0,\\
        0 
        & \text{if } 0 < \ast \text{ is odd},\\
        \lan p \ran / \lan p^n \ran 
        & \text{if } 0 < \ast \text{ is even}.
    \end{cases}
\end{align*}
\end{lemma}

\begin{proof}
The action of $C_n = W$ on $\tilde{F_1} \iso \lan pu \ran \x C_{p^n-1}$
is trivial on $\lan pu \ran$ and acts on $C_{p^n-1}$ by $\zeta \mto
\zeta^p$.

For $t$ a generator of $C_n$, written additively, and $N =
\sum_{i=0}^{n-1} t^i$, $H^\ast(C_n, \tilde{F_1})$ is the cohomology of
the complex
\[
    \xymatrix{
        \tilde{F_1} \ar[r]^{1-t}
        & \tilde{F_1} \ar[r]^{N}
        & \tilde{F_1} \ar[r]^{1-t}
        & \ldots\ .
    }
\]
Using additive notation for $\tilde{F_1} \iso \Zbb \x \Zbb/p^n-1$, we
obtain
\begin{align*}
    (1-t)(1, 0) &= (0, 0), &
    (1-t)(0, 1) &= (0, 1-p),\\
    N(1, 0) &= (n, 0), &
    N(0, 1) &= (0, \frac{p^n-1}{p-1}),
\end{align*}
and the desired result for $H^\ast(W, \tilde{F_1})$ follows.

Now let $L = \Qbb_p(F_0) = \Qbb_p(\tilde{F_1})$ and $K = L^W$. Then
\[
    H^0(W, \Qbb_p(F_0)^\x) 
    = K^\x 
    = \Qbb_p^\x
\]
and $H^1(W, \Qbb_p(\tilde{F_1})) = 0$ by Hilbert's theorem 90.
Furthermore, since $L/K$ is unramified, we know from proposition
\ref{256} that the valuation map induces an isomorphism
\[
    H^2(W, L^\x) \iso H^2(W, \frac{1}{e(L)}\Zbb)
    \iso \Zbb/|W|\Zbb.
\]
Here $e(L) = 1$, and as $v(p) = 1$, the element $p$ represents a
generator of the cyclic group $H^2(W, L^\x)$. The result then follows
from the periodicity of the cohomology.
\end{proof}

\begin{corollary} \label{216}
If $\alpha = 0$ and $W \subset C_n$, then $i_W^\ast$ is an
isomorphism.
\end{corollary}

\begin{proof}
Let $L = \Qbb_p(F_0)$ and $K = L^{C_n} = \Qbb_p$. Since $L/K$ is
unramified, $\Ocal_K^\x$ is in the image of the norm by proposition
\ref{256} and $u \in N_{L/K}(L^\x)$. Hence $i_{C_n}^\ast(pu) = p$ and
$i_{C_n}^\ast$ is an isomorphism. For any subgroup $W \subset C_n$, it
follows from proposition \ref{214} that $i_W^\ast$ is an epimorphism,
and hence from lemma \ref{215} that it is an isomorphism.
\end{proof}

\begin{example} \label{251}
Let $\alpha = 0$ and $F_0 \iso C_{p^n-1}$ generated by a primitive
$(p^n\!-\!1)$-th root of unity $\omega$. Since $\Qbb_p(F_0)/\Qbb_p$ is a
maximal unramified commutative extension in $\Dbb_n$, we have
$\tilde{F_0} = \tilde{F_1} = \tilde{F_2}$.  Furthermore, as noted in
remark \ref{334}, there is an element $\xi_u$ in $\Dbb_n^\x$ that
generates the Frobenius $\sigma$ of $\Qbb_p(\omega)$ in such a way that
\[
    \Dbb_n
    \iso \Qbb_p(\omega)\lan \xi_u \ran 
    / (\xi_u^n = pu,\ \xi_u x = x^\sigma \xi_u)
    \qq \text{and} \qq
    \omega^\sigma = \omega^p.
\]
Hence for any $u \in \Zbb_p^\x$, $\tilde{F_3} \iso F_0 \rtimes \lan
\xi_u \ran$. In $\Gbb_n(u)$, we therefore have an extension
\[
    1 \raa F_0 \raa F_3 \raa C_n \raa 1
\]
with $C_n \iso Gal(\Qbb_p(F_0)/\Qbb_p)$ acting faithfully on the kernel
and
\[
    F_3
    = \lan \omega, \bar{\xi}_u \ 
      |\ \omega^{p^n-1} = \bar{\xi}_u^n = 1,\ 
      \bar{\xi}_u\omega\bar{\xi}_u^{-1}=\omega^p \ran
    \iso C_{p^n-1} \rtimes C_n
\]
for $\bar{\xi}_u$ the class of $\xi_u$ in $\Gbb_n(u)$.
\end{example}

\subsection*{The case $\alpha \geq 1$}

For the rest of the section we let $\alpha \geq 1$. The Galois group $G$
of $\Qbb_p(F_0)/\Qbb_p$ decomposes canonically as
\[
    G = G_p \x G_{p'},
\]
where $G_p = C_{p^{\alpha-1}} \x C_{\frac{n_\alpha}{m}}$ is the
$p$-torsion subgroup in the abelian group $G$ and $G_{p'} = C_{p-1} \x
C_m$ is the subgroup of elements of torsion prime to $p$. If $W$ is any
subgroup of $G$, then $W$ decomposes canonically in the same way as
\[
    W = W_p \x W_{p'}
    \qq \text{with} \q
    W_p \subset G_p
    \q \text{and} \q
    W_{p'} \subset G_{p'}.
\]
For any such $W \subset G$ we define the groups
\[
    W_0 := W_{p'}, \qq
    W_1 := W_0 \x (W_p \cap Aut(C_{p^\alpha}))
    \qq \text{and} \qq
    W_2 := W.
\]

\begin{proposition} \label{243}
If $\alpha \geq 1$, then $i_{W_0}^\ast$ is always an isomorphism.
\end{proposition}

\begin{proof}
The inclusion $\tilde{F_1} \subset \Qbb_p(F_0)^\x$ induces a short exact
sequence
\[
    1 \raa \tilde{F_1} \raa \Qbb_p(F_0)^\x 
    \raa \Qbb_p(F_0)^\x/\tilde{F_1} \raa 1,
\]
which induces a long exact sequence
\[
    \xymatrix{
        \ar[r]
        & H^1(W_0, \Qbb_p(F_0)^\x/\tilde{F_1}) \ar[r]
        & H^2(W_0, \tilde{F_1}) \ar[d]^{i_{W_0}^\ast} \\
        && H^2(W_0, \Qbb_p(F_0)^\x) \ar[r]
        & H^2(W_0, \Qbb_p(F_0)^\x/\tilde{F_1}) \ar[r]
        & 
    }
\]
The group $\Qbb_p(F_0)^\x/\tilde{F_1}$ fits into an exact sequence
\[
    1 \raa \Zbb_p(F_0)^\x/F_0
    \raa \Qbb_p(F_0)^\x/\tilde{F_1}
    \raa \Zbb/\Zbb\lan x_1 \ran
    \raa 1,
\]
induced by the exact sequences
\begin{align*}
    1 \raa F_0 \raa &\tilde{F_1} 
    \overset{v}{\raa} \Zbb\lan x_1 \ran
    \raa 1,\\[1ex]
    1 \raa \Zbb_p(F_0)^\x
    \raa &\Qbb_p(F_0)^\x
    \overset{v}{\raa} \Zbb
    \raa 1.
\end{align*}
Note that the group $\Zbb_p(F_0)^\x/F_0$ is free over $\Zbb_p$, while as
$\tilde{F_1}$ is maximal the quotient $\Zbb/\Zbb\lan x_1 \ran =
\Zbb/nv(x_1)\Zbb$ is a $p$-torsion group. Since $|W_0|$ is prime to $p$
we get
\[
    H^\ast(W_0, \Zbb_p(F_0)^\x/F_0)
    = H^\ast(W_0, \Zbb/\Zbb\lan x_1 \ran)
    = 0
    \qq \text{for } \ast > 0.
\]
Thus
\[
    H^\ast(W_0, \Qbb_p(F_0)^\x/\tilde{F_1}) 
    = 0
    \qq \text{for } \ast > 0,
\]
and the result follows.
\end{proof}

\begin{lemma} \label{221}
If $\alpha \geq 1$ and $W = C_{p-1} \subset Aut(C_{p^\alpha})$, then
\begin{align*}
    H^\ast(W, \tilde{F_1}) &\iso
    \begin{cases}
        \lan pu \ran \x C_{p^{n_\alpha}-1} 
        & \text{if } \ast = 0,\\
        0 
        & \text{if } 0 < \ast \text{ is odd},\\
        C_{p^{n_\alpha}-1} \ox C_{p-1} 
        \iso C_{p-1}
        & \text{if } 0 < \ast \text{ is even};
    \end{cases} \\[2ex]
    H^\ast(W, \Qbb_p(F_0)^\x) &\iso
    \begin{cases}
        (\Qbb_p(F_0)^{C_{p-1}})^\x 
        & \text{if } \ast = 0,\\
        0 
        & \text{if } 0 < \ast \text{ is odd},\\
        (\Qbb_p(F_0)^{C_{p-1}})^\x 
        / N_W(\Qbb_p(F_0)^\x)
        \iso C_{p-1}
        & \text{if } 0 < \ast \text{ is even}.
    \end{cases}
\end{align*}
\end{lemma}

\begin{proof}
Consider the short exact sequence
\[
    1 \raa F_0 \raa \tilde{F_1} \raa \Zbb\lan x_1 \ran \raa 1;
\]
it induces a long exact sequence in cohomology
\[
    H^0(W, F_0) \raa 
    H^0(W, \tilde{F_1}) \raa 
    H^0(W, \Zbb\lan x_1 \ran) \raa 
    H^1(W, F_0) \raa 
    \ldots,
\]
where the action of $W$ is trivial on $\Zbb\lan x_1 \ran$, while
faithful on the first factor of $F_0 \iso C_{p^\alpha} \x
C_{p^{n_\alpha}-1}$ and trivial on the second factor. Note that the
first factor of $F_0$ splits off and has trivial cohomology. Hence for
$t$ a generator of $W$, written additively, and $N = \sum_{i=0}^{p-2}
t^i$, the cohomology $H^\ast(W, F_0)$ can be calculated from the
additive complex
\[
    \xymatrix{
        \Zbb/(p^{n_\alpha}-1) \ar[r]^{1-t}
        & \Zbb/(p^{n_\alpha}-1) \ar[r]^{N}
        & \Zbb/(p^{n_\alpha}-1) \ar[r]^{\qq 1-t}
        & \ldots
    }
\]
with
\[
    (1-t)(1) = 0
    \qq \text{and} \qq
    N(1) = p-1;
\]
while $H^\ast(W, \Zbb\lan x_1 \ran)$ can be calculated from the additive
complex
\[
    \xymatrix{
        \Zbb \ar[r]^{1-t}
        & \Zbb \ar[r]^{N}
        & \Zbb \ar[r]^{1-t}
        & \ldots
    }
\]
with
\[
    (1-t)(1) = 0
    \qq \text{and} \qq
    N(1) = p-1.
\]
Consequently
\begin{align*}
    H^\ast(W, F_0) &\iso
    \begin{cases}
        C_{p^{n_\alpha}-1} 
        & \text{if } \ast = 0,\\
        C_{p^{n_\alpha}-1} \ast C_{p-1} \iso C_{p-1} 
        & \text{if } 0 < \ast \text{ is odd},\\
        C_{p^{n_\alpha}-1} \ox C_{p-1} \iso C_{p-1}
        & \text{if } 0 < \ast \text{ is even};
    \end{cases} \\[2ex]
    H^\ast(W, \Zbb\lan x_1 \ran) &\iso
    \begin{cases}
        \Zbb\lan x_1 \ran 
        & \text{if } \ast = 0,\\
        0 
        & \text{if } 0 < \ast \text{ is odd},\\
        C_{p-1}
        & \text{if } 0 < \ast \text{ is even},
    \end{cases}
\end{align*}
where $C_{p^{n_\alpha}-1} \ast C_{p-1}$ denotes the kernel of the
$(p\!-\!1)$-th power map on $C_{p^{n_\alpha}-1}$.

Clearly, $\lan pu \ran \x C_{p^{n_\alpha}-1} \subset \tilde{F_1}^{W}$.
Since $x_1^{p-1} = pu\delta$ with $\delta$ a $(p\!-\!1)$-th root of unity
in $\Qbb_p^\x$, we know that $\Qbb_p(x_1)^W = \Qbb_p$. Consider an
element $z = x_1^l y_1 y_2$ in $\tilde{F_1}^W$ with $y_1 \in \lan
\zeta_{p^\alpha} \ran$ and $y_2 \in \lan \zeta_{p^{n_\alpha}-1} \ran$.
Then for $\sigma \in W$, $y_2$ is invariant under $\sigma$, and we have
\[
    \left( \frac{\sigma(x_1)}{x_1} \right)^l
    = \frac{y_1}{\sigma(y_1)}.
\]
Since the order of $\frac{y_1}{\sigma(y_1)}$ divides a power of $p$ and
the order of $\frac{\sigma(x_1)}{x_1}$ divides $p-1$, we must have
$\frac{y_1}{\sigma(y_1)} = 1$, and hence $y_1 = 1$. Therefore $x_1^l$
is invariant under $\sigma$, and as $\Qbb_p(x_1)^W = \Qbb_p(\zeta_p)^W =
\Qbb_p$, we know that $l \equiv 0 \mod p-1$.  It follows that the
valuation of $\tilde{F_1}^W$ is integral, and therefore
\[
    H^0(W, \tilde{F_1}) 
    = \tilde{F_1}^W 
    = \lan pu \ran \x C_{p^{n_\alpha}-1}.
\]
Since the image of $H^0(W, \tilde{F_1})$ in $H^0(W, \Zbb\lan x_1 \ran)
\iso \Zbb\lan x_1 \ran$ is $\Zbb\lan pu \ran$, the group $H^0(W,
\Zbb\lan x_1 \ran)$ surjects onto $H^1(W, F_0) \iso C_{p-1}$, and
therefore $H^1(W, \tilde{F_1}) = 0$. By the periodicity of the
cohomology, the map
\[
    H^2(W, \Zbb\lan x_1 \ran) \raa H^3(W, F_0)
\]
is an isomorphism, and as $H^1(W, \Zbb\lan x_1 \ran) =0$ we
obtain
\[
    H^2(W, \tilde{F_1})
    \iso H^2(W, F_0) 
    \iso C_{p^{n_\alpha}-1} \ox C_{p-1}.
\]

Finally, the triviality of $H^1(W, \Qbb_p(F_0)^\x)$ is a direct
consequence of Hilbert's theorem $90$, while the remaining cases for
$H^\ast(W, \Qbb_p(F_0)^\x)$ follow from the characterisation of the
Brauer group in terms of the invariants and the norm relative to the
Galois group $W$ as given by theorem \ref{068} and corollary \ref{072}.
\end{proof}

\begin{lemma} \label{244}
If $\alpha = 1$, $C_{n_1} = Gal(\Qbb_p(C_{p^{n_1}-1})/\Qbb_p)$, and
$C_{p-1} = Aut(C_p)$, then
\begin{align*}
    H^\ast(C_{n_1}, \tilde{F_1}^{C_{p-1}}) &\iso
    \begin{cases}
        \lan pu \ran \x C_{p-1} 
        & \text{if } \ast = 0,\\
        0 
        & \text{if } 0 < \ast \text{ is odd},\\
        \lan pu \ran / \lan (pu)^{n_1} \ran 
        & \text{if } 0 < \ast \text{ is even};
    \end{cases} \\[2ex]
    H^\ast(C_{n_1}, (\Qbb_p(F_0)^{C_{p-1}})^\x) &\iso
    \begin{cases}
        \Qbb_p^\x 
        & \text{if } \ast = 0,\\
        0 
        & \text{if } 0 < \ast \text{ is odd},\\
        \lan p \ran / \lan p^{n_1} \ran 
        & \text{if } 0 < \ast \text{ is even}.
    \end{cases}
\end{align*}
\end{lemma}

\begin{proof}
The action of $C_{n_1}$ on $\tilde{F_1}^{C_{p-1}} \iso \lan pu \ran \x
C_{p^{n_1}-1}$ is trivial on the first factor and acts on
$C_{p^{n_1}-1}$ by $\zeta \mto \zeta^p$.

Let $t$ be generator of $C_{n_1}$, written additively, and $N =
\sum_{i=0}^{n_1-1} t^i$. Using additive notation for
$\tilde{F_0}^{C_{p-1}} \iso \Zbb \x \Zbb/p^{n_1}-1$, we obtain
\begin{align*}
    (1-t)(1, 0) &= (0, 0)
    & (1-t)(0, 1) &= (0, 1-p),\\
    N(1, 0) &= (n_1, 0),
    & N(0, 1) &= (0, \frac{p^{n_1}-1}{p-1}),
\end{align*}
and the desired result for $H^\ast(C_{n_1}, \tilde{F_1}^{C_{p-1}})$
follows.

Now for $L = \Qbb_p(\tilde{F_0})^{C_{p-1}}$ and $K = L^{C_{n_1}}$, we
have
\[
    H^0(C_{n_1}, L^\x) 
    = K^\x 
    = \Qbb_p^\x
\]
and $H^1(C_{n_1}, L^\x) = 0$ by Hilbert's theorem 90. Furthermore as
$L/K$ is unramified, we know from proposition \ref{256} that the
valuation map induces an isomorphism
\[
    H^2(C_{n_1}, L^\x) \iso H^2(C_{n_1}, \frac{1}{e(L)}\Zbb)
    \iso \Zbb/n_1\Zbb.
\]
Here $e(L) = 1$, and as $v(p) = 1$, the element $p$ represents a
generator of the cyclic group $H^2(C_{n_1}, L^\x)$. The result then
follows from the periodicity of the cohomology.
\end{proof}

\begin{corollary} \label{245}
If $\alpha = 1$, then $H^\ast(C_{n_1}, \tilde{F_1}^{C_{p-1}}) \ra
H^\ast(C_{n_1}, (\Qbb_p(F_0)^{C_{p-1}})^\x)$ is an isomorphism for $0 <
\ast$ even.
\end{corollary}

\begin{proof}
Let $L = \Qbb_p(\tilde{F_0})^{C_{p-1}}$ and $K = L^{C_{n_1}} = \Qbb_p$.
Because the extension $L/K$ is unramified, proposition \ref{256} implies
that the group of units of the ring of integers $\Ocal_K$ of $K$ is
contained in the norm $N_{L/K}(L^\x)$. As $p$ is a uniformizing element
of $\Qbb_p^\x = K^\x$, it is a generator of the cyclic group
$K^\x/N_{L/K}(L^\x) \iso H^2(C_{n_1}, L)$, and the result follows.
\end{proof}

\begin{lemma} \label{246}
If $\alpha \geq 2$, $W_0 = G_{p'}$ and $C_{p^{\alpha-1}} \subset
Aut(C_{p^\alpha})$, then
\begin{align*}
    H^\ast(C_{p^{\alpha-1}}, \tilde{F_1}^{W_0}) &\iso
    \begin{cases}
        \lan pu \ran \x C_{p^{\frac{n_\alpha}{m}}-1}
        & \text{if } \ast = 0,\\
        0 
        & \text{if } 0 < \ast \text{ is odd},\\
        \lan pu \ran / \lan (pu)^{p^{\alpha-1}} \ran 
        \iso C_{p^{\alpha-1}}
        & \text{if } 0 < \ast \text{ is even};
    \end{cases} \\[2ex]
    H^\ast(C_{p^{\alpha-1}}, (\Qbb_p(F_0)^{W_0})^\x) &\iso
    \begin{cases}
        (\Qbb_p(F_0)^{ W_0 \x C_{p^{\alpha-1}} })^\x 
        & \text{if } \ast = 0,\\
        0 
        & \text{if } 0 < \ast \text{ is odd},\\
        C_{p^{\alpha-1}}
        & \text{if } 0 < \ast \text{ is even}.
    \end{cases}
\end{align*}
\end{lemma}

\begin{proof}
The action of $C_{p-1} \subset W_0 \cap Aut(C_{p^\alpha})$ on
$\tilde{F_1} = \lan x_1 \ran \x C_{p^{n_\alpha}-1} \x C_{p^\alpha}$
being faithful on the first and last factors, we have $\tilde{F_1}^{W_0}
\iso \lan pu \ran \x C_{p^{\frac{n_\alpha}{m}}-1}$, and consequently the
action of $C_{p^{\alpha-1}}$ on $\tilde{F_1}^{W_0}$ is trivial.

Let $t$ be generator of $C_{p^{\alpha-1}}$, written additively, and $N =
\sum_{i=0}^{p^{\alpha-1}-1} t^i$. Using additive notation for
$\tilde{F_1}^{W_0} \iso \Zbb \x \Zbb/(p^{\frac{n_\alpha}{m}}\!-\!1)$ we
obtain
\begin{align*}
    (1-t)(1, 0) &= (0, 0), & 
    (1-t)(0, 1) &= (0, 0),\\[1ex]
    N(1, 0) &= (p^{\alpha-1}, 0), &
    N(0, 1) &= (0, p^{\alpha-1}),
\end{align*}
and the desired result for $H^\ast(C_{p^{\alpha-1}},
\tilde{F_1}^{W_0})$ follows.

The case of $H^\ast(C_{p^{\alpha-1}}, (\Qbb_p(F_0)^{W_0})^\x)$ follows
from Hilbert's theorem 90 when $\ast$ is odd, while the case where $0 <
\ast$ is even is given by the isomorphism
\[
    (\Qbb_p(F_0)^{ W_0 \x C_{p^{\alpha-1}} })^\x 
        / N_{ C_{p^{\alpha-1}} }((\Qbb_p(F_0)^{W_0})^\x)
        \iso C_{p^{\alpha-1}}.
\]
\end{proof}

In view of theorem \ref{250}, we are only interested in the case where
$W_1$ is maximal, that is $W_1 = G_{p'} \x (G_p \cap
Aut(C_{p^\alpha}))$. 

\begin{corollary} \label{247}
If $\alpha \geq 2$, $W_0 = G_{p'}$ and $|W_1/W_0| = p^{\alpha-1}$, then
\[
    H^\ast(W_1/W_0, \tilde{F_1}^{W_0}) \raa 
    H^\ast(W_1/W_0, (\Qbb_p(F_0)^{W_0})^\x)
    \q \text{for } 0 < \ast \text{ even}
\]
is surjective if and only if it is an isomorphism, and this is true if
and only if
\[
    u \not\in
    \mu(\Zbb_p^\x) 
    \x \{ x \in \Zbb_p^\x\ |\ x \equiv 1 \mod (p^2) \}
    \qq \text{and} \qq
    \alpha = k.
\]
\end{corollary}

\begin{proof}
The first assertion is an obvious consequence of lemma \ref{246}. Let
\[
    M := \Qbb_p(F_0), \qq
    L := M^{W_0}
    \qq \text{and} \qq
    K := L^{ C_{p^{\alpha-1}} } = M^{W_1}.
\]
Since $L/K$ is totally ramified, we know from proposition \ref{256} that
\[
    H^2(C_{p^{\alpha-1}}, L^\x)
    \iso H^2(C_{p^{\alpha-1}}, \Ocal_L^\x),
\]
and as $N_{G/W_1} \circ N_{ C_{p^{\alpha-1}} } (\Ocal_L^\x) =
N_{G/W_0}(\Ocal_L^\x)$, we may consider the homomorphism
\[
    \tau: H^2(C_{p^{\alpha-1}}, L^\x) \raa 
    \Zbb_p^\x/ N_{G/W_0}(\Ocal_L^\x)
\]
given as the composite
\[
    H^2(C_{p^{\alpha-1}}, L^\x)
    \iso H^2(C_{p^{\alpha-1}}, \Ocal_L^\x)
    \iso (\Ocal_L^\x)^{ C_{p^{\alpha-1}} }
      / N_{ C_{p^{\alpha-1}} }(\Ocal_L^\x)
    \overset{N_{G/W_1}}{\raa} \Zbb_p^\x/ N_{G/W_0}(\Ocal_L^\x).
\]

We claim that $\tau$ is an isomorphism. Because $Gal(L/\Qbb_p)$
preserves $\Ocal_L^\x$, and hence $l^\x$ for $l$ the residue field of
$L$, we have an epimorphism $Gal(L/\Qbb_p) \ra Gal(l/\Fbb_p)$ whose
kernel will be denoted $A$. Since $K$ is the maximal unramified
subextension of $L/\Qbb_p$, we may consider the short exact sequences
\[
    \xymatrix{
        1 \ar[r]
        & \Zbb_p^\x/N_{G/W_0}(\Ocal_L^\x) \ar[r] \ar[d]
        & \Qbb_p^\x/N_{G/W_0}(L^\x) 
        \ar[r]^{v} \ar[d]_{\iso}^{(\_, L/\Qbb_p)}
        & \Zbb/v(N_{G/W_0}(L^\x)) 
        \ar[r] \ar[d]_{\iso}^{\sigma^{(\_)}}
        & 1\\
        1 \ar[r]
        & A \ar[r] \ar[d]_{\iso}
        & Gal(L/\Qbb_p) \ar@{=}[d] \ar[r]^{red}
        & Gal(l/\Fbb_p) \ar[d]_{\iso} \ar[r]
        & 1\\
        1 \ar[r]
        & Gal(L/K) \ar[r]
        & Gal(L/\Qbb_p) \ar[r]^{pr}
        & Gal(K/\Qbb_p) \ar[r]
        & 1,
    }
\]
where the bottom two squares commute, the middle vertical isomorphism is
the norm residue symbol of $L/\Qbb_p$ as defined in \cite{serre4}
section 2.2, the top left hand vertical map is its restriction, and
where the top right hand vertical isomorphism is given by the power map
of the Frobenius automorphism $\sigma \in Gal(l/\Fbb_p)$. By local class
field theory (see for example \cite{koch} chapter 2 \S 1.3) we have
\[
    pr(x, L/\Qbb_p) = (x, K/\Qbb_p)
    \qq \text{for all } x \in \Qbb_p^\x/N_{G/W_0}(L^\x). 
\]
On the other hand \cite{serre4} proposition 2 shows that
\[
    (x, K/\Qbb_p) = \sigma^{v(x)}
    \qq \text{for all } x \in \Qbb_p^\x/N_{G/W_0}(L^\x). 
\]
It follows that the top right hand square in the above diagram, and
hence the diagram itself, is commutative. From the five lemma, the top
left hand vertical map of the diagram is an isomorphism, and as
$Gal(L/K) \iso C_{p^{\alpha-1}}$, we get
\[
    N_{G/W_0}(\Ocal_L^\x)
    = \mu(\Zbb_p^\x) \x U_\alpha(\Zbb_p^\x)
\]
as a subgroup of index $p^{\alpha-1}$ in $\Zbb_p^\x$. By corollary
\ref{072}, we know that $H^2(C_{p^{\alpha-1}}, L^\x)$ has order
$p^{\alpha-1}$. The norm $N_{G/W_1}: \Ocal_K^\x \ra \Zbb_p^\x$ being
surjective by proposition \ref{257}, it follows that $\tau$ is an
isomorphism.

As a consequence of this latter result, our map of interest
\[
    i^\ast:
    H^2(C_{p^{\alpha-1}}, \tilde{F_1}^{W_0}) 
    \iso \lan pu \ran / \lan (pu)^{p^{\alpha-1}} \ran
    \raa H^2(C_{p^{\alpha-1}}, L^\x)
\]
is surjective (and hence an isomorphism) if and only if $\tau i^\ast$ is
surjective, that is, if and only if $\tau(i^\ast(pu))$ is a generator of
the cyclic group $\Zbb_p^\x/N_{G/W_0}(\Ocal_L^\x) \iso
U_1(\Zbb_p^\x)/U_\alpha(\Zbb_p^\x)$.  In fact $i^\ast(pu)=i^\ast(u)$.
Indeed, as seen in example \ref{258}, there is an element $\tilde{x} \in
\Qbb_p(\zeta_{p^\alpha})$ such that
\[
    p = N_{\Qbb_p(\zeta_{p^\alpha})/\Qbb_p}(\tilde{x})
    = N_{ C_{p^{\alpha-1}} }(x)
\]
for $x = N_{C_{p-1}}(\tilde{x})$ in $\Qbb_p(\zeta_{p^\alpha})^{C_{p-1}}
\subset L$; by remark \ref{271} we then have
\[
    p \in N_{C_{p-1}}(L^\x)
    \qq \text{and} \qq
    i^\ast(pu)=i^\ast(u).
\]
Furthermore as $\mu(\Zbb_p^\x) \subset N_{G/W_0}(\Ocal_L^\x)$, we have
$\tau(u) = \tau(u_1)$ for $u_1$ the component of $u \in \Zbb_p^\x$ in
$U_1(\Zbb_p^\x)$ via the isomorphism $\Zbb_p^\x \iso \mu(\Zbb_p) \x
U_1(\Zbb_p^\x)$ of proposition \ref{341}. Letting $z \in \Zbb_p$ be such
that $u_1 = 1+zp$, we finally obtain that
\begin{align*}
    \tau(pu) = N_{G/W_1}(u_1)
    = u_1^{|G/W_1|}
    \equiv 1+z|G/W_1|p \q \mod p^2
\end{align*}
is a generator if and only if
\[
    u \not\in
    \mu(\Zbb_p^\x) 
    \x \{ x \in \Zbb_p^\x\ |\ x \equiv 1 \mod (p^2) \}
    \qq \text{and} \qq
    \l| G/W_1 \r| \not\equiv 0 \mod p,
\]
the latter condition being equivalent to $\alpha = k$ (or $n_\alpha =
m$).
\end{proof}

\begin{lemma} \label{248}
If $\alpha \geq 2$, $W_0 = G_{p'}$, $|W/W_1| \neq 1$ and $L =
\Qbb_p(F_0)^{W_1}$ with $v(L) = \frac{1}{e(L)}\Zbb$, then
\begin{align*}
    H^\ast(W/W_1, \tilde{F_0}^{W_1}) &\iso
    \begin{cases}
        \lan pu \ran \x C_{p-1} 
        & \text{if } \ast = 0,\\
        0 
        & \text{if } 0 < \ast \text{ is odd},\\
        \lan pu \ran / \lan (pu)^{|W/W_1|} \ran 
        & \text{if } 0 < \ast \text{ is even};
    \end{cases} \\[2ex]
    H^\ast(W/W_1, L^\x) &\iso
    \begin{cases}
        (L^{W/W_1})^\x 
        & \text{if } \ast = 0,\\
        0 
        & \text{if } 0 < \ast \text{ is odd},\\
        \lan \pi \ran / \lan \pi^{|W/W_1|} \ran 
        & \text{if } 0 < \ast \text{ is even},
    \end{cases}
\end{align*}
for $\pi$ a unifomizing element of $L^{W/W_1}$.
\end{lemma}

\begin{proof}
Since $W_0 = G_{p'} \subset W_1$, none of the elements of $C_{p^\alpha}$
are left invariant by $W_1$, and we have $\tilde{F_0}^{W_1} \iso \lan pu
\ran \x C_{p^{\frac{n_\alpha}{m}}-1}$. The action of $W/W_1$ on this
group is trivial on the first factor and acts on
$C_{p^{\frac{n_\alpha}{m}}-1}$ by $\zeta \mto \zeta^p$.

Let $t$ be generator of $W/W_1$, written additively, and $N =
\sum_{i=0}^{|W/W_1|-1} t^i$. Using additive notation for
$\tilde{F_0}^{W_1} \iso \Zbb \x \Zbb/(p^{\frac{n_\alpha}{m}}-1)$, we get
\begin{align*}
    (1-t)(1, 0) &= (0, 0),
    & (1-t)(0, 1) &= (0, 1-p),\\
    N(1, 0) &= (|W/W_1|, 0),
    & N(0, 1) &= (0, \frac{p^{\frac{n_\alpha}{m}}-1}{p-1}),
\end{align*}
and the desired result for $H^\ast(W/W_1, \tilde{F_0}^{W_1})$ follows.

Now let $K = L^{W/W_1}$. Then
\[
    H^0(W/W_1, L^\x) = K^\x
    \qq \text{and} \qq
    H^1(W/W_1, L^\x) = 0   
\]
by Hilbert's theorem 90. Finally, as $L/K$ is unramified, proposition
\ref{256} yields
\[
    H^2(W/W_1, L) 
    \iso H^2(G, \frac{1}{e(L)}\Zbb)
    \iso \frac{1}{e(L)}\Zbb / |W/W_1| \cdot \frac{1}{e(L)}\Zbb
    \iso \lan \pi_K \ran / \lan \pi_K^{|W/W_1|} \ran
\]
for $\pi_K$ a uniformizing element of $K$. The result then follows from
periodicity of the cohomology.
\end{proof}

\begin{corollary} \label{249}
If $\alpha \geq 2$, $W_0 = G_{p'}$, $|W/W_1| \neq 1$ and $L =
\Qbb_p(F_0)^{W_1}$ with $v(L) = \frac{1}{e(L)}\Zbb$, then
\[
    H^\ast(W/W_1, \tilde{F_1}^{W_1}) \raa H^\ast(W/W_1, L^\x)
    \q \text{for } 0 < \ast \text{ even}
\]
is surjective if and only if it is an isomorphism, and this is true if
and only if $e(L)$ divides $p-1$.
\end{corollary}

\begin{proof}
The short exact sequence
\[
    1 \raa \tilde{F_0} \raa \tilde{F_1} \raa \Zbb/p-1 \raa 1
\]
induces a long exact sequence
\[
    1 \raa \tilde{F_0}^{W_1} \raa \tilde{F_1}^{W_1} 
    \raa (\Zbb/p-1)^{W_1} \raa H^1(W_1, \tilde{F_0}^{W_1})
    \raa \ldots,
\]
which in turn induces a short exact sequence
\[
    1 \raa \tilde{F_0}^{W_1} \raa \tilde{F_1}^{W_1} 
    \raa I \raa 1
\]
where $|I|$ divides $p-1$. Since $W/W_1$ is a $p$-group, $|W/W_1|$ is
prime to $p-1$ and we have $H^\ast(W/W_1, I) = 0$ for $\ast \geq 1$.
Hence
\[
    H^\ast(W/W_1, \tilde{F_1}^{W_1})
    \iso H^\ast(W/W_1, \tilde{F_0}^{W_1})
    \q \text{for } \ast \geq 2,
\]
and by the periodicity of the cohomology of the finite cyclic group
$W/W_1$ this is also true for $\ast = 1$. For $\ast=2$, we are
interested in the image of this group in $H^2(W/W_1, L^\x)$. Let $K =
L^{W/W_1}$ and $M = \Qbb_p(F_0)$. From lemma \ref{248} we have
\[
    H^2(W/W_1, L^\x) 
    \iso \frac{1}{e(L)}\Zbb / \frac{|W/W_1|}{e(L)}\Zbb,
\]
and we know that $e(L)$ divides $e(M) = (p-1)p^{\alpha-1}$. Because
$L/K$ is unramified, the group $\Ocal_K^\x$ is contained in the norm
$N_{L/K}(L^\x)$ by proposition \ref{256}. The map
\[
    H^2(W/W_1, \tilde{F_1}^{W_1}) \raa H^2(W/W_1, L^\x)
\]
is therefore surjective if and only if $v(pu)=v(p)=1$ is a generator of
$\frac{1}{e(L)}\Zbb / \frac{|W/W_1|}{e(L)}\Zbb$, and this is true if and
only if $p$ does not divide $e(L)$.
\end{proof}

\begin{theorem} \label{250}
Let $p$ be an odd prime, $n = (p-1)p^{k-1}m$ with $m$ prime to $p$, $u
\in \Zbb_p^\x$, $F_0 = C_{p^\alpha} \x C_{p^{n_\alpha}-1}$ be a maximal
abelian finite subgroup in $\Sbb_n$, $G = Gal(\Qbb_p(F_0)/\Qbb_p)$,
$G_{p'}$ be the $p'$-part of $G$, and let $\tilde{F_1} = \lan x_1 \ran
\x F_0 \subset \Qbb_p(F_0)^\x$ be maximal as a subgroup of
$\Qbb_p(F_0)^\x$ having $\tilde{F_0}$ as subgroup of finite index.
\begin{enumerate}
    \item For any $0 \leq \alpha \leq k$, there is an extension of
    $\tilde{F_1}$ by $G_{p'}$; this extension is unique up to
    conjugation.

    \item If $\alpha \leq 1$, there is an extension of $\tilde{F_1}$ by
    $G$; this maximal extension is unique up to conjugation.

    \item If $\alpha \geq 2$, there is an extension of $\tilde{F_1}$ by
    $G$ if and only if    
    \[
        \alpha = k
        \qq \text{and} \qq
        u \not\in \mu(\Zbb_p^\x) 
        \x \{ x \in \Zbb_p^\x\ |\ x \equiv 1 \mod (p^2) \},
    \]
    in which case this maximal extension is unique up to conjugation.
\end{enumerate}
\end{theorem}

\begin{proof}
1) From corollary \ref{216} and proposition \ref{243} we know that the
map  
\[
    i_{G_{p'}}^\ast:
    H^2(G_{p'}, \tilde{F_1}) \raa H^2(G_{p'}, \Qbb_p(F_0)^\x)
\]
is an isomorphism. Existence and uniqueness up to conjugation of an
extension of $\tilde{F_1}$ by $G_{p'}$ then follow from corollary
\ref{192}.

2) The case $\alpha=0$ follows from corollary \ref{216} and corollary
\ref{192}. Now assume that $\alpha = 1$ and that $W = G = C_{p-1} \x
C_{n_1}$. We have a short exact sequence
\[
    1 \raa C_{p-1} \raa W \raa C_{n_1} \raa 1,
\]
which gives rise to the Hochschild-Serre spectral sequences (see
\cite{brown} section VII.6)
\begin{align*}
    E_2^{s, t} 
    \iso H^s(C_{n_1}, H^t(C_{p-1}, \tilde{F_1})) \q
    &\Raa \q H^{s+t}(W, \tilde{F_1}),\\[1ex]
    E_2^{s, t} 
    \iso H^s(C_{n_1}, H^t(C_{p-1}, \Qbb_p(F_0)^\x)) \q
    &\Raa \q H^{s+t}(W, \Qbb_p(F_0)^\x).
\end{align*}
By lemma \ref{221} and proposition \ref{243}, each map $E_2^{s, t} \ra
E_2^{s, t}$ is an isomorphism for $t > 0$. Moreover, by lemma \ref{221}
we have
\[
    H^0(C_{p-1}, \tilde{F_1}) 
    = \tilde{F_1}^{C_{p-1}}
    \iso \lan pu \ran \x C_{p^{n_1}-1}
\]
and
\[
    H^0(C_{p-1}, \Qbb_p(F_0)^\x)
    = (\Qbb_p(F_0)^{C_{p-1}})^\x
    \iso \Qbb_p(C_{p^{n_1}-1})^\x.
\]
Then lemma \ref{244} and corollary \ref{245} imply that the map $E_2^{s,
t} \ra E_2^{s, t}$ is an isomorphism as well for $t=0$ and $s > 0$. It
follows that
\[
    i_W^\ast: H^\ast(W, \tilde{F_1}) \raa H^\ast(W, \Qbb_p(F_0)^\x)
\]
is an isomorphism for $\ast > 0$. Existence and uniqueness up to
conjugation then follows from corollary \ref{192}.

3) Assume that $\alpha \geq 2$ and that $W \subset G$ is such that $W_0
= G_{p'}$ with $|W_1/W_0| \neq 1$ and $|W/W_1| \neq 1$. We have a short
exact sequence
\[
    1 \raa W_0 \raa W_1 \raa W_1/W_0 \raa 1,
\]
which gives rise to the spectral sequences
\begin{align*}
    E_2^{s, t} 
    \iso H^s(W_1/W_0, H^t(W_0, \tilde{F_1})) \q
    &\Raa \q H^{s+t}(W_1, \tilde{F_1}),\\[1ex]
    E_2^{s, t} 
    \iso H^s(W_1/W_0, H^t(W_0, \Qbb_p(F_0)^\x)) \q
    &\Raa \q H^{s+t}(W_1, \Qbb_p(F_0)^\x).
\end{align*}
By proposition \ref{243} each map $E_2^{s, t} \ra E_2^{s, t}$ is an
isomorphism for $t > 0$. Moreover, we know from lemma \ref{246} and
corollary \ref{247} that when $t=0$ and $s>0$, a necessary and sufficient
condition for
\[
    H^s(W_1/W_0, \tilde{F_1}^{W_0}) \raa 
    H^s(W_1/W_0, (\Qbb_p(F_0)^{W_0})^\x)
\]
to be surjective (and hence an isomorphism) is that $u$ is a topological
generator in $\Zbb_p^\x/\mu(\Zbb_p)$ and $\alpha = k$. The map
\[
    H^\ast(W_1, \tilde{F_1}) \raa H^\ast(W_1, \Qbb_p(F_0)^\x),
    \q \text{for } \ast > 0
\]
is thus surjective if and only if it is an isomorphism, and this is true
if and only if 
\[
    u \not\in \mu(\Zbb_p^\x) \x \{ x \in
    \Zbb_p^\x\ |\ x \equiv 1 \mod (p^2) \}
    \qq \text{and} \qq
    \alpha = k.
\]
Now assuming these conditions are satisfied, the short exact sequence
\[
    1 \raa W_1 \raa W \raa W/W_1 \raa 1
\]
induces spectral sequences
\begin{align*}
    E_2^{s, t} 
    \iso H^s(W/W_1, H^t(W_1, \tilde{F_1})) \q
    &\Raa \q H^{s+t}(W, \tilde{F_1}),\\[1ex]
    E_2^{s, t} 
    \iso H^s(W/W_1, H^t(W_1, \Qbb_p(F_0)^\x)) \q
    &\Raa \q H^{s+t}(W, \Qbb_p(F_0)^\x),
\end{align*}
where each map $E_2^{s, t} \ra E_2^{s, t}$ is an isomorphism for $t>0$.
Furthermore, lemma \ref{248} and corollary \ref{249} imply that in case
$t=0$ and $s>0$, the map
\[
    H^s(W/W_1, \tilde{F_1}^{W_1}) 
    \raa H^s(W/W_1, (\Qbb_p(F_0)^{W_1})^\x) 
\]
is surjective (and hence an isomorphism) if and only if
$e(\Qbb_p(F_0)^{W_1})$ divides $p-1$. In particular for $W = G$, the map
\[
    i_{G}^\ast: H^2(G, \tilde{F_1}) \raa H^2(G, \Qbb_p(F_0)^\x)
\]
is surjective (and hence an isomorphism) if and only if $W_1$ is
realized and $e(\Qbb_p(F_0)^{W_1})$ divides $p-1$, that is, if and only
if $W_1$ is realized and $W_1/W_0 \iso C_{p^\alpha}$. The result then
follows from theorem \ref{190} and corollary \ref{192}.
\end{proof}

\section{Extensions of maximal abelian finite subgroups 
of $\Sbb_n$ for $p=2$} 

In this section, we assume $p = 2$, $F_0$ to be a maximal abelian
finite subgroup of $\Sbb_n$, and $\tilde{F_1}$ to be maximal as a
subgroup of $\Qbb_2(F_0)^\x$ having $\tilde{F_0}$ as a subgroup of
finite index; in other words
\[
    F_0 \iso C_{2^\alpha} \x C_{2^{n_\alpha}-1}
    \qq \text{with} \q 
    1 \leq \alpha \leq k,\
    n_\alpha = \frac{n}{\phi(2^\alpha)}.
\]
By corollary \ref{294}, we have $\tilde{F_1} = \lan x_1 \ran \x F_0$
with
\[
    v(x_1) =
    \begin{cases}
        1 
        & \text{if } \alpha \leq 1,
        \text{ or if } u \equiv \pm 3 \mod 8 
        \text{ and } n_\alpha \text{ is odd},\\
        \frac{1}{2}
        & \text{if } \alpha \geq 2 
        \text{ and either } u \equiv \pm 1 \mod 8
        \text{ or } n_\alpha \text{ is even}.
    \end{cases}
\]

\begin{remark} \label{295}
Since $n_\alpha = 2^{k-\alpha}m$ with $m$ odd, we have
\[
    n_\alpha \equiv 1 \mod 2
    \qq \Lra \qq
    \alpha = k.
\]
\end{remark}

\noindent By remark \ref{297}, we may in fact choose $x_1 \in
\tilde{F_1}$ to be $x_1 = 2u$ in the cases where its valuation is $1$,
otherwise to be $x_1 = (1+i)t$ for $i \in \Qbb_2(F_0)^\x$ a primitive
$4$-th root of unity and
\[
    t \in
    \begin{cases}
        \Zbb_2^\x & \text{if } u \equiv \pm 1 \mod 8,\\
        \Zbb_2(\zeta_3)^\x & \text{if } u \equiv \pm 3 \mod 8,
    \end{cases}
    \qq \text{with} \qq
    t^2 =
    \begin{cases}
        u & \text{if } u \equiv 1 \text{ or } -3 \mod 8,\\
        -u & \text{if } u \equiv -1 \text{ or } 3 \mod 8.
    \end{cases}
\]
By definition $\Qbb_2(F_0) = \Qbb_2(\tilde{F_1})$, and because the
latter is a maximal subfield of $\Dbb_n$ we have $\tilde{F_1} =
\tilde{F_2}$. We let
\[
    G 
    := Gal(\Qbb_2(F_0)/\Qbb_2) 
    \iso
    \begin{cases}
        C_n 
        & \text{if } \alpha = 1,\\
        C_{n_\alpha} \x C_{2^{\alpha-2}} \x C_{2}
        & \text{if } \alpha \geq 2,
    \end{cases}
\]
as given by proposition \ref{342}. From our choice of $x_1$, we know
that $\tilde{F_1}$ is stable under the action of a subgroup $W \subset
G$: if $x_1 = 2u$ this is clear, and if $v(x_1) = \frac{1}{2}$ and
$\sigma \in W$ we have $\frac{\sigma(x_1)}{x_1} \in F_0$ and hence
$\sigma(x_1) \in x_1 F_0 \subset \tilde{F_1}$.  The goal of the section
is to determine necessary and sufficient conditions on $n$, $u$ and
$\alpha$ for the homomorphism
\[
    i_G^\ast: H^2(G, \tilde{F_1}) \raa H^2(G, \Qbb_2(F_0)^\x)
\]
to be surjective, and whenever this happens, we want to determine its
kernel. This is done via the analysis of
\[
    i_W^\ast: H^2(W, \tilde{F_1}) \raa H^2(W, \Qbb_2(F_0)^\x)
\]
for suitable subgroups $W \subset G$.

\subsection*{The case $\alpha = 1$}

The situation is much simpler when the $2$-Sylow subgroup of $F_0$ is
contained in $\Qbb_2^\x$. Recall that $C_{2^\alpha} \ast C_{n}$ denotes
the kernel of the $n$-th power map on $C_{2^\alpha}$.

\begin{lemma} \label{231}
If $\alpha \leq 1$ and $W = C_n$, then $\tilde{F_1} = \lan 2u \ran \x
F_0$ and
\begin{align*}
    H^\ast(W, \tilde{F_1}) &\iso
    \begin{cases}
        \lan 2u \ran \x C_{2^\alpha} 
        & \text{if } \ast = 0,\\
        C_{2^\alpha} \ast C_n 
        & \text{if } 0 < \ast \text{ is odd},\\
        \lan 2u \ran / \lan (2u)^n \ran \x C_{2^\alpha} \ox C_n 
        & \text{if } 0 < \ast \text{ is even};
    \end{cases} \\[2ex]
    H^\ast(W, \Qbb_2(F_0)^\x) &\iso
    \begin{cases}
        \Qbb_2^\x 
        & \text{if } \ast = 0,\\
        0 
        & \text{if } 0 < \ast \text{ is odd},\\
        \lan 2 \ran / \lan 2^n \ran 
        & \text{if } 0 < \ast \text{ is even}.
    \end{cases}
\end{align*}
\end{lemma}

\begin{proof}
We know from corollary \ref{294} that $\tilde{F_1} = \tilde{F_0}$. The
action of $W = C_n$ on $\tilde{F_1} \iso \lan 2u \ran \x C_{2^\alpha} \x
C_{2^n-1}$ is trivial on $\lan 2u \ran \x C_{2^\alpha}$ and acts on
$C_{2^n-1}$ by $\zeta \mto \zeta^2$.

For $t$ a generator of $C_n$, written additively, and $N =
\sum_{i=0}^{n-1} t^i$, $H^\ast(C_n, \tilde{F_1})$ is the cohomology of
the complex
\[
    \xymatrix{
        \tilde{F_1} \ar[r]^{1-t}
        & \tilde{F_1} \ar[r]^{N}
        & \tilde{F_1} \ar[r]^{1-t}
        & \ldots\ .
    }
\]
Using additive notation for $\tilde{F_1} \iso \Zbb \x \Zbb/2^\alpha \x
\Zbb/2^n-1$, we obtain
\begin{align*}
    (1-t)(1, 0, 0) &= (0, 0, 0),
    & N(1, 0, 0) &= (n, 0, 0),\\
    (1-t)(0, 1, 0) &= (0, 0, 0),
    & N(0, 1, 0) &= (0, n, 0),\\
    (1-t)(0, 0, 1) &= (0, 0, -1),
    & N(0, 0, 1) &= (0, 0, 2^n-1),
\end{align*}
and the desired result for $H^\ast(W, \tilde{F_1})$ follows.

Now let $L = \Qbb_2(F_0) = \Qbb_2(\tilde{F_1})$ and $K = L^W$. Then
\[
    H^0(W, \Qbb_2(F_0)^\x) 
    = K^\x 
    = \Qbb_2^\x
\]
and $H^1(W, \Qbb_2(F_0)^\x) = 0$ by Hilbert's theorem 90. Furthermore as
$L/K$ is unramified, proposition \ref{256} imply
\[
    H^2(W, \Qbb_2(F_0)^\x) 
    \iso \lan 2 \ran / \lan 2^n \ran
\]
as desired.
\end{proof}

\begin{corollary} \label{232}
If $\alpha = 1$ and $W \subset C_n$, then $i_W^\ast: H^2(W, \tilde{F_1})
\ra H^2(W, \Qbb_2(F_0)^\x)$ is an epimorphism.  It is an isomorphism if
and only if $n$ is odd. If $n$ is even, its kernel is $\{\pm 1\}$.
\end{corollary}

\begin{proof}
First assume that $W = C_n$ with $L = \Qbb_2(F_0)$ and $K = L^W$. As
$L/K$ is unramified, proposition \ref{256} yields $u \in N_{C_n}(L^\x)$.
Hence $i_{C_n}^\ast$ is surjective by lemma \ref{231}. The case $W
\subset C_n$ follows from proposition \ref{214}, and the other
assertions are clear.
\end{proof}

\begin{example} \label{263}
When $\alpha = 1$, the group $F_0 \iso C_2 \x C_{2^n-1}$ is generated by
$-\omega$ for $\omega$ a $(2^n\!-\!1)$-th root of unity in $\Sbb_n$. Here
$\Qbb_2(F_0)/\Qbb_2$ is a maximal unramified commutative extension in
$\Dbb_n$ and $\tilde{F_0} = \tilde{F_1} = \tilde{F_2}$. Now for any $u
\in \Zbb_2^\x$ there are elements $\xi_u$ and $\xi_{-u}$ of valuation
$\frac{1}{n}$ in $N_{\Dbb_n^\x}(F_0)$ such that
\[
    \xi_{u}^n = 2u, \qq
    \xi_{-u}^n = -2u
    \qq \text{and} \qq
    \xi_{\pm u} \omega \xi_{\pm u}^{-1} = \omega^2,
\]
with $\tilde{F_3^+} = \lan \xi_{u} \ran \x F_0$ and $\tilde{F_3^-} =
\lan \xi_{-u} \ran \x F_0$. In $\Gbb_n(u)$, this gives extensions
\[
    1 \raa F_0 \raa F_3^{\pm} \raa C_n \raa 1,
\]
having classes in
\[
    H^2(C_n, F_0)
    \iso H^2(C_n, C_2) \op H^2(C_n, C_{2^n-1})
    \iso H^2(C_n, C_2)
    \iso
    \begin{cases}
        0 & \text{if } n \text{ is odd},\\
        \Zbb/2 & \text{if } n \text{ is even}.
    \end{cases}
\]
One of the extensions is a semi-direct product, represented by 
\[
    \lan -\omega, \bar{\xi}_u \ran
    \iso C_{2(2^n-1)} \rtimes C_n,
\]
for $\bar{\xi}_u$ the class of $\xi_u$ in $\Gbb_n(u)$. When $n$ is even,
we have 
\[
    (-\bar{\xi}_{-u})^n
    = (-1)^n (\bar{\xi}_{-u})^n
    = -1
\]
for $\bar{\xi}_{-u}$ the class of $\xi_{-u}$ in $\Gbb_n(u)$. The
respective $2$-Sylow subgroups of $\lan -\omega, \bar{\xi}_u \ran$ and
$\lan -\omega, \bar{\xi}_{-u} \ran$ are $C_2 \x C_{2^{k-1}}$ and
$C_{2^k}$ which are clearly not isomorphic.
\end{example}

\subsection*{The case $\alpha \geq 2$}

We let $\alpha \geq 2$. In this case $\Qbb_2(i) \subset \Qbb_2(F_0)$.

\begin{proposition} \label{235}
If $\alpha \geq 2$ and $W_0$ is a subgroup of odd order in $C_{n_\alpha}
\subset Aut(C_{2^{n_\alpha}-1})$, then $i_W^\ast: H^2(W_0, \tilde{F_1})
\ra H^2(W_0, \Qbb_2(F_0)^\x)$ is an isomorphism.
\end{proposition}

\begin{proof}
We may use the same argument as proposition \ref{243}. Using that
$\alpha \geq 2$, we know that $\Zbb/\Zbb\lan x_1 \ran$ is either trivial
or a $2$-torsion group, while $\Zbb_2(F_0)^\x/F_0$ is free over
$\Zbb_2$. Hence
\[
    H^\ast(W_0, \Zbb_2(F_0)^\x/F_0)
    = H^\ast(W_0, \Zbb/\Zbb\lan x_1 \ran)
    = H^\ast(W_0, \Qbb_2(F_0)^\x/\tilde{F_1}) 
    = 0
    \qq \text{for } \ast > 0,
\]
and the result follows.
\end{proof}

\begin{lemma} \label{233}
If $\alpha \geq 2$, $u \equiv \pm 3 \mod 8$, $n_\alpha$ is odd and $W =
C_{n_\alpha} \subset Aut(C_{2^{n_\alpha}-1})$, then $\tilde{F_1} =
\tilde{F_0}$ and
\begin{align*}
    H^\ast(W, \tilde{F_1}) &\iso
    \begin{cases}
        \lan 2u \ran \x C_{2^\alpha} 
        & \text{if } \ast = 0,\\
        C_{2^\alpha} \ast C_{n_\alpha} 
        & \text{if } 0 < \ast \text{ is odd},\\
        \lan 2u \ran / \lan (2u)^{n_\alpha} \ran 
        \x C_{2^\alpha} \ox C_{n_\alpha} 
        & \text{if } 0 < \ast \text{ is even};
    \end{cases} \\[2ex]
    H^\ast(W, \Qbb_2(F_0)^\x) &\iso
    \begin{cases}
        \Qbb_2(\zeta_{2^\alpha})^\x 
        & \text{if } \ast = 0,\\
        0 
        & \text{if } 0 < \ast \text{ is odd},\\
        \lan \zeta_{2^\alpha}-1 \ran 
        / \lan (\zeta_{2^\alpha}-1)^{n_\alpha} \ran 
        & \text{if } 0 < \ast \text{ is even}.
    \end{cases}
\end{align*}
\end{lemma}

\begin{proof}
We know from corollary \ref{294} that $\tilde{F_1} = \tilde{F_0}$. The
calculations for $H^\ast(W, \tilde{F_1})$ and $H^\ast(W, \Qbb_2(F_0))$
are identical to that of lemma \ref{231}, except that $2$ is replaced
with $(\zeta_{2^\alpha}-1)$ in the second case.
\end{proof}

\begin{lemma} \label{234}
If $\alpha \geq 2$, $u \equiv \pm 1 \mod 8$ and $W = C_{n_\alpha}
\subset Aut(C_{2^{n_\alpha}-1})$, then $\tilde{F_1} = \lan x_1 \ran \x
F_0$ with $v(x_1)=\frac{1}{2}$ and
\begin{align*}
    H^\ast(W, \tilde{F_1}) &\iso
    \begin{cases}
        \lan x_1 \ran \x C_{2^\alpha} 
        & \text{if } \ast = 0,\\
        C_{2^\alpha} \ast C_{n_\alpha} 
        & \text{if } 0 < \ast \text{ is odd},\\
        \lan x_1 \ran / \lan x_1^{n_\alpha} \ran 
        \x C_{2^\alpha} \ox C_{n_\alpha} 
        & \text{if } 0 < \ast \text{ is even};
    \end{cases} \\[2ex]
    H^\ast(W, \Qbb_2(F_0)^\x) &\iso
    \begin{cases}
        \Qbb_2(\zeta_{2^\alpha})^\x 
        & \text{if } \ast = 0,\\
        0 
        & \text{if } 0 < \ast \text{ is odd},\\
        \lan \zeta_{2^\alpha}-1 \ran 
        / \lan (\zeta_{2^\alpha}-1)^{n_\alpha} \ran 
        & \text{if } 0 < \ast \text{ is even}.
    \end{cases}
\end{align*}
\end{lemma}

\begin{proof}
We know from corollary \ref{294} that $\tilde{F_1} = \lan x_1 \ran \x
F_0$ with $v(x_1) = \frac{1}{2}$. The action of $C_{n_\alpha}$ on
$\tilde{F_1} \iso \lan x_1 \ran \x C_{2^\alpha} \x C_{2^{n_\alpha}-1}$
is trivial on the first two factors and acts on the third by $\zeta \mto
\zeta^2$.

Let $t$ be generator of $C_{n_\alpha}$, written additively, and $N =
\sum_{i=0}^{n_\alpha-1} t^i$. Using additive notation for $\tilde{F_1}
\iso \Zbb \x \Zbb/2^\alpha \x \Zbb/2^{n_\alpha}-1$, we obtain
\begin{align*}
    (1-t)(1, 0, 0) &= (0, 0, 0),
    & (1-t)(0, 1, 0) &= (0, 0, 0),
    & (1-t)(0, 0, 1) &= (0, 0, -1),\\[1ex]
    N(1, 0, 0) &= (n_\alpha, 0, 0),
    & N(0, 1, 0) &= (0, n_\alpha, 0),
    & N(0, 0, 1) &= (0, 0, 0),
\end{align*}
and the desired result for $H^\ast(W, \tilde{F_1})$ follows.

Now for $L = \Qbb_2(F_0)$ and $K = L^W = \Qbb_2(\zeta_{2^\alpha})$, we
have
\[
    H^0(W, \Qbb_2(\tilde{F_1})) 
    = \Qbb_2(Ker(1-t))^\x 
    = \Qbb_2(\zeta_{2^\alpha})^\x
\]
and $H^1(W, \Qbb_2(F_0)^\x) = 0$ by Hilbert's theorem 90.
Furthermore, as $L/K$ is unramified, $\zeta_{2^\alpha}-1$ is a
uniformizing element of $L$ and proposition \ref{256} implies
\[
    H^2(W, \Qbb_2(F_0)^\x) 
    \iso \lan \zeta_{2^\alpha}-1 \ran 
    / \lan (\zeta_{2^\alpha}-1)^{n_\alpha} \ran
\]
as desired.
\end{proof}

\begin{lemma} \label{236}
If $\alpha \geq 3$, $u \equiv \pm 3 \mod 8$, $n_\alpha$ is odd and $W =
C_{2^{\alpha-2}} \subset Aut(C_{2^\alpha})$ is generated by $\zeta \mto
\zeta^5$, then $\tilde{F_1} = \tilde{F_0}$ and
\begin{align*}
    H^\ast(W, \tilde{F_1}) &\iso
    \begin{cases}
        \lan 2u \ran \x C_{4} \x C_{2^{n_\alpha}-1} 
        & \text{if } \ast = 0,\\
        0 
        & \text{if } 0 < \ast \text{ is odd},\\
        \lan 2u \ran / \lan (2u)^{2^{\alpha-2}} \ran 
        \iso C_{2^{\alpha-2}}
        & \text{if } 0 < \ast \text{ is even};
    \end{cases} \\[2ex]
    H^\ast(W, \Qbb_2(F_0)^\x) &\iso
    \begin{cases}
        (\Qbb_2(F_0)^{ C_{2^{\alpha-2}} })^\x 
        & \text{if } \ast = 0,\\
        0 
        & \text{if } 0 < \ast \text{ is odd},\\
        (\Qbb_2(F_0)^{ C_{2^{\alpha-2}} })^\x 
        / N_W(\Qbb_2(F_0)^\x) \iso C_{2^{\alpha-2}} 
        & \text{if } 0 < \ast \text{ is even}.
    \end{cases}
\end{align*}
\end{lemma}

\begin{proof}
We know from corollary \ref{294} that $\tilde{F_1} = \tilde{F_0}$. The
action of $C_{2^{\alpha-2}}$ on $\tilde{F_1} \iso \lan 2u \ran \x
C_{2^\alpha} \x C_{2^{n_\alpha}-1}$ is trivial on $\lan 2u \ran \x
C_{2^{n_\alpha}-1}$ and acts on $C_{2^{\alpha}}$ by $\zeta \mto
\zeta^{5}$.

For $t$ a generator of $C_{2^{\alpha-2}}$, written additively, and $N =
\sum_{i=0}^{2^{\alpha-2}-1} t^i$, we obtain
\begin{align*}
    (1-t)(1, 0, 0) &= (0, 0, 0),
    & N(1, 0, 0) &= (2^{\alpha-2}, 0, 0),\\
    (1-t)(0, 1, 0) &= (0, -4, 0),
    & N(0, 1, 0) &= (0, 2^{\alpha-2}, 0),\\
    (1-t)(0, 0, 1) &= (0, 0, 0),
    & N(0, 0, 1) &= (0, 0, 2^{\alpha-2}),
\end{align*}
and the desired result for $H^\ast(W, \tilde{F_1})$ follows.

The case of $H^\ast(W, \Qbb_2(F_0)^\x)$ for $0 < \ast$ odd follows from
Hilbert's theorem 90, and the rest is clear.
\end{proof}

\begin{lemma} \label{237}
Let $\alpha \geq 3$, and assume either $u \equiv \pm 1 \mod 8$ or $u
\equiv \pm 3 \mod 8$ with $n_\alpha$ even. If $W = C_{2^{\alpha-2}}
\subset Aut(C_{2^\alpha})$ is generated by $\zeta \mto \zeta^5$, then
$\tilde{F_1} = \lan x_1 \ran \x F_0$ with $v(x_1) = \frac{1}{2}$ and
\begin{align*}
    H^\ast(W, \tilde{F_1}) &\iso
    \begin{cases}
        \lan x_1 \ran \x C_{4} \x C_{2^{n_\alpha}-1} 
        & \text{if } \ast = 0,\\
        0 
        & \text{if } 0 < \ast \text{ is odd},\\
        \lan x_1 \ran / \lan x_1^{2^{\alpha-2}} \ran 
        \iso C_{2^{\alpha-2}}
        & \text{if } 0 < \ast \text{ is even};
    \end{cases} \\[2ex]
    H^\ast(W, \Qbb_2(F_0)^\x) &\iso
    \begin{cases}
        (\Qbb_2(F_0)^{ C_{2^{\alpha-2}} })^\x 
        & \text{if } \ast = 0,\\
        0 
        & \text{if } 0 < \ast \text{ is odd},\\
        (\Qbb_2(F_0)^{ C_{2^{\alpha-2}} })^\x 
        / N_W(\Qbb_2(F_0)^\x) \iso C_{2^{\alpha-2}} 
        & \text{if } 0 < \ast \text{ is even}.
    \end{cases}
\end{align*}
\end{lemma}

\begin{proof}
We know from corollary \ref{294} that $\tilde{F_1} = \lan x_1 \ran \x
F_0$ with $v(x_1) = \frac{1}{2}$. The calculations are identical to that
of lemma \ref{236}, except that $2u$ is replaced by $x_1$ for the
calculation of $H^\ast(W, \tilde{F_1})$.
\end{proof}

\begin{corollary} \label{238}
If $\alpha \geq 3$ and $W = C_{2^{\alpha-2}} \subset Aut(C_{2^\alpha})$
is generated by $\zeta \mto \zeta^5$, then $i_W^\ast: H^2(W,
\tilde{F_1}) \ra H^2(W, \Qbb_2(F_0)^\x)$ is never surjective.
\end{corollary}

\begin{proof}
Let $L:= \Qbb_2(F_0)$ and $K:=L^{W}$. Since $L/K$ is totally ramified,
we know from proposition \ref{256} that $H^2(W, L^\x) \iso H^2(W,
\Ocal_L^\x)$. As $N_{G/W} \circ N_W (\Ocal_L^\x) = N_G(\Ocal_L^\x)$, we
may consider the homomorphism
\[
    \tau: H^2(W, L^\x) \raa \Zbb_2^\x/N_G(\Ocal_L^\x)
\]
given by the norm
\[
    N_{G/W}: H^2(W, \Ocal_L^\x) 
    \iso (\Ocal_K^\x)/N_W(\Ocal_L^\x) \raa \Zbb_2^\x/N_G(\Ocal_L^\x).
\]
In order to analyse this homomorphism, we consider the short exact
sequences
\[
    \xymatrix{
        1 \ar[r]
        & \Zbb_2^\x/N_{G}(\Ocal_L^\x) \ar[r] \ar[d]
        & \Qbb_2^\x/N_{G}(L^\x) 
        \ar[r]^{v} \ar[d]_{\iso}^{(\_, L/\Qbb_2)}
        & \Zbb/v(N_{G}(L^\x)) 
        \ar[r] \ar[d]_{\iso}^{\sigma^{(\_)}}
        & 1\\
        1 \ar[r]
        & Gal(L/L^{C_{2^{\alpha-2}} \x C_2}) 
        \ar[r]
        & Gal(L/\Qbb_2) \ar[r]^{pr}
        & Gal(l/\Fbb_2) \ar[r]
        & 1\\
    }
\]
where 
\[
    Gal(L/L^{C_{2^{\alpha-2}} \x C_2}) 
    \iso C_{2^{\alpha-2}} \x C_2, \qq
    Gal(L^{C_{2^{\alpha-2}} \x C_2}/\Qbb_2) 
    \iso Gal(l/\Fbb_2) \iso C_{n_\alpha}
\]
for $l$ the residue field of $L$, where the middle vertical isomorphism
is the norm residue symbol of $L/\Qbb_2$ as defined in \cite{serre4}
section 2.2, the left hand vertical map is its restriction, and where
the right hand vertical isomorphism is given by the power map of the
Frobenius automorphism $\sigma \in Gal(l/\Fbb_2)$.  We know from local
class field theory (see for example \cite{koch} chapter 2 \S 1.3) that
\[
    pr(x, L/\Qbb_2) = (x, L^{C_{2^{\alpha-2}} \x C_2}/\Qbb_2)
    \qq \text{for all } x \in \Qbb_2^\x/N_{G}(L^\x). 
\]
On the other hand \cite{serre4} proposition 2 shows that
\[
    (x, L^{C_{2^{\alpha-2}} \x C_2}/\Qbb_2) = \sigma^{v(x)}
    \qq \text{for all } x \in \Qbb_2^\x/N_{G}(L^\x). 
\]
Thus the right hand square in the above diagram, and hence the diagram
itself, is commutative. The five lemma then implies
\[
    \Zbb_2^\x/N_G(\Ocal_L^\x) \iso C_{2^{\alpha-2}} \x C_2
    \qq \text{and} \qq
    N_G(\Ocal_L^\x) = U_\alpha(\Zbb_2^\x).
\]
The image of $\tau$ however is $U_2(\Zbb_2^\x)/U_\alpha(\Zbb_2^\x)$.  To
see this, consider the tower of extensions
\[
    \xymatrix{
        \Qbb_2 \ar@{-}[r]^{C_2}
        & \Qbb_2(i) \ar@{-}[r]^{C_{n_\alpha}}
        & K \ar@{-}[r]^{W}
        & L.
    }
\]
Since $K/\Qbb_2(i)$ is unramified, we know from proposition \ref{257}
that
\[
    N_{C_{n_\alpha}}:
    \Ocal_K^\x \raa \Zbb_2(i)^\x
\]
is surjective. Hence for any $a_1, a_2 \in \Zbb_2$, there exists an
element $x = 1+a(1+i)$ in $\Ocal_K^\x$ with $a \in \Zbb_2$ such
that
\[
    N_{C_{n_\alpha}}(x) = 1+(a_1+a_2i)(1+i)\ \in \Zbb_2(i)^\x.
\]
Therefore
\begin{align} \label{272}
    N_{G/W}(x)\ 
    \notag &=\ N_{C_2}(1+(a_1+a_2 i)(1+i))\\ 
    \notag &=\ [1+(a_1+a_2 i)(1+i)][1+(a_1-a_2 i)(1-i)]\\
    \notag &=\ 1 + 2(a_1^2 + a_2^2 + a_1 - a_2)\\
    \notag &=\ 1 + 2(a_1^2 + a_1) + 2(a_2^2 - a_2)\\
    &\equiv\ 1 \mod 4,
\end{align}
and the map $\tau: H^2(W, \Ocal_L^\x) \ra
U_2(\Zbb_2^\x)/U_\alpha(\Zbb_2^\x)$ is an isomorphism. 

By to lemma \ref{236} and \ref{237}, the map $i_W^\ast$ is therefore
surjective if and only if $\tau(x_1)$ is a generator of
$U_2(\Zbb_2^\x)/U_\alpha(\Zbb_2^\x)$. Recall that
\[
    x_1 =
    \begin{cases}
        2u & \text{if } u \equiv \pm 3 \mod 8
        \text{ and } n_\alpha \text{ is odd},\\
        (1+i)t & \text{otherwise},
    \end{cases}
\]
with
\[
    t^2 =
    \begin{cases}
        u & \text{if } u \equiv 1 \text{ or } -3 \mod 8,\\
        -u & \text{if } u \equiv -1 \text{ or } 3 \mod 8.
    \end{cases}
\]
Since both $2$ and $1+i$ belong to
$N_{\Qbb_2(\zeta_{2^\alpha})/\Qbb_2(i)} (\Qbb_2(\zeta_{2^\alpha}))$
according to example \ref{258}, it follows by remark \ref{271} that $2$
and $1+i$ both belong to $N_{L/K}(L^\x)$. Thus if $u \equiv \pm 3 \mod
8$ with $n_\alpha$ odd, we have
\[
    \tau(2u) = \tau(u) = u^{2n_\alpha} \equiv 1\ \mod 8.
\]
On the other hand if $u \equiv \pm 1 \mod 8$, then
\begin{align*}
    \tau(x_1) &= \tau(t) = t^{2n_\alpha} =
    \begin{cases}
        u^{n_\alpha} & \text{if } u \equiv 1 \mod 8,\\
        (-u)^{n_\alpha} & \text{if } (-u) \equiv 1 \mod 8,
    \end{cases} \\[1ex]
    &\equiv 1\ \mod 8.
\end{align*}
Finally if $u \equiv \pm 3 \mod 8$ with $n_\alpha$ even, there is a
subgroup of index $2$ in $G/W$ which acts trivially on $t$, and we have
\begin{align*}
    \tau(x_1) 
    &= \tau(t) 
    = (t(-t))^{n_\alpha} 
    = (-1)^{n_\alpha} t^{2n_\alpha}
    \equiv (\pm 3)^{n_\alpha}
    \equiv 1 \q \mod 8.
\end{align*}
In any case, the map $i_W^\ast$ is never surjective.
\end{proof}

\begin{lemma} \label{239}
Let $\alpha \geq 2$, $u \equiv \pm 3 \mod 8$, $n_\alpha$ be odd, and let
$C_2 \subset Aut(C_{2^\alpha})$ be generated by $\zeta \mto \zeta^{-1}$.
If $W_0$ is a subgroup of odd order in $G$, then $\tilde{F_1}^{W_0} =
\tilde{F_0}^{W_0}$ and
\begin{align*}
    H^\ast(C_2, \tilde{F_1}^{W_0}) &\iso
    \begin{cases}
        \lan 2u \ran \x (C_{2^\alpha} \ast C_2) \x
        C_{2^{\frac{n_\alpha}{|W_0|}}-1} 
        & \text{if } \ast = 0,\\
        C_{2^\alpha} \ox C_2 
        & \text{if } 0 < \ast \text{ is odd},\\
        \lan 2u \ran / \lan (2u)^{2} \ran 
        \x (C_{2^\alpha} \ast C_2)
        & \text{if } 0 < \ast \text{ is even};
    \end{cases} \\[2ex]
    H^\ast(C_2, (\Qbb_2(F_0)^{W_0})^\x) &\iso
    \begin{cases}
        (\Qbb_2(F_0)^{ C_{2} })^\x 
        & \text{if } \ast = 0,\\
        0 
        & \text{if } 0 < \ast \text{ is odd},\\
        (\Qbb_2(F_0)^{ C_{2} })^\x 
        / N_W(\Qbb_2(F_0)^\x) \iso C_{2} 
        & \text{if } 0 < \ast \text{ is even}.
    \end{cases}
\end{align*}
\end{lemma}

\begin{proof}
We know from corollary \ref{294} that $\tilde{F_1} = \tilde{F_0}$. The
action of $C_2$ on $\tilde{F_1}^{W_0} \iso \lan 2u \ran \x C_{2^\alpha}
\x C_{2^{\frac{n_\alpha}{|W_0|}}-1}$ is trivial on the first and last
factors and acts on the second by $\zeta \mto \zeta^{-1}$.

Let $t$ the generator of $C_2$, written additively, and $N = 1+t$.
Using additive notation for $\tilde{F_1}^{W_0} \iso \Zbb \x
\Zbb/2^\alpha \x \Zbb/(2^{\frac{n_\alpha}{|W_0|}}-1)$, we obtain
\begin{align*}
    (1-t)(1, 0, 0) &= (0, 0, 0),
    & N(1, 0, 0) &= (2, 0, 0),\\
    (1-t)(0, 1, 0) &= (0, 2, 0),
    & N(0, 1, 0) &= (0, 0, 0),\\
    (1-t)(0, 0, 1) &= (0, 0, 0),
    & N(0, 0, 1) &= (0, 0, 2),
\end{align*}
and the desired result for $H^\ast(C_2, \tilde{F_1}^{W_0})$ follows.

The case of $H^\ast(C_2, (\Qbb_2(F_0)^{W_0})^\x)$ for $0 < \ast$ odd
follows from Hilbert's theorem 90, and the rest is clear.
\end{proof}

\begin{lemma} \label{241}
Let $\alpha=2$ and either $u \equiv \pm 1 \mod 8$ or $u \equiv \pm 3
\mod 8$ with $n_\alpha$ even. If $C_2 \subset Aut(C_{2^\alpha})$ is
generated by $\zeta \mto \zeta^{-1}$, and if $W_0$ is a subgroup of odd
order in $G$, then $\tilde{F_1}^{W_0} = \lan x_1 \ran \x F_0^{W_0}$ with
$v(x_1)=\frac{1}{2}$ and
\begin{align*}
    H^\ast(C_2, \tilde{F_1}^{W_0}) &\iso
    \begin{cases}
        \lan 2u \ran \x (C_{2^\alpha} \ast C_2) \x
        C_{2^{\frac{n_\alpha}{|W_0|}}-1} 
        & \text{if } \ast = 0,\\
        0 
        & \text{if } 0 < \ast \text{ odd},\\
        C_{2^\alpha} \ast C_2
        & \text{if } 0 < \ast \text{ even};\\
    \end{cases} \\[2ex]
    H^\ast(C_2, (\Qbb_2(F_0)^{W_0})^\x) &\iso
    \begin{cases}
        (\Qbb_2(F_0)^{ C_{2} })^\x 
        & \text{if } \ast = 0,\\
        0 
        & \text{if } 0 < \ast \text{ odd},\\
        (\Qbb_2(F_0)^{ C_{2} })^\x 
        / N_W(\Qbb_2(F_0)^\x) \iso C_{2} 
        & \text{if } 0 < \ast \text{ even}.
    \end{cases}
\end{align*}
\end{lemma}

\begin{proof}
We know that $\tilde{F_1}^{W_0} = \lan x_1 \ran \x F_0^{W_0}$ with
$v(x_1) = \frac{1}{2}$. The action of $C_2$ on $\tilde{F_1}^{W_0} \iso
\lan x_1 \ran \x C_{2^\alpha} \x C_{2^{\frac{n_\alpha}{|W_0|}}-1}$ is
trivial on the last factor, acts on $C_{2^\alpha}$ by $\zeta_{2^\alpha}
\mto \zeta_{2^\alpha}^{-1}$ on the second, and sends $x_1$ to $-i x_1$.

Note that the last factor splits off and has trivial cohomology. Hence
for $t$ a generator of $C_2$, written additively, and $N = 1+t$, the
cohomology $H^\ast(C_2, \tilde{F_1}^{W_0})$ can be calculated from the
additive complex
\[
    \xymatrix{
        \Zbb \x \Zbb/4 \ar[r]^{1-t}
        & \Zbb \x \Zbb/4 \ar[r]^{N}
        & \Zbb \x \Zbb/4 \ar[r]^{\q 1-t}
        & \ldots\ ,
    }
\]
where
\[
    t(1,0) = (1, 1)
    \qq \text{and} \qq
    t(0,1) = (0,-1).
\]
Therefore
\begin{align*}
    (1-t)(1, 0) &= (0, -1), &
    (1-t)(0, 1) &= (0, 2),\\
    N(1, 0) &= (2, 1), &
    N(0, 1) &= (0, 0).
\end{align*}
Hence
\begin{align*}
    Ker(1-t) 
    &= \lan (2, 1), (0, 2) \ran,
    & Im(1-t) 
    &= \lan (0, 1) \ran, \\
    Ker(N) 
    &= \lan (0, 1) \ran,
    & Im(N) 
    &= \lan (2, 1) \ran,
\end{align*}
and the desired result for $H^\ast(C_2, \tilde{F_1}^{W_0})$ follows.

The case of $H^\ast(C_2, (\Qbb_p(F_0)^{W_0})^\x)$ for $0 < \ast$ odd
follows from Hilbert's theorem 90, and the rest is clear.
\end{proof}

We have seen in corollary \ref{238} that $i_G^\ast$ is not surjective
whenever $\alpha \geq 3$. Thus the case $\alpha = 2$ is all that we want
to consider in the following corollary.

\begin{corollary} \label{242}
Let $\alpha = 2$, $C_2 = Aut(C_{2^\alpha})$ and let $W_0$ be a subgroup
of odd order in $G$. Then $H^2(C_2, \tilde{F_1}^{W_0}) \ra H^2(C_2,
(\Qbb_2(F_0)^{W_0})^\x)$ is surjective if and only if $\alpha = k$. In
this case, its kernel is isomorphic to $C_2$ if $u \equiv \pm 3 \mod 8$
and it is an isomorphism if $u \equiv \pm 1 \mod 8$.
\end{corollary}

\begin{proof}
Let $L:=\Qbb_2(F_0)^{W_0}$, $K:=L^{C_2}$ and $H:= G/W_0 =
Gal(L/\Qbb_p)$. Note that $L/K$ is totally ramified. Similarly to
corollary \ref{238}, we may consider the homomorphism
\[
    \tau: H^2(C_2, L^\x) \raa \Zbb_2^\x/N_H(\Ocal_L^\x)
\]
given by the norm
\[
    N_{H/C_2}: H^2(C_2, L^\x) 
    \iso H^2(C_2, \Ocal_L^\x) 
    \iso (\Ocal_K^\x)/N_{C_2}(\Ocal_L^\x) 
    \raa \Zbb_2^\x/N_H(\Ocal_L^\x).
\]
Here again, as in corollary \ref{238}, we have short exact sequences
forming a commutative diagram
\[
    \xymatrix{
        1 \ar[r]
        & \Zbb_2^\x/N_{H}(\Ocal_L^\x) \ar[r] \ar[d]_{\iso}
        & \Qbb_2^\x/N_{H}(L^\x) 
        \ar[r]^{v} \ar[d]_{\iso}^{(\_, L/\Qbb_2)}
        & \Zbb/v(N_{H}(L^\x)) 
        \ar[r] \ar[d]_{\iso}^{\sigma^{(\_)}}
        & 1\\
        1 \ar[r]
        & Gal(L/K) 
        \ar[r]
        & Gal(L/\Qbb_2) \ar[r]^{pr}
        & Gal(l/\Fbb_2) \ar[r]
        & 1\\
    }
\]
where
\[
    Gal(L/K) \iso C_2
    \qq \text{and} \qq
    Gal(l/\Fbb_2) \iso C_{\frac{n_\alpha}{|W_0|}},
\]
for $l$ the residue field of $L$. Since $L^{C_{n_\alpha}/ W_0} =
\Qbb_2(F_0)^{C_{n_\alpha}} = \Qbb_2(i)$, and since $L/\Qbb_2(i)$ is
unramified, we know from proposition \ref{257} that 
\[
    N_{H/C_2} 
    : \Ocal_L^\x \raa \Zbb_2(i)^\x
\]
is surjective; consequently
\[
    N_H(\Ocal_L^\x) 
    = N_{C_2} \circ N_{H/C_2}(\Ocal_L^\x)
    = N_{C_2}(\Zbb_2(i)^\x).
\]
Furthermore, as in (\ref{272}), for any elements $a_1, a_2 \in \Zbb_2$
we have 
\[
    N_{C_2}(1+(a_1+a_2 i)(1+i))
    \equiv 1 \mod 4.
\]
Hence $N_H(\Ocal_L^\x) = U_2(\Zbb_2^\x)$ and the map
\[
    \tau: H^2(C_2, L^\x) \raa \Zbb_2^\x/U_2(\Zbb_2^\x) = \{\pm 1\}
\]
is surjective by proposition \ref{257}.

Using lemma \ref{239} and \ref{241}, the map $i^\ast: H^2(C_2,
\tilde{F_1}^{W_0}) \ra H^2(C_2, L^\x)$ is therefore surjective if and
only if
\[
    -1 =
    \begin{cases}
        \tau(2u) \text{ or } \tau(-1) 
        & \text{if } u \equiv \pm 3 \mod 8 
        \text{ and } n_\alpha \text{ odd},\\
        \tau(-1) 
        & \text{otherwise}.
    \end{cases}
\]
Since $N_{C_2}(1+i) = (1+i)(1-i) = 2$, remark \ref{271} implies
\[
    \tau(2u) = \tau(u) = u^{|H/C_2|}
    \qq \text{and} \qq
    \tau(-1) = (-1)^{|H/C_2|}.
\]
Hence $\tau(-1) = -1$ if and only if 
\[
    |H/C_2| = |C_{n_\alpha}/W_0|
    \text{ is odd}
    \qq \Lra \qq
    n_\alpha \text{ is odd}
    \qq \Lra \qq
    \alpha = k,
\]
and the result follows.
\end{proof}

\begin{theorem} \label{259}
Let $p=2$, $n = 2^{k-1}m$ with $m$ odd, $u \in \Zbb_2^\x$, $F_0 =
C_{2^\alpha} \x C_{2^{n_\alpha}-1}$ be a maximal abelian finite
subgroup of $\Sbb_n$, $G = Gal(\Qbb_2(F_0)/\Qbb_2)$, $G_{2'}$ be the odd
part of $G$, and let $\tilde{F_1} = \lan x_1 \ran \x F_0 \subset
\Qbb_2(F_0)^\x$ be maximal as a subgroup of $\Qbb_2(F_0)^\x$ having
$\tilde{F_0}$ as subgroup of finite index. 
\begin{enumerate}
    \item For any $1 \leq \alpha \leq k$, there is an extension of
    $\tilde{F_1}$ by $G_{2'}$; this extension is unique up to
    conjugation.
    
    \item If $\alpha = 1$, there is an extension of $\tilde{F_1}$ by
    $G$; the number of such extensions up to conjugation is
    \[
        \begin{cases}
            1 & \text{if } n \text{ is odd},\\
            2 & \text{if } n \text{ is even}.
        \end{cases}
    \]

    \item If $\alpha = 2$, there is an extension of $\tilde{F_1}$ by $G$
    if and only if $k=2$; the number of such extensions up to
    conjugation is
    \[
        \begin{cases}
            1 & \text{if } u \equiv \pm 1 \mod 8,\\
            2 & \text{if } u \not\equiv \pm 1 \mod 8.
        \end{cases}
    \]

    \item If $\alpha \geq 3$, there is no extension of $\tilde{F_1}$ by
    $G$.
\end{enumerate}
\end{theorem}

\begin{proof}
1) From corollary \ref{232} and proposition \ref{235} we know that
\[
    i_{G_{2'}}^\ast:
    H^2(G_{2'}, \tilde{F_1}) \raa H^2(G_{2'}, \Qbb_2(F_0)^\x)
\]
is an isomorphism. Existence and uniqueness up to conjugation then
follows from corollary \ref{192}.

2) This follows from corollary \ref{232} and \ref{192}.

3) Let $\alpha = 2$. Applying proposition \ref{214} and corollary
\ref{192} together with corollary \ref{242} in the case where $W_0$ is
trivial, we obtain that $\tilde{F_1}$ can never be extended by $G$ when
$n_\alpha$ is even. Assume then that $n_\alpha$ is odd. In this case $G$
decomposes canonically as
\[
    G = G_{2'} \x C_2
    \qq \text{with} \q
    G_{2'} = C_{n_\alpha}.
\]
In particular, there is a short exact sequence
\[
    1 \raa C_{n_\alpha} \raa G \raa C_2 \raa 1
\]
which gives rise to the Hochschild-Serre spectral sequences (see
\cite{brown} section VII.6)
\begin{align*}
    E_2^{s, t} 
    \iso H^s(C_2, H^t(C_{n_\alpha}, \tilde{F_1})) \q
    &\Raa \q H^{s+t}(G, \tilde{F_1}),\\[1ex]
    E_2^{s, t} 
    \iso H^s(C_2, H^t(C_{n_\alpha}, \Qbb_2(F_0)^\x)) \q
    &\Raa \q H^{s+t}(G, \Qbb_2(F_0)^\x).
\end{align*}
From lemma \ref{233}, \ref{234} and proposition \ref{235}, each map
$E_2^{s, t} \ra E_2^{s, t}$ is an isomorphism for $t>0$. We also have
\[
    H^0(C_{n_\alpha}, \tilde{F_1})
    = \tilde{F_1}^{C_{n_\alpha}}
    = \lan x_1 \ran \x C_{2^\alpha}
\]
and
\[
    H^0(C_{n_\alpha}, \Qbb_2(F_0)^\x)
    = (\Qbb_2(F_0)^{C_{n_\alpha}})^\x
    = \Qbb_2(i)^\x.
\]
Then corollary \ref{242} applied to the case $W_0 = C_{n_\alpha}$ shows
that the map
\[
    H^s(C_2, \tilde{F_1}^{C_{n_\alpha}})
    \raa H^s(C_2, (\Qbb_2(F_0)^{C_{n_\alpha}})^\x)
\]
is surjective when $t=0$ and $s>0$; its kernel is trivial if $u \equiv
\pm 1 \mod 8$, otherwise it is of cardinality $2$. In fact, since
$n_\alpha$ is odd, all the terms $E_2^{s, t}$ for which $s>0$ and $t>0$
are trivial. By the results of lemma \ref{233} and \ref{234}, the
non-trivial terms for which $s=0$ are of odd order, and the non-trivial
terms for which $t=0$ are powers of $2$. Hence all differentials of the
spectral sequences are trivial and $E_2^{\ast, \ast} = E_\infty^{\ast,
\ast}$. Consequently, $i_G^\ast$ is surjective if and only if $n_\alpha$
is odd, that is, if and only if $\alpha=k$. The result then follows from
corollary \ref{192}.

4) By corollary \ref{238} and proposition \ref{214} the map
\[
    i_{G}^\ast:
    H^2(G, \tilde{F_1}) \raa H^2(G, \Qbb_2(F_0)^\x)
\]
is never surjective if $\alpha \geq 3$. The result is then a consequence
of corollary \ref{192}.
\end{proof}

\section{Extensions of maximal finite subgroups of $\Sbb_n$ 
containing $Q_8$} 

In this section, we establish under what condition a maximal finite
subgroup $G$ of $\Sbb_n$ with a quaternionic $2$-Sylow subgroup extends
to a subgroup of order $n|G|$ in $\Gbb_n(u)$. Recall from theorem
\ref{032} that such a $G$ exists if and only if $p=2$ and $n=2m$ with
$m$ odd, in which case
\[
    G 
    \iso Q_8 \rtimes C_{3(2^m-1)}
    \iso T_{24} \x C_{2^m-1}.
\]

\begin{theorem} \label{260}
Let $p=2$, $n=2m$ with $m$ odd, and $u \in \Zbb_2^\x$. A subgroup $G$
isomorphic to $T_{24} \x C_{2^m-1}$ in $\Sbb_n$ extends to a maximal
finite subgroup $F$ of order $n|G|=48m(2^m-1)$ in $\Gbb_n(u)$ if and
only if $u \equiv \pm 1 \mod 8$; this extension is unique up to
conjugation. Moreover if $u \not\equiv \pm 1 \mod 8$ and $G'$ is a
subgroup isomorphic to $Q_8 \x C_{2^m-1}$ in $\Sbb_n$, there is no
extension of $G'$ of order $n|G'|$ in $\Gbb_n(u)$.
\end{theorem}

\begin{proof}
Let $i, j, \zeta_3, \zeta_{2^m-1}$ be elements of respective order $4$,
$4$, $3$ and $2^m-1$ generating $G$, and let $T:=\lan i, j, \zeta_3 \ran
\iso T_{24}$. We first establish the structure of the centralizer of
$G$. By the centralizer theorem \ref{007}, there is a $\Qbb_2$-algebra
isomorphism
\[
    \Dbb_n 
    \iso \Qbb_2(T) \ox_{\Qbb_2} C_{\Dbb_n}(T),
\]
where $C_{\Dbb_n}(T)$ is a central division algebra of dimension $m^2$
over $\Qbb_2$. Note that the commutative extension
$\Qbb_2(\zeta_{2^m-1})/\Qbb_2$ is maximal unramified in $C_{\Dbb_n}(T)$.
Consequently
\[
    C_{\Dbb_n^\x}(G) \iso \Qbb_2(\zeta_{2^m-1})^\x, \qq
    C_{\Sbb_n}(G) \iso \Zbb_2(\zeta_{2^m-1})^\x,
\]
and as $\Qbb_2(\zeta_{2^m-1})/\Qbb_2$ is unramified we have
$C_{\Sbb_n}(G) \iso C_{\Gbb_n(u)}(G)$.

We now show the existence of the desired extension of order $n|G|$
assuming $u \equiv \pm 1 \mod 8$, that is $u \equiv \pm 1 \mod
(\Zbb_2^\x)^2$. Let $t \in \Zbb_2^\x$ be such that
\[
    t^2 =
    \begin{cases}
        u & \text{if } u \equiv 1 \mod 8,\\
        -u & \text{if } u \equiv -1 \mod 8.
    \end{cases}
\]
The valuation map gives rise to a short exact sequence
\[
    1 \raa N_{\Sbb_n}(G) \raa N_{\Gbb_n(u)}(G) \overset{v}{\raa} 
    \frac{1}{n}\Zbb/\Zbb \iso \Zbb/n \raa 1.
\]
Let $\xi_u \in C_{\Dbb_n^\x}(T)$ be an element satisfying $\xi_u^m=2u$
and acting on $\zeta_{2^m-1}$ by raising it to its square. Consider the
element $(1+i)tj\xi_u \in \Dbb_n^\x$. It becomes a generator in $\Zbb/n$
as
\[
    v((1+i)tj\xi_u)
    = v(1+i) + v(\xi_u)
    = \frac{1}{2} + \frac{1}{m}
    = \frac{m+2}{n},
\]
where $m+2$ is prime to $n$. Furthermore as $t, \xi_u$ commute with $i,
j$, we have
\begin{align*}
    [(1+i)tj\xi_u]^n\
    &=\ [(1+i)j(1+i)jt^2\xi_u^2]^m\\[1ex]
    &=\ [(1+i)(1-i)j^2t^2\xi_u^2]^m\\[1ex]
    &=\ [-2t^2\xi_u^2]^m \\[1ex]
    &=\
    \begin{cases}
        -(2u)^{m+2} & \text{if } u \equiv 1 \mod 8,\\
        (2u)^{m+2} & \text{if } u \equiv -1 \mod 8,
    \end{cases}
\end{align*}
and it is easy to check that $(1+i)tj\xi_u \in N_{\Dbb_n^\x}(G)$. This
shows the existence of $F$ in the case $u \equiv \pm 1 \mod 8$.

We proceed to the non-existence part of the result for $u \not\equiv 1
\mod 8$. First note that there is a short exact sequence
\[
    1 \raa C_{\Dbb_n^\x}(G) \raa N_{\Dbb_n^\x}(G) \overset{\rho}{\raa}
    Aut(T_{24}) \x Gal(\Qbb_2(\zeta_{2^m-1})/\Qbb_2) \raa 1,
\]
where $|Aut(T_{24})| = 24$ and $Gal(\Qbb_2(\zeta_{2^m-1})/\Qbb_2)$ is
cyclic of order $m$. Indeed, if $x \in N_{\Dbb_n^\x}(G)$, then the
conjugation action by $x$ preserves both $G$ and its $2$-Sylow subgroup
$Q$. Consequently $\Qbb_2(Q)^\x = \Qbb_2(T)^\x$ and $C_{\Dbb_n^\x}(G)$
are preserved as well. As for the surjectivity of $\rho$, we know from
the Skolem-Noether theorem that the restriction of $\rho$ to
$\Qbb_2(T)^\x \subset N_{\Dbb_n^\x}(G)$ is surjective on $Aut(T_{24})$,
while by definition the element $\xi_u \in N_{\Dbb_n^\x}(G)$ maps to a
generator of $Gal(\Qbb_2(\zeta_{2^m-1})/\Qbb_2)$. Now since
\[
    C_{\Dbb_n^\x}(G) = \Qbb_2(\zeta_{2^m-1})^\x
    \qq \text{and} \qq
    v(N_{\Dbb_n^\x}(G)) = \frac{1}{n}\Zbb
\]
as shown in proposition \ref{163}, we know that
\[
    N_{\Dbb_{n}^\x}(G) 
    = \lan C_{\Dbb_n^\x}(G), G, (1+i), \xi_u \ran
    = \lan \Zbb_2[\zeta_{2^m-1}]^\x, T, (1+i), \xi_u \ran.
\]
In the case $u \not\equiv \pm 1 \mod 8$, we claim that there is no $x
\in N_{\Dbb_n^\x}(G)$ such that
\[
    v(x) = \frac{1}{n}
    \qq \text{and} \qq
    x^n \in \lan G, 2u \ran.
\]
Indeed, if such an $x$ existed, there would be a $y \in T$ and a $z \in
\Zbb_2[\zeta_{2^m-1}]^\x$ such that $x^m=(1+i)yz$, in which case
\begin{align*}
    x^{2m}\
    &=\ (1+i)y(1+i)yz^2\\
    &=\ (1+i)^2 (1+i)^{-1} y (1+i) y z^2\\
    &=\ 2i \sigma(y)y z^2
\end{align*}
would belong to $2z^2T$, for $\sigma$ the automorphism of $T$ induced by
the conjugation by $(1+i)^{-1}$. In this case $2z^2 \in \lan G, 2u \ran
\cap \Qbb_2(\zeta_{2^m-1})^\x$, and there would be a $g \in G$ with
\[
    2z^2=g(2u)
    \qq \Lra \qq
    z^2=gu.
\]
Since both $z^2$ and $u$ are in $\Zbb_2(\zeta_{2^m-1})^\x$, so does $g$
and $z^2 = \pm u$. As shown in corollary \ref{115}, this is impossible
since $m$ is odd and $u \not\equiv \pm 1 \mod (\Zbb_2^\x)^2$. It follows
that $G$ cannot be extended as a subgroup of order $n|G|$ in $\Gbb_n(u)$
when $u \not\equiv \pm 1 \mod 8$. In fact, the argument also shows the
corresponding result for $G'$: since $\Qbb_2(Q_8) = \Qbb_2(T_{24})$, we
have 
\[
    \Qbb_2(G') 
    = \Qbb_2(G)
    \qq \text{and} \qq
    N_{\Dbb_n^\x}(G')
    = N_{\Dbb_n^\x}(G),
\]
and there is no $x$ of valuation $\frac{1}{n}$ in $N_{\Dbb_n^\x}(G')$
such that $x^n \in \lan G', 2u \ran \subset \lan G, 2u \ran$.

It remains to verify the statement on uniqueness when $u \equiv \pm 1
\mod 8$. For a finite group $F$ of order $n|G|$ extending $G$, we have
$F \in N_{\Dbb_n^\x}(G)$. Let
\[
    A := F \cap Ker(\rho)
    = \lan 2u, -1, \zeta_{2^m-1} \ran
    \qq \text{and} \qq
    B := F/A
\]
Applying theorem \ref{206} to the case $F \in
\Gcal_\rho(N_{\Dbb_n^\x}(G), A, B)$, it is enough to check that the
cohomology group $H^1(B, Ker(\rho)/A)$ is trivial. As 
\[
    |B| < \infty, \qq
    Ker(\rho)/A 
    = \Qbb_2(\zeta_{2^m-1})^\x/A 
    \iso \Zbb_2^m,
\]
and because the $B$-module structure is trivial, we obtain
\[
    H^1(B, Ker(\rho)/A)
    \iso Hom(B, \Zbb_2^m)
    = 0.
\]
\end{proof}

\section{Example of the case $n=2$} 

In this section, we illustrate the situation for $n=2$ and we find the
finite subgroups of $\Gbb_2(u)$ up to conjugation for $p \in \{2, 3\}$,
that is for those primes $p$ for which $p-1$ divides $n$. 

For a given $p$, we let $\omega \in \Sbb_2$ be a primitive
$(p^2\!-\!1)$-th root of unity and $\sigma$ be the Frobenius automorphism
of $\Qbb_p(\omega)/\Qbb_p$. For each $u \in \Zbb_p^\x$, we let $\xi_u
\in \Dbb_2^\x$ be an element associated to $\sigma$ such that $\xi_u^2 =
pu$. As in example \ref{251} and \ref{263} the multiplicative subgroups
in the division algebra
\[
    \Dbb_2
    \iso \Qbb_p(\omega)\lan \xi_u \ran /
    (\xi_u^2 = pu, \xi_u x = x^{\sigma} \xi_u),
    \qq x \in \Qbb_p(\omega),
\]
which correspond to finite subgroups of $\Gbb_2(u)$ are easily
expressible in terms of $\xi_u$ and $\omega$. This allows to determine
the conjugacy classes of those finite subgroups explicitly.

\subsection*{The case $p=3$}

Let $p=3$. Here $k=1$, $m=1$ and $\alpha \in \{0, 1\}$. \medskip

1) If $\alpha=0$, then $F_0 = \lan \omega \ran \iso C_8$ and
$\tilde{F_0} = F_0 \x \lan 3u \ran$. As shown in example \ref{251}
\[
    \tilde{F_0}=\tilde{F_1}=\tilde{F_2}, \qq 
    \tilde{F_3} = \lan \tilde{F_0}, \xi_u \ran    
    \q \text{with} \q \xi_u^2 = 3u,
\]
and for $\bar{\xi}_u$ the class of $\xi_u$ in $\Gbb_n(u)$ the group
\[
    F_3
    = \lan \omega, \bar{\xi}_u \ran 
    \iso\ SD_{16} \label{321}
\]
is a semidihedral group of order $16$.

\medskip

2) If $\alpha=1$, then $F_0 = \lan \zeta_3 \ran \x \lan \omega^4 \ran
\iso C_6$ where $\omega^4=-1$. The primitive third root of unity
$\zeta_3 \in \Sbb_2$ may be given by
\[
    \zeta_3 = -\frac{1}{2}(1+\omega S)
    \qq \text{for} \q S = \xi_1.
\]
In this case
\[
    \zeta_3^{-1} = \zeta_3^2 = -\frac{1}{2}(1-\omega S)
    \qq \text{and} \qq
    \zeta_3^2-\zeta_3 = \omega S.
\]
According to theorem \ref{250} there is no restriction on $u$, and $x_1$
can be chosen as $x_1 = (\zeta_3^2-\zeta_3)t$ with $t\in \Zbb_3^\x$ such
that
\[
    t^2 =
    \begin{cases}
        u & \text{if } u \equiv 1 \mod 3,\\
        -u & \text{if } u \equiv -1 \mod 3.
    \end{cases}
\]
Indeed,
\[
    x_1^2 
    = (\omega S)^2 t^2
    = \omega^4 S^2 t^2
    = -3 t^2,
\]
so that $v(x_1) = \frac{1}{2}$, and we have
\begin{align*}
    x_1 \zeta_3 x_1^{-1}
    &= - \frac{1}{2} (1+(\omega S)^2(\omega S)^{-1})
    = \zeta_3,\\[1ex]
    x_1 \omega^2 x_1^{-1}
    &= \omega S \omega^2 S^{-1} \omega^{-1}
    = (\omega^2)^3.
\end{align*}
Hence $\tilde{F_1} = \tilde{F_2} = \lan x_1 \ran \x \lan \zeta_3 \ran
\x \lan \omega^4 \ran$, where
\[
    x_1^2 =
    \begin{cases}
        -3u & \text{if } u \equiv 1 \mod 3,\\
        3u & \text{if } u \equiv -1 \mod 3,
    \end{cases}
    \qq \text{and} \qq
    \bar{x}_1^2 =
    \begin{cases}
        -1 & \text{if } u \equiv 1 \mod 3,\\
        1 & \text{if } u \equiv -1 \mod 3,
    \end{cases}
\]
for $\bar{x}_1$ the class of $x_1$ in $\Gbb_2(u)$.  Furthermore
$\omega^2 \in N_{\Dbb_2^\x}(\tilde{F_1})$ given that $\omega^2 \zeta_3
\omega^{-2} = \zeta_3^2$, and
\[
    \omega^2 x_1 \omega^{-2}
    = \omega^2 \zeta_3(1-\zeta_3) \omega^{-2}
    = \zeta_3^2(1-\zeta_3^2)
    = (\zeta_3 + \zeta_3^2)(\zeta_3 - \zeta_3^2)
    = -x_1.
\]
Thus $\tilde{F_3} = \lan \tilde{F_1}, \omega^2 \ran$ and $F_3 = \lan
\bar{x}_1, \zeta_3, \omega^2 \ran$ is a maximal finite subgroup of order
$24$ in $\Gbb_2(u)$.

\medskip

We let
\[
    D_8 
    \iso \lan a, b\ |\ a^4=b^2=1, bab^{-1} = a^{-1} \ran
    \label{320} 
\]
denote the dihedral group of order $8$.

\begin{theorem} \label{261}
Let $n=2$, $p=3$ and $u \in \Zbb_p^\x$. The conjugacy classes of maximal
finite subgroups $F$ of $\Gbb_2(u)$ are represented by
\[
    SD_{16}
    \qq \text{and} \qq
    \begin{cases}
        C_3 \rtimes Q_8 & \text{if } u \equiv 1 \mod 3,\\
        C_3 \rtimes D_8 & \text{if } u \equiv -1 \mod 3.
    \end{cases}
\]
\end{theorem}

\begin{proof}
We first consider the cases where $F_0$ is such that $[\Qbb_3(F_0) :
\Qbb_3]=2$; by theorem \ref{332} we may assume that $F_0$ is maximal.
The first class originates from the case $\alpha=0$; its existence and
uniqueness follow from example \ref{251} and theorem \ref{250}. 

Suppose then that $\alpha=1$. If $u \equiv 1 \mod 3$, the $2$-Sylow
subgroup $\lan \omega^2, \bar{x}_1 \ran$ of $F_3$ is isomorphic to
$Q_8$. As the latter does not contain a subgroup isomorphic to $C_2 \x
C_2$, the short exact sequence
\[
    1 \raa F_2 = \lan \zeta_3, \bar{x}_1 \ran
    \raa F_3 \raa C_2 \raa 1
\]
does not split. However, $\lan \zeta_3 \ran$ being normal in $F_3$, we
obtain $F_3 \iso C_3 \rtimes Q_8$. On the other hand if $u \equiv -1
\mod 3$, the group $F_3$ contains a subgroup isomorphic to $C_2 \x C_2$.
In this case we have a split extension
\[
    1 \raa F_2 = \lan \zeta_3, -1, \bar{x}_1 \ran
    \raa F_3 \raa C_2 \raa 1
\]
with a $2$-Sylow subgroup isomorphic to $D_8 \iso \lan \omega^2,
\bar{x}_1\ |\ (\omega^2)^4=\bar{x}_1^2=1, \bar{x}_1 \omega^2
\bar{x}_1^{-1} = \omega^{-2} \ran$, and $F_3 \iso C_3 \rtimes D_8$.
Uniqueness of the class of $F_3$ in $\Gbb_2(u)$ follows from theorem
\ref{250}.

It remains to consider the case where $F_0 = \{\pm 1\} \iso C_2$, that
is, $F_0$ is maximal such that $\Qbb_3(F_0) = \Qbb_3$. Then obviously
$\tilde{F_0} = \tilde{F_1}$. Because $\Qbb_3^\x/(\Qbb_3^\x)^2 \iso
\Zbb/2\Zbb \x \{\pm 1\}$ is represented by the elements of the set
$\{\pm 1, \pm 3\}$, we know that there are three possible quadratic
extensions of $\Qbb_3$ given by
\[
    L_v := \Qbb_3/(X^2-v)
    \qq \text{for } v \in \{-1, \pm 3\};
\]
each of them is unique up to conjugation. Among these $L_{-1} =
\Qbb_3(\zeta_8)$ and $L_{-3} = \Qbb_3(\zeta_3)$ have already been
considered. 

Hence suppose $v=3$ and let $x_2 := Xt$ with $t \in \Zbb_3^\x$ such that
\[
    t^2
    =
    \begin{cases}
        u & \text{if } u \equiv 1 \mod 3,\\
        -u & \text{if } u \equiv -1 \mod 3.
    \end{cases}
\]
Then
\[
    x_2^2
    = 3t^2
    =
    \begin{cases}
        3u \equiv 1 \mod \lan 3u \ran & \text{if } u \equiv 1 \mod 3,\\
        -3u \equiv -1 \mod \lan 3u \ran & \text{if } u \equiv -1 \mod 3,
    \end{cases}
\]
and we have an extension
\[
    1 \raa \tilde{F_1} = \lan 2u, \pm 1 \ran
    \raa \tilde{F_2} = \lan x_2, \pm 1 \ran
    \raa C_2 \raa 1,
\]
where
\[
    F_2
    \iso
    \begin{cases}
        C_2 \x C_2 & \text{if } u \equiv 1 \mod 3,\\
        C_4 & \text{if } u \equiv -1 \mod 3.
    \end{cases}
\]
By corollary \ref{188}, this group is unique up to conjugation. Because
the group $Aut(F_0)$ is trivial, proposition \ref{329} implies $F_3 =
F_2$.  This class however is neither new nor maximal. Indeed, for the
group $\lan \omega, \xi_u \ran \subset \Dbb_2^\x$ whose corresponding
group $\lan \omega, \bar{\xi}_u \ran$ in $\Gbb_n(u)$ represents the
class $SD_{16}$ found above, one can take
\[
    x_2 =
    \begin{cases}
        \xi_u & \text{if } u \equiv 1 \mod 3,\\
        \omega\xi_u & \text{if } u \equiv -1 \mod 3,
    \end{cases}
\]
in order to see that $F_2 \subset SD_{16}$.
\end{proof}

\subsection*{The case $p=2$}

Let $p=2$. Here $k=2$, $m=1$ and $\alpha \in \{1, 2\}$. \medskip

1) If $\alpha=1$, then $F_0 = \lan -\omega \ran \iso C_6$ and
$\tilde{F_0} = F_0 \x \lan 2u \ran$. As shown in example \ref{263}
\[
    \tilde{F_0}=\tilde{F_1}=\tilde{F_2}, \qq
    \tilde{F_3^{\pm}} 
    = \lan \tilde{F_0}, \xi_{\pm u} \ran
    \q \text{with } \xi_{\pm u}^2 = \pm 2u,
\]
and we have
\[
    F_3^+ 
    = \lan -\omega, \bar{\xi}_u\ \ran
    \iso C_6 \rtimes C_2,
    \qq
    F_3^- 
    = \lan -\omega, \bar{\xi}_{-u}\ \ran
    \iso C_3 \rtimes C_4,
\]
for $\bar{\xi}_{\pm u}$ the class of $\xi_{\pm u}$ in $\Gbb_n(u)$.

\medskip

2) If $\alpha=2$, then $F_0 = C_4 \subset T_{24}$ with $C_4 = \lan i
\ran$ and $T_{24} = \lan i, j \ran \rtimes \lan \zeta_3 \ran$.
According to theorem \ref{259} and \ref{260}, a finite maximal
extension of $F_0$ in $\Gbb_2(u)$ is an extension of $T_{24}$ if and
only if $u \equiv \pm 1 \mod 8$. Let
\[
    x_1 =
    \begin{cases}
        (1+i)t \text{ with } t^2=u & \text{if } u \equiv 1 \mod 8,\\
        (1+i)t \text{ with } t^2=-u & \text{if } u \equiv -1 \mod 8,\\
        2u & \text{if } u \equiv \pm 3 \mod 8.
    \end{cases}
\]
Then we know that $\tilde{F_1} = \tilde{F_2} = \lan x_1 \ran \x F_0$. In
case $u \equiv \pm 1 \mod 8$, we have $\tilde{F_3} = \tilde{F_2}$ and we
find $x_1^2 = 2it^2$, $x_1^4 = -4u^2$ and $x_1^8 = (2u)^4$, so that the
group $F_3$ is cyclic of order $8$; it is unique up to conjugation by
corollary \ref{183}.

\medskip

We let
\[
    O_{48} 
    \iso \lan a, b, c\ |\ a^2 = b^3 = c^4 = abc \ran
    \label{322}
\]
denote the binary octahedral group of order $48$.

\begin{theorem} \label{264}
Let $n=2$, $p=2$ and $u \in \Zbb_2^\x$. The conjugacy classes of maximal
finite subgroups $F$ of $\Gbb_2(u)$ are represented by
\[
    \begin{cases}
        C_{6} \rtimes C_2,\ O_{48} 
          & \text{if } u \equiv 1 \mod 8,\\
        C_{3} \rtimes C_4,\ T_{24} \rtimes C_2 
          & \text{if } u \equiv -1 \mod 8,\\
        C_{3} \rtimes C_4,\ C_{6} \rtimes C_2,\ 
        D_8  \text{ and } T_{24}
          & \text{if } u \equiv 3 \mod 8,\\
        C_{3} \rtimes C_4,\ C_{6} \rtimes C_2,\ 
        Q_8 \text{ and } T_{24}
          & \text{if } u \equiv -3 \mod 8.
    \end{cases}
\]
\end{theorem}

\begin{proof}
We first consider the cases where $F_0$ is such that $[\Qbb_2(F_0) :
\Qbb_2]=2$; by theorem \ref{332} we may assume that $F_0$ is maximal.
The classes $C_{6} \rtimes C_2$ and $C_{3} \rtimes C_4$ originate from
the case $\alpha = 1$. They are respectively represented by
\[
    F_3^+ 
    = \lan -\omega, \bar{\xi}_u\ \ran
    \qq \text{and} \qq
    F_3^- 
    = \lan -\omega, \bar{\xi}_{-u}\ \ran.
\]
Their existence and uniqueness follow from example \ref{263} and theorem
\ref{259}. We will now analyse the case where $F_0 = \lan i \ran \iso
C_4$.

Suppose that $u \equiv \pm 1 \mod 8$. Then
\[
    x_1^2 = (1+i)^2t^2 = 2it^2
    \equiv
    \begin{cases}
        i \mod \lan 2u \ran & \text{if } u \equiv 1 \mod 8,\\
        -i \mod \lan 2u \ran & \text{if } u \equiv -1 \mod 8,
    \end{cases}
    \qq
    x_1^4 \equiv -1 \mod \lan 2u \ran,
\]
and
\[
    x_1 i x_1^{-1} = i, \qq
    x_1 j x_1^{-1}
    = (1+i)j\frac{(1-i)}{2}
    = \frac{(1+i)^2}{2}j
    = ij
    = k.
\]
Therefore, we have a chain of subgroups
\[
    \tilde{F_0} = \lan i, 2u \ran
    \subsetneq \tilde{F_1} = \tilde{F_2} = \lan i, x_1 \ran
    \subsetneq \tilde{F_3} = \lan i, j, x_1 \ran,
\]
where $\tilde{F_i}$ is normal in $\tilde{F_{i+1}}$ for $1 \leq i \leq
3$, and where $|\tilde{F_1}/\tilde{F_0}|=|\tilde{F_3}/\tilde{F_2}|=2$.
Because $x_1^2 \equiv \pm i \mod \lan 2u \ran$ and $x_1^4 \equiv -1 \mod
\lan 2u \ran$, we know that for $\bar{x}_1$ the class of $x_1$ in
$\Gbb_n(u)$ we have $F_1 \iso C_8$ and there is an extension
\[
    1 \raa F_1 = \lan \bar{x}_1 \ran
    \raa F_3 = \lan \bar{x}_1, j \ran
    \raa C_2 \raa 1,
\]
where $j\bar{x}_1 \in F_3$ maps non-trivially to the quotient group. As
\begin{align*}
    (j x_1)^2
    &= j (x_1 j x_1^{-1}) x_1^2
    = j(ij)(2it^2)
    = -2t^2 \\[1ex]
    &\equiv
    \begin{cases}
        -1 \mod \lan 2u \ran & \text{if } u \equiv 1 \mod 8,\\
        1 \mod \lan 2u \ran & \text{if } u \equiv -1 \mod 8,
    \end{cases}
\end{align*}
and since
\begin{align*}
    (j x_1) x_1 (j x_1)^{-1}
    &= j x_1 j^{-1}
    = -(j x_1)^2 x_1^{-1}
    = 2t^2 x_1^{-1} \\[1ex]
    &\equiv
    \begin{cases}
        x_1^{-1} \mod \lan 2u \ran & \text{if } u \equiv 1 \mod 8,\\
        -x_1^{-1} \mod \lan 2u \ran & \text{if } u \equiv -1 \mod 8,
    \end{cases}
\end{align*}
we find
\[
    F_3 \iso
    \begin{cases}
        Q_{16} & \text{if } u \equiv 1 \mod 8,\\
        C_8 \rtimes C_2 = SD_{16} & \text{if } u \equiv -1 \mod 8.
    \end{cases}
\]
Clearly, $F_3$ is a $2$-Sylow subgroup of $F := \lan F_3, \omega \ran$
and $T_{24} = \lan i, j, \omega \ran \subset F$. As seen above, $x_1$
and $j x_1$ both belong to $N_{\Dbb_n^\x}(\lan i, j \ran) =
N_{\Dbb_n^\x}(\lan i, j, \omega \ran)$, and there is an extension
\[
    1 \raa T_{24} = \lan i, j, \omega \ran
    \raa F \raa C_2 \raa 1,
\]
where $x_1, j x_1 \in F$ are mapped non-trivially to the quotient group.

Assume for the moment that $u \equiv 1 \mod 8$. We let $a := \bar{x}_1$,
so that $a^2 = i$, and we consider the element of order $6$
\[
    b := \frac{1}{2}(1+i+j+k)\ \in T_{24} \subset F.
\]
Then we can take $\omega = -b$ and we easily check that
\[
    b^{-1}
    = -\omega^2
    = \frac{1}{2}(1-i-j-k), \qq
    b^{-1}a^2b
    = j.
\]
In particular $F = \lan a, b \ran$ is generated by the elements $a$ and
$b$ of respective order $8$ and $6$. These two elements interact via $ba
= -a^{-1}b^{-1}$ since
\[
    (bx_1)^2
    = \frac{1}{4}(2i+2j)^2 u
    = (i+j)^2 u
    = -2u
    \equiv -1 \q \mod \lan 2u \ran.
\]
Letting $c:=ba$, it follows that
\[
    F
    = \lan a, b\ |\ (ba)^2 = b^3 = a^4 = -1 \ran
    = \lan a, b, c\ |\ c^2 = b^3 = a^4 = cba \ran
\]
is isomorphic to the binary octahedral group $O_{48}$. Uniqueness of $F$
up to conjugation is given by theorem \ref{260}; its class is clearly
maximal. In fact since 
\begin{align*}
    (jx_1) \omega (jx_1)^{-1}\
    &=\ -\frac{1}{2} jx_1(1+i+j+k)x_1^{-1}j^{-1}\\
    &=\ -\frac{1}{2} j(1+i+k-j)j^{-1}\\
    &=\ -\frac{1}{2} (1-i-k-j)\\
    &=\ \omega^2,
\end{align*}
we may take $\xi_{-u} = jx_1$ in order to find that $F$ contains $F_3^-
= \lan b, j \bar{x}_1 \ran$. On the other hand, $F$ does not have a
subgroup isomorphic to $F_3^+$ since its $2$-Sylow subgroup $F_3 \iso
Q_{16}$ has no subgroup isomorphic to $C_2 \x C_2$. The class of $F_3^+$
is therefore maximal when $[\Qbb_2(F_0):\Qbb_2]=2$ and $u \equiv 1 \mod
8$.

Now assume $u \equiv -1 \mod 8$. Then $(j x_1)^2 \equiv 1 \mod \lan 2u
\ran$, in which case
\[
    F 
    = \lan \bar{x}_1, j, \omega \ran
    \iso T_{24} \rtimes C_2.
\]
The above calculations show that we may take $\xi_{u} = j x_1$ in order
to find that $F_3^+ = \lan b, j\bar{x}_1 \ran$ is a subgroup of $F$. On
the other hand one easily verifies that the group $\lan x_1, i, j \ran$
does not have an element of valuation $\frac{1}{2}$ which has order $4$
modulo $\lan 2u \ran$. This means that the class of $F$ does not contain
that of $F_3^-$. The latter is therefore maximal when
$[\Qbb_2(F_0):\Qbb_2]=2$ and $u \equiv -1 \mod 8$.

We now suppose $u \equiv \pm 3 \mod 8$. By theorem \ref{032}, a maximal
finite subgroup $F$ of $\Gbb_2(u)$ containing $F_0 = \lan i \ran \iso
C_4$ satisfies $C_4 \subset F \cap \Sbb_2 \subset T_{24}$. If $F \subset
\Sbb_2$, then $F \iso T_{24}$ contains the subgroup $F \cap S_n \iso
Q_8$ as in lemma \ref{328}.b. Otherwise if $F \not\subset \Sbb_2$, the
$2$-Sylow subgroup of $F \cap \Sbb_2$ must be $C_4$ by theorem
\ref{260}, and we have a chain of subgroups
\[
    \tilde{F_0}
    = \tilde{F_1}
    = \tilde{F_2}
    \subset \tilde{F_3}
    = \tilde{F},
\]
where $\tilde{F_3}/\tilde{F_0}$ is a cyclic group of order at most $2$.
We are thus looking for an element $x_3 \in \Dbb_2^\x$ such that $x_3^2
\in \tilde{F_0} = F_0 \x \lan 2u \ran$. By the Skolem-Noether theorem,
there is a short exact sequence
\[
        1 \raa
        C_{\Dbb_2^\x}(F_0) = \Qbb_2(i)^\x \raa
        N_{\Dbb_2^\x}(F_0) = \lan \Qbb_2(i)^\x, j \ran \raa
        C_2 \raa 1,
\]
where $j$ is mapped non-trivially to the quotient group. Hence $x_3$ is
of the form $x_3 = j^\epsilon z$ for $\epsilon \in \{ \pm 1 \}$ and $z
\in \Qbb_2(i)^\x$. We have
\[
    x_3^2
    = j^\epsilon z j^\epsilon z
    = -(j^\epsilon z j^{-\epsilon})z
    = - N(z)
\]
for $N: \Qbb_2(i)^\x \ra \Qbb_2^\x$ the norm of the extension
$\Qbb_2(i)/\Qbb_2$. In the proof of corollary \ref{242} we have shown
that $N(\Qbb_2(i)^\x) = \lan 2 \ran \x U_2(\Zbb_2(i)^\x)$. Since 
\[
    N(2+i)
    = (2+i)(2-i)
    = 5
    \equiv -3 \q \mod 8,
\]
we have $-6 \in N(\Qbb_2(i)^\x)$. We may therefore choose $z$ such that
\[
    x_3^2 =
    \begin{cases}
        2u & \text{if } u \equiv 3 \mod 8,\\
        -2u & \text{if } u \equiv -3 \mod 8.
    \end{cases}
\]
In this case $\tilde{F_3} = \lan i, x_3 \ran$, and for $\bar{x}_3$ the
class of $x_3$ in $\Gbb_2(u)$ we get
\[
    F_3 =
    \begin{cases}
        \lan i, \bar{x}_3\ |\ i^4=1,\ 
        \bar{x}_3 i \bar{x}_3^{-1} = i^{-1},\ 
        \bar{x}_3^2 = 1 \ran
        \iso D_8
        & \text{if } u \equiv 3 \mod 8,\\[1ex]
        \lan i, \bar{x}_3\ |\ i^4=1,\ 
        \bar{x}_3 i \bar{x}_3^{-1} = i^{-1},\ 
        \bar{x}_3^2 = -1 \ran
        \iso Q_8
        & \text{if } u \equiv -3 \mod 8,
    \end{cases}
\]
as a maximal finite subgroup of $\Gbb_2(u)$.  Since $v(x_3) =
\frac{1}{2}$, the conjugacy classes of $F_3$ and $T_{24} \cap S_2 \iso
Q_8$ must be distinct (although they are isomorphic if $u \equiv -3 \mod
8$). By theorem \ref{259}, $F_3$ and $T_{24}$ represent the only two
maximal classes containing $\lan i \ran$ when $u \equiv \pm 3 \mod 8$.
The maximality of $F_3^+$ and $F_3^-$ in this case is obvious.

It remains to consider the cases where $F_0 = \{\pm 1\} \iso C_2$, that
is, $F_0$ is maximal such that $\Qbb_2(F_0) = \Qbb_2$. Then obviously
$\tilde{F_0} = \tilde{F_1} = \lan 2u, \pm 1 \ran \iso \Zbb \x C_2$.
Because $\Qbb_2^\x/(\Qbb_2^\x)^2 \iso \Zbb/2\Zbb \x
\Zbb_2^\x/U_3(\Zbb_2^\x)$ is represented by the elements of the set
$\{\pm 1, \pm 2, \pm 3, \pm 6 \}$, we know that there are seven possible
quadratic extensions of $\Qbb_2$ given by
\[
    L_v := \Qbb_2/(X^2-v)
    \qq \text{for } v \in \{-1, \pm 2, \pm 3, \pm 6\};
\]
each of them is unique up to conjugation. Among these $L_{-1} =
\Qbb_2(\zeta_4)$ and $L_{-3} = \Qbb_2(\zeta_3)$ have already been
considered. Furthermore if $v=3$, and if $a, b \in \Qbb_2$, the element
\[
    (a+bX)^2 = a^2+3b^2+2abX
\]
cannot belong to $\tilde{F_1} = \lan 2u, \pm 1 \ran$ and the later can
never be extended non-trivially to some $\tilde{F_2}$. 

Let us then consider the cases where $v \in \{\pm 2, \pm 6\}$. If $u
\equiv \pm \frac{v}{2} \mod 8$, we let $x_2 := Xt$ with $t \in
\Zbb_2^\x$ such that
\[
    t^2
    =
    \begin{cases}
        \frac{2u}{v} & \text{if } u \equiv \frac{v}{2} \mod 8,\\
        -\frac{2u}{v} & \text{if } u \equiv -\frac{v}{2} \mod 8.
    \end{cases}
\]
Then
\[
    x_2^2
    = vt^2
    =
    \begin{cases}
        2u \equiv 1 \mod \lan 2u \ran 
        & \text{if } u \equiv \frac{v}{2} \mod 8,\\
        -2u \equiv -1 \mod \lan 2u \ran
        & \text{if } u \equiv -\frac{v}{2} \mod 8,
    \end{cases}
\]
and for $\bar{x}_2$ the class of $x_2$ in $\Gbb_2(u)$ we have
\[
    F_2
    = \lan \bar{x}_2, \pm 1 \ran
    \iso
    \begin{cases}
        C_2 \x C_2 
        & \text{if } u \equiv \frac{v}{2} \mod 8,\\
        C_4 
        & \text{if } u \equiv -\frac{v}{2} \mod 8.
    \end{cases}
\]
These classes however are not new: in the case $[\Qbb_2(F_0):\Qbb_2]=2$
and $\alpha = 1$ treated above, considering the situation where
\[
    x_2
    =
    \begin{cases}
        \xi_u 
        & \text{if } u \equiv 1 \text{ or } -3 \mod 8,\\
        \xi_{-u}
        & \text{if } u \equiv -1 \text{ or } 3 \mod 8,
    \end{cases}
\]
we see that $C_2 \x C_2 \subset F_3^+$ and $C_4 \subset F_3^-$. We also
know from corollary \ref{188} that the group $F_2$ is unique up to
conjugation. On the other hand if $u \not\equiv \pm \frac{v}{2} \mod 8$,
that is if $v \not\equiv \pm 2u \mod 8$, there is no $x \in L_v$ such
that $x^2 \in \lan 2u \ran \mod \{\pm 1\}$ and we have $\tilde{F_2} =
\tilde{F_1} = \tilde{F_0}$. 

Finally, because $Aut(F_0)$ is trivial independently of the value of
$u$, it follows from proposition \ref{329} that $F_3 = F_2$.  
\end{proof}

\begin{remark} \label{265}
For $\alpha=2$, we have shown
\[
    F_3 \iso
    \begin{cases}
        Q_{16} & \text{if } u \equiv 1 \mod 8,\\
        SD_{16} & \text{if } u \equiv -1 \mod 8,\\
        D_8 & \text{if } u \equiv 3 \mod 8,\\
        Q_8 & \text{if } u \equiv -3 \mod 8.
    \end{cases}
\]
When $u \equiv \pm 3$, the second conjugacy class obtained in theorem
\ref{259}.3 is not maximal as a finite subgroup of $\Gbb_2(u)$. It is
contained in $T_{24}$ and is represented by $T_{24} \cap S_2 = Q_8$. It
comes from the existence of an element $j$ of valuation zero in
$\Dbb_2^\x$ which induces the action of $Gal(\Qbb_2(i)/\Qbb_2)$ on $F_0
= \lan i \ran$ given by $i \mto -i$.
\end{remark}


\appendix\renewcommand{\chaptername}{Appendix}

\chapter{Simple algebras} 
\label{333}

We provide here the essential background and some classic results on
finite dimensional simple algebras. An overview of the subject can be
found in \cite{platonov}.

\vspace{2ex}

\begin{definition}
Let $A$ be an associative ring with unit.
\begin{itemize}
    \item $A$ is called \emph{simple} if the only two sided ideals of
    $A$ are $A$ itself and the zero ideal.

    \item $A$ is a \emph{skew field} if for every non-zero element $a$
    of $A$ there is an element $a^{-1} \in A$ satisfying
    \[
        aa^{-1} = 1 = a^{-1}a.
    \]
\end{itemize}
\end{definition}

Clearly, a commutative skew field is a field, and the set of non-zero
elements $A^\x$ of a skew field $A$ forms a group under multiplication.
On the other hand, the center $Z(A)$ of a simple ring $A$ is a field, as
for any non-zero element $a$ in $Z(A)$ the two sided ideal $aA$ is $A$
by simplicity, and its inverse $a^{-1}$ exists in $Z(A)$. In particular,
a simple ring $A$ is an algebra over any subfield $K$ of $Z(A)$.

\begin{definition}
A finite dimensional simple algebra $A$ over a field $K$ which is also a
skew field is a \emph{division algebra} over $K$. When $K = Z(A)$, the
division algebra $A$ is said to be \emph{central} and is also referred
to as an \emph{Azumaya algebra}.
\end{definition}

\begin{example} \label{001}
The algebra $M_n(K)$ of all $n \x n$ matrices over a field $K$ is a
simple algebra. To see this consider the canonical basis $\{e_{ij}\}$ of
$M_n(K)$, where $e_{ij}$ denotes the matrix having zero coefficients
everywhere except $1$ for the entry on the $i$-th row and $j$-th column.
We need to show that given a non-zero two-sided ideal $I$ of $M_n(K)$,
every $e_{ij}$ belongs to $I$.  Since
\[
    e_{ij}e_{kl} =
    \begin{cases}
    e_{il} & \text{if } j = k, \\
    0 & \text{if } j \neq k,
    \end{cases}
\]
we only have to show that $I$ contains at least on of the $e_{ij}$.  Let
\[
    a = \sum_{i,j=1}^n a_{ij}e_{ij}\ \in I
\]
be an element of $I$ with $a_{ij} \in K$ and $a_{kl} \neq 0$ for some $1
\leq k,l \leq n$. Then
\[
    a_{kl}e_{kl} = e_{kk}ae_{ll} \in I
\]
and $e_{kl} \in I$ as desired. It is clear however that when $n \geq 2$,
$M_n(K)$ is not a division algebra.
\end{example}

\begin{example} \label{002}
When $K$ is an algebraically closed field, there is no $K$-division
algebra other than $K$ itself, for if $A$ is such an algebra we must
have $K(a) = K$ for every element $a$ in $A$. 
\end{example}

\begin{proposition} \label{006}
If $A$ is a division algebra over a field $K$, then any $K$-subalgebra
$B$ of $A$ is itself a division algebra.
\end{proposition}

\begin{proof}
For any non-zero element $x \in B$, we must show that $x^{-1} \in A$ is
an element of $B$. Since $B$ is of finite dimension over $K$, the
elements of the sequence $1, x, x^2, \ldots$ are linearly dependent via
a polynomial in $B$ we can assume to be unitary and with a non-zero
constant term; in other words
\[
    x^m + b_{m-1}x^{m-1} + \ldots + b_1x + b_0 
    = 0 \q 
    \text{with } b_i \in B \text{ and } b_0 \neq 0.
\]
Hence
\[
    x(x^{m-1} + b_{n-1}x^{n-2} + \ldots + b_1) 
    = -b_0,
\]
and therefore
\[
    x^{-1} 
    = -b^{-1}_0 (x^{m-1} + b_{n-1}x^{n-2} + \ldots + b_1)\ \in B
\]
as desired.
\end{proof}

The following classic result reduces the study of finite dimensional
simple algebras to the particular case of division algebras. A proof can
be found in \cite{kersten} theorem 2.5 or \cite{reiner} section 7a.

\begin{theorem}[Wedderburn] \label{003}
A finite dimensional simple algebra $A$ over a field $K$ is isomorphic
as a $K$-algebra to $M_n(D)$ for $D$ a $K$-division algebra. The integer
$n$ is unique and $D$ is unique up to isomorphism.
\end{theorem}

\begin{corollary} \label{004}
The dimension of a central simple algebra is a square.
\end{corollary}

\begin{proof}
If $A$ is a central simple algebra of dimension $[A:K]$ over a field
$K$ and if $\bar{K}$ denotes the algebraic closure of the latter, we
obtain a central simple algebra $A \ox_K \bar{K}$ of the same dimension
\[
    [A \ox_K \bar{K}: \bar{K}] = [A:K].
\]
By Wedderburn's theorem $A \ox_K \bar{K}$ is $\bar{K}$-isomorphic to
$M_n(D)$ for $D$ a central division algebra over $\bar{K}$. Because
$\bar{K}$ is algebraically closed, we have $D = \bar{K}$ by example
\ref{002}. This implies that $A \ox_K \bar{K}$ has dimension $n^2$ over
$\bar{K}$.
\end{proof}

From the Wedderburn theorem, we know that if $A$ is a central simple
algebra of dimension $n^2$ over $K$, then $A \iso M_r(D)$ for $D$ an
Azumaya algebra over $K$, and there is an integer $m$ with
\[
    n^2 = [A:K] = r^2[D:K] = r^2 m^2.
\]
The skewfield $D$ is called the \emph{skewfield part} of $A$, the
integer $\text{deg}(A) = n$ \label{335} is the \emph{degree} of $A$ and
$\text{ind}(A) = m$ \label{336} is its \emph{index}. 

\medskip

Another classic result we use in the text is the following. For an
algebra $A$ and a subalgebra $B$ of $A$, we denote by
\[
    C_A(B) = \{ a \in A\ |\ ab = ba \text{ for any } b \in B \}
\]
the centralizer of $B$ in $A$, and we denote by $B^{op}$ the opposite
ring of $B$. As shown in \cite{kersten} theorem 8.4, we have:

\begin{theorem}[Centralizer] \label{007}
Let $A$ be a central simple algebra of finite dimension over a field
$K$, and let $B$ be a simple subalgebra of $A$. Then
\begin{enumerate}
    \item there is a $K$-algebra homomorphism $C_A(B) \ox_K M_{[B:K]}(K)
    \iso A \ox_K B^{op}$; 

    \item $C_A(B)$ is a central simple algebra over $Z(B)$;

    \item $C_A(C_A(B)) = B$;

    \item $C_A(B) \ox_{Z(B)} B \iso C_A(Z(B))$ via the map
    \[
        C_A(B) \x B \ra C_A(Z(B))\ :\ (x,b) \mapsto xb.
    \]
\end{enumerate}
In particular if $B$ is central over $K$, then
\[
    Z(B) = K, \q C_A(Z(B)) = A
    \q \text{and} \q
    [A:K] = [B:K][C_A(B):K].
\]
\end{theorem}

\begin{corollary} \label{008}
The degree of a commutative extension $L$ of $K$ contained in a finite
dimensional central simple $K$-algebra $A$ divides $\text{deg}(A)$.
\end{corollary}

\begin{proof}
Because $L \subset C_A(L)$, we have
\[
    [C_A(L):K] = [C_A(L):L][L:K],
\]
and therefore
\[
    [A:K] = [L:K][C_A(L):K] = [L:K]^2 [C_A(L):L].
\]
\end{proof}

Thus the problem of describing subfields of finite dimensional central
simple algebras is reduced to the problem of describing their maximal
subfields, that is, those subfields of $A$ containing $K$ that are not
properly contained in a subfield of $A$. Because $A$ is assumed to be of
finite dimension, maximal subfields always exist in $A$.

\begin{proposition} \label{009}
If $L$ is a maximal subfield of a finite dimensional central simple
$K$-algebra $A$, then $C_A(L) \iso M_n(L)$. In particular, if $A$ is an
Azumaya algebra, then
\[
    C_A(L) = L
    \qq \text{and} \qq
    [L:K] = [A:K]^{\frac{1}{2}} = \text{ind}(A).
\]
\end{proposition}

\begin{proof}
According to the Wedderburn theorem, if the first assertion was not true
we would have $C_A(L) \iso M_n(D)$ for a noncommutative division algebra
$D$ over $L$.  This division algebra would then contain a subfield
properly containing $L$, and this would contradict the maximality of $L$
in $A$. Furthermore if $A$ is a skew field, we must have $n=1$, so that
$C_A(L) = L$. By the centralizer theorem,
\[
    [A:K] = [C_A(L):K][L:K] = [L:K]^2,
\]
as desired.
\end{proof}

We end the section by stating one of the most useful results in the
theory of simple algebras. See \cite{reiner} section 7d or
\cite{kersten} section 8 for proofs.

\begin{theorem}[Skolem-Noether] \label{010}
Let $A$ be a finite dimensional central simple algebra over a field $K$
and let $B$ be a simple $K$-subalgebra of $A$. If $\phi : B \ra A$ is a
$K$-algebra homomorphism, then there exists a unit $a \in A^\x$
satisfying
\[
    \phi(b) = aba^{-1}
    \qq \text{for all}\ b \in B.
\]
In particular, every $K$-isomorphism between subalgebras of $A$ can be
extended to an inner automorphism of $A$.
\end{theorem}


\chapter{Brauer groups of local fields} 
\label{060}

We collect here the needed results on Brauer groups, cyclic algebras and
local class field theory. More details can be found in \cite{reiner}
chapter 7.

\vspace{5ex}

\section{Brauer groups} 
\label{061}

Let $K$ be a field and let $A, B$ be a central simple $K$-algebras. We
say that $A$ and $B$ are equivalent, denoted $A \sim B$, if their
skewfield parts are $K$-isomorphic, in other words if there is an
isomorphism of $K$-algebras
\[
    A \ox_K M_r(K) \iso B \ox_K M_s(K)
\]
for some integers $r$ and $s$. Let $[A]$ and $[B]$ denote the respective
equivalence classes of $A$ and $B$. Under multiplication defined by
\[
    [A] \cdot [B] = [A \ox_K B],
\]
the set of classes of central simple $K$-algebras forms an abelian group
denoted $Br(K)$; \label{323} it is called the \emph{Brauer group} of
$K$. Clearly, its unit is $[K]$.

For an extension $L$ of $K$, there is a group homomorphism
\[
    Br(K) \raa Br(L)\ :\ [A] \mtoo [L \ox_K A],
\]
whose kernel $Br(L/K) = Br(L,K)$ is the \emph{relative Brauer group} of
$L$ over $K$. Thus $[A] \in Br(L/K)$ if and only if $L \ox_K A \iso
M_r(L)$ for some integer $r$, in which case we say that $L$
\emph{splits} $A$, or is a \emph{splitting field} of $A$. As shown in
\cite{reiner} theorem 28.5 and remark 28.9, we have the following:

\begin{proposition} \label{062}
For $D$ a central division algebra over $K$, a field $L$ splits $D$ if
and only if it embeds as a maximal subfield of $D$.
\end{proposition}

For $[A] \in Br(K)$, we define the \emph{exponent} \label{302} $exp[A]$
of $[A]$ to be the order of $[A]$ in $Br(K)$, and we define the
\emph{index} $ind[A]$ of $[A]$ to be the index of the skewfield part of
$A$, that is
\[
    ind[A] = ind(D) = [D:K]^{\frac{1}{2}}
\]
for $D$ a division algebra equivalent to $A$ in $Br(K)$. As given in
\cite{reiner} theorem 29.22, we have:

\begin{proposition} \label{064}
For any $[A]$ in $Br(K)$, $ind[A]$ is a multiple of $exp[A]$.
\end{proposition}

\section{Crossed algebras} 
\label{338}

Let $L$ be a Galois extension of $K$ with Galois group $G = Gal(L/K)$.
We define an algebra
\[
    A = \sum_{\sigma \in G} Lu_\sigma
\]
having as $L$-basis a set of symbols $\{u_\sigma\ |\ \sigma \in G\}$
satisfying
\[
    \sigma(x) u_\sigma = u_\sigma x, \qq
    u_\sigma u_\tau = f_{\sigma,\tau} u_{\sigma\tau},
    \qq \text{and} \qq
    \rho(f_{\sigma,\tau}) f_{\rho,\sigma\tau} 
    = f_{\rho,\sigma} f_{\rho\sigma,\tau}
\]
for $x \in L$, $\rho, \sigma, \tau \in G$ and $f_{\sigma,\tau} \in
L^\x$. A map $f: G \x G \ra L^\x$ satisfying this third condition is a
\emph{factor set} from $G$ to $L^\x$.  Given such an $f$, the algebra
$A$ thus constructed is a \emph{crossed(-product) algebra} and is denoted
$(L/K, f)$. 

According to \cite{reiner} theorem 29.6, for each $f$, $(L/K, f)$ is a
finite dimensional central simple algebra over $K$ having $L$ as maximal
subfield. 

\begin{proposition} \label{331}
If $A = (L/K, f)$ and $exp[A]=[L:K]$, then $A$ is a division algebra.
\end{proposition}

\begin{proof}
Let $n = [L:K]$, so that $[A:K] = n^2$, and let $D$ be the skewfield
part of $A$ with $A \iso M_r(D)$ and $m = ind[D]$. Then $n=mr$, and
$exp[A]$ divides $m$ by proposition \ref{064}. Because $exp[A] = n$, we
have $m=n$ and $r=1$, in which case $A$ is a division algebra.
\end{proof}

We also know from \cite{reiner} theorem 29.6 that the set of factor sets
from $G$ to $L^\x$ can be partitioned under an equivalence relation to
form a multiplicative group of classes $[f]$, isomorphic to the second
cohomology group $H^2(G,L^\x)$, in such a way that two crossed algebras
$(L/K, e)$, $(L/K, f)$ are $K$-isomorphic if and only if $[e]=[f]$.
Then by \cite{reiner} theorem 29.12 we have the following:

\begin{theorem} \label{063}
Let $L$ be a finite Galois extension of a field $K$ with Galois group
$G$. Then
\[
    H^2(G,L^\x) \iso Br(L/K)
\]
given by mapping $[f] \in H^2(G,L^\x)$ onto the class $[(L/K, f)] \in
Br(L/K)$.
\end{theorem}

\begin{remark} \label{254}
As noted in remark (i) following theorem 29.13 of \cite{reiner}, if $K
\subset K' \subset L$ are finite Galois extensions with Galois groups $G
= Gal(K/L)$ and $G' = Gal(K'/L)$, then there is a commutative diagram
\[
    \xymatrix
    {
        H^2(G, L^\x) \ar[r]^{\iso} \ar[d]_{\text{res}} 
        & Br(L/K) \ar[d]^{\_ \ox_K K'} \\
        H^2(G', L^\x) \ar[r]^{\iso} & Br(L/K')
    }
\]
where the left hand vertical map is the restriction homomorphism
induced by the inclusion $G \subset G'$.
\end{remark}

\section{Cyclic algebras} 
\label{065}

Let $L$ be a finite Galois extension of a field $K$ with cyclic Galois
group $G = Gal(L/K)$ of order $n$ generated by $\sigma$; such an
extension is called \emph{cyclic}. Let $a$ be an element of $K^\x$ and
form the associative $K$-algebra
\[
    A = (L/K, \sigma, a)
    = \sum_{i=0}^{n-1} L u^i,
\]
for an element $u$ satisfying $u x = \sigma(x) u$ and $u^n=a$ for all $x
\in L$, where $u^0$ is identified with the unit of $A$.  Such a
$K$-algebra is called \emph{cyclic}. 

As explained in \cite{reiner} section 30, $A$ is isomorphic to the
crossed algebra $(L/K, f)$ where the factor set $f$ from $G$ to $L^\x$
is given by
\[
    f_{\sigma^i, \sigma^j} =
    \begin{cases}
        1 & \text{if } i+j < n,\\
        a & \text{if } i+j \geq n,
    \end{cases}
\]
for $0 \leq i,j \leq n-1$. In particular, $A$ is a central simple
$K$-algebra split by $L$.  Conversely, \cite{reiner} theorem 30.3
establishes that if $L/K$ is a cyclic extension with Galois group $G$ of
order $n$ generated by $\sigma$, and if $f$ is a factor set from $G$ to
$L^\x$, then the crossed algebra $(L/K, f)$ is isomorphic to the cyclic
algebra $(L/K, \sigma, a)$ for
\[
    a = \prod_{i=0}^{n-1} f_{\sigma^i, \sigma}\ \in K^\x.
\]

According to \cite{reiner} theorem 30.4, we have:

\begin{proposition} \label{066}
Let $L/K$ be a cyclic extension with Galois group of order $n$ generated
by $\sigma$, and let $a, b \in K^\x$. Then
\begin{enumerate}
    \item $(L/K, \sigma, a) \iso (L/K, \sigma^s, a^s)$ for any integer
    $s$ prime to $n$;

    \item $(L/K, \sigma, 1) \iso M_n(K)$;

    \item $(L/K, \sigma, a) \iso (L/K, \sigma, b)$ if and only if
    $\frac{a}{b}$ belongs to the norm $N_{L/K}(L^\x)$. In particular,
    $(L/K, \sigma, a) \iso K$ if and only if $a \in N_{L/K}(L^\x)$;

    \item $(L/K, \sigma, a) \ox_K (L/K, \sigma, b) \iso (L/K, \sigma,
    ab)$.
\end{enumerate}
\end{proposition}

\begin{corollary} \label{067}
Let $A = (L/K, \sigma, a)$ be a cyclic algebra. Then $exp[A]$ is the
smallest positive integer $s$ such that $a^s \in N_{L/K}(L^\x)$.
\end{corollary}

\begin{proof}
Since $[A]^s = [(L/K, \sigma, a^s)]$, we have $[A]^s = 1$ if and only if
$a^s \in N_{L/K}(L^\x)$.
\end{proof}

We know from class field theory and theorem \ref{063} that the map
\[
    K^\x \raa Br(L/K)\ :\ a \mtoo [(L/K, \sigma, a)]
\]
is an epimorphism of group which induces an isomorphism:

\begin{theorem} \label{068}
If $L/K$ is a cyclic extension with Galois group $G$, then
\[
    H^2(G, L^\x) \iso Br(L/K) \iso K^\x/N_{L/K}(L^\x).
\]
\qed
\end{theorem}

\section{The local case} 
\label{069}

Suppose that $K$ is a local field with residue field of cardinality $q$
and a uniformizing element $\pi_K$. Let $n$ be a positive integer, $K_n$
an unramified extension of degree $n$ over $K$, and let $\sigma \in
Gal(K_n/K) \iso \Zbb/n$ be the Frobenius of this extension. For a
positive integer $r$, we consider the cyclic algebra $A = (K_n/K,
\sigma, \pi_K^r)$ and we define the \emph{Hasse invariant} of $A$ to be
\[
    inv_K (K_n/K, \sigma, \pi_K^r) = \frac{r}{n}.
\]
By \cite{reiner} theorem 31.1 and 31.5, we know that the isomorphism
class of $A$ only depends on $r$ modulo $n$, and that the skewfield part
of $A$ has the same invariant as $A$. Consequently, the invariant of $A$
only depends on the class $[A]$ in $Br(K)$ and there is a well defined
map
\[
    inv_K:\ Br(K) \raa \Qbb/\Zbb;
\]
it is in fact an isomorphism by \cite{reiner} theorem 31.8:

\begin{theorem} \label{070}
If $K$ is a local field, then $Br(K) \iso \Qbb/\Zbb$ via $inv_K$.
\end{theorem}

By \cite{reiner} theorem 31.9, we have:

\begin{theorem} \label{071}
Let $L$ be a finite extension of degree $m$ over a local field $K$.
There is a commutative diagram
\[
    \xymatrix
    {
        Br(K) \ar[r]^{inv_K}_{\iso} \ar[d]_{L \ox_K \_} 
        & \Qbb/\Zbb \ar[d]^{\cdot m} \\
        Br(L) \ar[r]^{inv_L}_{\iso} & \Qbb/\Zbb 
    }
\]
where the right hand vertical map is multiplication by $m$.
\end{theorem}

\begin{corollary} \label{072}
Let $L$ be a finite Galois extension of degree $m$ over a local field
$K$ with Galois group $G$. Then
\[
    H^2(G, L^\x) \iso Br(L/K) \iso \Zbb/m.
\]
\end{corollary}

\begin{proof}
By theorem \ref{071} and the definition of $Br(L/K)$, there is a
commutative diagram
\[
    \xymatrix
    {
        1 \ar[r] & Br(L/K) \ar[r] 
        & Br(K) \ar[r] \ar[d]^{\iso} & Br(L) \ar[d]^{\iso} \\
        && \Qbb/\Zbb \ar[r]_{m} & \Qbb/\Zbb
    }
\]
where the top row is exact, the bottom map is multiplication by $m$, and
the vertical maps are isomorphisms. Hence $Br(L/K)$ is isomorphic to the
kernel of the bottom map.
\end{proof}

\begin{corollary} \label{255}
If $K \subset K' \subset L$ are finite Galois extensions of local fields
with Galois groups $G = Gal(L/K)$ and $G' = Gal(L/K')$, then the
restriction map 
\[
    H^2(G, L^\x) \raa H^2(G', L^\x)
\]
induced by the inclusion $G' \subset G$ is surjective.
\end{corollary}

\begin{proof}
The diagram
\[
    \xymatrix
    {
        H^2(G, L^\x) \ar[r]^{\iso} \ar[d]_{\text{res}}
        & Br(L/K) \ar[r]^{\text{inc}} \ar[d]
        & Br(K) \ar[r]^{\_ \ox_K L} \ar[d]_{\_ \ox_K K'}
        & Br(L) \ar[r]^{\iso} \ar@{=}[d]
        & \Qbb/\Zbb \ar@{=}[d] \\ 
        H^2(G', L^\x) \ar[r]^{\iso}
        & Br(L/K') \ar[r]^{\text{inc}}
        & Br(K') \ar[r]^{\_ \ox_K' L}
        & Br(L) \ar[r]^{\iso}
        & \Qbb/\Zbb
    }
\]
given by theorem \ref{063} and \ref{071} is commutative by remark
\ref{254}. By corollary \ref{072}, the relative Brauer groups $Br(L/K)$
and $Br(L/K')$ are cyclic of order $|G|$ and $|G'|$ respectively, and
the second square in the above diagram may be identified with the
commutative square
\[
    \xymatrix{
        \Zbb/|G| \ar[d] \ar[r]^{\text{inc}}
        & \Qbb/\Zbb \ar[d]^{\cdot \frac{|G|}{|G'|}} \\
        \Zbb/|G'| \ar[r]^{\text{inc}}
        & \Qbb/\Zbb,
    }
\]
where the right hand vertical map is multiplication by
$\frac{|G|}{|G'|}$ according to theorem \ref{071}. In particular this
latter map is surjective and sends $\frac{1}{|G|}$ to $\frac{1}{|G'|}$.
Hence the generator of $\Zbb/|G|$ associated to $\frac{1}{|G|}$ must be
sent to a generator of $\Zbb/|G'|$. The second vertical map in the first
diagram given above is therefore surjective and the result follows.
\end{proof}

\begin{proposition} \label{256}
Let $L/K$ be a finite Galois extension of local fields of characteristic
zero with cyclic Galois group $G$.
\begin{enumerate}
    \item If $L/K$ is unramified, the valuation map induces an
    isomorphism
    \[
        H^2(G, L^\x) 
        \iso H^2(G, \frac{1}{e(L)}\Zbb)
        \iso \lan \pi_K \ran / \lan \pi_K^{|G|} \ran,
    \]
    for $e(L)$ the ramification index of $L/\Qbb_p$ and $\pi_K$ a
    uniformizing element of $K$.

    \item If $L/K$ is totally ramified, the valuation map induces an
    isomorphism
    \[
        H^2(G, L^\x) \iso H^2(G, \Ocal_L^\x),
    \]
    for $\Ocal_L$ the ring of integers of $L$.
\end{enumerate}
\end{proposition}

\begin{proof}
The valuation map $v = v_{\Qbb_p}: L^\x \ra \frac{1}{e(L)}\Zbb$ is
surjective and induces a short exact sequence
\[
    1 \raa \Ocal_L^\x \raa L^\x \raa \frac{1}{e(L)}\Zbb \raa 1,
\]
which in turns induces a long exact sequence
\[
    H^1(G, \frac{1}{e(L)}\Zbb) \ra
    H^2(G, \Ocal_L^\x) \ra
    H^2(G, L^\x) \ra
    H^2(G, \frac{1}{e(L)}\Zbb) \ra
    H^3(G, \Ocal_L^\x).
\]

If $L/K$ is unramified, \cite{serre4} proposition 1 says that $H^i(G,
\Ocal_L^\x)$ is trivial for all $i \in \Zbb$ and hence yields the
result.

If $L/K$ is totally ramified, there are unifomizing elements $\pi_K$ of
$K$ and $\pi_L$ of $L$ such that
\[
    \pi_K = (\pi_L)^{|G|}.
\]
Therefore
\[
    v(\pi_K) = |G|v(\pi_L) = |G| \cdot \frac{1}{e(L)}\Zbb,
\]
and consequently the map $H^2(G, L^\x) \ra H^2(G, \frac{1}{e(L)}\Zbb)$
is trivial. Moreover, since $G$ is finite and $\frac{1}{e(L)}\Zbb$ is
infinite, we have $H^1(G, \frac{1}{e(L)}\Zbb) = 0$ and the result
follows.
\end{proof}

\begin{example} \label{258}
For any prime $p$ and $\alpha \geq 1$, we have
\[
    p 
    \in N_{\Qbb_p(\zeta_{p^\alpha})/\Qbb_p}(\Qbb_p(\zeta_{p^\alpha})^\x).
\]
Indeed, for $1 \leq r \leq \alpha-1$ let $\sigma$ be a generator of
$Gal(\Qbb_p(\zeta_{p^{r+1}}) / \Qbb_p(\zeta_{p^{r}}))$ satisfying 
\[
    \sigma(\zeta_{p^{r+1}}) = \zeta_{p^{r+1}}\zeta_p, 
\]
and define 
\[
    \Sigma_i(X_1, \ldots, X_p)
\]
to be the homogeneous symmetric polynomial of degree $i$ in $p$
variables $X_1, \ldots, X_p$, so that
\[
    \prod_{i=1}^p (X-X_i)\
    =\ \sum_{i=1}^p (-1)^i \Sigma_i(X_1, \ldots, X_n)X^{n-i}.
\]
Then for $1 \leq k \leq p-1$ we have
\begin{align*}
    N_{\Qbb_p(\zeta_{p^{r+1}}) 
    / \Qbb_p(\zeta_{p^{r}})}(1-\zeta_{p^{r+1}}^k)\
    &=\ \prod_{j=0}^{p-1} (1-\sigma^j(\zeta_{p^{r+1}}^k))\\
    &=\ \sum_{i=0}^p (-1)^i\ \Sigma_i(\zeta_{p^{r+1}}^k,\
    \sigma(\zeta_{p^{r+1}}^k),\ \ldots,\ \sigma^{p-1}(\zeta_{p^{r+1}}^k))\\
    &=\ \sum_{i=0}^p (-1)^i\ \Sigma_i(\zeta_{p^{r+1}}^k,\
    \sigma(\zeta_{p^{r+1}}\zeta_{p})^k,\ \ldots,\
    \sigma(\zeta_{p^{r+1}}\zeta_{p}^{p-1})^k)\\
    &=\ \sum_{i=0}^p (-1)^i\ \zeta_{p^{r+1}}^{ik}\
    \Sigma_i(1,\ \sigma(\zeta_{p}^k),\ 
    \ldots,\ \sigma(\zeta_{p}^{(p-1)k}))\\
    &=\ 1-\zeta_{p^r}^k,
\end{align*}
where the last equality is a consequence of the fact that
\[
    \Sigma_i(1,\ \sigma(\zeta_{p}^k),\ 
    \ldots,\ \sigma(\zeta_{p}^{(p-1)k}))\
    =\
    \begin{cases}
        1 & \text{if } i=0,p,\\
        0 & \text{if } i\neq 0,p.
    \end{cases}
\]
As shown in corollary \ref{093}
\[
    p 
    = \prod_{k=1}^{p-1} (\zeta_p^k-1) 
    = N_{\Qbb_p(\zeta_p)/\Qbb_p}(\zeta_p-1).
\]
Consequently 
\[
    p \in 
    N_{\Qbb_p(\zeta_{p^\alpha})/\Qbb_p}
    (\Qbb_p(\zeta_{p^\alpha})^\x)
    \qq \text{and} \qq
    p \in N_{\Qbb_p(\zeta_{p^\alpha})/\Qbb_p(\zeta_p)}
    (\Qbb_p(\zeta_{p^\alpha})^\x).
\]
Moreover if $p=2$, we have 
\[
    2, (1 \pm \zeta_4) 
    \in N_{\Qbb_2(\zeta_{2^\alpha}) / \Qbb_2(\zeta_4)} 
    (\Qbb_2(\zeta_{2^\alpha})^\x)
    \qq \text{for} \q \alpha \geq 2.
\]
\end{example}

For a local field $K$ of characteristic zero with uniformizing element
$\pi_K$ and ring of integers $\Ocal_K$, we let
\begin{align*} \label{340}
    U_i(\Ocal_K^\x) 
    &= \{ x \in \Ocal_K^\x\ |\ v_K(x-1) \geq i \}\\[1ex] 
    &= \{ x \in \Ocal_K^\x\ |\ x \equiv 1 \mod \pi_K^i\},
    \q i \geq 0,
\end{align*}
be the $i$-th group in the filtration
\[
    \Ocal_K^\x 
    = U_0(\Ocal_K^\x) 
    \supset U_1(\Ocal_K^\x) 
    \supset U_2(\Ocal_K^\x) 
    \supset \ldots
\]

\begin{proposition} \label{257}
Let $L/K$ be a finite Galois extension of local fields of characteristic
zero with Galois group $G$. If $L/K$ is unramified, then the trace
\[
    Tr_G = Tr_{L/K}: l \raa k
\]
is surjective on the residue fields, and the norm
\[
    N_G = N_{L/K}: \Ocal_L^\x \raa \Ocal_K^\x
\]
is surjective on the groups of units of the rings of integers.
\end{proposition}

\begin{proof}
Since $G = Gal(l/k)$ is cyclic, Hilbert's theorem 90 yields $H^1(G, l) =
0$. Let $t$ denote a generator of $G$, and let $Tr := Tr_G$. In the
periodic complex
\[
    \xymatrix{
        l \ar[r]^{1-t}
        & l \ar[r]^{Tr}
        & l \ar[r]^{1-t}
        & l \ar[r]^{Tr}
        & \ldots
    }
\]
we have $Ker(Tr) = Im(1-t)$. Hence
\[
    |Ker(1-t)|
    = \frac{|l|}{|Im(1-t)|}
    = \frac{|l|}{|Ker(Tr)|}
    = |Im(Tr)|,
\]
and $H^2(G, l) = 0$. Because $Ker(1-t)=k$, it follows that $Im(Tr_G) =
k$.

In order to show the second assertion, we first note that for any $i
\geq 1$ the norm $N_G$ becomes the trace
\[
    Tr: U_i(\Ocal_L^\x)/U_{i+1}(\Ocal_L^\x)
    \raa U_i(\Ocal_K^\x)/U_{i+1}(\Ocal_K^\x)
\]
on the successive quotients of the filtration of the units of the rings
of integers; these maps are surjective by the first assertion. For each
$i \geq 1$, consider the commutative diagram
\[
    \xymatrix{
        1 \ar[r]
        & U_i(\Ocal_L^\x)/U_{i+1}(\Ocal_L^\x) \ar[r] \ar@{->>}[d]^{Tr}
        & U_1(\Ocal_L^\x)/U_{i+1}(\Ocal_L^\x) \ar[r] \ar[d]
        & U_1(\Ocal_L^\x)/U_{i}(\Ocal_L^\x) \ar[r] \ar[d]
        & 1 \\
        1 \ar[r]
        & U_i(\Ocal_K^\x)/U_{i+1}(\Ocal_K^\x) \ar[r]
        & U_1(\Ocal_K^\x)/U_{i+1}(\Ocal_K^\x) \ar[r]
        & U_1(\Ocal_K^\x)/U_{i}(\Ocal_K^\x) \ar[r]
        & 1,
    }
\]
where the horizontal lines are exact and the vertical maps are induced
by the norm. If $i=1$ the vertical maps are obviously surjective.
Moreover if $i \geq 2$ and if the vertical map on the right hand side is
surjective, then the middle one is also surjective by the five lemma.
We conclude by induction on $i$ that $N_G$ is surjective on
$U_1(\Ocal_K^\x)$, and consequently on $\Ocal_K^\x$.
\end{proof}

\begin{remark} \label{271}
According to classical Galois theory (see for example \cite{lang}
chapter VI theorem 1.12), if $K_1$ and $K_2$ are extensions of $\Qbb_p$
such that $K_1K_2 = L$, $K_1 \cap K_2 = K$ and $L/K$ is Galois with
abelian Galois group, then any $x \in K$ such that $x \in
N_{K_1/K}(K_1^\x)$ satisfies $x \in N_{L/K_2}(K^\x)$.
\end{remark}


\chapter{Division algebras over local fields} 
\label{033}

We provide here a short account on division algebras over local fields.
The reader may refer to \cite{reiner} chapter 3 for more details.

Let $K$ be a local field with residue field of cardinality $q$, let
$\pi_K$ be a uniformizing element of $K$, and let $D$ be a central
division algebra of dimension $n^2$ over $K$. As shown in \cite{reiner}
theorem 12.10, the normalized valuation $v_K: \pi_K \mapsto 1$ on $K$
extends in a unique way to a valuation $v = v_D$ \label{325} on $D$. By
\cite{reiner} section 13, we know that the skew field $D$ is complete
with respect to $v$ and that the maximal order $\Ocal_D$ of $D$ is of
degree $n^2$ over the ring of integers $\Ocal_K$ of $K$. Let $d$ and $k$
denote the residue fields of $D$ and $K$ respectively. By \cite{reiner}
theorem 13.3 we have
\[
    n^2 = ef,
\]
where 
\begin{itemize}\label{301}
    \item $e = e(D/K) = |v(D^\x)/v(K^\x)|$ denotes the
    \emph{ramification index} of $D$ over $K$;

    \item $f = f(D/K) = [d:k]$ denotes the \emph{inertial degree} of $D$
    over $K$.
\end{itemize}

\begin{proposition} \label{011}
If $D$ is a central division algebra of dimension $n^2$ over a local
field $K$, then
\[
    e(D/K) = f(D/K) = n.
\]
\end{proposition}

\begin{proof}
Because there exists an element $x \in D$ such that $v(x) = e(D/K)^{-1}$
and as $x$ belongs to a commutative subfield of degree at most $n$ over
$K$, it follows that $e(D/K) \leq n$. On the other hand $k$ is a finite
field and $d = k(\bar{y})$ is a commutative field, for $\bar{y}$ the
image in $d$ of some suitable $y \in D$. Hence $f(D/K) \leq n$ and the
result follows.
\end{proof}

Since $[d:k] = n$, we can find an $x \in \Ocal_D$ such that $k(\bar{x})
= d$. Let $K_n = K(x)$. Because $K_n$ is commutative, $[K_n:K] \leq n$.
On the other hand, $\bar{x}$ is an element of the residue field $k_n$ of
$K_n$, while $k_n = d$, so that $[k_n:k] = n$. It follows that $K_n$ is
a maximal unramified extension of degree $n$ over $K$ in $D$. Such a
$K_n$ is referred to as an \emph{inertia field} of $D$. Of course the
above construction of $K_n$ is not unique, but the Skolem-Noether
theorem implies that all inertia fields are conjugate. 

Let $\omega \in D^\x$ be a root of unity satisfying
\[
    K(\omega) = K_n;
\]
in particular $\omega$ is of order $q^n-1$. According to \cite{reiner}
theorem 14.5, there exists a uniformizing element $\pi$ of $D$
satisfying
\[
    \pi^n = \pi_K
    \qq \text{and} \qq
    \pi \omega \pi^{-1} = \omega^{q^s},
\]
where $s < n$ is a positive integer prime to $n$, uniquely determined by
$D$, which does not depend upon the choice of $\omega$ or $\pi$. Let $r
\in \Zbb$ be such that $rs \equiv 1 \mod n$; in particular $r$ is prime
to $n$. Using \cite{reiner} theorem 31.1 and proposition \ref{066}, we
know that $D$ is isomorphic to the cyclic algebra
\[
    D
    \iso (K_n/K, \sigma^s, \pi_K)
    \iso (K_n/K, \sigma, \pi_K^r),
\]
and is classified up to isomorphism by its invariant 
\[
    inv_K(D) = \frac{r}{n}\ \in \Qbb/\Zbb.
\]
In other words we have:

\begin{theorem} \label{012}
All Azumaya algebras over a local field $K$ are classified up to
isomorphism, via $inv_K$, by the elements of the additive group
$\Qbb/\Zbb$.
\end{theorem}

\begin{notation} \label{324}
For a class in $\Qbb/\Zbb$ represented by an element $r/n \in \Qbb$ with
$(r;n) = 1$ and $1 \leq r < n$, the corresponding Azumaya algebra is
denoted $D(K,r/n)$. When $K = \Qbb_p$, $r=1$ and $p$ is understood, we
write $\Dbb_n = D(\Qbb_p,1/n)$.
\end{notation}

\begin{corollary} \label{084}
If $D$ is a central division algebra over a local field, then $exp[D] =
ind[D]$.
\end{corollary}

\begin{proof}
Suppose $inv_K(D) = \frac{r}{n}$, where $K$ denotes the center of $D$
and $[D:K] = n^2$. By definition $ind[D] = n$. We know from proposition
\ref{064} that $exp[A]$ must divide $n$. Because $r$ is prime to $n$, it
follows that $exp[A] = n$. 
\end{proof}

\begin{remark} \label{334}
Suppose $inv_K(D) = \frac{r}{n}$. By the Skolem-Noether theorem, the
Frobenius automorphism $\sigma$ of $K(\omega) = K_n$ is given by
\[
    \sigma(x) = \xi x \xi^{-1}
\]
for a suitable element $\xi \in D^\x$ determined up to multiplication by
an element of $K(\omega)^\x$. Then clearly the image of $v(\xi)$ in 
\[
    \frac{1}{n}\Zbb / \Zbb \subset \Qbb/\Zbb
\]
is none other than the invariant of $D$. Furthermore, as $\sigma^n$ is
the identity on the inertia field $K(\omega)$, we know that $\xi^n$
commutes with all elements of $K(\omega)$ and hence belongs to
$K(\omega)$. Because
\[
    v(\xi) = \frac{1}{n} v(\xi^n),
\]
we have $v(\xi) = r/n$. Hence $\xi^n = \pi_K^r u$ for a unit $u \in
K(\omega)^\x$. In this case,
\[
    D 
    \iso D(K,r/n) 
    \iso K(\omega)\lan \xi \ran 
    / (\xi^n = \pi_K^r, \xi x = x^\sigma \xi)
\]
as mentioned in the paragraph following the proof of \cite{reiner}
theorem 14.5. 
\end{remark}

So far, we have dealt with unramified extensions of the base field $K$,
but there are in $D$ many more commutative subfields. It can in fact be
shown that all extensions of $K$ of degree dividing $n$ exist; see
\cite{reiner} theorem 31.11, \cite{hazewinkel} 23.1.4 and 23.1.7, or
\cite{serre4} section 1 for proofs.

\begin{theorem}[Embedding] \label{013}
If $D$ is a central division algebra of dimension $n^2$ over a local
field $K$, then the degree of a commutative extension $L$ of $K$ in $D$
divides $n$, and any extension $L$ of $K$ whose degree divides $n$
embeds as a commutative subfield of $D$.
\end{theorem}

In particular, a local field $L$ of characteristic zero embeds in some
$\Dbb_n$, in which case its group of units $L^\x$ is a subgroup of
$\Dbb_n^\x$. The structure of $L^\x$, both algebraically and
topologically, is well known and is recorded below; see for example
\cite{neukirch} chapter II proposition 5.3 and 5.7.

\begin{proposition} \label{341}
Let $L$ be a local field of characteristic zero with residue field $l
\iso \Fbb_{p^f}$, roots of unity $\mu(L)$ and uniformizing element
$\pi_L$. Then
\begin{align*}
    L^\x\
    &=\ \lan \pi_L \ran \x \Ocal_L^\x\\
    &=\ \lan \pi_L \ran \x l^\x \x U_1(\Ocal_L^\x)\\
    &\iso\ \Zbb \x \mu(L) \x \Zbb_p^{[L:\Qbb_p]}.
\end{align*}
\end{proposition}

The most frequently encountered fields are the cyclotomic extensions of
$\Qbb_p$. Recall the following result from \cite{neukirch} chapter II
proposition 7.12 and 7.13.

\begin{proposition} \label{342}
Let $\zeta$ be a primitive $k$-th root of unity for $k=\beta p^\alpha
\geq 1$ with $(\beta;p)=1$, and let $f$ be the smallest positive integer
such that $p^f \equiv 1 \mod \beta$. Then $\Qbb_p(\zeta)/\Qbb_p$ is a
Galois extension with ramification index $\phi(p^\alpha)$ and residue
degree $f$, where
\begin{align*}
    \mu(\Qbb_p(\zeta)) 
    &\iso 
    \begin{cases}
        \Zbb/p^\alpha(p^f\!-\!1) 
        & \text{if } p>2 \text{ or } \alpha \geq 1,\\
        \Zbb/2(2^f\!-\!1) 
        & \text{if } p=2 \text{ and } \alpha = 0,
    \end{cases}\\[1ex]
    Gal(\Qbb_p(\zeta)/\Qbb_p) 
    &\iso 
    \begin{cases}
        (\Zbb/p^\alpha)^\x \x \Zbb/f
        & \text{if } \alpha \geq 1,\\
        \Zbb/f
        & \text{if } \alpha = 0.
    \end{cases}
\end{align*}
\end{proposition}

\begin{corollary} \label{343}
We have
\[
    \Qbb_p(\zeta)^\x
    \iso \Zbb \x \Zbb_p[\zeta]^\x
    \iso 
    \begin{cases}
        \Zbb \x \Zbb/p^\alpha(p^f\!-\!1) \x \Zbb_p^{\phi(p^\alpha)f}
        & \text{if } p>2 \text{ or } \alpha \geq 1,\\
        \Zbb \x \Zbb/2(2^f\!-\!1) \x \Zbb_2^{f}
        & \text{if } p=2 \text{ and } \alpha = 0.
    \end{cases}
\]
\end{corollary}

\begin{proof}
This follows from proposition \ref{341} and \ref{342}.
\end{proof}

We end the section by analysing the invariant of some embeddings that
are useful in the text.

\begin{proposition} \label{014}
Let $D$ be a central division algebra of invariant $\frac{r}{n}$ over a
local field $K$ for $r$ prime to $n$, let $L \subset D$ be a commutative
extension of $K$, and let $m$ be such that $n = m[L:K]$. Then $C_D(L)$
is a central division algebra of invariant $\frac{r}{m}$ over $L$. 
\end{proposition}

\begin{proof}
Using the centraliser theorem \ref{007}, we know that $C_D(L)$ is a
central division algebra of dimension $m^2$ over $L$, and we have
\begin{align*}
    D \ox_K L\
    &\iso\ C_D(L) \ox_K M_{n/m}(K)\\
    &\iso\ C_D(L) \ox_L L \ox_K M_{n/m}(K)\\
    &\iso\ C_D(L) \ox_L M_{n/m}(L).
\end{align*}
Hence the invariant of $C_D(L)$ is that of $D \ox_K L$, which is
$\frac{r}{n}[L:K]$ by theorem \ref{071}.
\end{proof}

\begin{proposition} \label{015}
For any prime $p$, $\Dbb_m$ embeds as a $\Qbb_p$-subalgebra of $\Dbb_n$
if and only if $n=km$ with $k \equiv 1 \mod m$.
\end{proposition}

\begin{proof}
If $D(\Qbb_p, 1/m)$ embeds as a $\Qbb_p$-subalgebra of $D(\Qbb_p,1/n)$,
then the centralizer theorem provides an isomorphism
\[
    D(\Qbb_p,1/n) 
    \iso D(\Qbb_p, 1/m) \ox_{\Qbb_p} C_{D_n}(D(\Qbb_p, 1/m)),
\]
so that there is an integer $k$ satisfying $n = km$. Because
$C_{D_n}(D(\Qbb_p, 1/m))$ is a central division algebra over $\Qbb_p$,
we also know the existence of an integer $l$ such that
\[
    C_{D_n}(D(\Qbb_p, 1/m)) \iso D(\Qbb_p, l/k).
\]
The law on the Brauer group $\Qbb/\Zbb$ being defined as such a tensor
product over the $\Qbb_p$-Azumaya algebra classes (see appendix
\ref{060}), it follows that
\begin{align}
    \frac{1}{n}
    \equiv \frac{1}{m} + \frac{l}{k}\q \mod \Zbb. \tag{$\ast$}
\end{align}
Consequently $1 \equiv k + lm \mod n$, and $k \equiv 1$ mod $m$.

Conversely, if $n = km$ with $k \equiv 1 \mod m$, there is an integer
$l$ prime to $k$ such that $1 \equiv k+lm \mod n$. It follows that
($\ast$) is verified and $D(\Qbb_p, 1/m)$ embeds as a
$\Qbb_p$-subalgebra of $D(\Qbb, 1/n)$.
\end{proof}

\begin{corollary} \label{040}
When $p=2$, $\Dbb_2$ embeds in $\Dbb_n$ if and only if $n \equiv 2 \mod
4$.  \qed
\end{corollary}


\chapter{Endomorphisms of formal group laws} 
\label{273}

We give here a short account on endomorphisms of formal group laws of
finite height $n$ defined over a field of characteristic $p > 0$. We
summarize how these occurs as elements of the central division algebra
$\Dbb_n = D(\Qbb_p, 1/n)$ of invariant $\frac{1}{n}$ over $\Qbb_p$. The
reader may refer to \cite{hazewinkel} or \cite{frohlich} for more
details. 

\bigskip

\begin{definition}
Let $R$ be a commutative ring with unit. A \emph{formal group law} over
$R$ is a power series $F = F(X,Y) = X +_F Y \in R[[X, Y]]$ satisfying
\begin{itemize}
    \item $F(X, 0) = F(0, X) = X$,

    \item $F(X, Y) = F(Y, X)$, and

    \item $F(X, F(Y, Z)) = F(F(X, Y), Z)$ in $R[[X, Y, Z]]$.
\end{itemize}
We denote by $FGL(R)$ the set of formal group laws defined over $R$. For
$F, G \in FGL(R)$, a \emph{homomorphism} from $F$ to $G$ is a power
series $f=f(X) \in R[[X]]$ without constant term such that $f(F(X, Y)) =
G(f(X), f(Y))$. It is an \emph{isomorphism} if it is invertible, that
is, if the coefficient of $X$ is a unit in $R$. 
\end{definition}

The set $Hom_R(F,G)$ of homomorphisms from $F$ to $G$ forms an abelian
group under formal addition
\[
    G(f(X),g(X)) = f(X) +_G g(X).
\]
When $F=G$, the group $End_R(F) = Hom_R(F,F)$ becomes a ring via the
composition of series. Its group of units is written $End_R(F)^\x =
Aut_R(F)$. For an integer $n \in \Zbb$, we define the \emph{$n$-series}
$[n]_F$ to be the image of $n$ in $End_R(F)$ via the canonical ring
homomorphism $\Zbb \ra End_R(F)$, in other words
\[
    [n]_F(X) 
    = \underbrace{X +_F \ldots +_F X}_{\text{n times}}.
\]
As shown in \cite{frohlich} chapter I \S 3, when $R=k$ is a field
of characteristic $p > 0$, any homomorphism $f \in Hom_k(F,G)$ can be
written as a series
\[
    f(X) = \sum_{i \geq 1} a_i X^{ip^n}
\]
for some integer $n = ht(f) \in \Nbb^* \cup \{\infty\}$ defined as the
\emph{height} of $f$, where by convention $ht(f) = \infty$ if $f=0$. For
$F \in FGL(k)$ we then define $ht(F)$ to be the height of $[p]_F$. As
shown in \cite{frohlich} chapter III \S 2, this induces a valuation $ht$
on $End_k(F)$ which turns $End_k(F)$ into a complete local ring. In
particular, the definition of $[n]_F$ extends to the $p$-adic integers
$\Zbb_p$, and $ht(f)=0$ if and only if $f$ is invertible.

Let us fix a separably closed field $K$ of characteristic $p>0$. As
shown in \cite{frohlich} chapter III \S 2, we have the following three
results: the first two provide a classification of the $K$-isomorphism
classes of formal group laws defined over $K$ and the third one
describes the endomorphism ring as a subring of the central division
algebra of Hasse invariant $1/ht(F)$ over $\Qbb_p$.

\begin{theorem}[existence] \label{274}
For a positive integer $n$, there exists a formal group law $F_n \in
FGL(\Fbb_p)$ such that $[p]_{F_n}(X) = X^{p^n}$; it is the Honda formal
group law of height $n$.
\end{theorem}

\begin{theorem}[Lazard] \label{275}
Two formal group laws $F,G \in FGL(K)$ are $K$-isomorphic if and only if
$ht(F)=ht(G)$.
\end{theorem}

\begin{theorem}[Dieudonné - Lubin] \label{276}
For a formal group law $F \in FGL(K)$ of finite height $n$, the ring
$End_K(F)$ is isomorphic to the maximal order $\Ocal_n$ of the central
division algebra $\Dbb_n = D(\Qbb_p, 1/n)$ of invariant $\frac{1}{n}$
over $\Qbb_p$.
\end{theorem}

We now describe the image in $\Dbb_n$ of the ring of endomorphisms
defined over a finite subfield of $K$. For this we identify $\Ocal_n$
with $End_K(F_n)$ and fix two integers $n, r \geq 1$.  Let $v$ denote
the unique extension to $\Dbb_n^\x$ of the $p$-adic valuation $p \mto 1$
on $\Qbb_p^\x$. Let $\Ccal_r$ be the set of conjugacy classes of
elements of valuation $\frac{r}{n}$ in $\Ocal_n$, and let
$\Ical(\Fbb_{p^r}, n)$ denote the set of $\Fbb_{p^r}$-isomorphism
classes of formal group laws of height $n$. Define the map
\[
    \Phi: \Ical(\Fbb_{p^r}, n) \raa \Ccal_r
\]
by assigning to a formal group law $F \in FGL(\Fbb_{p^r})$ of height $n$
and a $K$-isomorphism $f: F_n \ra F$, the conjugacy class of $\xi_F^r
\in \Ocal_n^\x$ the element associated to the endomorphism $f^{-1}
X^{p^r} f$. Then $\Phi$ is a bijection (see \cite{hazewinkel} 24.4.2, or
\cite{frohlich} chapter III \S 3 theorem 2).

\begin{theorem} \label{277}
The map
\[
    End_{\Fbb_{p^r}}(F) 
    \raa C_{\Ocal_n}(\xi_F^r)\ 
    :\ x \mtoo f^{-1}xf
\]
is a ring isomorphism from $End_{\Fbb_{p^r}}(F)$ to the subring of all
elements of $\Ocal_n$ commuting with $\xi_F^r$.
\end{theorem}

\begin{proof}
In $End_K(F) \iso \Ocal_n$, the ring $End_{\Fbb_{p^r}}(F)$ is
characterized by $\xi_F^r x = x \xi_F^r$, as a series $g(X) \in K[[X]]$
satisfies $g(X)^{p^r} = g(X^{p^r})$ if and only if its coefficients are
in $\Fbb_{p^r}$.
\end{proof}

In other words if $m = [\Qbb_p(\xi_F^r):\Qbb_p]$, then $m$ divides $n$
and $End_{\Fbb_{p^r}}(F)$ is isomorphic to the maximal order of the
division algebra
\[
    D(\Qbb_p(\xi_F^r), m/n) 
    \iso C_{\Dbb_n}(\xi_F^r) 
    \subset D_n.
\]
In particular $End_{\Fbb_{p^r}}(F)$ is the ring of integers of the
$\Qbb_p$-algebra $End_{\Fbb_{p^r}}(F) \ox_{\Zbb_p} \Qbb_p$.

\begin{corollary} \label{278}
There exists a formal group law $F$ defined over $\Fbb_{p^r}$ and of
height $n$ such that
\[
    End_{\Fbb_{p^r}}(F) 
    \iso End_{K}(F) 
    \iso \Ocal_n
\]
if and only if $r$ is a multiple of $n$.
\end{corollary}

\begin{proof}
This follows from the fact that the valuation group of the center
$\Qbb_p$ of $\Dbb_n$ is $\Zbb$, and hence that $End_{\Fbb_{p^r}}(F) \iso
\Ocal_n$ if and only if $End_{\Fbb_{p^r}}(F) \subset \Zbb_p$.
\end{proof}

\begin{corollary} \label{279}
If $r=1$, then $End_{\Fbb_{p}}(F)$ is commutative and its field of
fractions is totally ramified of degree $n$ over $\Qbb_p$.
\end{corollary}

\begin{proof}
In this case, the element $\xi_F \in \Dbb_n$ has valuation
$\frac{1}{n}$. Hence $\Qbb_p(\xi_F)$ has ramification index at least $n$
over $\Qbb_p$. Since $\Qbb_p(\xi_F)$ is a commutative subfield of
$\Dbb_n$, we have $[\Qbb_p(\xi_F):\Qbb_p] \leq n$ and
$\Qbb_p(\xi_F)/\Qbb_p$ is totally ramified of degree $n$.  The
commutativity of $End_{\Fbb_p}(F)$ follows from the fact that the
centralizer of $\Qbb_p(\xi_F)$ in $\Dbb_n$ is $\Qbb_p(\xi_F)$ itself.
\end{proof}

Generally $End_K(F) \iso \Ocal_n$ for $F$ a formal group law of height
$n$. If $F$ is already defined over $\Fbb_p$, the element $\xi_F \in
\Ocal_n$ corresponds to the Frobenius endomorphism $X^p \in End_{K}(F)$.

\begin{proposition} \label{280}
If $F$ is defined over $\Fbb_p$, then $End_K(F) = End_{\Fbb_{p^n}}(F)$
if and only if the minimal polynomial of $\xi_F \in \Ocal_n$ over
$\Zbb_p$ is $\xi_F^n-up$ with $u \in \Zbb_p^\x$.
\end{proposition}

\begin{proof}
One has $End_{\Fbb_{p^n}}(F) = C_{End_K(F)}(\xi_F^n)$, and therefore
$End_K(F) = End_{\Fbb_{p^n}}(F)$ if and only if $\xi_F^n$ is central.
The result then follows form the fact that the center of $End_K(F)$ is
$\Zbb_p$ and the valuation of $\xi_F^n$ is equal to the valuation of
$p$.
\end{proof}

From appendix \ref{033}, we know that
\[
    \Ocal_n
    \iso \Zbb_p(\omega)\lan \xi_F \ran /
    (\xi_F^n = pu, \xi_F x \xi_F^{-1} = \sigma(x)),
    \qq x \in \Zbb_p(\omega),
\]
for a primitive $(p^n\!-\!1)$-th root of unity $\omega$ and $\sigma \in
Gal(\Zbb_p(\omega)/\Zbb_p) \iso Gal(\Fbb_{p^n}/\Fbb_p)$ the Frobenius
automorphism. Here $\sigma$ lifts to an action on $\Ocal_n$ given by
\[
    \sigma\left( \sum_{i \geq 0} x_i \xi_F^i \right)
    = \sum_{i \geq 0} \sigma(x_i) \xi_F^i,
    \qq x_i \in \Zbb_p(\omega).
\]
Since $\xi_F^n = pu$, we know that $v(\xi_F) = \frac{1}{n}$. Thus the
valuation map and the canonical projection $\pi: \Dbb_n^\x \ra
\Dbb_n^\x/\lan pu \ran$ induce the exact commutative diagram
\[
    \xymatrix{
    && 1 \ar[d] & 1 \ar[d] \\
    && \lan pu \ran \ar[d] \ar[r] & \lan pu \ran \ar[d] \\
    1 \ar[r] & \Ocal_n^\x \ar[d] \ar[r] 
      & \Dbb_n^\x \ar[d]^{\pi} \ar[r]^{v} 
      & \lan \xi_F \ran \ar[d] \ar[r] & 1\\
    1 \ar[r] & \Ocal_n^\x \ar[r] 
      & \Dbb_n^\x/\lan pu \ran \ar[r]^{v} \ar[d] 
      & \lan \xi_F \ran/\lan pu \ran \ar[r] \ar[d] & 1 \\
    && 1 & 1 \\
    }
\]
in which the bottom horizontal sequence splits and the group $\lan \xi_F
\ran/\lan pu \ran \iso Gal(\Fbb_{p^n}/\Fbb_p)$ acts on $\Ocal_n^\x \iso
\Sbb_n$ by the above given action. It follows that
\[
    \Dbb_n^\x/\lan pu \ran
    \iso \Sbb_n \rtimes_F Gal(\Fbb_{p^n}/\Fbb_p)
    \iso \Gbb_n(u).
\]



\chapter*{Notations}
\addcontentsline{toc}{chapter}{Notations}

\subsection*{Integers}

\noindent
\begin{tabular*}{150mm}{rl@{\extracolsep{\fill}}l}
$n$
    & a positive integer
    & \rule{16mm}{0mm}\\
$p$
    & a prime
    & \\
$k$
    & the maximal integer such that $p^k$ divides $n$ 
    & p. \pageref{131}\\
$m$
    & the positive integer $\frac{n}{\phi(p^k)}$ when $p-1$ divides $n$ 
    & p. \pageref{131}\\
$n_\alpha$
    & the positive integer $\frac{n}{\phi(p^\alpha)}$ 
    for $0 \leq \alpha \leq k$ 
    & p. \pageref{131}\\
$(a;b)$
    & the greatest common divisor of $a$ and $b$
    & \\
$r_i$
    & the positive integer $|F_i/F_{i-1}|$ for $1 \leq i \leq 3$
    & \\
$[A:K]$
    & the dimension of $A$ over $K$
    & \\
$deg(A)$
    & the degree of $A$
    & p. \pageref{335}\\
$ind(A)$
    & the index of $A$
    & p. \pageref{336}\\
$exp(A)$
    & the exponent of $A$
    & p. \pageref{302}\\
$e(D/K)$ 
    & the ramification index of $D$ over $K$
    & p. \pageref{301}\\
$f(D/K)$ 
    & the inertial degree of $D$ over $K$
    & p. \pageref{301}\\
\end{tabular*}

\subsection*{Elements}

\noindent
\begin{tabular*}{150mm}{rl@{\extracolsep{\fill}}l}
$u$
    & a unit in $\Zbb_p^\x$
    & \rule{16mm}{0mm}\\
$S$
    & an element of $\Dbb_n^\x$ generating the Frobenius 
    such that $S^n=p$
    & p. \pageref{303}\\
$\zeta_i$
    & a $i$-th root of unity
    & \\
$x_i$
    & an element of $\Dbb_n^\x$ such that 
    $v(x_i) = (\prod_{k=1}^i r_i)^{-1}$ 
    and $x_i^{r_i} \in \tilde{F_{i-1}}$
    & p. \pageref{181}, \pageref{186}\\
$\epsilon_{p^\alpha}$
    & an element satisfying $(\zeta_{p^\alpha}-1)^{\phi(p^\alpha)} 
    = p \epsilon_{p^\alpha}$
    & p. \pageref{304}\\
$\pi_{\alpha}$
    & the element $\zeta_{p^\alpha}-1$
    & p. \pageref{305}\\
$\pi_K$
    & a uniformizing element of $K$
    & \\
\end{tabular*}

\subsection*{Sets}

\noindent
\begin{tabular*}{150mm}{rl@{\extracolsep{\fill}}l}
$C_G(H)$
    & the centralizer of $H$ in $G$
    & \rule{16mm}{0mm}\\
$N_G(H)$
    & the normalizer of $H$ in $G$
    & \\
$S/\sim_G$
    & the set of orbits with respect to the $G$-action on $S$
    & \\
$\Fcal(G)$
    & the set of all finite subgroups of $G$
    & p. \pageref{306}\\
$\tilde{\Fcal}_u(G)$
    & the set of all subgroups of $G$ containing $\lan pu \ran$\\
    & \q as a subgroup of finite index
    & p. \pageref{306}\\
$\tilde{\Fcal}_u(\Qbb_p(F_0), \tilde{F_0})$
    & as defined in
    & p. \pageref{307}\\
$\tilde{\Fcal}_u(\Qbb_p(F_0), \tilde{F_0}, r_1)$
    & as defined in
    & p. \pageref{308}\\
$\tilde{\Fcal}_u(C_{\Dbb_n^\x}(F_0), \tilde{F_1})$
    & as defined in
    & p. \pageref{309}\\
$\tilde{\Fcal}_u(C_{\Dbb_n^\x}(F_0), \tilde{F_1}, r_2)$
    & as defined in
    & p. \pageref{310}\\
$\tilde{\Fcal}_u(C_{\Dbb_n^\x}(F_0), \tilde{F_1}, L)$
    & as defined in
    & p. \pageref{311}\\
$\tilde{\Fcal}_u(N_{\Dbb_n^\x}(F_0), \tilde{F_2})$
    & as defined in
    & p. \pageref{312}\\
$\tilde{\Fcal}_u(N_{\Dbb_n^\x}(F_0), \tilde{F_2}, W)$
    & as defined in
    & p. \pageref{313}\\
\end{tabular*}

\pagebreak
\subsection*{Groups}

\noindent
\begin{tabular*}{150mm}{rl@{\extracolsep{\fill}}l}
$\Sbb_n$
    & the $n$-th (classical) Morava stabilizer group
    & p. \pageref{314}\\
$S_n$
    & the $p$-Sylow subgroup of $\Sbb_n$
    & p. \pageref{315}\\
$\Gbb_n(u)$
    & the $n$-th extended Morava stabilizer group associated to $u$
    & p. \pageref{316}\\
$\mu(R)$
    & the roots of unity in $R$
    & \\
$\mu_i(R)$
    & the $i$-th roots of unity in $R$
    & \\
$F_i$
    & the $i$-th subgroup of $\Gbb_n(u)$ 
    associated to a finite $F \subset \Gbb_n(u)$
    & p. \pageref{317}\\
$\tilde{F_i}$
    & the $i$-th subgroup of $\Dbb_n^\x$, 
    associated to a finite $F \subset \Gbb_n(u)$
    & p. \pageref{318}\\
$Z(G)$
    & the center of $G$
    & \\
$\Zbb\lan x \ran$
    & the infinite cyclic group generated by $x$
    & \\
$C_n$
    & the cyclic group of order $n$
    & \rule{10mm}{0mm}\\
$C_n \ast C_m$
    & the kernel of the $m$-th power map on $C_n$
    & \\
$Q_{2^n}$
    & the (generalized) quaternionic group order $2^n$
    & p. \pageref{019}\\
$T_{24}$
    & the binary tetrahedral group of order $24$
    & p. \pageref{319}\\
$D_{8}$
    & the dihedral group of order $8$
    & p. \pageref{320}\\
$SD_{16}$
    & the semidihedral group of order $16$
    & p. \pageref{321}\\
$O_{48}$
    & the binary octahedral group of order $48$
    & p. \pageref{322}\\
$Br(K)$
    & the Brauer group of $K$
    & p. \pageref{323}\\
$Br(L/K)$
    & the relative Brauer group of $L$ over $K$
    & p. \pageref{323}\\
\end{tabular*}

\subsection*{Rings, fields}

\noindent
\begin{tabular*}{150mm}{rl@{\extracolsep{\fill}}l}
$\Fbb_{p^n}$
    & the finite field with $p^n$ elements
    & \rule{10mm}{0mm}\\
$\Qbb_p$
    & the field of $p$-adic numbers
    & \\
$\Zbb_p$
    & the ring of $p$-adic integers
    & \\
$\Wbb(R)$
    & the ring of Witt vectors over $R$
    & \\
$D(K, r/n)$ 
    & the $K$-central division algebra of invariant $\frac{r}{n}$
    & p. \pageref{324}\\
$\Dbb_n$ 
    & the $\Qbb_p$-central division algebra of invariant $\frac{1}{n}$
    & p. \pageref{324}\\
$\Ocal_n$
    & the maximal order of $\Dbb_n$
    & \\
$\Ocal_K$
    & the ring of integers of the field $K$
    & \\
$U_i(\Ocal_K^\x)$
    & the $i$-th filtration group 
    $ \{ x \in \Ocal_K^\x\ |\ v_K(x-1) \geq i \}$ 
    & p. \pageref{340} \\
$R(G)$
    & the $R$-algebra generated by $G$
    & p. \pageref{170} \\
$R[G]$
    & the group ring generated by $G$
    & \\
$R_\alpha$
    & the ring $\Zbb_p[\zeta_{p^\alpha}]$
    & p. \pageref{305}\\
\end{tabular*}

\subsection*{Maps}

\noindent
\begin{tabular*}{150mm}{rl@{\extracolsep{\fill}}l}
$v$
    & the valuation $p \mto 1$ relative to $\Qbb_p$
    & \rule{10mm}{0mm}\\
$v_K$
    & the valuation $\pi_K \mto 1$ relative to the field $K$
    & \\
$v_D$
    & the valuation $\pi_{Z(D)} \mto 1$ relative to 
    the division algebra $D$
    & p. \pageref{325}\\
$\phi$
    & Euler's totient function
    & p. \pageref{337}\\
$N_{L/K}$
    & the norm of the extension $L/K$
    & \\
$Tr_{L/K}$
    & the trace of the extension $L/K$
    & \\
$N_{G}$
    & the norm relative to the Galois group $G$
    & \\
$Tr_{G}$
    & the trace relative to the Galois group $G$
    & \\
\end{tabular*}


\renewcommand{\bibname}{\LARGE\bf Bibliography \vspace{-4mm}}


\end{document}